\UseRawInputEncoding
\documentclass[leqno,11pt]{amsart}

\usepackage{amssymb, amsmath}
\usepackage{amssymb,amsfonts}
\usepackage{amsmath,latexsym}
\usepackage{pdfsync}
\usepackage{bbm}
\usepackage{xcolor}

\def\inte#1{
\displaystyle\mathop{#1\kern0pt}^\circ }

\let\pa=\partial

\def\cA{\mathcal{A}}

\let\grad\nabla

\def\virgp{\raise 2pt\hbox{,}}
\def\cdotpv{\raise 2pt\hbox{;}}

\def\eqdefa{\buildrel\hbox{\footnotesize def}\over =}

\def\C{\mathop{\bf C\kern 0pt}\nolimits}
\def\DD{\mathop{\bf D\kern 0pt}\nolimits}
\def\K{\mathop{\bf K\kern 0pt}\nolimits}
\def\N{\mathop{\bf N\kern 0pt}\nolimits}
\def\Q{\mathop{\bf Q\kern 0pt}\nolimits}
\def\R{\mathop{\bf R\kern 0pt}\nolimits}
\def\SS{\mathop{\bf S\kern 0pt}\nolimits}
\def\ZZ{\mathop{\bf Z\kern 0pt}\nolimits}
\def\TT{\mathop{\bf T\kern 0pt}\nolimits}

\def\dive{\mathop{\rm div}\nolimits}

\def\na{\nabla}

\newcommand{\beq}{\begin{equation}}
\newcommand{\eeq}{\end{equation}}
\newcommand{\ben}{\begin{eqnarray}}
\newcommand{\een}{\end{eqnarray}}
\newcommand{\beno}{\begin{eqnarray*}}
\newcommand{\eeno}{\end{eqnarray*}}

\newtheorem{thm}{Theorem}[section]
\newtheorem{lem}{Lemma}[section]
\newtheorem{rmk}{Remark}[section]
\newtheorem{col}{Corollary}[section]
\newtheorem{prop}{Proposition}[section]
\renewcommand{\theequation}{\thesection.\arabic{equation}}

\begin{document}

\title{Global stability of the compressible viscous surface waves in an infinite layer}

\author[G. Gui]{Guilong Gui}
\address[G. Gui]{School of Mathematics and Computational Science, Xiangtan University, Xiangtan 411105, China.} \email{glgui@amss.ac.cn}
\author[Z. Zhang]{Zhifei Zhang} \address[Z. Zhang]{School of Mathematical Science, Peking University, Beijing 100871, China.} \email{zfzhang@math.pku.edu.cn}

\begin{abstract}

We investigate in this paper the global stability of the compressible viscous surface waves in the absence of surface tension effect with a steady-state violating Rayleigh-Taylor instability and the reference domain being the horizontal infinite layer. The fluid dynamics are governed by the 3-D gravity-driven isentropic compressible Navier-Stokes equations. We develop a mathematical approach to establish global well-posedness of free boundary problems of the multi-dimensional compressible Navier-Stokes system based on the Lagrangian framework, which requires no nonlinear compatibility conditions on the initial data.

\end{abstract}

\date{}

\maketitle

\vskip 0.2cm


\vskip 0.2cm


\renewcommand{\theequation}{\thesection.\arabic{equation}}
\setcounter{equation}{0}

\section{Introduction}

\subsection{Formulation in Eulerian Coordinates}
We consider in this paper the global existence of time-dependent flows of an
viscous isentropic compressible fluid in a moving domain $\Omega(t)$
with an upper free surface $\Sigma_{F}(t)$ and a fixed bottom $\Sigma_B$
\begin{equation}\label{VFS-eqns-1}
\begin{cases}
&\partial_t\rho+\nabla\cdot(\rho\,u)=0 \quad\mbox{in} \quad \Omega(t),\\
  &\rho(\partial_t u + (u\cdot \grad) u )+ \nabla\, p(\rho)- \nabla\cdot\mathbb{S}(u)=-g\,\rho\,e_1 \quad\mbox{in} \quad \Omega(t),\\
 & (p(\rho)\,\mathbb{I}-\mathbb{S}(u))n(t)=p_{\mbox{\tiny atm}}n(t)\quad \mbox{on} \quad \Sigma_{F}(t),\\
 &  \mathcal{V}(\Sigma_{F}(t))=u\cdot n(t) \quad \mbox{on} \quad \Sigma_{F}(t),\\
 &u|_{\Sigma_B}=0,
\end{cases}
\end{equation}
where the positive scalar function $\rho$ is the density of the fluid, the velocity $u=(u^1, u^2, u^3)$, $u^i=u^i(x_1, x_2, x_3)$, $i, j=1, 2, 3$, and the constant $g > 0$ stands for the strength of gravity. we denote $n(t)$ the outward-pointing unit normal on $\Sigma_F(t)$, $\mathbb{I}$ the $3\times3$ identity matrix. The stress tensor $\mathbb{S}(u)$ is defined by $\mathbb{S}(u)=\varepsilon\mathbb{D}(u)+(\delta-\frac{2}{3}\varepsilon)(\nabla \cdot u)\mathbb{I}$, where the strain tensor $\mathbb{D}(u)=\nabla u+\nabla u^{T}$ and dynamic viscosity $\varepsilon$ and bulk viscosity $\delta$ are constants which satisfy the following relationship
 \begin{equation*}
 \varepsilon>0,\quad \delta\geq 0.
 \end{equation*}
  The deviatoric (trace-free) part of the strain tensor $\mathbb{D}(u)$ is then $\mathbb{D}^0(u)=\mathbb{D}(u)-\frac{2}{3}\dive\, u\, \mathbb{I}$. The viscous stress tensor in fluid is then given by $\mathbb{S}(u)=\varepsilon\,\mathbb{D}^0(u)+\delta \dive\, u\, \, \mathbb{I}$. The scalar function $p$ is the pressure which is a function of density $p=p(\rho)>0$, and the pressure function is assumed to be smooth, positive, and strictly increasing. We denote $ \mathcal{V}(\Sigma_{F}(t))$ the outer-normal velocity of the free surface $\Sigma_{F}(t)$. The tensor $(p\,\mathbb{I} - \mathbb{S}(u))$ is known as the viscous stress tensor. The first equation in \eqref{VFS-eqns-1} is the conservation of mass. The second equation in \eqref{VFS-eqns-1} is the momentum conservation law, where the gravity is the only external force, which points in the negative $x_1$ direction (as the vertical direction). The third equation in \eqref{VFS-eqns-1} means the fluid satisfies the kinetic boundary condition on the free boundary $\Sigma_F(t)$, where $p_{atm}$ stands for the atmospheric pressure, assumed to be constant. The kinematic boundary condition (the forth equation in \eqref{VFS-eqns-1}) states that the free boundary $\Sigma_F(t)$ is moving with speed equal to the normal component of the fluid velocity. The fifth equation in \eqref{VFS-eqns-1} implies that the fluid is no-slip, no-penetrated on the fixed bottom boundary. Here the effect of surface tension is neglected on the free surface.

\subsection{Steady-state}

Let the equilibrium domain $\Omega \subset \mathbb{R}^3$ be the horizontal infinite slab
\begin{equation}\label{def-domain-1}
\begin{split}
&\Omega=\{x=(x_1, x_h)|-\underline{b}<x_1<0,\quad x_h=(x_2, x_3)\in\mathbb{R}^2\}
\end{split}
\end{equation}
with the bottom $\Sigma_b=\{x_1=-\underline{b}\}$ and the top surface  $\Sigma_0=\{x_1=0\}$, where the positive constant $\underline{b}$ will be the depth of the fluid at infinity.

We seek a steady-state equilibrium solution to \eqref{VFS-eqns-1} with $\bar{u}=0$, $\bar{\rho}=\bar{\rho}(x_1)$ and the equilibrium domain $\Omega$ in \eqref{def-domain-1}. Then the system \eqref{VFS-eqns-1} reduces to the following ODE with the equilibrium density $\bar{\rho}=\bar{\rho}(x_1)$
\begin{equation}\label{steady-ODE-1}
\begin{cases}
&\frac{d p(\bar{\rho}(x_1))}{dx_1}=-g\,\bar{\rho}(x_1),\quad \forall\, x_1 \in (-\underline{b}, 0),\\
&p(\bar{\rho}(0))=p_{\text{atm}}.
\end{cases}
\end{equation}
Motivated by \cite{Guo-Tice-2010}, we may claim that the necessary and sufficient for the existence of an equilibrium to \eqref{steady-ODE-1} are as follows:
\begin{equation}\label{steady-ODE-2}
\begin{split}
&(1). \quad p_{\text{atm}} \in p(\mathbb{R}^+), \quad \text{which defines} \quad\bar{\rho}_1\eqdefa p^{-1}(p_{\text{atm}}),\\
&(2). \quad 0<\underline{b}<\frac{1}{g}\int_{\bar{\rho}_1}^{+\infty}\frac{p'(s)}{s}\,ds.
\end{split}
\end{equation}
In fact, since the pressure function $p=p(\cdot)$ is smooth, positive, and strictly increasing in $\mathbb{R}^{+}$, then second equation in \eqref{steady-ODE-1} holds if and only if $p_{\text{atm}} \in p(\mathbb{R}^+)$, which defines $\bar{\rho}_1= p^{-1}(p_{\text{atm}})$, that is, the first condition in \eqref{steady-ODE-2} holds.
On the other hand, we introduce the enthalpy function $h:\,(0, +\infty)\rightarrow \mathbb{R}$ given by
\begin{equation}\label{def-enth-1}
\begin{split}
 &h(z)\eqdefa\int_{\bar{\rho}_1}^z\frac{p'(s)}{s}\,ds,
 \end{split}
\end{equation}
which is smooth, strictly increasing, and positive on $[\bar{\rho}_1,\,+\infty)$ according to the assumptions on the pressure function $p$.
From \eqref{steady-ODE-1}, we get
\begin{equation}\label{steady-ODE-3}
\begin{cases}
 &\frac{d h(\bar{\rho}(x_1))}{dx_1}=-g,\quad \forall\, x_1 \in (-\underline{b}, 0),\\
 &\bar{\rho}(0)=\bar{\rho}_1.
 \end{cases}
\end{equation}
Solve this ODE to find $\bar{\rho}(x_1)=h^{-1}(-g\,x_1)$, which gives a well-defined, smooth, and decreasing function $\bar{\rho}: [-\underline{b}, \,0]\rightarrow [\bar{\rho}_1, \,+\infty)$ if and only if
\begin{equation}\label{steady-ODE-4}
 g\,\underline{b} \in h((\bar{\rho}_1, +\infty)),
\end{equation}
that is, the second condition in \eqref{steady-ODE-2} holds.

Throughout the rest of the paper we will assume that the conditions in \eqref{steady-ODE-2}
are satisfied for a fixed $\underline{b}>0$, which then uniquely determine the equilibrium
density $\bar{\rho}(x_1)$. We consider only the case that the density is away from the vacuum:
\begin{equation}\label{steady-ODE-5}
\inf_{x_1 \in [-\underline{b},\, 0]}\bar{\rho}(x_1)>0.
\end{equation}
Furthermore, we assume that
\begin{equation}\label{steady-ODE-6}
\inf_{x_1 \in [-\underline{b}, \, 0]}(-\bar{\rho}'(x_1))>0,
\end{equation}
which excludes the Rayleigh-Taylor gravitational instability.

Under the assumptions \eqref{steady-ODE-2}, \eqref{steady-ODE-5}, and \eqref{steady-ODE-6}, we focus in the paper on the global stability issue of the steady-state $u=0$, $\rho=\bar{\rho}(x_1)$, and $\Omega(t)=\Omega$ defined on \eqref{def-domain-1}.

The problem can be equivalently stated as follows. Given an initial domain $\Omega_0 \subset \mathbb{R}^3$ bounded by a bottom surface $\Sigma_{B}$, and a top surface $\Sigma_F(0)$, as well as an initial velocity field $u_0$, where the upper boundary does not touch the bottom, we wish to find for each $t\in [0, T]$, a domain $\Omega(t)$ initially near the equilibrium domain $\Omega$, $(u(t, \cdot), \,\rho(t, \cdot))$ defined on $\Omega(t)$ near the steady-state $(u,\, \rho)=(0, \,\bar{\rho}(x_1))$ defined on $\Omega$, and a transformation
$\bar{\eta}(t, \cdot):\,\Omega_0\rightarrow \mathbb{R}^3$ so that
\begin{equation}\label{VFS-eqns-12}
\begin{cases}
&\Omega(t)=\bar{\eta}(t, \Omega_0), \quad\bar{\eta}(t, \Sigma_B)=\Sigma_B,\\
&\partial_t\bar{\eta}=u\circ\bar{\eta},\\
&\partial_t\rho+\nabla\cdot(\rho\,u)=0 \quad\mbox{in} \quad \Omega(t),\\
  &\rho(\partial_t u + (u\cdot \grad) u )+ \nabla\, p(\rho)- \nabla\cdot\mathbb{S}(u)=-g\,\rho\,\nabla\,x_1 \quad\mbox{in} \quad \Omega(t),\\
 & (p(\rho)\,\mathbb{I}-\mathbb{S}(u))n(t)=p_{\mbox{\tiny atm}}n(t)\quad \mbox{on} \quad \Sigma_{F}(t),\\
 &u|_{\Sigma_B}=0,\\
 &u|_{t=0}=u_0,\quad \rho|_{t=0}={\rho}_0,\quad  \bar{\eta}|_{t=0}=x.
\end{cases}
\end{equation}

 We assume that the initial domain $\Omega_0$ is the image of $\Omega$ under a diffeomorphism $\overline{\sigma}: \Omega \rightarrow  \Omega_0$, where $\overline{\sigma}(\Sigma_b)=\Sigma_B$, $\overline{\sigma}(\Sigma_0)=\Sigma_F(0)$, $\overline{\sigma}$ is of the form $\overline{\sigma}(x)=x+\xi_0(x)$, and $\xi_0$ satisfies $\xi_0^1 \in H^s(\Sigma_0)$, and $\Lambda_h^{s-1} \nabla\xi_0\in H^1(\Omega)$ with $s>2$,  $\xi_0^1$ stands for the vertical component of the vector $\xi_0$, $\nabla=(\partial_1, \partial_2, \partial_3)^T$, and $\nabla_h=(\partial_2, \partial_3)^T$.

\subsection{Known results}

The free boundary problems of the Navier-Stokes equations have attracted much attention. The quite direct approach used to solve these problems  is Lagrangian. In Lagrangian coordinates, the geometry of the domain is encoded in the flow map $\bar{\eta}: (0, T)\times\Omega_0 \rightarrow\Omega(t)$ satisfying $\partial_t\bar{\eta}(t, x) = v(t, x)$ which gives the trajectory of a particle located at $x\in \Omega_0=\Omega(0)$ at $t = 0$, where $v = u \circ \bar{\eta}$ is the Lagrangian velocity field in the fixed domain $\Omega_0$ with $u$ being the Eulerian velocity field in the moving domain $\Omega(t)$.

In Lagrangian coordinates, Solonnikov \cite{Solonnikov-1977} proved the local well-posedness of the viscous surface problem in H\"{o}lder spaces for the incompressible fluid motion in a bounded domain whose entire boundary is a free surface. While in a horizontal infinite domain, Beale \cite{Beale-1981} obtained the local well-posedness in $L^2$-based space-time Sobolev spaces under the assumption with necessary compatibility condition, which was extended to $L^p$-based  space-time Sobolev spaces by Abels \cite{Abels-2005}. However, Beale gave the negative answer to the global well-posedness in \cite{Beale-1981}. As a matter of fact, he proved a non-decay theorem which showed that a {\it reasonable} extension to small-data global well-posedness with decay of the free surface fails. In order to get the global well-posedness of the incompressible viscous surface problem, Beale \cite{Beale-1984} turned to introduce another method---the flattening transformation to successful solve the viscous surface problem with surface tension in the space-time Sobolev spaces. The incompressible viscous surface problem (without surface tension) was studied by Sylvester \cite{Sylvester-1990} and Tani-Tanaka \cite{Tani-Tanaka-1995} in the Beale-Solonnikov functional framework, where higher regularity and more compatibility conditions were required. Beale and Nishida studied the decay properties of solutions to incompressible viscous surface waves with surface tension in \cite{Beale-Nishida-1985}, and then Hataya \cite{Hataya-2011} gave a complete proof to their decay estimates provided $\xi_0^1 \in L^1(\Sigma_0)$ and initial data is small. This result also suggests that, even if $\xi_0^1 \in L^1(\Sigma_0)$, for the solution $u$ to the viscous surface waves without surface tension, both $\|u\|_{L^{1}(\mathbb{R}^+;H^{2}(\Omega))}$ and  $\|\xi^1\|_{L^{2}(\mathbb{R}^+; L^{2}(\Sigma_0))}$ may still  blow-up.

By using the flattening transformation, Guo and Tice \cite{Guo-Tice-1, Guo-Tice-2} employed the geometric structure in the Eulerian coordinates to study the local well-posedness of the incompressible viscous surface waves in Sobolev spaces for small initial data, and then introduced the two-tier energy method to get the algebraic decay rate of the solutions, which leads to the construction of global-in-time solutions.  Wu \cite{Wu-2014} extended the local well-posedness result in \cite{Guo-Tice-2} from small data to general data, and
 Wang \cite{Wang-2020} studied the anisotropic decay of the global solution. We mention that, the two-tier energy method in \cite{Guo-Tice-1, Guo-Tice-2} couples the decay of low-order energy and the boundedness of high-order energy, where higher regularity and more compatibility conditions, as well as the negative horizontal regularity of the solutions were needed, which also avoids to contradict Beale's non-decay theorem. Ren, Xiang and Zhang \cite{Ren-X-Zhang-2019} proved the local well-posedness of the viscous surface wave equation in low regularity Sobolev spaces in Eulerian coordinates, where only the first-order time-derivative of the velocity is involved in the compatibility condition, but the global existence of the solution is left open.

Recently, Masmoudi and Rousset \cite{Masm-Rou-2017} studied the inviscid limit of the free boundary incompressible Navier-Stokes equations, where they proved the existence of solutions on a uniform time interval based on Sobolev conormal spaces, in which the minimal normal regularity of the velocity was required, and then justified the inviscid limit.

More recently, the first author of this paper developed a mathematical approach to establish global well-posedness of the incompressible viscous surface waves based on the Lagrangian framework \cite{Gui-2020}, where no nonlinear compatibility conditions on the initial data were required, which gives a positive answer to the Beale's question \cite{Beale-1981} about the existence of the global strong solutions to the free boundary incompressible system in Lagrangian coordinates.

The free boundary problems corresponding to a single horizontally periodic layer of compressible viscous fluid with surface tension and without the gravitation have been also studied by several authors. Jin \cite{Jin-2003} and Jin-Padula \cite{Jin-Padula-2002} produced global-in-time solutions using Lagrangian coordinates, and Tanaka and Tani produced global solutions with temperature dependence. Zadrzynska \cite{Zadrzynska2001, Zajaczkowski1999} proved the global well-posedness of the viscous compressible Fluid bounded by a free boundary. Guo and Tice \cite{Guo-Tice-2010} studied the linear theory of the problem for two layers of compressible viscous fluids.

The free boundary problem allows for the development of the classical Rayleigh-Taylor instability \cite{Rayleigh-1883, Taylor-1950}, at least when the equilibrium has a heavier fluid on top and a lighter one below and there is a downward gravitational force.

There are also some global results for compressible viscous surface-internal wave problems with gravitation and without the surface tension by using the flattening transformation. For instance, in a horizontal periodic domain with finite depth, Wang, Tice, and Kim \cite{WTK-2016, WTK-2014} obtained the local and global well-posedness and decay for two layers compressible fluid, and Huang and Luo \cite{HLuo-2021} establish the global well-posedness of the single layer of compressible viscous heat-conducting surface wave without surface tension provided that the initial data are close to a nontrivial equilibrium state. Tani and Tanaka \cite{TaniT-2003} studied the solvability of the compressible heat-conducting surface wave problem with the surface tension in Lagrangian coordinates.

\subsection{Our goal}

There are still some problems left to be solved in the study of global well-posedness of the compressible viscous surface waves without surface tension.

First, it is of interest to know whether we can get the global strong solution of this system in a horizontally infinite layer, in particular, via the Lagrangian approach.

On the other hand, in previous works, the global well-posedness was established for the initial data which has
high normal regularity and some compatibility conditions in terms of the time-derivatives of the velocity on the initial data are needed. Usually, it is not easy to verify the valid of the compatibility condition in terms of the time-derivatives of the velocity, which is essentially some type of the nonlinear compatibility conditions from the momentum equation in \eqref{VFS-eqns-1}. In the present case, a natural and important question is whether a corresponding well-posedness result can be obtained with low normal regularity and without any compatibility conditions in terms of the time-derivatives of the velocity (nonlinear compatibility conditions) on the initial data.

The aim of this paper is to give positive answers to these two problems. We will employ the Lagrangian approach to study the global well-posedness of the compressible viscous surface waves without surface tension in the infinite layer, where the assumptions of low normal regularity and no any compatibility conditions in terms of the time-derivatives of the velocity on the initial data are required.

\subsection{Formulation of the system in Lagrangian Coordinates}

Let us now, in more detail, introduce the Lagrangian coordinates in which the free boundary becomes fixed.

\subsubsection{The flow map}

Let $\eta\eqdefa \bar{\eta} \circ \bar{\sigma}$ be a position of the fluid particle $x$ in the equilibrium domain $\Omega$ at time $t$ so that
\begin{equation*}
\begin{cases}
&\frac{d}{dt}\eta(t, x)=u(t, \eta(t, x)), \quad t>0, \, x\in \Omega,\\
&\eta|_{t=0}=x+\xi_0(x), \quad x\in \Omega,
\end{cases}
\end{equation*}
then the displacement $\xi(t, x)\eqdefa \eta(t, x)-x$ satisfies
\begin{equation*}
\begin{cases}
&\frac{d}{dt}\xi(t, x)=u(t, x+\xi(t, x)),\\
&\xi|_{t=0}=\xi_0.
\end{cases}
\end{equation*}
We define Lagrangian quantities the velocity $v$ and the pressure $q$ in fluid as (where $x=(x_1, x_2, x_3)^T\in \Omega$):
\begin{equation*}
\begin{split}
&v(t, x)\eqdefa u(t, \eta(t, x)),\quad q(t, x)\eqdefa p(t, \eta(t, x)).
\end{split}
\end{equation*}
Denote the Jacobian of the flow map $\eta$ by
$J \eqdefa \mbox{det}(D\eta)$.
Define $\mathcal{A} \eqdefa  (D\eta)^{-T}$, then according to definitions of the flow map $\eta$ and the displacement $\xi$, we may get the identities:
\begin{equation*}
\mathcal{A}_{i}^k \partial_{k} \eta^j=\mathcal{A}_{k}^j \partial_{i} \eta^k=\delta_i^j, \quad \partial_k(J\mathcal{A}_{i}^k)=0,\quad\partial_{i} \eta^j=\delta_i^j+\partial_{i} \xi^j, \quad \mathcal{A}_{i}^j=\delta_i^j-\mathcal{A}_{i}^k \partial_k\xi^j.
\end{equation*}

Simple computation implies that
$ J=1+\grad \cdot\xi+\mathcal{B}_{00}+\mathcal{B}_{000}$
with
\begin{equation*}\label{J-expression-2a}
  \begin{split}
&\mathcal{B}_{00}:=\partial_1\xi^1 \nabla_h\cdot \xi^h-\partial_1\xi^h\cdot\nabla_h\xi^1+\nabla_h^{\perp}\xi^2\cdot\nabla_h\xi^3,\\
&\mathcal{B}_{000}:=\partial_1\xi^1 \nabla_h^{\perp}\,\xi^2\cdot \nabla_h\,\xi^3+\partial_1\xi^2\nabla_h^{\perp}\,\xi^3\cdot \nabla_h\,\xi^1+\partial_1\xi^3 \nabla_h^{\perp}\,\xi^1\cdot \nabla_h\,\xi^2
  \end{split}
\end{equation*}
with $\xi^h\eqdefa (\xi^2, \xi^3)^T$,
$\nabla_h\eqdefa (\partial_2,\, \partial_3)^T,\quad \nabla_h^{\perp}\eqdefa (-\partial_3,\, \partial_2)^T$.

Usually, we set $a_{ij}\eqdefa\,J\,\mathcal{A}_i^j$ and $\widetilde{\mathcal{A}}:=\mathcal{A}-\mathbb{I}$.

\subsubsection{Derivatives of $J$ and $\mathcal{A}$ in Lagrangian coordinates}

Next, we give some useful equations which we often use in what follows.

Since $\mathcal{A}(D\eta)^T=I$, differentiating it with respect to $t$ and $x$ once yields
\begin{equation}\label{identity-Lagrangian-1}
\begin{split}
&\partial_t \mathcal{A}_{i}^j=-\mathcal{A}_{k}^j\mathcal{A}_{i}^{m} \partial_{m}v^k,\,\partial_{s} \mathcal{A}_{i}^j=-\mathcal{A}_{k}^j\mathcal{A}_{i}^{m} \partial_{m}\partial_{s}\xi^k,
\end{split}
\end{equation}
where we used the fact $\partial_t\eta=v$ in the first equation in \eqref{identity-Lagrangian-1}.
Whence differentiate the Jacobian determinant $J$, we get
\begin{equation*}
\begin{split}
&\partial_t J =J \mathcal{A}_{i}^j \partial_j v^i, \quad \partial_{k} J =J \mathcal{A}_{i}^j \partial_j\partial_{k} \xi^i.
\end{split}
\end{equation*}
Moreover, it is easy to verify the following Piola identity:
\begin{equation*}
\begin{split}
&\partial_j (J \mathcal{A}_{i}^j) =0 \quad \forall \,i = 1, 2, 3.
\end{split}
\end{equation*}
Here and in what follows, the subscript notation for vectors and tensors as well as the Einstein summation convention has been adopted unless otherwise specified.

\subsubsection{Navier-Stokes equations in Lagrangian coordinates}

Under Lagrangian coordinates, we may introduce the differential operators with
their actions given by $(\nabla_{\mathcal{A}}f)_i=\mathcal{A}_i^j  \partial_jf$, $ \mathbb{D}_{\mathcal{A}} (v)=\nabla_{\mathcal{A}} v+(\nabla_{\mathcal{A}} v)^T$, $\Delta_{\mathcal{A}} f=\nabla_{\mathcal{A}}\cdot \nabla_{\mathcal{A}} f$, so the Lagrangian version of the system \eqref{VFS-eqns-1} can be written on the fixed reference domain $\Omega$ as
\begin{equation}\label{eqns-pert-1aa}
\begin{cases}
&\partial_t\xi=v\quad \mbox{in}\quad  \Omega,\\
&\partial_tf+f\nabla_{\mathcal{A}} \cdot v=0,\quad \mbox{in}\quad  \Omega,\\
  &f\partial_t v + \nabla_{\mathcal{A}}\,p(f)-\grad_{\mathcal{A}} \cdot \mathbb{S}_{\mathcal{A}}(v)=-g\,f\,\nabla_{\mathcal{A}}\,\eta^1\quad \mbox{in}\quad  \Omega,\\
     & \bigg(p(f)\,\mathbb{I}-\mathbb{S}_{\mathcal{A}}(v)\bigg)\mathcal{N}=p_{\text{atm}}\mathcal{N}\quad \mbox{on} \quad \Sigma_0,\\
       & v|_{\Sigma_{b}}=0,\\
 &\xi|_{t=0}=\xi_0, \,   v|_{t=0}=v_0,
\end{cases}
\end{equation}
where $n_0=(1, 0, 0)^T$ is the outward-pointing unit normal vector on the interface $\Sigma_{0}$, $\mathcal{N}:= J\mathcal{A}\,n_0$ stands for the outward-pointing normal vector on the moving interface $\Sigma_F(t)$.

One may readily check from the definition of $J$ that
$\pa_t J=\na_{J\cA}\cdot v$,
which together with the equation of $f$ in \eqref{eqns-pert-1aa} yields $\pa_t(fJ)=0$.
Hence, we find
\begin{align}\label{mass-cons-1}
Jf(t,x)=(Jf)(0,x)=\det (D\eta_0)\rho_0(\eta_0),
\end{align}
where $\rho_0$ is a given initial density function. We are interested in the initial density $\rho_0$ satisfying
\begin{equation}\label{assum-init-density-1}
\begin{split}
&\rho_0(\eta_0)\det (D\eta_0)=\bar{\rho}(x_1)\quad \text{in} \quad \Omega.
\end{split}
\end{equation}
Thus, it follows from \eqref{mass-cons-1} that
\begin{equation*}
\begin{split}
J\,f=\bar{\rho}(x_1),
\end{split}
\end{equation*}
which implies that
\begin{equation*}
\begin{split}
   f= J^{-1}\,\bar{\rho}(x_1),\quad p(f)=p( J^{-1}\,\bar{\rho}).
\end{split}
\end{equation*}
\begin{rmk}\label{rmk-mass-cons}
For any smooth, bounded subdomain $\mathcal{O}$ of $\Omega$, we know that $\eta_0(\mathcal{O})$ is a subdomain of $\Omega(0)$ if $\eta_0$ is a diffeomorphism from $\Omega$ to $\Omega(0)$. Hence, by using change of variables, we get
\begin{equation}\label{mass-cons-2}
\begin{split}
 \int_{\eta_0(\mathcal{O})} \rho_0(y) \,dy= \int_{\mathcal{O}} \rho_0(\eta_0)\,\det (D\eta_0) \,dx.
\end{split}
\end{equation}
Therefore, the assumption \eqref{assum-init-density-1} is equivalent to the mass conservation law \cite{DM-1990}
\begin{equation}\label{mass-cons-3}
\begin{split}
 \int_{\eta_0(\mathcal{O})} \rho_0(y) \,dy= \int_{\mathcal{O}} \bar{\rho}\,dx \quad \forall \,\, \mathcal{O} \subset \Omega.
\end{split}
\end{equation}
\end{rmk}

Thanks to the definition of the enthalpy function \eqref{def-enth-1}, we get
\begin{equation*}
\begin{split}
  \nabla_{\mathcal{A}}\,p(f)+g\,f\nabla_{\mathcal{A}}\,\eta^1=f\nabla_{\mathcal{A}}\bigg(h(f)-h(\bar{\rho}(x_1))+g\,\xi^1\bigg).
 \end{split}
\end{equation*}
Define
\begin{equation}\label{def-q-Q-1}
  q \eqdefa h(f)-h(\bar{\rho}(x_1))+g\,\xi^1, \quad  {Q} \eqdefa h(f)-h(\bar{\rho}(x_1)),
\end{equation}
then we have
\begin{equation*}\label{pressure-exp-enth-1}
\nabla_{\mathcal{A}}\,p(f)+g\,f\nabla_{\mathcal{A}}\,\eta^1 =f\nabla_{\mathcal{A}}q,\quad {Q} =q-g\,\xi^1,
\end{equation*}
which implies that the momentum equations in \eqref{eqns-pert-1aa} can be read as
\begin{equation*}
\bar{\rho}(x_1)\partial_t v + \bar{\rho}(x_1)\nabla_{\mathcal{A}}\,q-\grad_{J\mathcal{A}} \cdot \mathbb{S}_{\mathcal{A}}(v)=0\quad \mbox{in}\quad  \Omega.
\end{equation*}
By the definition of $q$ in \eqref{def-q-Q-1}, we have
\begin{equation*}
\partial_tq = \partial_th(f)+g\,v^1,
\end{equation*}
which along with the equation of $f$ in \eqref{eqns-pert-1aa} and the definition of the enthalpy function \eqref{def-enth-1} implies
\begin{equation}\label{eqns-q-form-1}
\begin{split}
  & \partial_tq = -p'(f) \nabla_{\mathcal{A}}\cdot v  +g\,v^1,\\
   \end{split}
\end{equation}
and then
\begin{equation}\label{eqns-q-form-2}
\begin{split}
  &  \nabla_{\mathcal{A}}\cdot v =- (p'(f))^{-1}\partial_t {Q}.
   \end{split}
\end{equation}
From \eqref{eqns-q-form-1}, we get
\begin{equation*}
\begin{split}
 & \partial_tq =-p'(\bar{\rho}(x_1)) \nabla_{\mathcal{A}}\cdot v +(p'(\bar{\rho}(x_1))-p'(f)) \nabla_{\mathcal{A}}\cdot v +g\,v^1,
   \end{split}
\end{equation*}
which leads to
\begin{equation}\label{eqns-q-form-3}
\begin{split}
\nabla \cdot (\bar{\rho}(x_1)\,v)&=g^{-1}\, \bar{\rho}'(x_1) \partial_tq+B_{1, i}^{h, j}\partial_jv^i
   \end{split}
\end{equation}
with
\begin{equation*}
\begin{split}
B_{1, i}^{h, j}\partial_jv^i\eqdefa\, g^{-1}\bar{\rho}'(x_1)   (p'(f)-p'(\bar{\rho}(x_1)) \nabla_{\mathcal{A}}\cdot v -\bar{\rho}(x_1)\nabla_{\widetilde{\mathcal{A}}}\cdot \,v.
   \end{split}
\end{equation*}

On the other hand, due to the classical inverse function theorem and $h'(z)=\frac{p'(z)}{z}>0$ ($\forall \,\,z>0$), there is a unique inverse function $\varphi$ of the equation ${Q}= h(f)-h(\bar{\rho}(x_1))$ so that
\begin{equation*}
\begin{split}
  &f=\varphi({Q})=h^{-1}({Q}+h(\bar{\rho}(x_1))),\quad \varphi(0)=\bar{\rho}(x_1),
   \end{split}
\end{equation*}
and
\begin{equation*}
\begin{split}
    &\varphi'(y)=\frac{1}{h'(\varphi(y))}=\frac{\varphi(y)}{p'(\varphi(y))},\quad \varphi'(0)=\frac{\varphi(0)}{p'(\varphi(0))}=\frac{\bar{\rho}(x_1)}{p'(\bar{\rho}(x_1))}.
   \end{split}
\end{equation*}
Hence, we obtain
\begin{equation*}
\begin{split}
  &p(f)-p(\bar{\rho}(x_1))=(p\circ \varphi)({Q})-(p\circ \varphi)(0)=(p\circ \varphi)'(0){Q}+{Q}^2 \int_0^1(p\circ \varphi)''(s\,{Q})(1-s)\,ds,
 \end{split}
\end{equation*}
which follows that
\begin{equation*}
p(f)-p(\bar{\rho}(x_1))=\bar{\rho}(x_1){Q}+{Q}^2 \mathcal{R}(Q)
\end{equation*}
with
\begin{equation*}
\mathcal{R}(Q)\eqdefa\int_0^1(p\circ \varphi)''(s\,{Q})(1-s)\,ds.
\end{equation*}
Hence, on the boundary $\Sigma_0$, one has
\begin{equation*}
p(f)-p(\bar{\rho}(x_1))=\bar{\rho}(0){Q}+{Q}^2 \mathcal{R}(Q) \quad \text{on} \quad \Sigma_0.
\end{equation*}
Therefore, under the assumptions \eqref{assum-init-density-1} and $\Omega_0=\bar{\sigma}(\Omega)$, the system \eqref{VFS-eqns-12} can be equivalently read as
\begin{equation}\label{eqns-pert-1}
\begin{cases}
&\partial_t\xi=v\quad \mbox{in}\quad  \Omega,\\
  &\bar{\rho}(x_1)\partial_t v + \bar{\rho}(x_1)\nabla_{\mathcal{A}}\,q-\grad_{J\mathcal{A}} \cdot \mathbb{S}_{\mathcal{A}}(v)=0\quad \mbox{in}\quad  \Omega,\\
  & \nabla \cdot (\bar{\rho}(x_1)\,v)=g^{-1}\, \bar{\rho}'(x_1) \partial_tq+B_{1, i}^{h, j}\partial_jv^i\quad \mbox{in}\quad  \Omega,\\
     & \bigg((\bar{\rho}(0){Q}+{Q}^2 \mathcal{R}(Q))\,\mathbb{I}-\mathbb{S}_{\mathcal{A}}(v)\bigg)\mathcal{N}=0\quad \mbox{on} \quad \Sigma_0,\\
       & v|_{\Sigma_{b}}=0,\\
 &\xi|_{t=0}=\xi_0, \,   v|_{t=0}=v_0.
\end{cases}
\end{equation}

\subsection{Our results}\label{subsec-our-result}

In order to state our results, we must explain our notation for differential operators, Sobolev spaces and norms and describe the regularity assumptions which will be imposed.

We denote the operator $\mathcal{P}(\partial_h)$ the horizontal derivatives of some functions, here $\mathcal{P}(\partial_h)$ is a pseudo-differential operator with the symbol  $\mathcal{P}(\varsigma_h)$ depending only on the horizontal frequency $ \varsigma_h=(\varsigma_2,\, \varsigma_3)^T$, that is,
 \begin{equation*}
\begin{split}
 \mathcal{P}(\partial_h)f(x_1,x_h):=\mathcal{F}^{-1}_{\varsigma_h\rightarrow x_h}\bigg(\mathcal{P}(\varsigma_h)\mathcal{F}_{x_h\rightarrow \varsigma_h}(f(x_1,x_h))\bigg),
\end{split}
\end{equation*}
where $\mathcal{F}(f)$ (or $\mathcal{F}^{-1}(f)$) is the Fourier (or inverse Fourier) transform of a function $f$.
In particular, we denote $\dot{\Lambda}_h^{\sigma}$ (or $\Lambda_h^{\sigma}$) the homogeneous (or nonhomogeneous) horizontal differential operator with the symbol $|\varsigma_h|^{\sigma}$ (or $\langle\varsigma_h\rangle^{\sigma}$) respectively, where ${\sigma} \in \mathbb{R}$, $\langle\varsigma_h\rangle\eqdefa (1+|\varsigma_h|^2)^{1/2}$.

We take $H^k(\Omega)$ with $k \in \mathbb{N}$ and $H^{\sigma}(\Sigma_0)$ with ${\sigma}\in \mathbb{R}$ to be the usual Sobolev spaces, and the homogeneous space $\mathring{H}^1(\Omega)=\{f\in L^1_{\text{loc}}(\Omega):\,\nabla\,f \in L^2(\Omega)\}$ with the norm $\|f\|_{\mathring{H}^1(\Omega)}=\|\nabla\,f \|_{L^2(\Omega)}$. For $ f\in {}_{0}{H}^{k}(\Omega)$, it means $f \in {H}^{k}(\Omega)$ with zero bottom boundary value $f|_{\Sigma_b}=0$. For a functional space $Y$, when we write $f \in \Lambda_h^{-\sigma}Y$ (or $f \in \dot{\Lambda}_h^{-\sigma}Y$) we mean $\Lambda_h^{\sigma}\,f\in Y$ (or $\dot{\Lambda}_h^{\sigma}f \in Y$).

\subsubsection{Local well-posedness}

 Inspired by the works in \cite{Gui-WW-2019} and \cite{GuiL-2022}, we may get the following local well-posedness of the system \eqref{eqns-pert-1}, which proof is very similar to the one in \cite{GuiL-2022}, and  we omit it.
\begin{thm}[Local well-posedness]\label{thm-local}
Let $s>2$. Assume $(\xi_0,\,v_0)$ satisfies ${\Lambda_h^{s-1}}\nabla\,\xi_0 \in  H^1(\Omega)$, $\xi^1_0 \in H^{s}(\Sigma_0)$,  ${\Lambda_h^{s-1}}v_0\in {}_{0}{H}^{1}(\Omega)$. There exists a positive constant $\epsilon_0$ such that the viscous surface wave problem \eqref{eqns-pert-1} with initial data $(\xi_0, v_0)$ has a unique solution $(\xi,\,v,\,q)$ (depending continuously on the initial data) satisfying
\begin{equation*}
\begin{split}
& {\Lambda_h^{s-1}}\nabla\,\xi \in \mathcal{C}([0, T_0];  H^1(\Omega)), \,\, \xi^1 \in \mathcal{C}([0, T_0]; H^{s}(\Sigma_0)),\\
&{\Lambda_h^{s-1}}v\in(\mathcal{C}([0, T_0];  {}_{0}{H}^{1}(\Omega))\cap L^2([0, T_0]; {H}^2(\Omega)),\quad {\Lambda_h^{s-1}} q\in L^2([0, T_0];  {H}^1(\Omega))
\end{split}
\end{equation*}
with $T_0=\min\{1, (\|\Lambda_h^{s-1}\nabla\,\xi_0\|_{H^1(\Omega)}^2+\|\xi^1_0 \|_{H^{s}(\Sigma_0)}^2)^{-1}\}$ provided \begin{equation*}
\|\Lambda_h^{s-1}\nabla\,\xi_0\|_{H^1(\Omega)}^2+\|\xi^1_0 \|_{H^{s}(\Sigma_0)}^2+\|\Lambda_h^{s-1}v_0\|_{H^{1}(\Omega)}^2 <\epsilon_0,
\end{equation*}
 and, for any $T\in (0, T_0]$, the solution satisfies the estimate
\begin{equation}\label{def-energy-local}
\begin{split}
&\sup_{t \in[0, T]}(\|\Lambda_h^{s-1}\nabla\,\xi(t)\|_{H^1(\Omega)}^2+\|\xi^1(t) \|_{H^{N}(\Sigma_0)}^2+\|\Lambda_h^{s-1}v(t)\|_{H^{1}(\Omega)}^2 )\\
&\qquad\qquad \qquad \qquad \qquad \qquad \qquad +\|\Lambda_h^{s-1}(\nabla\,v,\,q)\|_{L^2([0, T_0]; {H}^1(\Omega))}^2 \\
&\leq C\bigg(\|\Lambda_h^{s-2}\nabla\,\xi_0\|_{H^1(\Omega)}^2+\|\xi^1_0 \|_{H^{s-1}(\Sigma_0)}^2+\|\Lambda_h^{s-1}v_0\|_{H^{1}(\Omega)}^2\\
&\qquad \qquad \qquad \qquad \qquad \qquad +T\,(\|\Lambda_h^{s-1}\nabla\,\xi_0\|_{H^1(\Omega)}^2+\|\xi^1_0 \|_{H^{s}(\Sigma_0)}^2)\bigg),
\end{split}
\end{equation}
And if the maximal time $T^{\ast} $of the existence of the solution $(\xi,\,v,\,q)$ is finite: $T^{\ast} <+\infty $, then
\begin{equation}\label{loc-blow-crit}
\begin{split}
\lim_{t  \nearrow T^{\ast}}(\|\Lambda_h^{s-1}\nabla\,\xi(t)\|_{H^1(\Omega)}+\|\xi^1(t) \|_{H^{s}(\Sigma_0)}+\|\Lambda_h^{s-1}v(t)\|_{H^{1}(\Omega)}+\|J^{-1}\|_{L^\infty(\Omega)})=+\infty.
\end{split}
\end{equation}
Moreover, if in addition $(\xi_0,\,v_0) \in \mathfrak{F}_{s+\ell_0}$, then $(\xi,\,v) \in \mathcal{C}([0, T_0]; \mathfrak{F}_{s+\ell_0})$.
\end{thm}
Notice that the definition of $\mathfrak{F}_{s+\ell_0}$ in Theorem \ref{thm-local} is be defined in \eqref{def-spaceF-1} below.
\begin{rmk}\label{rmk-euler-1}
According to \eqref{def-energy-local}, we see that, if $\epsilon_0$ in Theorem \ref{thm-local} is small enough, there is a positive $\delta_0 \in (0, \frac{1}{2})$ such that  $1-\delta_0\leq \sup_{(t, x)\in[0,T]\times \Omega}J(t, x)\leq 2$, which implies that the flow-map $\eta(t, x)$ defines a diffeomorphism from the equilibrium domain $\Omega$ to the moving domain $\Omega(t)$ with the boundary $\Sigma_F(t)$. From this, together with the fact that $\eta_0$ is a diffeomorphism from the equilibrium domain $\Omega$ to the initial domain $\Omega(0)$, we deduce a diffeomorphism from the initial domain $\Omega(0)$ to the evolving domain $\Omega(t)$ for any $t \in [0, T]$. Denote the inverse of the flow map $\eta(t, x)$ by $\eta^{-1}(t, y)$ for $t \in [0, T]$ so that if $y = \eta(t, x)$ for $y \in \Omega(t)$ and $t \in [0, T]$, then $x = \eta^{-1}(t, y) \in \Omega$.
\end{rmk}

\subsubsection{Global well-posedness theorem}

Let us explain that, at the beginning of this subsection, how to define the energy and dissipation functions.

Let $s>2$ and $\lambda\in (0, 1)$ satisfy $s+\lambda>3$, $\ell_0 \in (\frac{1}{2},\,\min\{\frac{1}{2}+\frac{1}{4}\lambda,\,\frac{s+\lambda-2}{2}\}]$. Taking $\sigma_0:=2\ell_0-\lambda$, we have $\sigma_0+\lambda=2\ell_0$, $\sigma_0 \in (1-\lambda, \,\min\{1-\frac{1}{2}\lambda,\,s-2\}]$.

We define the decay instantaneous energy $\dot{\mathcal{E}}_{s}$ (with $s>2$) as
\begin{equation*}
\begin{split}
 \dot{\mathcal{E}}_{s}\eqdefa  &
 \|\dot{\Lambda}_h^{\sigma_0}\Lambda_h^{s-\sigma_0}(v,\, q)\|_{L^2(\Omega)}^2+ \|\dot{\Lambda}_h^{\sigma_0}\Lambda_h^{s-\sigma_0}\xi^1\|_{L^2(\Sigma_0)}^2 \\
&\qquad +\|\dot{\Lambda}_h^{\sigma_0-1}\Lambda_h^{s+\frac{1}{2}-\sigma_0} {Q}\|_{L^2(\Sigma_0)}^2 +\|\dot{\Lambda}_h^{\sigma_0}{\Lambda}_h^{s-1-\sigma_0}(\nabla\, v,\,\partial_1q)\|_{L^2(\Omega)}^2,
\end{split}
\end{equation*}
and the decay instantaneous dissipation $\dot{\mathcal{D}}_{s}$ is defined by
\begin{equation*}
\begin{split}
\dot{\mathcal{D}}_{s}\eqdefa & \|\dot{\Lambda}_h^{\sigma_0}\Lambda_h^{s-\sigma_0}\nabla\,v\|_{L^2(\Omega)}^2 +\|\dot{\Lambda}_h^{\sigma_0-1}\Lambda_h^{s+\frac{1}{2}-\sigma_0}({Q},\,\partial_t{Q})\|_{L^2(\Sigma_0)}^2
\\
&\qquad +\|\dot{\Lambda}_h^{\sigma_0+1}\Lambda_h^{s-2-\sigma_0}\xi^1\|_{H^\frac{1}{2}(\Sigma_0)}^2+  \|\dot{\Lambda}_h^{\sigma_0}{\Lambda}_h^{s-1-\sigma_0}
(\partial_t\partial_1\,q,
\, \nabla^2\,v,\,\nabla\,q,\,\partial_tv)\|_{L^2(\Omega)}^2.
\end{split}
\end{equation*}
On the other hand, we define the bounded instantaneous energy $\mathcal{E}_{s}$  and dissipation $\mathcal{D}_{s}$ (with $s>2$) including the negative horizontal derivative of $v$ and $\xi^1|_{\Sigma_0}$
\begin{equation}\label{def-bdd-energy-dissi-v-1}
\begin{split}
&\mathcal{E}_{s}\eqdefa \|\dot{\Lambda}_h^{-\lambda}\,(\nabla\,v,\,q)\|_{L^2(\Omega)}^2
+\|\dot{\Lambda}_h^{-\lambda}\,(Q,\,\xi^1)\|_{L^2(\Sigma_0)}^2+ \dot{\mathcal{E}}_{s},\\
&\mathcal{D}_{s}\eqdefa\|(\dot{\Lambda}_h^{-\lambda}(\nabla\,v,\,\partial_tv)\|_{L^2(\Omega)}^2
+\|\dot{\Lambda}_h^{-\lambda}({Q},\,\partial_t{Q})\|_{L^2(\Sigma_0)}^2+\dot{\mathcal{D}}_{s}.
\end{split}
\end{equation}

We also define instantaneous quantities $E_s=E_{s}(\xi,\,v)$ (with $s>2$) in terms of the velocity $v$ and the displacement $\xi$:
\begin{equation*}
\begin{split}
 &E_s\eqdefa \|\dot{\Lambda}_h^{\sigma_0}{\Lambda}_h^{s-1-\sigma_0}\nabla\xi\|_{H^1(\Omega)}^2+\mathcal{E}_{s}.
\end{split}
\end{equation*}
From this, we define the Hilbert functional space $\mathfrak{F}_N$ by
 \begin{equation}\label{def-spaceF-1}
\begin{split}
\mathfrak{F}_s\eqdefa \{(\xi,\,v)\in \mathring{H}^1(\Omega)\times {}_{0}H^1(\Omega):\,{E}_{s}(\xi,\,v)<+\infty\}
\end{split}
\end{equation}
equipped with the norm
$ \|(\xi,\,v)\|_{\mathfrak{F}_s}\eqdefa {E}_{s}(\xi, \,v)^{\frac{1}{2}}$.

\begin{rmk}\label{rmk-bdd-energy-1}
Thanks to the definitions above, for $s>2$, one can see
\begin{equation*}
\begin{split}
&{\mathcal{E}}_{s+\ell_0}\thicksim \|\dot{\Lambda}_h^{s+\ell_0-1 }(\nabla\, v,\,\nabla\,q)\|_{L^2(\Omega)}^2+ \|(\dot{\Lambda}_h^{s+\ell_0}\xi^1,\,\dot{\Lambda}_h^{s+\ell_0- \frac{1}{2}}{Q})\|_{L^2(\Sigma_0)}^2 +  {\mathcal{E}}_{s},\\
&\dot{\mathcal{E}}_{s+\ell_0}\thicksim\|\dot{\Lambda}_h^{s+\ell_0-1 }(\nabla\, v,\,\nabla\,q)\|_{L^2(\Omega)}^2+ \|(\dot{\Lambda}_h^{s+\ell_0}\xi^1,\,\dot{\Lambda}_h^{s+\ell_0- \frac{1}{2}}{Q})\|_{L^2(\Sigma_0)}^2+ \dot{\mathcal{E}}_{s},\\
&{E}_{s+\ell_0}\thicksim \|\dot{\Lambda}_h^{s-1+\ell_0 }\nabla\xi\|_{H^1(\Omega)}^2+\mathcal{E}_{s+\ell_0}+E_{s},
\end{split}
\end{equation*}
and
\begin{equation*}
\begin{split}
&{\mathcal{D}}_{s+\ell_0}\thicksim \|\dot{\Lambda}_h^{s+\ell_0}\nabla\,v\|_{L^2(\Omega)}^2 +\|\dot{\Lambda}_h^{s+\ell_0- \frac{1}{2}}({Q},\,\partial_t{Q})\|_{L^2(\Sigma_0)}^2
+\|\dot{\Lambda}_h^{s +\ell_0-1}\xi^1\|_{H^\frac{1}{2}(\Sigma_0)}^2\\
&\qquad\qquad+  \|\dot{\Lambda}_h^{s+\ell_0-1 }
(\partial_t\partial_1\,q,
\, \nabla^2\,v,\,\nabla\,q,\,\partial_tv)\|_{L^2(\Omega)}^2+{\mathcal{D}}_{s}, \\
&\dot{\mathcal{D}}_{s+\ell_0}\thicksim \|\dot{\Lambda}_h^{s+\ell_0}\nabla\,v\|_{L^2(\Omega)}^2 +\|\dot{\Lambda}_h^{s+\ell_0- \frac{1}{2}}({Q},\,\partial_t{Q})\|_{L^2(\Sigma_0)}^2
+\|\dot{\Lambda}_h^{s +\ell_0-1}\xi^1\|_{H^\frac{1}{2}(\Sigma_0)}^2\\
&\qquad\qquad+  \|\dot{\Lambda}_h^{s+\ell_0-1 }
(\partial_t\partial_1\,q,
\, \nabla^2\,v,\,\nabla\,q,\,\partial_tv)\|_{L^2(\Omega)}^2+\dot{\mathcal{D}}_{s}.
\end{split}
\end{equation*}
\end{rmk}

The main result of this paper states as follows, the proof of which will be presented in Section \ref{sect-proof-mainthm}.
\begin{thm}[Global well-posedness]\label{thm-main}
Let $s>2$ and $\lambda\in (0, 1)$ satisfy $s+\lambda>3$, $\ell_0 \in (\frac{1}{2},\,\min\{\frac{1}{2}+\frac{1}{4}\lambda,\,\frac{s+\lambda-2}{2}\}]$. If $(\xi_0,\,v_0) \in \mathfrak{F}_{s+\ell_0}$, then there exists a small positive constant $\nu_0$ such that the compressible viscous surface wave problem \eqref{eqns-pert-1} is globally well-posed in $\mathfrak{F}_{s+\ell_0}$ provided $E_{s+\ell_0}(0)\leq \nu_0$. Moreover, for any $t>0$, there hold
\begin{equation}\label{totalenerg-high-thm-1}
\begin{split}
&\mathcal{E}_{s+\ell_0}(t) \lesssim\mathcal{ E}_{s+\ell_0}(0),\,E_{s}(t) \lesssim \,E_{s}(0)+\mathcal{ E}_{s+\ell_0}(0),\\
&E_{s+\ell_0}(t) \lesssim \,E_{s+\ell_0}(0)+(1+t)\mathcal{ E}_{s+\ell_0}(0),\\
&\sup_{\tau\in[0, t]}\bigg((1+\tau)^{2\ell_0}\dot{\mathcal{E}}_{s}(\tau)\bigg)
+\int_0^t\bigg((1+\tau)^{\frac{1}{2}+\ell_0}\dot{\mathcal{D}}_{s}(\tau)
+\mathcal{D}_{s+\ell_0}(\tau)\bigg)\,d\tau \lesssim \mathcal{E}_{s+\ell_0}(0).
 \end{split}
\end{equation}
\end{thm}

\begin{rmk}\label{rmk-mainthm-2} There are three main differences between the work of Jang-Tice-Wang \cite{WTK-2014} and our result Theorem \ref{thm-main}. The first one is that, Jang-Tice-Wang \cite{WTK-2014} studied the global well-posedness of the system \eqref{VFS-eqns-1} via its Eulerian and geometric formulations, while we directly use the Lagrangian approach to get its global well-posedness result, which also tells us what happens to the entire flow map $\eta$ in $\Omega$. On the other hand, consideration in \cite{WTK-2014} was about the system \eqref{VFS-eqns-1} in horizontally periodic layer, which method used doesn't work to the horizontally infinite layer case we study in the paper. Finally, compared to the work of \cite{WTK-2014} in which high total regularity of the initial data and compatibility conditions about the time derivatives of the velocity are imposed, there is no requirement about the compatibility conditions in terms of $\partial_tv$ in Theorem \ref{thm-main}, and we need only to ask the minimal normal regularity of the velocity (more precisely, no more than the first order normal derivatives on the initial velocity are required in the theorem).
 \end{rmk}

\begin{rmk}\label{rmk-mainthm-3}
(i)There is no requirement about the compatibility conditions in terms of $\partial_tv$ in Theorem \ref{thm-main}, which {\bf plays a significant role} in the study of the vacuum free boundary problem of the compressible Navier-Stokes system (see Gui-Wang-Wang \cite{Gui-WW-2019} for example).

(ii)The Lagrangian method used in the paper is a new approach to study the global well-posedness of the incompressible or compressible Navier-Stokes equations with free boundary problem, which can be applied to investigate the well-posedness of many types of free boundary problem of fluid dynamical systems.

(iii)Some anisotropic algebraic decay estimates in terms of the energy and dissipation are uncovered in Theorem \ref{thm-main}, and this kind of result may be helpful to understand the possible blow-up or decay mechanism of the solution.
 \end{rmk}

\subsection{Plan of the paper}

The rest of the paper is organized as follows. Section \ref{sect-form} introduces some new variables according to the kinetic boundary condition $\eqref{VFS-eqns-1}_3$ in \eqref{VFS-eqns-1}, and gives the linearized form of the viscous surface waves \eqref{VFS-eqns-1} in Lagrangian coordinates. Next, we first recall, in Section \ref{sect-tool}, some necessary estimates, and then derive the basic estimates in terms of the flow map. In Section \ref{sect-energy}, we apply the four new good unknowns introduced first in the paper to get energy estimates of the horizontal derivatives of the velocity and its gradient. Estimates of second normal derivatives of the velocity as well as their high-order horizontal regularities are showed by using the Stokes estimates in Section \ref{sect-stokes}. As a consequence, the total energy estimates are stated in Section \ref{sect-total}. Finally, in Section \ref{sect-proof-mainthm}, we complete the proof of Theorem \ref{thm-main}.

\subsection{Notations}

\medbreak  Let us end this introduction by some notations that will be used in all that follows.

For operators $A,B,$ we denote $[A;B]=AB-BA$ to be the  commutator of $A$ and $B.$ We denote $1+t$ by $\langle t\rangle$. For~$a\lesssim b$, we mean that there is a uniform constant $C,$ which may be different on different lines, such that $a\leq Cb$.  The notation $a\thicksim b$ means both $a\lesssim b$ and $b\lesssim a$. Throughout the paper, the subscript notation for vectors and tensors as well as the Einstein summation convention has been adopted unless otherwise specified, Einstein¡¯s summation convention means that repeated Latin indices $i,\, j,\, k$, etc., are summed from $1$ to $3$, and repeated Greek indices $\alpha, \,\beta, \,\gamma$, etc., are summed from $2$ to $3$. In the vector $v=(v^1, v^2, v^3)^T$, we denote the vertical component by $v^1$, and its horizontal component by $v^h=(v^2, v^3)^T$.

\renewcommand{\theequation}{\thesection.\arabic{equation}}
\setcounter{equation}{0}

\section{Linearized form of the system in Lagrangian coordinates}\label{sect-form}

\subsection{Some new unknowns from the boundary}
In this section, we want to introduce four new good unknowns related to the normal derivatives $\partial_1v$ and the pressure $q$, which plays a significant role for the proof of the main result of the paper.

Notice that $\mathcal{N}=J\mathcal{A}\,n_0$ is the outer normal vector field on the interface $\Sigma_0$,
\begin{equation*}
  \begin{split}
&\mathcal{N}= J\mathcal{A}\,n_0=a_{j1}e_j=a_{11}e_1+a_{21}e_2+a_{31}e_3.
\end{split}
\end{equation*}
Let $\tau_2$ and $\tau_3$ be two independent tangential vector fields on $\Sigma_0$
\begin{equation*}
  \begin{split}
&\tau_2=-a_{21}e_1+a_{11}e_2,\quad \tau_3=-a_{31}e_1+a_{11}e_3,
\end{split}
\end{equation*}
which implies that $\mathcal{N}\cdot \tau_{\alpha}=0$ with $\alpha=2, 3$.

By a direct computation, we get
\begin{equation}\label{deform-bdry-1}
  \begin{split}
&\bigl(\mathbb{D}_{J\mathcal{A}}(v)\mathcal{N}\bigr)^1=(a_{11}^2+|\overrightarrow{\rm{a}_1}|^2)\partial_1 v^1+a_{11}a_{\alpha\,1}\partial_1 v^\alpha+\mathcal{B}_{1, i}^{\alpha}\partial_{\alpha}v^{i}, \\
&\bigl(\mathbb{D}_{J\mathcal{A}}(v)\mathcal{N}\bigr)^2=a_{21}a_{11}\,\partial_1 v^1+(|\overrightarrow{\rm{a}_1}|^2+a_{21}^2)\partial_1 v^2+a_{21}a_{31}\,\partial_1 v^3+\partial_2 v^1+\mathcal{B}_{2, i}^{\alpha}\partial_{\alpha}v^{i}, \\
&\bigl(\mathbb{D}_{J\mathcal{A}}(v)\mathcal{N}\bigr)^3=a_{31}a_{11}\partial_1 v^1+a_{31}a_{21}\partial_1 v^2+(|\overrightarrow{\rm{a}_1}|^2+a_{31}^2)\partial_1 v^3+\partial_3 v^1+\mathcal{B}_{3, i}^{\alpha}\partial_{\alpha}v^{i},\\
  \end{split}
\end{equation}
where
\begin{equation}\label{def-a1-row-1}
\overrightarrow{\rm{a}_1}:=(a_{11}, a_{21}, a_{31})^T,\quad |\overrightarrow{\rm{a}_1}| :=(\sum_{i=1}^3a_{i1}^2)^{1/2},
\end{equation}
and
\begin{equation*}
  \begin{split}
&\mathcal{B}_{1, i}^{\alpha}\partial_{\alpha}v^{i}\eqdefa a_{i\beta}a_{i1}\partial_{\beta} v^1+a_{1\beta}a_{i\,1}\,\partial_\beta v^i,\\
&\mathcal{B}_{2, i}^{\alpha}\partial_{\alpha}v^{i}\eqdefa (a_{22}a_{11}-1)\,\partial_2 v^1+a_{23}a_{11}\,\partial_3 v^1+ a_{i\beta}a_{i1}\partial_{\beta} v^2+a_{2\beta}a_{\gamma\,1}\,\partial_{\beta} v^\gamma,\\
&\mathcal{B}_{3, i}^{\alpha}\partial_{\alpha}v^{i}\eqdefa a_{32}a_{11}\partial_2 v^1+a_{3\beta}a_{\alpha\,1}\partial_{\beta} v^{\alpha}+ a_{i\beta}a_{i1} \partial_{\beta} v^3+ (a_{33}a_{11}-1) \,\partial_3 v^1.
\end{split}
\end{equation*}

Hence, multiplying the boundary equation on $\Sigma_0$ in \eqref{eqns-pert-1} by  the tangential vectors $\tau_{\alpha}$ (with $\alpha=2, 3$) implies
\begin{equation}\label{re-bdry-eqns-1}
  \begin{split}
&0=\tau_2\cdot(\mathbb{D}_{J\mathcal{A}}(v)\mathcal{N})
=-a_{21} |\overrightarrow{\rm{a}_1}|^2 \partial_1 v^1+a_{11} |\overrightarrow{\rm{a}_1}|^2\partial_1 v^2\\
&\qquad\qquad\qquad\qquad\qquad+a_{11}\partial_2 v^1+(-a_{21}\mathcal{B}_{1, i}^{\alpha}+a_{11}\mathcal{B}_{2, i}^{\alpha})\partial_{\alpha}v^{i},
\end{split}
\end{equation}
and
\begin{equation}\label{re-bdry-eqns-2}
  \begin{split}
&0=\tau_3\cdot(\mathbb{D}_{J\mathcal{A}}(v)\mathcal{N})
=-a_{31} |\overrightarrow{\rm{a}_1}|^2 \partial_1 v^1+a_{11} |\overrightarrow{\rm{a}_1}|^2\partial_1 v^3\\
&\qquad\qquad\qquad\qquad\qquad+a_{11}\partial_3 v^1+(-a_{31}\mathcal{B}_{1, i}^{\alpha}+a_{11}\mathcal{B}_{3, i}^{\alpha})\partial_{\alpha}v^{i}.
\end{split}
\end{equation}
Thanks to \eqref{deform-bdry-1}, one can see
\begin{equation*}
  \begin{split}
 \mathcal{N}\cdot\bigl(\mathbb{D}_{J\mathcal{A}}(v)\mathcal{N}\bigr) =&
2a_{11} |\overrightarrow{\rm{a}_1}|^2 \partial_1 v^1+2a_{21} |\overrightarrow{\rm{a}_1}|^2 \partial_1 v^2+2a_{31} |\overrightarrow{\rm{a}_1}|^2 \partial_1 v^3\\
  &+a_{21}\partial_2 v^1+a_{31}\partial_3 v^1+a_{11}\mathcal{B}_{1, i}^{\alpha}\partial_{\alpha}v^{i} +a_{21}\mathcal{B}_{2, i}^{\alpha}\partial_{\alpha}v^{i} +a_{31}\mathcal{B}_{3, i}^{\alpha}\partial_{\alpha}v^{i}.
  \end{split}
\end{equation*}
In the fluid, we write
\begin{equation}\label{div-cond-fluid-1}
  \begin{split}
&\grad_{J\mathcal{A}}\cdot v=a_{11}\partial_1v^1
+a_{\alpha\,1}\partial_1v^{\alpha}+ a_{22} \partial_2v^2+ a_{33} \partial_3v^3+\mathcal{B}_{4, i}^{\alpha}\partial_{\alpha}v^{i}
  \end{split}
\end{equation}
with the nonlinear terms including horizontal derivatives of the velocity $\partial_{\alpha} v$
\begin{equation*}
  \begin{split}
&\mathcal{B}_{4, i}^{\alpha}\partial_{\alpha}v^{i}\eqdefa a_{1\beta}\partial_{\beta}v^1+a_{23}\partial_3v^2+a_{32}\partial_2v^3.
  \end{split}
\end{equation*}
Combining \eqref{div-cond-fluid-1} with \eqref{eqns-q-form-2} gives rise to
\begin{equation}\label{comp-cond-fluid-2}
  \begin{split}
&a_{11}\partial_1v^1
+a_{\alpha\,1}\partial_1v^{\alpha}=-(a_{22} \partial_2v^2+ a_{33} \partial_3v^3)-J(p'(f))^{-1} \partial_t{Q}-\mathcal{B}_{4, i}^{\alpha}\partial_{\alpha}v^{i}.
  \end{split}
\end{equation}
Thanks to \eqref{re-bdry-eqns-1}, \eqref{re-bdry-eqns-2}, and \eqref{comp-cond-fluid-2}, we solve them in terms of the vertical derivative of $v$ to conclude that on the interface $\Sigma_0$
\begin{equation}\label{v-com-interface-1}
  \begin{cases}
&\partial_1v^1=-\nabla_h\cdot v^h+\mathcal{B}_{6, i}^{\alpha}\partial_{\alpha}v^{i}-a_{11}|\overrightarrow{\rm{a}_1}|^{-2}J(p'(f))^{-1}\partial_t{Q},\\
&  \partial_1v^{2}=-\partial_2v^1+\mathcal{B}_{7, i}^{\alpha}\partial_{\alpha}v^{i}-a_{21}J(p'(f))^{-1}\partial_t{Q},\\
&  \partial_1v^{3}=-\partial_3v^1+\mathcal{B}_{8, i}^{\alpha}\partial_{\alpha}v^{i}-a_{31}J(p'(f))^{-1}\partial_t{Q},\\
\end{cases}
\end{equation}
where
\begin{equation*}
  \begin{split}
  &\mathcal{B}_{6, i}^{\alpha}\partial_{\alpha}v^{i}\eqdefa -|\overrightarrow{\rm{a}_1}|^{-2}((a_{11}a_{22}-|\overrightarrow{\rm{a}_1}|^{2}) \partial_2v^2+ (a_{11}a_{33}-|\overrightarrow{\rm{a}_1}|^{2}) \partial_3v^3)-a_{11}|\overrightarrow{\rm{a}_1}|^{-2}\mathcal{B}_{4, i}^{\alpha}\partial_{\alpha}v^{i}\\
&+ a_{11}|\overrightarrow{\rm{a}_1}|^{-4}(a_{21}\partial_2 v^1+a_{31}\partial_3 v^1)-  ( (a_{21}^2+a_{31}^2)\mathcal{B}_{1, i}^{\alpha}- a_{11}(a_{21}\mathcal{B}_{2, i}^{\alpha}+a_{31}\mathcal{B}_{3, i}^{\alpha}))\partial_{\alpha}v^{i},\\
&\mathcal{B}_{7, i}^{\alpha}\partial_{\alpha}v^{i}\eqdefa \frac{ a_{21}^2}{|\overrightarrow{\rm{a}_1}|^{2}}\partial_2 v^1+\frac{ a_{21}a_{31}}{|\overrightarrow{\rm{a}_1}|^{2}} \partial_3 v^1 -a_{21}  \bigg((a_{22} \partial_2v^2+ a_{33} \partial_3v^3) +\mathcal{B}_{4, i}^{\alpha}\partial_{\alpha}v^{i}\bigg) \\
&\qquad\qquad\quad+|\overrightarrow{\rm{a}_1}|^{-2}(a_{11} a_{21}\mathcal{B}_{1, i}^{\alpha}-(a_{11}^2+a_{31}^2)\mathcal{B}_{2, i}^{\alpha}+a_{21} a_{31}\mathcal{B}_{3, i}^{\alpha})\partial_{\alpha}v^{i},\\
&\mathcal{B}_{8, i}^{\alpha}\partial_{\alpha}v^{i}\eqdefa \frac{a_{21}a_{31} }{|\overrightarrow{\rm{a}_1}|^{2}} \partial_2 v^1+ \frac{ a_{31}^2}{|\overrightarrow{\rm{a}_1}|^{2}}\partial_3 v^1-a_{31} \bigg((a_{22} \partial_2v^2+ a_{33} \partial_3v^3) +\mathcal{B}_{4, i}^{\alpha}\partial_{\alpha}v^{i}\bigg)\\
&\qquad\qquad\quad+|\overrightarrow{\rm{a}_1}|^{-2}\bigg(a_{31}(a_{11}\mathcal{B}_{1, i}^{\alpha}+a_{21}\mathcal{B}_{2, i}^{\alpha})-(a_{11}^2+a_{21}^2)\mathcal{B}_{3, i}^{\alpha}\bigg)\partial_{\alpha}v^{i}.
\end{split}
\end{equation*}

Define
\begin{equation*}
  \begin{split}
&\widetilde{\mathcal{B}}_{1}\eqdefa -(a_{11}|\overrightarrow{\rm{a}_1}|^{-2}J(p'(f))^{-1}-(p'(\bar{\rho}(x_1)))^{-1}),\quad \widetilde{\mathcal{B}}_{2}\eqdefa -a_{21}J(p'(f))^{-1},\\
& \widetilde{\mathcal{B}}_{3}\eqdefa -a_{31}J(p'(f))^{-1},\quad \widetilde{\mathcal{B}}_{4}\eqdefa -(\delta+\frac{4}{3}\varepsilon)\, ((p'(f))^{-1}-(p'(\bar{\rho}(x_1)))^{-1}).
\end{split}
\end{equation*}
then we have
\begin{equation*}
  \begin{cases}
&\partial_1v^1=-\nabla_h\cdot v^h+\mathcal{B}_{6, i}^{\alpha}\partial_{\alpha}v^{i}-(p'(\bar{\rho}(0)))^{-1}\partial_t{Q}+\widetilde{\mathcal{B}}_{1}\partial_t{Q},\\
&  \partial_1v^{2}=-\partial_2v^1+\mathcal{B}_{7, i}^{\alpha}\partial_{\alpha}v^{i}+\widetilde{\mathcal{B}}_{2}\partial_t{Q},\\
&  \partial_1v^{3}=-\partial_3v^1+\mathcal{B}_{8, i}^{\alpha}\partial_{\alpha}v^{i}+\widetilde{\mathcal{B}}_{3}\partial_t{Q}.
\end{cases}
\end{equation*}

On the other hand, from the boundary condition on $\Sigma_0$ in \eqref{eqns-pert-1}, one finds
\begin{equation*}
  \begin{split}
&\bigg(\bar{\rho}(0)\,{Q}+{Q}^2\mathcal{R}({Q})-(\delta-\frac{2}{3}\varepsilon)\,\nabla_{\mathcal{A}}\cdot v\bigg)\, \mathcal{N}\cdot\mathcal{N}- \varepsilon\mathcal{N}\cdot(\mathbb{D}_{\mathcal{A}}(v) \mathcal{N})=0,
  \end{split}
\end{equation*}
which implies
\begin{equation*}\label{bdry-q-1}
\begin{split}
  & J\, \bigg(\bar{\rho}(0)\,{Q}+{Q}^2\mathcal{R}({Q})\bigg)=
  -(\delta+\frac{4}{3}\varepsilon)\,J(p'(f))^{-1}\partial_t{Q}\\
  &\qquad-2\varepsilon(a_{22} \partial_2v^2+ a_{33} \partial_3v^3) -2\varepsilon\mathcal{B}_{4, i}^{\alpha}\partial_{\alpha}v^{i}+\varepsilon  |\overrightarrow{\rm{a}_1}|^{-2}(a_{\beta\,1}\partial_{\beta} v^1 + a_{j1}\mathcal{B}_{j, i}^{\alpha}\partial_{\alpha}v^{i}).
  \end{split}
\end{equation*}
Hence, we get
\begin{equation}\label{bdry-tq-1}
\begin{split}
  &- (\delta+\frac{4}{3}\varepsilon) (p'(f))^{-1}\partial_t{Q}= \bar{\rho}(0)\,{Q}+{Q}^2\mathcal{R}({Q})+2\varepsilon \nabla_h\cdot\,v^h-\mathcal{B}_{9, i}^{\alpha}\partial_{\alpha}v^{i}
  \end{split}
\end{equation}
with
\begin{equation*}
\begin{split}
  \mathcal{B}_{9, i}^{\alpha}\partial_{\alpha}v^{i} \eqdefa &-2\varepsilon  ((J^{-1}a_{22} -1)\partial_2v^2+(J^{-1} a_{33} -1)\partial_3v^3) -2\varepsilon  J^{-1}\mathcal{B}_{4, i}^{\alpha}\partial_{\alpha}v^{i}\\
  &+\varepsilon  |\overrightarrow{\rm{a}_1}|^{-2}  J^{-1}(a_{\beta\,1}\partial_{\beta} v^1 + a_{j1}\mathcal{B}_{j, i}^{\alpha}\partial_{\alpha}v^{i}).
\end{split}
\end{equation*}

 Plugging \eqref{bdry-tq-1} into \eqref{v-com-interface-1} shows that, on the boundary $\Sigma_0$,
\begin{equation}\label{d-1-v-interface-form-1}
  \begin{cases}
&\partial_1v^1=-\frac{\delta-\frac{2}{3}\varepsilon}{\delta+\frac{4}{3}\varepsilon} \nabla_h\cdot v^h+\mathcal{B}_{11, i}^{\alpha}\partial_{\alpha}v^{i}+ \frac{a_{11}J \bar{\rho}(x_1)}{|\overrightarrow{\rm{a}_1}|^{2}(\delta+\frac{4}{3}\varepsilon)}\,{Q}
+\frac{a_{11}|\overrightarrow{\rm{a}_1}|^{-2}J}{(\delta+\frac{4}{3}\varepsilon)}
{Q}^2\mathcal{R}({Q}),\\
&  \partial_1v^{2}=-\partial_2v^1+\mathcal{B}_{12, i}^{\alpha}\partial_{\alpha}v^{i}+\frac{a_{21}J\bar{\rho}(x_1)}{(\delta+\frac{4}{3}\varepsilon)} \,{Q} +\frac{a_{21}J}{(\delta+\frac{4}{3}\varepsilon)}{Q}^2\mathcal{R}({Q}),\\
&  \partial_1v^{3}=-\partial_3v^1+\mathcal{B}_{13, i}^{\alpha}\partial_{\alpha}v^{i}+\frac{a_{31}J\bar{\rho}(x_1)}{(\delta+\frac{4}{3}\varepsilon)} \,{Q}+\frac{a_{31}J}{(\delta+\frac{4}{3}\varepsilon)}{Q}^2\mathcal{R}({Q}),\\
\end{cases}
\end{equation}
where
\begin{equation*}
  \begin{split}
&\mathcal{B}_{11, i}^{\alpha}\partial_{\alpha}v^{i}\eqdefa\mathcal{B}_{6, i}^{\alpha}\partial_{\alpha}v^{i}+\frac{2\varepsilon (a_{11}|\overrightarrow{\rm{a}_1}|^{-2}J-1)}{(\delta+\frac{4}{3}\varepsilon)} \nabla_h\cdot\,v^h-\frac{a_{11}|\overrightarrow{\rm{a}_1}|^{-2}J}{(\delta+\frac{4}{3}\varepsilon)}\mathcal{B}_{9, i}^{\alpha}\partial_{\alpha}v^{i},\\
& \mathcal{B}_{12, i}^{\alpha}\partial_{\alpha}v^{i}\eqdefa  \mathcal{B}_{7, i}^{\alpha}\partial_{\alpha}v^{i}+\frac{2\varepsilon a_{21}J}{(\delta+\frac{4}{3}\varepsilon)} \nabla_h\cdot\,v^h-\frac{a_{21}J}{(\delta+\frac{4}{3}\varepsilon)}\mathcal{B}_{9, i}^{\alpha}\partial_{\alpha}v^{i},\\
& \mathcal{B}_{13, i}^{\alpha}\partial_{\alpha}v^{i}\eqdefa \mathcal{B}_{8, i}^{\alpha}\partial_{\alpha}v^{i}+\frac{2\varepsilon a_{31}J}{(\delta+\frac{4}{3}\varepsilon)} \nabla_h\cdot\,v^h-\frac{a_{31}J}{(\delta+\frac{4}{3}\varepsilon)}\mathcal{B}_{9, i}^{\alpha}\partial_{\alpha}v^{i}.
\end{split}
\end{equation*}
Due to \eqref{bdry-tq-1}, we find, on the boundary $\Sigma_0$,
\begin{equation}\label{eqns-Q-bdry-1}
  \begin{split}
&\bar{\rho}(0){Q}+(\delta+\frac{4}{3}\varepsilon)\, \frac{-\bar{\rho}'(0)}{g\,\bar{\rho}(0)}\partial_t{Q}=-2\varepsilon\,\nabla_h\cdot v^h+\mathcal{B}_{9, i}^{\alpha}\partial_{\alpha}v^{i}
-{Q}^2\mathcal{R}({Q})+\widetilde{\mathcal{B}}_{4}\partial_t{Q}
\end{split}
\end{equation}
with $\widetilde{\mathcal{B}}_{4}\eqdefa -(\delta+\frac{4}{3}\varepsilon)\, ((p'(f))^{-1}-(p'(\bar{\rho}(x_1)))^{-1})$.

On the other hand, substituting the first equation of \eqref{v-com-interface-1} into  \eqref{bdry-tq-1}, and then combining the result with \eqref{eqns-q-form-2}, we obtain
\begin{equation}\label{bdry-q-eqns-1}
\begin{split}
    &\bar{\rho}(0)\,{Q}-2\varepsilon \partial_1v^1-(\delta-\frac{2}{3}\varepsilon) \nabla \cdot v=\mathcal{B}(\partial_{\alpha}v, {Q}, \partial_t{Q})
  \end{split}
\end{equation}
with
\begin{equation*}
\begin{split}
&\mathcal{B}(\partial_{\alpha}v, {Q}, \partial_t{Q}):=-\widetilde{\mathcal{B}}_5\partial_t{Q} -(\delta+\frac{4}{3}\varepsilon)\mathcal{B}_{6, i}^{\alpha}\partial_{\alpha}v^{i} + \mathcal{B}_{9, i}^{\alpha}\partial_{\alpha}v^{i}-{Q}^2\mathcal{R}({Q}) \quad \text{and} \quad\\
  &\widetilde{\mathcal{B}}_5:=(\delta+\frac{4}{3}\varepsilon)\,( 1 -J a_{11}|\overrightarrow{\rm{a}_1}|^{-2})(p'(f))^{-1}.
\end{split}
\end{equation*}

According to \eqref{d-1-v-interface-form-1}, we introduce three new unknowns $\mathcal{G}^i$ with $i=1, 2, 3$ defined in the fluid by
\begin{equation}\label{d-1-v-unknown-1}
  \begin{cases}
&\mathcal{G}^1\eqdefa\partial_1v^1+\frac{\delta-\frac{2}{3}\varepsilon}{\delta+\frac{4}{3}\varepsilon} \nabla_h\cdot v^h-\frac{\bar{\rho}(0)}{(\delta+\frac{4}{3}\varepsilon)}\, \mathcal{H}({Q})-\widetilde{\mathcal{G}}^1,\\
&  \mathcal{G}^{2}\eqdefa\partial_1v^{2}+\partial_{2}v^1-\widetilde{\mathcal{G}}^2,\\
&  \mathcal{G}^{3}\eqdefa\partial_1v^{3}+\partial_{3}v^1-\widetilde{\mathcal{G}}^3\quad \text{in} \quad \Omega,
\end{cases}
\end{equation}
where for $\beta=2, 3$,
\begin{equation*}
  \begin{split}
  &\widetilde{\mathcal{G}}^j\eqdefa \mathcal{G}^j_{{Q}, 1}+\mathcal{G}^j_{\partial_hv}+\mathcal{G}^j_{{Q}, 2} \quad (\text{for}\,\, j=1, 2, 3),\\
& \mathcal{G}^{1}_{{Q}, 1}:=\frac{\bar{\rho}(0)\,(a_{11}|\overrightarrow{\rm{a}_1}|^{-2}J-1) }{(\delta+\frac{4}{3}\varepsilon)}\, \mathcal{H}({Q}),\quad \mathcal{G}^{\beta}_{{Q}, 1}:=\frac{\bar{\rho}(0)\,a_{\beta\,1}J}{(\delta+\frac{4}{3}\varepsilon)}  \mathcal{H}({Q}),\quad\mathcal{G}^{1}_{\partial_hv}:=\mathcal{B}_{11, i}^{\alpha}\partial_{\alpha}v^{i},\\
& \mathcal{G}^{\beta}_{\partial_hv}:=\mathcal{B}_{10+\beta, i}^{\alpha}\partial_{\alpha}v^{i}, \, \mathcal{G}^{1}_{{Q}, 2}:=\frac{a_{11}|\overrightarrow{\rm{a}_1}|^{-2}J}{(\delta+\frac{4}{3}\varepsilon)}
\mathcal{H}({Q}^2\mathcal{R}({Q})),\, \mathcal{G}^{\beta}_{{Q}, 2}:=\frac{a_{\beta\,1}J}{(\delta+\frac{4}{3}\varepsilon)}\mathcal{H}(
{Q}^2\mathcal{R}({Q})).
\end{split}
\end{equation*}
Due to \eqref{d-1-v-interface-form-1} and \eqref{d-1-v-unknown-1}, we find
\begin{equation}\label{unknown-bdry-1}
  \begin{split}
  &\mathcal{G}^1=0,\quad \mathcal{G}^2=0,\quad \mathcal{G}^3=0\quad\text{on}\quad \Sigma_0,
\end{split}
\end{equation}
and $\partial_1v^i$ with $i=1, 2, 3$ satisfies
\begin{equation}\label{partialv-q-1}
  \begin{cases}
  &\partial_1v^1=\mathcal{G}^1-\frac{\delta-\frac{2}{3}\varepsilon}{\delta+\frac{4}{3}\varepsilon} \nabla_h\cdot v^h+\frac{\bar{\rho}(0)}{(\delta+\frac{4}{3}\varepsilon)}\, \mathcal{H}({Q})+\widetilde{\mathcal{G}}^1,\\
&\partial_1v^{2}= \mathcal{G}^{2}-\partial_{2}v^1+\widetilde{\mathcal{G}}^2,\\
&\partial_1v^{3}=\mathcal{G}^{3}-\partial_{3}v^1+\widetilde{\mathcal{G}}^3\quad \text{in} \quad \Omega.
\end{cases}
\end{equation}

Combining \eqref{v-com-interface-1} with \eqref{bdry-q-eqns-1}, the boundary condition on $\Sigma_0$ in \eqref{eqns-pert-1} can be written as the linearized form
\begin{equation}\label{prtv-interface-1}
\begin{split}
\bar{\rho}(0)q\, e_1- \mathbb{S}(v)\,e_1=
\left(
  \begin{array}{c}
   g\bar{\rho}(0)\xi^1+\mathcal{B}(\partial_{\alpha}v, {Q}, \partial_t{Q})\\
    -\varepsilon\,(\mathcal{B}_{7, i}^{\alpha}\partial_{\alpha}v^{i}+\widetilde{\mathcal{B}}_{2}\partial_t{Q})\\
    - \varepsilon\,(\mathcal{B}_{8, i}^{\alpha}\partial_{\alpha}v^{i}+\widetilde{\mathcal{B}}_{3}\partial_t{Q}).
  \end{array}
\right)
\end{split}
\end{equation}

In order to extend the interface boundary forms of $\partial_1v$ and $q$ in \eqref{prtv-interface-1} to the interior domain of the fluid, let us first introduce $\mathcal{H}(f)$ as the harmonic extension of $f|_{\Sigma_0}$ into $\Omega$:
\begin{equation}\label{harmonic-ext-xi-1-1}
\begin{cases}
   &\Delta \mathcal{H}(f)=0 \quad\mbox{in} \quad \Omega,\\
   &\mathcal{H}(f)|_{\Sigma_0}=f|_{\Sigma_0},\quad \mathcal{H}(f)|_{\Sigma_b}=0.
\end{cases}
\end{equation}

These four new unknowns $\mathcal{G}^i$ with $i=1, 2, 3$ and ${Q}$ play an important role in the estimate of the $L^\infty_tL^2$ norms of $\nabla\,v$ as well as its tangential derivatives, as it is possible to close all the estimates by using integration in parts and the relations \eqref{partialv-q-1} with \eqref{unknown-bdry-1} when we  encounter the integrals including the normal derivative of $\partial_1v$ and $q$. More details will be showed in the proof of Lemma \ref{lem-pseudo-energy-tv-1}. This makes it possible to avoid using the compatibility conditions in terms of the acceleration $\partial_tv$ and its derivatives.

\subsection{Linearized form of the system}

Let's now derive the linearized form of the system \eqref{eqns-pert-1} in the Lagrangian coordinates, which is crucially used to recover the bounds for high order derivatives.

Under the Lagrangian coordinates, we divide the dissipative terms $\grad_{J\mathcal{A}}\grad_{\mathcal{A}}\cdot v$ and $\grad_{J\mathcal{A}} \cdot \mathbb{D}_{\mathcal{A}}(v)$ into linear and nonlinear parts
\begin{equation}\label{dissipative-split-1}
\begin{split}
 &(\grad_{J\mathcal{A}}\grad_{\mathcal{A}}\cdot v)^k=(\nabla\nabla\cdot v)^k+\widetilde{\mathfrak{g}}_{kk},\quad (\grad_{J\mathcal{A}} \cdot \mathbb{D}_{\mathcal{A}}(v))^{k}=(\nabla \cdot \mathbb{D}(v))^{k}+\mathfrak{g}_{kk}
\end{split}
\end{equation}
with $k=1, 2, 3$, where nonlinear parts $\widetilde{\mathfrak{g}}_{kk}$ and $\mathfrak{g}_{kk}$ have the forms
\begin{equation*}\label{re-cond-fluid-1}
  \begin{split}
\widetilde{\mathfrak{g}}_{11}\eqdefa &[J(\mathcal{A}_{1}^1)^2-1]\partial_1^2 v^1+J\mathcal{A}_{1}^1\mathcal{A}_{\alpha}^1\partial_1^2v^{\alpha}
+(J\mathcal{A}_{1}^1\mathcal{A}_{j}^{\alpha}+J\mathcal{A}_{1}^{\alpha}\mathcal{A}_{j}^1
-\delta_j^{\alpha}) \partial_1\partial_{\alpha}v^j\\
&+J\mathcal{A}_{1}^{\beta} \mathcal{A}_{j}^{\alpha} \partial_{\beta} \partial_{\alpha}v^j+J(\mathcal{A}_{1}^{1}\partial_{1}\mathcal{A}_{j}^{1} +\mathcal{A}_{1}^{\alpha}\partial_{\alpha}\mathcal{A}_{j}^{1}) \partial_{1}v^j+J\mathcal{A}_{1}^{k}\partial_{k}( \mathcal{A}_{j}^{\alpha}) \partial_{\alpha}v^j,\\
\widetilde{\mathfrak{g}}_{\beta\beta}\eqdefa &J\mathcal{A}_{\beta}^1 \mathcal{A}_{1}^1\partial_1^2 v^1J+\mathcal{A}_{\beta}^1\mathcal{A}_{\alpha}^1\partial_1^2v^{\alpha}
+(J\mathcal{A}_{\beta}^1\mathcal{A}_{j}^{\alpha}+J\mathcal{A}_{\beta}^{\alpha}\mathcal{A}_{j}^1
-\delta_{\beta}^{\alpha}\delta_{j}^1) \partial_1\partial_{\alpha}v^j\\
&+(J\mathcal{A}_{\beta}^{\gamma} \mathcal{A}_{j}^{\alpha} -\delta_{\beta}^{\gamma} \delta_{j}^{\alpha})\partial_{\gamma} \partial_{\alpha}v^j+J\mathcal{A}_{\beta}^{k}\partial_{k}( \mathcal{A}_{j}^1) \partial_1v^j+J\mathcal{A}_{\beta}^{k}\partial_{k} \mathcal{A}_{j}^{\alpha} \partial_{\alpha}v^j,
  \end{split}
\end{equation*}
\begin{equation*}
\begin{split}
 &\mathfrak{g}_{11}:=[J^{-1}|\overrightarrow{\rm{a}_1}|^2+J(\mathcal{A}_1^{1})^2-2]\partial_1^2 v^1+J\mathcal{A}_{\alpha}^{1}\mathcal{A}_1^1\partial_1^2 v^{\alpha}+2J\mathcal{A}_i^{\alpha}\mathcal{A}_i^1\partial_{\alpha} \partial_1 v^1\\
&+(J\mathcal{A}_i^{\alpha}\mathcal{A}_i^{\beta}-\delta_i^{\alpha}\delta_i^{\beta})\partial_{\alpha} \partial_{\beta} v^1+ J\mathcal{A}_i^{1}\mathcal{A}_1^{\alpha} \partial_1 \partial_{\alpha} v^i+(J\mathcal{A}_i^{\alpha}\mathcal{A}_1^1 -\delta_i^{\alpha}\delta_1^1)\partial_1 \partial_{\alpha} v^i+J\mathcal{A}_i^{\alpha}\mathcal{A}_1^{\beta}\partial_{\alpha} \partial_{\beta} v^i\\
&+J(\mathcal{A}_i^{m} \partial_m\mathcal{A}_i^1\partial_1 v^1+\mathcal{A}_i^{m} \partial_m\mathcal{A}_1^1\partial_1 v^i)+J(\mathcal{A}_i^{m} \partial_m\mathcal{A}_i^{\alpha}\partial_{\alpha} v^1+\mathcal{A}_i^{m} \partial_m\mathcal{A}_1^{\alpha}\partial_{\alpha} v^i),\\
 &\mathfrak{g}_{22}:= (J^{-1}|\overrightarrow{\rm{a}_1}|^2-1)\partial_1^2 v^2+J\mathcal{A}_i^{1}\mathcal{A}_2^1\partial_1^2v^i+2J\mathcal{A}_i^{\alpha}\mathcal{A}_i^1\partial_{\alpha} \partial_1 v^2\\
&+J(\mathcal{A}_i^{\alpha}\mathcal{A}_i^{\beta}-\delta_i^{\alpha}\delta_i^{\beta})\partial_{\alpha}\partial_{\beta} v^2+(J\mathcal{A}_i^{m}\mathcal{A}_2^2-\delta_i^{m})\partial_m \partial_2 v^i+J\mathcal{A}_i^{\alpha}\mathcal{A}_2^1\partial_{\alpha}\partial_1 v^i+J\mathcal{A}_i^{m}\mathcal{A}_2^3\partial_m \partial_3 v^i\\
&+ J(\mathcal{A}_i^{m} \partial_m\mathcal{A}_i^1\partial_1 v^2+J\mathcal{A}_i^{m} \partial_m\mathcal{A}_2^1\partial_1 v^i)+ (\mathcal{A}_i^{m} \partial_m\mathcal{A}_i^{\alpha}\partial_{\alpha} v^2+\mathcal{A}_i^{m} \partial_m\mathcal{A}_2^{\alpha}\partial_{\alpha} v^i),\\
 &\mathfrak{g}_{33}:=(J^{-1}|\overrightarrow{\rm{a}_1}|^2-1)\partial_1 ^2 v^3+J\mathcal{A}_i^{1}\mathcal{A}_3^1\partial_1^2 v^i +2J\mathcal{A}_i^{1}\mathcal{A}_i^{\alpha}\partial_1 \partial_{\alpha} v^3 \\
&+(J\mathcal{A}_i^{\beta}\mathcal{A}_i^{\alpha}-\delta_i^{\beta}\delta_i^{\alpha})\partial_{\beta} \partial_{\alpha} v^3+(J\mathcal{A}_i^{m}\mathcal{A}_3^3-\delta_i^{m})\partial_m \partial_3 v^i+J\mathcal{A}_i^{\alpha}\mathcal{A}_3^1\partial_{\alpha} \partial_1 v^i+J\mathcal{A}_i^{m}\mathcal{A}_3^2\partial_m \partial_2 v^i\\
&+ J(\mathcal{A}_i^{m} \partial_m\mathcal{A}_i^1\partial_1 v^3+J\mathcal{A}_i^{m} \partial_m\mathcal{A}_3^1\partial_1 v^i)+J(\mathcal{A}_i^{m} \partial_m\mathcal{A}_i^{\alpha}\partial_{\alpha} v^3+\mathcal{A}_i^{m} \partial_m\mathcal{A}_3^{\alpha}\partial_{\alpha} v^i).
\end{split}
\end{equation*}

The system \eqref{eqns-pert-1} can be written as the linearized form
\begin{equation}\label{eqns-linear-1}
\begin{cases}
&\partial_t\xi=v,\\
 & \bar{\rho}(x_1)\partial_t v +\bar{\rho}(x_1)\nabla\,q-  \nabla\cdot\mathbb{S}(v)=\mathfrak{g},\\
 &\nabla\,\cdot  (\bar{\rho}(x_1)\,v)=g^{-1}\, \bar{\rho}'(x_1) \partial_tq+B_{1, i}^{h, j}\partial_jv^i\quad \text{in} \quad \Omega,\\
  &\bar{\rho}(0)q\, e_1- \mathbb{S}(v)\,e_1=
\left(
  \begin{array}{c}
   \bar{\rho}(0)\xi^1+\mathcal{B}(\partial_{\alpha}v, {Q}, \partial_t{Q})\\
    -\varepsilon\,(\mathcal{B}_{7, i}^{\alpha}\partial_{\alpha}v^{i}+\widetilde{\mathcal{B}}_{2}\partial_t{Q})\\
    - \varepsilon\,(\mathcal{B}_{8, i}^{\alpha}\partial_{\alpha}v^{i}+\widetilde{\mathcal{B}}_{3}\partial_t{Q})
  \end{array}
\right) \quad \text{on} \quad \Sigma_0,\\
&v|_{\Sigma_b}=0,
    \end{cases}
\end{equation}
where $\mathfrak{g}=(\mathfrak{g}_1, \mathfrak{g}_2, \mathfrak{g}_3)^T$ with
\begin{equation}\label{def-g-linear-1}
\begin{split}
 &\mathfrak{g}_k:=-\bar{\rho}(x_1)\widetilde{\mathcal{A}}_k^i\partial_iq +(\delta-\frac{2}{3}\varepsilon)\widetilde{\mathfrak{g}}_{kk}+\varepsilon\,\mathfrak{g}_{kk} \quad \text{with} \quad k=1, 2, 3.
\end{split}
\end{equation}

We split the components of $\nabla\cdot\mathbb{S}(v)$ into two parts
\begin{equation*}
  \begin{split}
    &(\nabla\cdot\mathbb{S}(v))^1=(\delta+\frac{4}{3}\varepsilon)\partial_1^2v^1+\mathcal{L}_{1}(\partial_h\nabla\,v),\quad (\nabla\cdot\mathbb{S}(v))^2=\varepsilon\partial_1^2v^2+\mathcal{L}_{2}(\partial_h\nabla\,v),\\
     &(\nabla\cdot\mathbb{S}(v))^3=\varepsilon\partial_1^2v^3+\mathcal{L}_{3}(\partial_h\nabla\,v),
     \end{split}
\end{equation*}
where
\begin{equation*}
  \begin{split}
 &\mathcal{L}_{1}(\partial_h\nabla\,v):=(\delta-\frac{2}{3}\varepsilon)\partial_1\nabla_h\cdot v^h+\varepsilon \partial_2(\partial_1v^2+\partial_2v^1)+\varepsilon\partial_3(\partial_1v^3+\partial_3v^1),\\
  &\mathcal{L}_{2}(\partial_h\nabla\,v):=
(\delta-\frac{2}{3}\varepsilon)\partial_2\nabla_h\cdot v^h+(\delta+\frac{1}{3}\varepsilon) \partial_2\partial_1v^1+2\varepsilon\partial_2^2v^2+\varepsilon\partial_3(\partial_2v^3+\partial_3v^2),\\
  &\mathcal{L}_{3}(\partial_h\nabla\,v):=
  (\delta-\frac{2}{3}\varepsilon)\partial_3\nabla_h\cdot v^h+(\delta+\frac{1}{3}\varepsilon) \partial_3\partial_1v^1+2\varepsilon\partial_3^2v^3+\varepsilon\partial_2(\partial_2v^3+\partial_3v^2).
      \end{split}
\end{equation*}
From this, the momentum equations in \eqref{eqns-linear-1} are equivalent to the equations
\begin{equation}\label{prt1-q-vh-1}
\begin{cases}
&  \bar{\rho}(x_1)\partial_1q-  (\delta+\frac{4}{3}\varepsilon)\partial_1^2v^1=-\bar{\rho}(x_1)\partial_t v^1+\mathcal{L}_{1}(\partial_h\nabla\,v)+\mathfrak{g}_1,\\\\
 &\bar{\rho}(x_1)\partial_{\alpha}q-  \varepsilon\partial_1^2v^{\alpha}=-\bar{\rho}(x_1)\partial_t v^{\alpha}-\mathcal{L}_{\alpha}(\partial_h\nabla\,v) +\mathfrak{g}_{\alpha}\quad (\text{with}\quad \alpha=2, 3).
     \end{cases}
\end{equation}

\renewcommand{\theequation}{\thesection.\arabic{equation}}
\setcounter{equation}{0}

\section{Preliminary estimates}\label{sect-tool}

Let us first recall some basic estimates, which will be heavily used in the rest of the paper.

\begin{lem}[\cite{Alinhac-1986}, Theorem 2.61 in \cite{BCD}]\label{lem-composition-1}
Let $f$ be a smooth function on $\mathbb{R}$ vanishing at $0$, $s_1$, $s_2$ be two positive
real number, $s_1\in (0, 1)$, $s_2>0$. If $u$ belongs to $\dot{H}^{s_1}(\mathbb{R}^2) \cap \dot{H}^{s_2}(\mathbb{R}^2)\cap L^{\infty}(\mathbb{R}^2)$, then so does $f\circ u$,
and we have
\begin{equation*}
\|f\circ u\|_{\dot{H}^{s_i}} \leq C(f', \|u\|_{L^{\infty}})\|u\|_{\dot{H}^{s_i}}\quad \text{for} \quad i=1, 2.
\end{equation*}
\end{lem}

\begin{lem}(\cite{Alinhac-1986, BCD})\label{prop-composition-2}
Let $f$ be a smooth function such that $f'(0)=0$ . Let $s > 0$. For any couple $(u, v)$ of functions in $B^{s}_{p, r}\cap L^{\infty}$,
the function $f\circ v-f\circ u$ then belongs to $\dot{H}^{s}\cap L^{\infty}$ and
\begin{equation*}\begin{split}
&\|f\circ v-f\circ u\|_{\dot{H}^{s}} \leq C_{f''}(\|u\|_{L^{\infty}},  \|v\|_{L^{\infty}})\\
&\qquad \times (\|u-v\|_{\dot{H}^{s}}\sup_{\tau\in [0, 1]}\|v+\tau(u-v)\|_{L^{\infty}}+\|u-v\|_{L^{\infty}}\sup_{\tau\in [0, 1]}\|v+\tau(u-v)\|_{\dot{H}^{s}}),
\end{split}
\end{equation*}
where $C_{f''}(\cdot, \cdot)$ is a uniformly continuous function on $[0, \infty) \times [0, \infty)$ depending only on $f''$.
\end{lem}

\begin{lem}[Classical product laws in Sobolev spaces \cite{BCD, Taylor-2000}]\label{lem-product-law-1}
Let $s_0>2$, $p_1,\,p_2 \in [2, +\infty)$, $q_1,\, q_2 \in (2, +\infty]$, $\frac{1}{p_1}+\frac{1}{q_1}=\frac{1}{p_2}+\frac{1}{q_2}=\frac{1}{2}$, there hold
\begin{equation}\label{product-law-1}
\begin{split}
&\|f\,g\|_{\dot{H}^{s_1+s_2-1}(\mathbb{R}^2)} \lesssim  \|f\|_{\dot{H}^{s_1}(\mathbb{R}^2)} \|g\|_{\dot{H}^{s_2}(\mathbb{R}^2)} \quad \forall \,\, s_1,\,s_2 \in (-1, 1),\, s_1+s_2>0,\\
&\|\dot{\Lambda}_h^{\sigma}(f\,g)\|_{L^2(\mathbb{R}^2_h)} \lesssim
\|\dot{\Lambda}_h^{\sigma}f\|_{L^{p_1}(\mathbb{R}^2_h)}\|g\|_{L^{q_1}(\mathbb{R}^2_h)}
+\|\dot{\Lambda}_h^{\sigma}g\|_{L^{p_2}(\mathbb{R}^2_h)}\|f\|_{L^{q_2}(\mathbb{R}^2_h)} \quad \forall \,\, \sigma>0,\\
&\|\dot{\Lambda}_h^{s_1}(f\,g)\|_{L^2(\mathbb{R}^2_h)} \lesssim
\|\dot{\Lambda}_h^{s_1}f\|_{L^2(\mathbb{R}^2_h)}
\|\dot{\Lambda}_h^{s_2}\Lambda_h^{s_0-1-s_2}g\|_{L^2(\mathbb{R}^2_h)} \, \forall\,\, s_1 \in (-1, 1],\,s_2\in (0, 1).
\end{split}
\end{equation}
\end{lem}

\begin{lem}[Embedding inequality \cite{BCD}]\label{lem-embedding-ineq-1}
For any $\sigma \in (0, 1)$, $s_0>2$, there holds
\begin{equation}\label{embedding-ineq-001}
\begin{split}
&\|f\|_{L^\infty(\mathbb{R}_h^2)} \lesssim  \|\dot{\Lambda}_h^{\sigma} \Lambda_h^{s_0-1-\sigma} f\|_{L^2(\mathbb{R}_h^2)}.
\end{split}
\end{equation}
\end{lem}

In what follows, we assume $s_0 \in [2+\sigma_0, s]$.

As a corollary, we find immediately the anisotropic Sobolev inequality as follows
\begin{col}[Anisotropic Sobolev inequality]\label{cor-embedding-ineq-1}
For any $\sigma \in (0, 1)$, $s_0>2$, there holds
\begin{equation}\label{embedding-ineq-1}
\begin{split}
&\|f\|_{L^\infty(\Omega)}^2 \lesssim  \|\dot{\Lambda}_h^{\sigma} \Lambda_h^{s_0-1-\sigma} f\|_{L^2(\Omega)}\|\dot{\Lambda}_h^{\sigma} \Lambda_h^{s_0-1-\sigma} f\|_{H^1(\Omega)}.
\end{split}
\end{equation}
\end{col}

\begin{lem}\label{lem-productlaw-1aa}
Let $s>2$, if $(\lambda,\,\sigma_0) \in (0, 1)$ satisfies $1-\lambda< \sigma_0\leq \min\{1-\frac{1}{2}\lambda,\,s-2\}$, $\ell_0:=\frac{\sigma_0+\lambda}{2}$, $\sigma>\sigma_0$, then there hold that
\begin{equation}\label{fg-product-1}
  \begin{split}
    &\|\dot{\Lambda}_h^{-\lambda}(f_1\,f_2)\|_{L^2} \lesssim \min\{ \|\dot{\Lambda}_h^{\sigma_0-1}f_1 \|_{L^2}\| \dot{\Lambda}_h^{2(1-\ell_0)}f_2\|_{H^1},\,\|\dot{\Lambda}_h^{\sigma_0-1}f_1 \|_{H^1}\| \dot{\Lambda}_h^{2(1-\ell_0)}f_2\|_{L^2}\},\\
  &\|\dot{\Lambda}_h^{-\lambda}(f_1\,f_2)\|_{H^1}\lesssim\|\dot{\Lambda}_h^{\sigma_0-1}f_1 \|_{H^1}\| \dot{\Lambda}_h^{2(1-\ell_0)}f_2\|_{H^1},\\
  & \|\dot{\Lambda}_h^{\sigma_0}(f_1\,f_2)\|_{L^2}\lesssim \min\{\|\dot{\Lambda}_h^{\frac{1+\sigma_0}{2}}f_1 \|_{H^1}\|\dot{\Lambda}_h^{\frac{1+\sigma_0}{2}} f_2\|_{L^2},\,\|\dot{\Lambda}_h^{\frac{1+\sigma_0}{2}}f_1 \|_{L^2}\|\dot{\Lambda}_h^{\frac{1+\sigma_0}{2}} f_2\|_{H^1}\},\\
&\|\dot{\Lambda}_h^{\sigma_0}(f_1\,f_2)\|_{H^1} \lesssim \|\dot{\Lambda}_h^{\frac{1+\sigma_0}{2}}f_1 \|_{H^1}\|\dot{\Lambda}_h^{\frac{1+\sigma_0}{2}} f_2\|_{H^1},\\
 & \|\dot{\Lambda}_h^{\sigma}(f_1\,f_2)\|_{L^2}\lesssim \min\{\|\dot{\Lambda}_h^{\sigma}f_1 \|_{H^1}\|\dot{\Lambda}_h^{\sigma_0} \Lambda_h^{s_0-1-\sigma_0}  f_2\|_{L^2}+\|\dot{\Lambda}_h^{\sigma_0} \Lambda_h^{s_0-1-\sigma_0}f_1\|_{H^1}\|\dot{\Lambda}_h^{\sigma} f_2\|_{L^2},\,\\
 &\qquad \qquad\qquad\|\dot{\Lambda}_h^{\sigma}f_1 \|_{L^2}\|\dot{\Lambda}_h^{\sigma_0} \Lambda_h^{s_0-1-\sigma_0}  f_2\|_{H^1}+\|\dot{\Lambda}_h^{\sigma_0} \Lambda_h^{s_0-1-\sigma_0}f_1\|_{H^1}\|\dot{\Lambda}_h^{\sigma} f_2\|_{L^2}\},\\
&\|\dot{\Lambda}_h^{\sigma}(f_1\,f_2)\|_{H^1} \lesssim \|\dot{\Lambda}_h^{\sigma}f_1 \|_{H^1}\|\dot{\Lambda}_h^{\sigma_0} \Lambda_h^{s_0-1-\sigma_0} f_2\|_{H^1}+\|\dot{\Lambda}_h^{\sigma_0} \Lambda_h^{s_0-1-\sigma_0}f_1\|_{H^1}\|\dot{\Lambda}_h^{\sigma} f_2\|_{H^1}.
  \end{split}
\end{equation}
\end{lem}
\begin{proof}
For the estimates of $\dot{\Lambda}_h^{-\lambda}(f_1\,f_2)$, thanks to Lemma \ref{lem-product-law-1} again, we find
\begin{equation*}
  \begin{split}
  &\|\dot{\Lambda}_h^{-\lambda}(f_1\,f_2)\|_{L^2} \lesssim \min\{ \|\dot{\Lambda}_h^{\sigma_0-1}f_1 \|_{L^\infty_{x_1}L^2_h}\| \dot{\Lambda}_h^{2(1-\ell_0)}  f_2\|_{L^2},\,\|\dot{\Lambda}_h^{\sigma_0-1}f_1\|_{L^2} \| \dot{\Lambda}_h^{2(1-\ell_0)}  f_2\|_{L^\infty_{x_1}L^2_h}\}
  \end{split}
\end{equation*}
and
\begin{equation*}
  \begin{split}
&\|\dot{\Lambda}_h^{-\lambda}(f_1\,f_2)\|_{H^1}\lesssim \|\dot{\Lambda}_h^{-\lambda}(f_1\,f_2)\|_{L^2}
+\|\dot{\Lambda}_h^{-\lambda}(\nabla\,f_1\,f_2+f_1\nabla\,f_2)\|_{L^2}\\
&\lesssim \|\dot{\Lambda}_h^{\sigma_0-1}f_1 \|_{L^\infty_{x_1}L^2_h}\| \dot{\Lambda}_h^{2(1-\ell_0)}  f_2\|_{L^2}+\|\dot{\Lambda}_h^{\sigma_0-1}\nabla\,f_1 \|_{L^2}\| \dot{\Lambda}_h^{2(1-\ell_0)} f_2\|_{L^\infty_{x_1}L^2_h}\\
&\qquad\qquad+\|\dot{\Lambda}_h^{\sigma_0-1}f_1\|_{L^\infty_{x_1}L^2_h}\|\dot{\Lambda}_h^{2(1-\ell_0)}\nabla\, f_2\|_{L^2}.
  \end{split}
\end{equation*}
Due to the Sobolev embedding $H^1\hookrightarrow L^\infty_{x_1}L^2_h$, we get
\begin{equation*}
  \begin{split}
  &\|\dot{\Lambda}_h^{-\lambda}(f_1\,f_2)\|_{L^2} \lesssim \min\{ \|\dot{\Lambda}_h^{\sigma_0-1}f_1 \|_{H^1}\| \dot{\Lambda}_h^{2(1-\ell_0)} f_2\|_{L^2},\,\|\dot{\Lambda}_h^{\sigma_0-1}f_1\|_{L^2} \|  \dot{\Lambda}_h^{2(1-\ell_0)}  f_2\|_{H^1}\},\\
&\|\dot{\Lambda}_h^{-\lambda}(f_1\,f_2)\|_{H^1}\lesssim   \|\dot{\Lambda}_h^{\sigma_0-1}f_1 \|_{H^1}\| \dot{\Lambda}_h^{2(1-\ell_0)}f_2\|_{H^1},
  \end{split}
\end{equation*}
that is, the first two inequalities in \eqref{fg-product-1} hold true.

Applying Lemma \ref{lem-product-law-1} directly ensures that
\begin{equation*}
  \begin{split}
  & \|\dot{\Lambda}_h^{\sigma_0}(f_1\,f_2)\|_{L^2}\lesssim \min\{\|\dot{\Lambda}_h^{\frac{1+\sigma_0}{2}}f_1 \|_{H^1}\|\dot{\Lambda}_h^{\frac{1+\sigma_0}{2}} f_2\|_{L^2},\,\|\dot{\Lambda}_h^{\frac{1+\sigma_0}{2}}f_1 \|_{L^2}\|\dot{\Lambda}_h^{\frac{1+\sigma_0}{2}} f_2\|_{H^1}\},\\
&\|\dot{\Lambda}_h^{\sigma_0}(f_1\,f_2)\|_{H^1} \lesssim \|\dot{\Lambda}_h^{\frac{1+\sigma_0}{2}}f_1 \|_{H^1}\|\dot{\Lambda}_h^{\frac{1+\sigma_0}{2}} f_2\|_{H^1}.
  \end{split}
\end{equation*}
We use Lemma \ref{lem-product-law-1} to get
\begin{equation*}
  \begin{split}
\|\dot{\Lambda}_h^{\sigma}(f_1\,f_2)\|_{L^2}\lesssim \min\{&\|\dot{\Lambda}_h^{\sigma}f_1 \|_{L^\infty_{x_1}L^2_h}\| f_2\|_{L^2_{x_1}L^\infty_h}+\|f_1\|_{L^\infty}\|\dot{\Lambda}_h^{\sigma} f_2\|_{L^2},\\
&\|\dot{\Lambda}_h^{\sigma}f_1 \|_{L^2}\| f_2\|_{L^\infty}+\|f_1\|_{L^\infty}\|\dot{\Lambda}_h^{\sigma} f_2\|_{L^2}\}
  \end{split}
\end{equation*}
and
\begin{equation*}
  \begin{split}
\|\dot{\Lambda}_h^{\sigma}\nabla(f_1\,f_2)\|_{L^2}&\lesssim  \|\dot{\Lambda}_h^{\sigma}\nabla\,f_1  \|_{L^2}\| f_2\|_{L^\infty}+\|\nabla\,f_1\|_{L^2_{x_1}L^\infty_h}\|\dot{\Lambda}_h^{\sigma} f_2\|_{L^\infty_{x_1}L^2_h}
\\
&\qquad +\|\dot{\Lambda}_h^{\sigma}f_1 \|_{L^\infty_{x_1}L^2_h}\| \nabla\,f_2\|_{L^2_{x_1}L^\infty_h}+
\|f_1\|_{L^\infty}\|\dot{\Lambda}_h^{\sigma} \nabla\,f_2\|_{L^2}.
  \end{split}
\end{equation*}
which along with Corollary \ref{cor-embedding-ineq-1} and the Sobolev embedding $H^1(\Omega)\hookrightarrow L^\infty_{x_1}L^2_h(\Omega)$ gives rise to the last two inequalities in \eqref{fg-product-1}.
This completes the proof of Lemma \ref{lem-productlaw-1aa}.
\end{proof}

Motivated by Lemma 2.7 in \cite{Beale-1981}, we may get the following version of Korn's type inequality for the equilibrium domain $\Omega$.
\begin{lem}[Korn's lemma, Lemma 2.7 in \cite{Beale-1981}]\label{lem-korn-2}
Let $\Omega$ be the equilibrium domain given in \eqref{def-domain-1}, then there exists a positive constant $C$, independent of $u$, such that
\begin{equation}\label{korn-comp-2}
\|u\|_{H^1(\Omega)} \leq C\,\|\mathbb{D}^0(u)\|_{L^2(\Omega)}
\end{equation}
for all $u \in H^1(\Omega)$ with $u|_{\Sigma_b}=0$.
\end{lem}

We state now some estimates on the harmonic extension operator $\mathcal{H}$ defined in \eqref{harmonic-ext-xi-1-1}.
\begin{lem}[\cite{Gui-2020}]\label{lem-est-harmonic-ext-1}
Let $s\in \mathbb{R}$ and $r\geq 2$. For the harmonic extension $\mathcal{H}(f)$ of $f=f(t, x_1, x_h)$ (with $t \in\mathbb{R}^+$, $( x_1, x_h) \in \Omega$), there hold
\begin{equation*}
\begin{split}
 &\|\mathcal{H}(f)\|_{H^1(\Omega)} \lesssim \|f\|_{H^{\frac{1}{2}}(\Sigma_0)}, \quad \|\dot{\Lambda}_h^{s}\mathcal{H}(f)\|_{H^1(\Omega)} \lesssim \|\dot{\Lambda}_h^{s}f\|_{H^{\frac{1}{2}}(\Sigma_0)}, \\
 &\|\dot{\Lambda}_h^{s}\mathcal{H}(f)\|_{H^{2}(\Omega)} \lesssim \|\dot{\Lambda}_h^{s}f\|_{H^{\frac{3}{2}}(\Sigma_0)},\quad\|\mathcal{H}(f)\|_{H^{r}(\Omega)} \lesssim \|f\|_{H^{r-\frac{1}{2}}(\Sigma_0)},\\
 &\|\partial_t\mathcal{H}(f)\|_{H^1(\Omega)} \lesssim \|\partial_tf\|_{H^{\frac{1}{2}}(\Sigma_0)}.
\end{split}
\end{equation*}
\end{lem}

Let us now turn to estimate the coefficients of nonlinear terms in the system \eqref{eqns-pert-1}. First, thanks to Lemma \ref{embedding-ineq-1}, we get
\begin{equation}\label{size-xi-2}
\begin{split}
&\|\nabla\,\xi(t)\|_{L^\infty(\Omega)} \leq  C\| \dot{\Lambda}_h^{\sigma_0} \Lambda_h^{s_0-1-\sigma_0} \nabla\,\xi(t)\|_{H^1(\Omega)},
\end{split}
\end{equation}
which leads to
\begin{lem}\label{lem-size-xi}
There exists a positive constant $C$, independent of $\xi$ and $t$, such that
\begin{equation}\label{size-xi-1}
\|\nabla\xi(t)\|_{L^\infty(\Omega)} \leq C\,E_{s_0}^{\frac{1}{2}}(t).
\end{equation}
\end{lem}
Lemma \ref{lem-size-xi} tells us how the size of $\xi$ can control the variation of the deformation $D\xi(t)$ of the displacement $\xi$.

With Lemma \ref{lem-size-xi} in hand, we may prove that, the entries of the matrices $\widetilde{\mathcal{A}}:=\mathcal{A}-I$,  $J\mathcal{A}-I$, as well as the Jacobian $J$ of the flow map and its inverse $J^{-1}$, are anisotropic when we consider their low horizontal regularities, which is stated as follows.

\begin{lem}\label{lem-est-aij-1}
Let $s_0>2$, $i, \, j=1, 2, 3$, $\sigma_0 \in (0, 1)$, if $E_{s_0}(t) \leq 1$ for all time $t\in [0, T]$, then there hold
\begin{equation}\label{est-aij-J-1}
  \begin{split}
&(1). \|(J^{\pm 1}-1,\, \widetilde{\mathcal{A}}_{i}^{j},\,J\widetilde{\mathcal{A}}_i^j)\|_{L^\infty}\lesssim E_{s_0}^{\frac{1}{2}};\\
&(2). \,\|\dot{\Lambda}_h^{\sigma}\,(J^{\pm 1}-1,\, \widetilde{\mathcal{A}}_{i}^{j},\,J\widetilde{\mathcal{A}}_i^j)\|_{L^2}\lesssim E_{\sigma}^{\frac{1}{2}}, \\
&\qquad\qquad \qquad \qquad \|\dot{\Lambda}_h^{\sigma}\,(J^{\pm 1}-1,\, \widetilde{\mathcal{A}}_{i}^{j},\,J\widetilde{\mathcal{A}}_i^j)\|_{H^1}\lesssim E_{\sigma+1}^{\frac{1}{2}}\quad(\forall \, \sigma \geq \sigma_0);\\
&(3). \|\dot{\Lambda}_h^{\sigma}(J^{\pm 1}-1,\,\widetilde{\mathcal{A}}_{i}^{j},\,J\widetilde{\mathcal{A}}_i^j)\|_{H^1}\lesssim E_{\sigma+1}^{\frac{1}{2}} \\
&\qquad \qquad \qquad \quad \forall \, \sigma \in [\sigma_0-1, \sigma_0] \qquad \text{with}  \quad(i, j)\neq  (1, 2), (1, 3).
            \end{split}
\end{equation}
\end{lem}

\begin{proof}
For the first third inequalities in \eqref{est-aij-J-1}, we need only prove it for the term $a_{12}$, and the same proof remains valid for the others.

Indeed, for the first second in  \eqref{est-aij-J-1},
due to the expression of $a_{ij}$, we obtain from Lemma \ref{lem-product-law-1} that
\begin{equation}\label{est-aij-infty-1}
  \begin{split}
    & \|a_{12}\|_{L^\infty}\lesssim \|\partial_1\xi^2\|_{L^\infty}
  +\|\nabla\xi\|_{L^\infty}^2\lesssim E_{s_0}^{\frac{1}{2}}+E_{s_0} \lesssim  E_{s_0}^{\frac{1}{2}}.
            \end{split}
\end{equation}
Notice that
\begin{equation}\label{a12-Hk-1}
  \begin{split}
  \|\dot{\Lambda}_h^{\sigma}a_{12}\|_{L^2}&\lesssim \|\dot{\Lambda}_h^{\sigma} \partial_1\xi^2\|_{L^2}
+\|\dot{\Lambda}_h^{\sigma}(\nabla\xi\otimes\nabla\xi) \|_{L^2}+\|\dot{\Lambda}_h^{\sigma}(\nabla\xi\otimes\nabla^2\xi) \|_{L^2},\\
\|\dot{\Lambda}_h^{\sigma}a_{12}\|_{H^1} &\lesssim\|\dot{\Lambda}_h^{\sigma}a_{12}\|_{L^2}
+\|\dot{\Lambda}_h^{\sigma}\nabla\,a_{12}\|_{L^2}\\
&\lesssim \|\dot{\Lambda}_h^{\sigma} \partial_1\xi^2\|_{H^1}
+\|\dot{\Lambda}_h^{\sigma}(\nabla\xi\otimes\nabla\xi) \|_{L^2}+\|\dot{\Lambda}_h^{\sigma}(\nabla\xi\otimes\nabla^2\xi) \|_{L^2},
  \end{split}
\end{equation}
for $\sigma \geq \sigma_0$, thanks to \eqref{lem-product-law-1}, we may get
\begin{equation*}
  \begin{split}
  &\|\dot{\Lambda}_h^{\sigma}(\nabla\xi\otimes\nabla\xi) \|_{L^2}\lesssim \|\dot{\Lambda}_h^{\sigma}\nabla\xi\|_{L^2}\|\nabla\xi\|_{L^\infty}\lesssim E_{\sigma}^{\frac{1}{2}} E_{s_0}^{\frac{1}{2}},\\
&\|\dot{\Lambda}_h^{\sigma}(\nabla\xi\otimes\nabla\xi) \|_{L^2}+\|\dot{\Lambda}_h^{\sigma}(\nabla\xi\otimes\nabla^2\xi) \|_{L^2}\lesssim \|\dot{\Lambda}_h^{\sigma}\nabla\xi\|_{L^2}\|\nabla\xi\|_{L^\infty}\\
&\qquad\qquad\qquad+\|\dot{\Lambda}_h^{\sigma}\nabla\xi\|_{L^\infty_{x_1}L^2_h}
\|\nabla^2\xi\|_{L^2_{x_1}L^\infty_h}+\|\nabla\xi\|_{L^\infty}\|\dot{\Lambda}_h^{\sigma}\nabla^2\xi\|_{L^2}\lesssim E_{\sigma+1}^{\frac{1}{2}} E_{s_0}^{\frac{1}{2}},
  \end{split}
\end{equation*}
which along with \eqref{a12-Hk-1} follows $\|\dot{\Lambda}_h^{\sigma}a_{12}\|_{L^2}\lesssim E_{\sigma}^{\frac{1}{2}}$ and $\|\dot{\Lambda}_h^{\sigma}a_{12}\|_{H^1}\lesssim E_{\sigma+1}^{\frac{1}{2}}$.

For the third inequality in \eqref{est-aij-J-1}, we need only prove it for the term $a_{23}$ with $\sigma=\sigma_0-1$, and the same proof remains valid for the others. In fact, one has
\begin{equation}\label{a23-Hk-1}
  \begin{split}
\|\dot{\Lambda}_h^{\sigma_0-1}a_{23}\|_{H^1} &\lesssim \|\dot{\Lambda}_h^{\sigma_0-1} \partial_2\xi^3\|_{H^1}
+\|\dot{\Lambda}_h^{\sigma_0-1}(\nabla_h\xi\otimes\nabla\xi) \|_{L^2}\\
&\qquad\qquad +\|\dot{\Lambda}_h^{\sigma_0-1}(\nabla_h\xi\otimes\nabla^2\xi) \|_{L^2}+\|\dot{\Lambda}_h^{\sigma_0-1}(\nabla\nabla_h\xi\otimes\nabla\xi) \|_{L^2}.
  \end{split}
\end{equation}
Due to \eqref{fg-product-1}, we may get
\begin{equation*}
  \begin{split}
  &\|\dot{\Lambda}_h^{\sigma_0-1}(\nabla_h\xi\otimes\nabla\xi) \|_{L^2}+\|\dot{\Lambda}_h^{\sigma_0-1}(\nabla_h\xi\otimes\nabla^2\xi) \|_{L^2}+\|\dot{\Lambda}_h^{\sigma_0-1}(\nabla\nabla_h\xi\otimes\nabla\xi) \|_{L^2}\\
  &\lesssim \|\dot{\Lambda}_h^{\sigma_0-1}\nabla_h\xi\|_{H^1}
  \|\dot{\Lambda}_h^{\sigma_0}{\Lambda}_h^{s_0-1-\sigma_0}\nabla\xi\|_{H^1}\lesssim E_{\sigma_0}^{\frac{1}{2}} E_{s_0}^{\frac{1}{2}},
  \end{split}
\end{equation*}
which along with \eqref{a23-Hk-1} follows $\|\dot{\Lambda}_h^{\sigma_0-1}a_{23}\|_{H^1}\lesssim E_{\sigma_0}^{\frac{1}{2}}$.

The proof of the lemma is thus completed.
\end{proof}

With Lemma \ref{lem-est-aij-1} in hand, we estimate $\mathcal{B}_{j, i}^{\alpha}$ and $\mathcal{B}_{j, i}^{\alpha}\partial_\alpha\,v^i$ with $j=1, \cdots, 9.$ For simplicity, we denote $\mathcal{B}_{j, i}^{\alpha}$ by $B=B(\widetilde{\mathcal{A}})$.
\begin{lem}\label{lem-est-B-Bv-1}
Let $2 \leq  k\in \mathbb{N}$, $B=B(\widetilde{\mathcal{A}})$, if $(\lambda,\,\sigma_0) \in (0, 1)$ satisfies $1-\lambda< \sigma_0\leq 1-\frac{1}{2}\lambda$, and $E_{s_0}(t) \leq 1$ for all $t\in [0, T]$, then there hold that $\forall\,t\in [0, T]$
\begin{equation}\label{est-B-k-1}
  \begin{split}
  &\|\dot{\Lambda}_h^{\sigma}B(\widetilde{\mathcal{A}})\|_{H^1}\lesssim E_{\sigma+1}^{\frac{1}{2}}\quad\text{if} \quad  \sigma \geq  \sigma_0,\\
  & \|\dot{\Lambda}_h^{\sigma}(B(\widetilde{\mathcal{A}})\nabla_h\,v)\|_{H^1}\lesssim E_{s_0}^{\frac{1}{2}}\dot{\mathcal{D}}_{s_0}^{\frac{1}{2}} \quad\text{if} \quad  \sigma \in [-\lambda, \sigma_0],\\
  &\|\dot{\Lambda}_h^{\sigma}(B(\widetilde{\mathcal{A}})\nabla\,v)\|_{H^1}\lesssim E_{\sigma+1}^{\frac{1}{2}}\dot{\mathcal{D}}_{s_0}^{\frac{1}{2}}+E_{s_0}^{\frac{1}{2}} \dot{\mathcal{D}}_{\sigma+1}^{\frac{1}{2}}\quad\text{if} \quad  \sigma \geq \sigma_0,\\
 &\|\dot{\Lambda}_h^{\sigma}\partial_tB(\widetilde{\mathcal{A}})\|_{H^1}\lesssim \dot{\mathcal{D}}_{\sigma+1}^{\frac{1}{2}}+ E_{\sigma+1}^{\frac{1}{2}}\dot{\mathcal{D}}_{s_0}^{\frac{1}{2}}\quad\text{if} \quad  \sigma \geq \sigma_0.
  \end{split}
\end{equation}
If, moreover, all the terms in $B(\widetilde{\mathcal{A}})$ have at least one $\partial_{\alpha}\xi^i$-type factor with $i=1, 2, 3$ and $\alpha=2, 3$, then we have
\begin{equation}\label{est-sigma-B-2}
  \begin{split}
  &\|\dot{\Lambda}_h^{\sigma}B(\widetilde{\mathcal{A}})\|_{H^1}\lesssim E_{\sigma+1}^{\frac{1}{2}}\,\,\text{if}\, \, \sigma \geq  \sigma_0-1,\quad\|\dot{\Lambda}_h^{\sigma}(B(\widetilde{\mathcal{A}})\nabla\,v)\|_{H^1}\lesssim E_{s_0}^{\frac{1}{2}}\dot{\mathcal{D}}_{s_0}^{\frac{1}{2}} \,\,\text{if} \,\,  \sigma \in [-\lambda, \sigma_0].
  \end{split}
\end{equation}
\end{lem}
\begin{proof}
In view of Lemma  \ref{lem-est-aij-1}, it immediately follows the first inequalities in \eqref{est-B-k-1} and \eqref{est-sigma-B-2}. Let's now focus on the product estimates of $B(\widetilde{\mathcal{A}})\nabla\,v$ and $B(\widetilde{\mathcal{A}})\nabla_h\,v$.

We first use the product laws \eqref{fg-product-1} to get
\begin{equation*}
  \begin{split}
&\|\dot{\Lambda}_h^{-\lambda}(B(\widetilde{\mathcal{A}})\nabla_h\,v)\|_{H^1}\lesssim  \|\dot{\Lambda}_h^{\sigma_0-1}\nabla_h\,v\|_{H^1}
\|\dot{\Lambda}_h^{\sigma_0} \Lambda_h^{s_0-1-\sigma_0} B(\widetilde{\mathcal{A}})\|_{H^1}
\lesssim E_{s_0}^{\frac{1}{2}} \dot{\mathcal{D}}_{\sigma_0}^{\frac{1}{2}}.
  \end{split}
\end{equation*}
Similar estimates hold for $\|\dot{\Lambda}_h^{\sigma}(B(\widetilde{\mathcal{A}})\nabla_h\,v)\|_{H^1}$ with $\sigma \in [-\lambda, s_0]$.

On the other hand, thanks to \eqref{fg-product-1} again, for $ \sigma \geq \sigma_0$, it can be found that
\begin{equation}\label{est-Bv-large-4}
  \begin{split}
\|\dot{\Lambda}_h^{\sigma}(B(\widetilde{\mathcal{A}})\nabla_h\,v)\|_{H^1}&\lesssim \|\dot{\Lambda}_h^{\sigma}B(\widetilde{\mathcal{A}}) \|_{H^1}\|\dot{\Lambda}_h^{\sigma_0} \Lambda_h^{s_0-1-\sigma_0} \nabla_h\,v\|_{H^1}\\
&\quad+\|\dot{\Lambda}_h^{\sigma_0} \Lambda_h^{s_0-1-\sigma_0} B(\widetilde{\mathcal{A}})\|_{H^1}\|\dot{\Lambda}_h^{\sigma} \nabla_h\,v\|_{H^1}\lesssim E_{\sigma+1}^{\frac{1}{2}}\dot{\mathcal{D}}_{s_0}^{\frac{1}{2}} +E_{s_0}^{\frac{1}{2}} \dot{\mathcal{D}}_{\sigma+1}^{\frac{1}{2}} ,
  \end{split}
\end{equation}
which yields the third inequality in \eqref{est-B-k-1}.

Since $\partial_tB(\widetilde{\mathcal{A}})$ has the form like $\nabla\,v+\,B_1(\widetilde{\mathcal{A}})\nabla\,v$ with smooth function $B_1(\cdot)$, we immediately get the forth inequality in \eqref{est-B-k-1} from \eqref{est-Bv-large-4}.

Finally, for the second inequality in \eqref{est-sigma-B-2}, due to the same argument above, we need only prove it in the case of $\sigma=-\lambda$. Since all the terms in $B(\widetilde{\mathcal{A}})$ have at least one $\partial_{\alpha}\xi^i$-type factor, we get from the product law \eqref{fg-product-1} that
\begin{equation*}
  \begin{split}
\|\dot{\Lambda}_h^{-\lambda}(B(\widetilde{\mathcal{A}})\nabla \,v)\|_{H^1} \lesssim \|\dot{\Lambda}_h^{\sigma_0-1}B(\widetilde{\mathcal{A}}) \|_{H^1}\|\dot{\Lambda}_h^{2-\sigma_0-\lambda} \nabla \,v\|_{H^1}\lesssim E_{s_0}^{\frac{1}{2}} \dot{\mathcal{D}}_{s_0}^{\frac{1}{2}},
  \end{split}
\end{equation*}
so the third inequality in \eqref{est-sigma-B-2} holds true. This finishes the proof of Lemma \ref{lem-est-B-Bv-1}.
\end{proof}

Let us deal with the terms with respect to the source term $\mathfrak{g}$.

\begin{lem}\label{lem-est-g-1}
If $s_0>2$, $\lambda,\,\sigma_0 \in (0, 1)$ satisfies $1-\lambda< \sigma_0\leq 1-\frac{1}{2}\lambda$, $j=1, 2, 3$, and $E_{s_0}(t) \leq 1$ for all $t\in [0, T]$, then there hold that $\forall\,t\in [0, T]$
\begin{equation}\label{est-g-2}
\begin{split}
&\|\dot{\Lambda}_h^{-\lambda}\,(\widetilde{\mathfrak{g}}_{jj},\,\mathfrak{g}_{jj})\|_{L^2}
 +\|\dot{\Lambda}_h^{\sigma_0}\,(\widetilde{\mathfrak{g}}_{jj},\,\mathfrak{g}_{jj})\|_{L^2}\lesssim \,E_{s_0}^{\frac{1}{2}} \dot{\mathcal{D}}_{\sigma_0+1}^{\frac{1}{2}},\\
 &\|\dot{\Lambda}_h^{\sigma}\,(\widetilde{\mathfrak{g}}_{jj},\,\mathfrak{g}_{jj})\|_{L^2}\lesssim\,E_{s_0}^{\frac{1}{2}} \dot{\mathcal{D}}_{\sigma+1}^{\frac{1}{2}} \quad \forall \,\,\sigma \in [\sigma_0,\,s_0-1),\\
  &\|\dot{\Lambda}_h^{\sigma}\,(\widetilde{\mathfrak{g}}_{jj},\,\mathfrak{g}_{jj})\|_{L^2}\lesssim E_{\sigma+1}^{\frac{1}{2}} \dot{\mathcal{D}}_{s_0}^{\frac{1}{2}}+ E_{s_0}^{\frac{1}{2}} \dot{\mathcal{D}}_{\sigma+1}^{\frac{1}{2}} \quad \forall \,\, \sigma  \geq s_0-1,\\
   &\|\dot{\Lambda}_h^{-\lambda} \mathfrak{g}\|_{L^2}\lesssim E_{s_0}^{\frac{1}{2}}\,
(\dot{\mathcal{E}}_{\sigma_0+\frac{1}{2}}^{\frac{1}{2}}  + \dot{\mathcal{D}}_{\sigma_0+1}^{\frac{1}{2}}),\quad\|\dot{\Lambda}_h^{\sigma}\, \mathfrak{g} \|_{L^2}
 \lesssim\,E_{s_0}^{\frac{1}{2}} \dot{\mathcal{D}}_{\sigma+1}^{\frac{1}{2}} \quad \forall \,\,\sigma \in [\sigma_0,\,s_0-1),\\
  &\|\dot{\Lambda}_h^{\sigma}\,\mathfrak{g}\|_{L^2}\lesssim E_{\sigma+1}^{\frac{1}{2}} \dot{\mathcal{D}}_{s_0}^{\frac{1}{2}}+ E_{s_0}^{\frac{1}{2}} \dot{\mathcal{D}}_{\sigma+1}^{\frac{1}{2}} \quad \forall \,\, \sigma  \geq s_0-1.
\end{split}
\end{equation}
\end{lem}
\begin{proof}
We first bound the terms $\widetilde{\mathfrak{g}}_{jj}$ and $\mathfrak{g}_{jj}$ with $j=1,2,3$ of $\mathfrak{g}$ in \eqref{est-g-2}. For this, we split the terms of $\widetilde{\mathfrak{g}}_{jj}$ and $\mathfrak{g}_{jj}$ in \eqref{def-g-linear-1} into two types: the one is of the second-order derivative terms of $v$ and the other is of the first-order derivative terms of $v$.

 For the first one, notice that all the coefficients $B(\widetilde{\mathcal{A}})$ in the $B(\widetilde{\mathcal{A}})\partial_1^2v$-type terms have at least one $\partial_h\xi$-type factor, the product law \eqref{fg-product-1} ensures
\begin{equation*}
\begin{split}
 &\|\dot{\Lambda}_h^{-\lambda}\,(B(\widetilde{\mathcal{A}})\partial_1^2v)\|_{L^2}\lesssim \|\dot{\Lambda}_h^{1-\sigma_0-\lambda} B(\widetilde{\mathcal{A}}) \|_{H^1}
 \|\dot{\Lambda}_h^{\sigma_0}\partial_1^2v\|_{L^2},\\
  &\|\dot{\Lambda}_h^{\sigma}\,(B(\widetilde{\mathcal{A}})\partial_1^2v)\|_{L^2}\lesssim \|\dot{\Lambda}_h^{\sigma_0}  \Lambda_h^{s_0-1-\sigma_0}B(\widetilde{\mathcal{A}}) \|_{H^1}
 \|\dot{\Lambda}_h^{\sigma}\partial_1^2v\|_{L^2}\quad \forall\,\sigma \in [\sigma_0,\,1),
\end{split}
\end{equation*}
and for $\sigma \in [1,\,s_0-1)$
\begin{equation*}
\begin{split}
   &\|\dot{\Lambda}_h^{\sigma}\,(B(\widetilde{\mathcal{A}})\partial_1^2v)\|_{L^2}\lesssim \|\dot{\Lambda}_h^{\sigma-1}\,(B(\widetilde{\mathcal{A}})\partial_h\partial_1^2v)\|_{L^2}
   +\|\dot{\Lambda}_h^{\sigma-1}\,(\partial_hB(\widetilde{\mathcal{A}})\partial_1^2v)\|_{L^2}\\
   &\lesssim \|\dot{\Lambda}_h^{\sigma_0}  \Lambda_h^{s_0-1-\sigma_0}B(\widetilde{\mathcal{A}}) \|_{H^1}
 \|\dot{\Lambda}_h^{\sigma}\partial_1^2v\|_{L^2}+\|\dot{\Lambda}_h^{s_0-2}  \partial_hB(\widetilde{\mathcal{A}}) \|_{H^1}
 \|\dot{\Lambda}_h^{2+\sigma-s_0}\partial_1^2v\|_{L^2}.
\end{split}
\end{equation*}
While for any $\sigma \geq s_0-1$
\begin{equation*}
\begin{split}
  &\|\dot{\Lambda}_h^{\sigma}\,(B(\widetilde{\mathcal{A}})\partial_1^2v)\|_{L^2}\lesssim \|\dot{\Lambda}_h^{\sigma}B(\widetilde{\mathcal{A}}) \|_{H^1}\|\dot{\Lambda}_h^{\sigma_0} \Lambda_h^{s_0-1-\sigma_0} \partial_1^2v\|_{L^2}\\
  &\qquad\qquad\qquad\qquad\qquad+\|\dot{\Lambda}_h^{\sigma_0} \Lambda_h^{s_0-1-\sigma_0}B(\widetilde{\mathcal{A}})\|_{H^1}\|\dot{\Lambda}_h^{\sigma} \partial_1^2v\|_{L^2}.
\end{split}
\end{equation*}
Thanks to Lemma \ref{lem-est-aij-1} again, we prove
\begin{equation}\label{est-g-2-6a}
\begin{split}
 &\|\dot{\Lambda}_h^{-\lambda}\,(B(\widetilde{\mathcal{A}})\partial_1^2v)\|_{L^2}\lesssim   E_{s_0}^{\frac{1}{2}} \dot{\mathcal{D}}_{\sigma_0+1}^{\frac{1}{2}},\\
 & \|\dot{\Lambda}_h^{\sigma}\,(B(\widetilde{\mathcal{A}})\partial_1^2v)\|_{L^2} \lesssim  E_{s_0}^{\frac{1}{2}} \dot{\mathcal{D}}_{\sigma+1}^{\frac{1}{2}}\quad \forall\,\sigma \in [\sigma_0,\,s_0-1),\\
 & \|\dot{\Lambda}_h^{\sigma}\,(B(\widetilde{\mathcal{A}})\partial_1^2v)\|_{L^2}\lesssim \, E_{\sigma+1}^{\frac{1}{2}} \dot{\mathcal{D}}_{s_0}^{\frac{1}{2}}+ E_{s_0}^{\frac{1}{2}} \dot{\mathcal{D}}_{\sigma+1}^{\frac{1}{2}}\quad \forall\,\sigma \geq s_0-1.
\end{split}
\end{equation}

While for the terms like $B(\widetilde{\mathcal{A}})\partial_h\nabla\,v$, thanks to Lemmas \ref{lem-productlaw-1aa} and \ref{lem-est-aij-1}, one has
\begin{equation*}
\begin{split}
 &\|\dot{\Lambda}_h^{-\lambda}\,(B(\widetilde{\mathcal{A}})\partial_h\nabla\,v)\|_{L^2}\lesssim \|\dot{\Lambda}_h^{2-\sigma_0-\lambda}B(\widetilde{\mathcal{A}}) \|_{L^\infty_{x_1}L^2_h}
 \|\dot{\Lambda}_h^{\sigma_0-1}\,\partial_h\nabla\,v\|_{L^2},\\
   &\|\dot{\Lambda}_h^{\sigma}\,(B(\widetilde{\mathcal{A}})\partial_h\nabla\,v)\|_{L^2}\lesssim \|\dot{\Lambda}_h^{\sigma_0}  \Lambda_h^{s_0-1-\sigma_0}B(\widetilde{\mathcal{A}})\|_{H^1}
 \|\dot{\Lambda}_h^{\sigma}\partial_h\nabla\,v\|_{L^2}\quad \forall\,\sigma \in [\sigma_0,\,1),\\
 &\|\dot{\Lambda}_h^{\sigma}\,(B(\widetilde{\mathcal{A}})\partial_h\nabla\,v)\|_{L^2}\lesssim \|\dot{\Lambda}_h^{\sigma_0}  \Lambda_h^{s_0-1-\sigma_0}B(\widetilde{\mathcal{A}}) \|_{H^1}
 \|\dot{\Lambda}_h^{\sigma}\partial_h\nabla\,v\|_{L^2}\\
   &\qquad\qquad+\|\dot{\Lambda}_h^{s_0-2}  \partial_hB(\widetilde{\mathcal{A}}) \|_{H^1}
 \|\dot{\Lambda}_h^{2+\sigma-s_0}\partial_h\nabla\,v\|_{L^2}\quad \forall\,\sigma \in [1,\,s_0-1),\\
  &\|\dot{\Lambda}_h^{\sigma}\,(B(\widetilde{\mathcal{A}})\partial_h\nabla\,v)\|_{L^2}\lesssim \|\dot{\Lambda}_h^{\sigma}B(\widetilde{\mathcal{A}}) \|_{H^1}\|\dot{\Lambda}_h^{\sigma_0} \Lambda_h^{s_0-1-\sigma_0}\partial_h\nabla\,v\|_{L^2}\\
  &\qquad\qquad\qquad\qquad\qquad+\|\dot{\Lambda}_h^{\sigma_0} \Lambda_h^{s_0-1-\sigma_0}B(\widetilde{\mathcal{A}})\|_{H^1}\|\dot{\Lambda}_h^{\sigma} \partial_h\nabla\,v\|_{L^2}\quad \forall\,\sigma \geq s_0-1,
\end{split}
\end{equation*}
from this and the fact $1>2-\sigma_0-\lambda \geq \sigma_0$ and the Sobolev embedding $H^1(\Omega) \hookrightarrow L^\infty_{x_1}L^2_h(\Omega)$, we arrive at
\begin{equation}\label{est-g-2-6b}
\begin{split}
 &\|\dot{\Lambda}_h^{-\lambda}\,(B(\widetilde{\mathcal{A}})\partial_h\nabla\,v)\|_{L^2}\lesssim   E_{s_0}^{\frac{1}{2}} \dot{\mathcal{D}}_{\sigma_0+1}^{\frac{1}{2}},\\
 & \|\dot{\Lambda}_h^{\sigma}\,(B(\widetilde{\mathcal{A}})\partial_h\nabla\,v)\|_{L^2} \lesssim  E_{s_0}^{\frac{1}{2}} \dot{\mathcal{D}}_{\sigma+1}^{\frac{1}{2}}\quad \forall\,\sigma \in [\sigma_0,\,s_0-1),\\
 & \|\dot{\Lambda}_h^{\sigma}\,(B(\widetilde{\mathcal{A}})\partial_h\nabla\,v)\|_{L^2}\lesssim \, E_{\sigma+1}^{\frac{1}{2}} \dot{\mathcal{D}}_{s_0}^{\frac{1}{2}}+ E_{s_0}^{\frac{1}{2}} \dot{\mathcal{D}}_{\sigma+1}^{\frac{1}{2}}\quad \forall\,\sigma \geq s_0-1.
\end{split}
\end{equation}
For the first-order derivative terms of $v$ in $\mathfrak{g}_{jj}$ and $\widetilde{\mathfrak{g}}_{jj}$ with $j=1,2,3$, notice that all of them have the forms with $\mathcal{A}_{i}^{m}\partial_{m}\mathcal{A}_{j}^{1} \partial_{1}v$ or $\mathcal{A}_{i}^{m}\partial_{m} \mathcal{A}_{j}^{\alpha}\partial_{\alpha}v$, we then apply the product law \eqref{product-law-1} and Lemma \ref{lem-est-aij-1} to deduce
\begin{equation}\label{est-g-2-7a}
\begin{split}
 &\|\dot{\Lambda}_h^{-\lambda}\,(\mathcal{A}_{i}^{m}\partial_{m}\mathcal{A}_{j}^{1} \partial_{1}v)\|_{L^2}\lesssim \|\dot{\Lambda}_h^{1-\sigma_0-\lambda}(\mathcal{A}_{i}^{m}\partial_{m}\mathcal{A}_{j}^{1})\|_{L^2}\,
 \|\dot{\Lambda}_h^{\sigma_0}\partial_1 v\|_{L^\infty_{x_1}L^2_h}\\
 &\lesssim \|\dot{\Lambda}_h^{\sigma_0-1}\dot{\Lambda}_h^{2-2\sigma_0-\lambda}(\mathcal{A}_{i}^{m}\partial_{m}\mathcal{A}_{j}^{1})\|_{L^2}\,
 \|\dot{\Lambda}_h^{\sigma_0}\partial_1 v\|_{H^1}\lesssim\,E_{s_0}^{\frac{1}{2}} \dot{\mathcal{D}}_{\sigma_0+1}^{\frac{1}{2}},
\end{split}
\end{equation}
and
\begin{equation}\label{est-g-2-7b}
\begin{split}
 &\|\dot{\Lambda}_h^{-\lambda}\,(\mathcal{A}_{i}^{m}\partial_{m} \mathcal{A}_{j}^{\alpha}\partial_{\alpha}v)\|_{L^2}\lesssim \|\dot{\Lambda}_h^{\sigma_0-1} \partial_{\alpha}v\|_{L^\infty_{x_1}L^2_h}\,
 \|\dot{\Lambda}_h^{2-\sigma_0-\lambda}(\mathcal{A}_{i}^{m}\partial_{m} \mathcal{A}_{j}^{\alpha})\|_{L^2}\\
 &\lesssim \|\dot{\Lambda}_h^{2-\sigma_0-\lambda}(\mathcal{A}_{i}^{m}\partial_{m} \mathcal{A}_{j}^{\alpha})\|_{L^2}\,\|\dot{\Lambda}_h^{ \sigma_0} v\|_{H^1}\lesssim\,E_{s_0}^{\frac{1}{2}} \dot{\mathcal{D}}_{\sigma_0+1}^{\frac{1}{2}},
\end{split}
\end{equation}
where we have used the facts $1-\sigma_0>2-2\sigma_0-\lambda \geq 0$ and $\sigma_0 \leq 2-\sigma_0-\lambda \leq s_0-1$.

Thanks to the product law \eqref{product-law-1} and Lemma \ref{lem-est-aij-1}, we obtain, for $\sigma \in [\sigma_0,\,1)$
\begin{equation*}
\begin{split}
 \|\dot{\Lambda}_h^{\sigma}\,(\mathcal{A}_{i}^{m}\partial_{m}\mathcal{A}_{j}^{1} \partial_{1}v)\|_{L^2}&\lesssim \|\dot{\Lambda}_h^{1-\sigma}(\mathcal{A}_{i}^{m}\partial_{m}\mathcal{A}_{j}^{1})\|_{L^2}\,
 \|\dot{\Lambda}_h^{\sigma}\partial_1 v\|_{L^\infty_{x_1}L^2_h}\\
 &\lesssim \|\dot{\Lambda}_h^{\sigma_0-1}\dot{\Lambda}_h^{2-\sigma_0-\sigma}
 (\mathcal{A}_{i}^{m}\partial_{m}\mathcal{A}_{j}^{1})\|_{L^2}\,
 \|\dot{\Lambda}_h^{\sigma}\partial_1 v\|_{H^1},
\end{split}
\end{equation*}
for $\sigma \in [1,\,s_0-1)$
\begin{equation*}
\begin{split}
\|\dot{\Lambda}_h^{\sigma}\,(\mathcal{A}_{i}^{m}\partial_{m} \mathcal{A}_{j}^{\alpha}\partial_{\alpha}v)\|_{L^2}
&\lesssim \|\dot{\Lambda}_h^{\sigma-1}\,(\mathcal{A}_{i}^{m}\partial_{m} \mathcal{A}_{j}^{\alpha}\partial_{\alpha}\partial_hv)\|_{L^2}+\|\dot{\Lambda}_h^{\sigma-1}\,(\partial_h(\mathcal{A}_{i}^{m}\partial_{m} \mathcal{A}_{j}^{\alpha})\partial_{\alpha}v)\|_{L^2}\\
 &\lesssim\|\dot{\Lambda}_h^{\sigma-1} \partial_{\alpha}\partial_hv\|_{L^\infty_{x_1}L^2_h}\,
 \|\dot{\Lambda}_h^{\sigma_0}{\Lambda}_h^{s_0-1-\sigma_0}(\mathcal{A}_{i}^{m}\partial_{m} \mathcal{A}_{j}^{\alpha})\|_{L^2}\\
 &\qquad+\|\dot{\Lambda}_h^{2+\sigma-s_0} \partial_{\alpha}v\|_{L^\infty_{x_1}L^2_h}\,
 \|\dot{\Lambda}_h^{s_0-2}\partial_h(\mathcal{A}_{i}^{m}\partial_{m} \mathcal{A}_{j}^{\alpha})\|_{L^2},
\end{split}
\end{equation*}
and for $\sigma  \geq s_0-1$
\begin{equation*}
\begin{split}
 \|\dot{\Lambda}_h^{\sigma}\,(\mathcal{A}_{i}^{m}\partial_{m}\mathcal{A}_{j}^{1} \partial_{1}v)\|_{L^2} \lesssim& \|\dot{\Lambda}_h^{\sigma}(\mathcal{A}_{i}^{m}\partial_{m}\mathcal{A}_{j}^{1})\|_{L^2}\,
 \|\dot{\Lambda}_h^{\sigma_0}{\Lambda}_h^{s_0-1-\sigma_0} \partial_1 v\|_{L^\infty_{x_1}L^2_{x_h}}\\
 &+\|\dot{\Lambda}_h^{\sigma_0}{\Lambda}_h^{s_0-1-\sigma_0}
 (\mathcal{A}_{i}^{m}\partial_{m}\mathcal{A}_{j}^{1})\|_{L^2}\,
 \|\dot{\Lambda}_h^{\sigma} \partial_1 v\|_{L^\infty_{x_1}L^2_{x_h}},
\end{split}
\end{equation*}
which follows that
\begin{equation}\label{est-g-2-7c}
\begin{split}
 &\|\dot{\Lambda}_h^{\sigma}\,(\mathcal{A}_{i}^{m}\partial_{m}\mathcal{A}_{j}^{1} \partial_{1}v)\|_{L^2}\lesssim\,E_{s_0}^{\frac{1}{2}} \dot{\mathcal{D}}_{\sigma+1}^{\frac{1}{2}}\quad \forall \,\,\sigma \in [\sigma_0,\,1),\\
&\|\dot{\Lambda}_h^{\sigma}\,(\mathcal{A}_{i}^{m}\partial_{m} \mathcal{A}_{j}^{\alpha}\partial_{\alpha}v)\|_{L^2}\lesssim\,E_{s_0}^{\frac{1}{2}} \dot{\mathcal{D}}_{\sigma+1}^{\frac{1}{2}} \quad \forall \,\,\sigma \in [1,\,s_0-1),\\
& \|\dot{\Lambda}_h^{\sigma}\,(\mathcal{A}_{i}^{m}\partial_{m}\mathcal{A}_{j}^{1} \partial_{1}v)\|_{L^2}\lesssim E_{\sigma+1}^{\frac{1}{2}} \dot{\mathcal{D}}_{s_0}^{\frac{1}{2}}+ E_{s_0}^{\frac{1}{2}} \dot{\mathcal{D}}_{\sigma+1}^{\frac{1}{2}} \quad \forall \,\, \sigma  \geq s_0-1.
\end{split}
\end{equation}
Similarly, we may get
\begin{equation}\label{est-g-2-7d}
\begin{split}
 &\|\dot{\Lambda}_h^{\sigma}\,(\mathcal{A}_{i}^{m}\partial_{m}\mathcal{A}_{j}^{i} \partial_{i}v)\|_{L^2}\lesssim\,E_{s_0}^{\frac{1}{2}} \dot{\mathcal{D}}_{\sigma+1}^{\frac{1}{2}} \quad \forall \,\,\sigma \in [\sigma_0,\,s_0-1),
\end{split}
\end{equation}
and
\begin{equation}\label{est-g-2-7e}
\begin{split}
 &\|\dot{\Lambda}_h^{\sigma}\,(\mathcal{A}_{i}^{m}\partial_{m}\mathcal{A}_{j}^{k} \partial_{k}v)\|_{L^2}\lesssim E_{\sigma+1}^{\frac{1}{2}} \dot{\mathcal{D}}_{s_0}^{\frac{1}{2}}+ E_{s_0}^{\frac{1}{2}} \dot{\mathcal{D}}_{\sigma+1}^{\frac{1}{2}} \quad \forall \,\, \sigma  \geq s_0-1.
\end{split}
\end{equation}
Therefore, thanks to \eqref{est-g-2-6a}-\eqref{est-g-2-7e},  we obtain that for any $j=1,2,3$
\begin{equation*}
\begin{split}
 &\|\dot{\Lambda}_h^{-\lambda}\,(\widetilde{\mathfrak{g}}_{jj},\,\mathfrak{g}_{jj})\|_{L^2}
 +\|\dot{\Lambda}_h^{\sigma_0}\,(\widetilde{\mathfrak{g}}_{jj},\,\mathfrak{g}_{jj})\|_{L^2}\lesssim \,E_{s_0}^{\frac{1}{2}} \dot{\mathcal{D}}_{\sigma_0+1}^{\frac{1}{2}},\\
 &\|\dot{\Lambda}_h^{\sigma}\,(\widetilde{\mathfrak{g}}_{jj},\,\mathfrak{g}_{jj})\|_{L^2}\lesssim\,E_{s_0}^{\frac{1}{2}} \dot{\mathcal{D}}_{\sigma+1}^{\frac{1}{2}} \quad \forall \,\,\sigma \in [\sigma_0,\,s_0-1),\\
  &\|\dot{\Lambda}_h^{\sigma}\,(\widetilde{\mathfrak{g}}_{jj},\,\mathfrak{g}_{jj})\|_{L^2}\lesssim E_{\sigma+1}^{\frac{1}{2}} \dot{\mathcal{D}}_{s_0}^{\frac{1}{2}}+ E_{s_0}^{\frac{1}{2}} \dot{\mathcal{D}}_{\sigma+1}^{\frac{1}{2}} \quad \forall \,\, \sigma  \geq s_0-1.
\end{split}
\end{equation*}
For the pressure terms in $\mathfrak{g}$, from \eqref{product-law-1}, it can be obtained that
\begin{equation*}
\begin{split}
 &\|\dot{\Lambda}_h^{-\lambda}(\mathcal{A}_2^1\partial_1q)\|_{L^2}\lesssim \|\dot{\Lambda}_h^{1-\sigma_0-\lambda}\mathcal{A}_2^1\|_{L^\infty_{x_1}L^2_h}
 \|\dot{\Lambda}_h^{\sigma_0} \partial_1q \|_{L^2}\lesssim E_{s_0}^{\frac{1}{2}} \dot{\mathcal{D}}_{\sigma_0+1}^{\frac{1}{2}}.
\end{split}
\end{equation*}
While for the term about $\partial_hq$, we have
\begin{equation*}\label{est-g-2-10}
\begin{split}
 &\|\dot{\Lambda}_h^{-\lambda}( \mathcal{A}_1^h \partial_hq)\|_{L^2}\lesssim \|\dot{\Lambda}_h^{2-\sigma_0-\lambda}\mathcal{A}_1^h\|_{L^\infty_{x_1}L^2_h}
 \|\dot{\Lambda}_h^{\sigma_0-1} \partial_hq \|_{L^2}\lesssim E_{s_0}^{\frac{1}{2}}\,
 \|\dot{\Lambda}_h^{\sigma_0}q \|_{L^2}.\\
\end{split}
\end{equation*}
To control $ \|\dot{\Lambda}_h^{\sigma_0}q \|_{L^2}$, we employ Poincare's inequality and the relation $q={Q}+g\,\xi^1$ to get
\begin{equation*}
\begin{split}
 &\|\dot{\Lambda}_h^{\sigma_0}q\|_{L^2(\Omega)}\lesssim \|\dot{\Lambda}_h^{\sigma_0}q\|_{L^2(\Sigma_0)}+\|\dot{\Lambda}_h^{\sigma_0}\partial_1q \|_{L^2(\Omega)} \lesssim \|\dot{\Lambda}_h^{\sigma_0}\xi^1\|_{L^2(\Sigma_0)}+ \|\dot{\Lambda}_h^{\sigma_0} {Q}\|_{L^2(\Sigma_0)}+ \dot{\mathcal{D}}_{\sigma_0+1}^{\frac{1}{2}}.
\end{split}
\end{equation*}
It thus follows
\begin{equation*}\label{est-g-2-11}
\begin{split}
 &\|\dot{\Lambda}_h^{-\lambda}( \mathcal{A}_1^h \partial_hq)\|_{L^2}\lesssim E_{s_0}^{\frac{1}{2}}\,
(\|\dot{\Lambda}_h^{\sigma_0}\xi^1\|_{L^2(\Sigma_0)}+ \|\dot{\Lambda}_h^{\sigma_0} {Q}\|_{L^2(\Sigma_0)}+ \dot{\mathcal{D}}_{\sigma_0+1}^{\frac{1}{2}}).\\
\end{split}
\end{equation*}
Repeating the above argument to the other pressure terms in $\mathfrak{g}$ ensures that
\begin{equation*}
\begin{split}
 &\|\dot{\Lambda}_h^{-\lambda}(\mathcal{A}_2^1\partial_1q)\|_{L^2}
 +\|\dot{\Lambda}_h^{-\lambda}(\mathcal{A}_1^h\partial_hq)\|_{L^2}
 +\|\dot{\Lambda}_h^{-\lambda}((\mathcal{A}_1^1-1)\partial_1q)\|_{L^2} +\|\dot{\Lambda}_h^{-\lambda}((\mathcal{A}_2^2-1)\partial_2q)\|_{L^2}\\
 &+\|\dot{\Lambda}_h^{-\lambda}( \mathcal{A}_2^3 \partial_3q)\|_{L^2}+\|\dot{\Lambda}_h^{-\lambda}(\mathcal{A}_3^1\partial_1q)\|_{L^2}
 +\|\dot{\Lambda}_h^{-\lambda}((\mathcal{A}_3^3-1)\partial_3q)\|_{L^2}\\
 &\lesssim E_{s_0}^{\frac{1}{2}}\,
(\|\dot{\Lambda}_h^{\sigma_0}\xi^1\|_{L^2(\Sigma_0)}+ \|\dot{\Lambda}_h^{\sigma_0} {Q}\|_{L^2(\Sigma_0)}+ \dot{\mathcal{D}}_{\sigma_0+1}^{\frac{1}{2}}).
\end{split}
\end{equation*}
Consequently, we conclude that
\begin{equation*}
\begin{split}
&\|\dot{\Lambda}_h^{-\lambda} \mathfrak{g}\|_{L^2}\lesssim E_{s_0}^{\frac{1}{2}}\,
(\|\dot{\Lambda}_h^{\sigma_0}\xi^1\|_{L^2(\Sigma_0)}+ \|\dot{\Lambda}_h^{\sigma_0} {Q}\|_{L^2(\Sigma_0)}+ \dot{\mathcal{D}}_{\sigma_0+1}^{\frac{1}{2}})\lesssim E_{s_0}^{\frac{1}{2}}\,
(\dot{\mathcal{E}}_{\sigma_0+\frac{1}{2}}^{\frac{1}{2}}  + \dot{\mathcal{D}}_{\sigma_0+1}^{\frac{1}{2}}).
\end{split}
\end{equation*}
In the same manner it is easy to get
\begin{equation*}
\begin{split}
  &\|\dot{\Lambda}_h^{\sigma}\, \mathfrak{g} \|_{L^2}
 \lesssim\,E_{s_0}^{\frac{1}{2}} \dot{\mathcal{D}}_{\sigma+1}^{\frac{1}{2}} \quad \forall \,\,\sigma \in [\sigma_0,\,s_0-1),\\
  &\|\dot{\Lambda}_h^{\sigma}\,\mathfrak{g}\|_{L^2}\lesssim E_{\sigma+1}^{\frac{1}{2}} \dot{\mathcal{D}}_{s_0}^{\frac{1}{2}}+ E_{s_0}^{\frac{1}{2}} \dot{\mathcal{D}}_{\sigma+1}^{\frac{1}{2}} \quad \forall \,\, \sigma  \geq s_0-1.
\end{split}
\end{equation*}
The proof of the lemma is therefore accomplished.
\end{proof}

\begin{lem}\label{lem-est-G-norm-1}
Let $s_0>2$. If $(\lambda,\,\sigma_0) \in (0, 1)$ satisfies $1-\lambda< \sigma_0\leq 1-\frac{1}{2}\lambda$, $j=1, 2, 3$, $\sigma>1$, and $E_{s_0}(t) \leq 1$ for all $t\in [0, T]$, then there hold that $\forall\,t\in [0, T]$
\begin{equation}\label{est-GQ1-norm-1}
\begin{split}
 &\|\dot{\Lambda}_h^{-\lambda}(\mathcal{G}^{j}_{{Q}, 1},\,\mathcal{G}^{j}_{{Q}, 2})\|_{H^1}+\|\dot{\Lambda}_h^{\sigma_0}(\mathcal{G}^{j}_{{Q}, 1},\,\mathcal{G}^{j}_{{Q}, 2})\|_{H^1}\lesssim E_{s_0}^{\frac{1}{2}}\|\dot{\Lambda}_h^{\sigma_0} \Lambda_h^{\frac{1}{2}} {Q}\|_{L^2(\Sigma_0)},\\
     &\|\dot{\Lambda}_h^{-\lambda}\mathcal{G}^{j}_{\partial_hv} \|_{L^2} + \|\dot{\Lambda}_h^{\sigma_0}\mathcal{G}^{j}_{\partial_hv} \|_{L^2} \lesssim E_{s_0}^{\frac{1}{2}}\|\dot{\Lambda}_h^{\sigma_0} \Lambda_h v\|_{L^2},\\
    &\|\dot{\Lambda}_h^{-\lambda}\mathcal{G}^{j}_{\partial_hv} \|_{H^1} + \|\dot{\Lambda}_h^{\sigma_0}\mathcal{G}^{j}_{\partial_hv} \|_{H^1} \lesssim E_{s_0}^{\frac{1}{2}}\|\dot{\Lambda}_h^{\sigma_0} \Lambda_h\nabla\,v\|_{L^2},
  \end{split}
\end{equation}
\begin{equation}\label{est-GQ2-norm-1}
\begin{split}
  &\|\dot{\Lambda}_h^{\sigma}(\mathcal{G}^{j}_{{Q}, 1},\,\mathcal{G}^{j}_{{Q}, 2})\|_{L^2}\lesssim E_{\sigma+1}^{\frac{1}{2}} \|\dot{\Lambda}_h^{\sigma_0} \Lambda_h^{s_0-\frac{1}{2}-\sigma_0} {Q}\|_{L^2(\Sigma_0)}+ E_{s_0}^{\frac{1}{2}}   \|\dot{\Lambda}_h^{\sigma_0} \Lambda_h^{\sigma-\frac{1}{2}-\sigma_0}  {Q}\|_{L^2(\Sigma_0)},\\
   &\|\dot{\Lambda}_h^{\sigma}\mathcal{G}^{j}_{\partial_hv} \|_{L^2}\lesssim  E_{\sigma+1}^{\frac{1}{2}}\|\dot{\Lambda}_h^{\sigma_0} \Lambda_h^{s_0-\sigma_0} v\|_{L^2}+E_{s_0}^{\frac{1}{2}}\|\dot{\Lambda}_h^{\sigma+1}v\|_{L^2},
       \end{split}
\end{equation}
\begin{equation}\label{est-GQ2-s-norm-1}
\begin{split}
  &\|\dot{\Lambda}_h^{\sigma+\ell_0}(\mathcal{G}^{j}_{{Q}, 1},\,\mathcal{G}^{j}_{{Q}, 2})\|_{L^2}\lesssim E_{\sigma+1}^{\frac{1}{2}}\|\dot{\Lambda}_h^{\sigma_0} \Lambda_h^{s_0-\frac{1}{2}+\ell_0-\sigma_0} {Q}\|_{L^2(\Sigma_0)}+E_{s_0}^{\frac{1}{2}}\|\dot{\Lambda}_h^{\sigma+\ell_0-1} \Lambda_h^{\frac{1}{2}} {Q}\|_{L^2(\Sigma_0)},\\
   &\|\dot{\Lambda}_h^{\sigma+\ell_0}\mathcal{G}^{j}_{\partial_hv} \|_{L^2}\lesssim  E_{\sigma+1}^{\frac{1}{2}}\|\dot{\Lambda}_h^{\sigma_0} \Lambda_h^{s_0-1+\ell_0-\sigma_0} \nabla\,v\|_{L^2}+E_{s_0}^{\frac{1}{2}}\|\dot{\Lambda}_h^{\sigma+\ell_0}v\|_{L^2},
       \end{split}
\end{equation}
\begin{equation*}
\begin{split}
  &\|\dot{\Lambda}_h^{\sigma}(\mathcal{G}^{j}_{{Q}, 1},\,\mathcal{G}^{j}_{{Q}, 2},\,\mathcal{G}^{j}_{\partial_hv})\|_{H^1}\lesssim E_{\sigma+1}^{\frac{1}{2}}\dot{\mathcal{D}}_{s_0}^{\frac{1}{2}}
  +E_{s_0}^{\frac{1}{2}}\dot{\mathcal{D}}_{\sigma+1}^{\frac{1}{2}},\\
\end{split}
\end{equation*}
and
\begin{equation*}
\begin{split}
   &\|\dot{\Lambda}_h^{\sigma}\partial_t(\mathcal{G}^{j}_{{Q}, 1},\,\mathcal{G}^{j}_{{Q}, 2},\,\mathcal{G}^{j}_{\partial_hv})\|_{L^2}\lesssim E_{\sigma+1}^{\frac{1}{2}}\dot{\mathcal{D}}_{s_0}^{\frac{1}{2}}
  +E_{s_0}^{\frac{1}{2}}\dot{\mathcal{D}}_{\sigma+1}^{\frac{1}{2}},\\
     &\|\dot{\Lambda}_h^{-\lambda}\partial_t(\mathcal{G}^{j}_{{Q}, 1},\,\mathcal{G}^{j}_{{Q}, 2},\,\mathcal{G}^{j}_{\partial_hv})\|_{L^2}+\|\dot{\Lambda}_h^{\sigma_0}\partial_t(\mathcal{G}^{j}_{{Q}, 1},\,\mathcal{G}^{j}_{{Q}, 2},\,\mathcal{G}^{j}_{\partial_hv})\|_{L^2}\lesssim E_{s_0}^{\frac{1}{2}}\dot{\mathcal{D}}_{s_0}^{\frac{1}{2}}.
\end{split}
\end{equation*}
\end{lem}
\begin{proof}
Thanks to Lemmas \ref{lem-productlaw-1aa}, \ref{lem-est-harmonic-ext-1}, \ref{lem-est-aij-1}, and the fact $1-\lambda< \sigma_0\leq 1-\frac{1}{2}\lambda$, we have
\begin{equation}\label{exp-decom-G12-3}
  \begin{split}
    \|\dot{\Lambda}_h^{-\lambda}\mathcal{G}^{1}_{{Q}, 1}\|_{H^1}&\lesssim \|\dot{\Lambda}_h^{1-\sigma_0-\lambda}( a_{11}|\overrightarrow{\rm{a}_1}|^{-2}J -1)\|_{H^1}\, \|\dot{\Lambda}_h^{\sigma_0}\mathcal{H}({Q})\|_{H^1}\lesssim E_{s_0}^{\frac{1}{2}}\|\dot{\Lambda}_h^{\sigma_0} \Lambda_h^{\frac{1}{2}} {Q}\|_{L^2(\Sigma_0)},\\
        \|\dot{\Lambda}_h^{-\lambda}\mathcal{G}^{1}_{{Q}, 2}\|_{H^1}&\lesssim (1+\|\dot{\Lambda}_h^{\sigma_0}{\Lambda}_h^{s_0-1-\sigma_0}( a_{11}|\overrightarrow{\rm{a}_1}|^{-2}J -1)\|_{H^1})\, \|\dot{\Lambda}_h^{-\lambda}\mathcal{H}({Q}^2\mathcal{R}({Q}))\|_{H^1}.
\end{split}
\end{equation}
Owing to Lemmas \ref{lem-productlaw-1aa} and \ref{lem-est-harmonic-ext-1} again, one obtains
\begin{equation*}
  \begin{split}
 \|\dot{\Lambda}_h^{-\lambda}\mathcal{H}({Q}^2\mathcal{R}({Q}))\|_{H^1}&\lesssim  \|\dot{\Lambda}_h^{-\lambda}{\Lambda}_h^{\frac{1}{2}}({Q}^2\mathcal{R}({Q}))\|_{L^2(\Sigma_0)}\\
 &\lesssim \|\dot{\Lambda}_h^{1-\sigma_0-\lambda}{\Lambda}_h^{\frac{1}{2}}({Q}\mathcal{R}({Q}))\|_{L^2(\Sigma_0)}
 \|\dot{\Lambda}_h^{\sigma_0}Q\|_{L^2(\Sigma_0)}\lesssim E_{s_0}^{\frac{1}{2}}\|\dot{\Lambda}_h^{\sigma_0} \Lambda_h^{\frac{1}{2}} {Q}\|_{L^2(\Sigma_0)},
\end{split}
\end{equation*}
which along with \eqref{exp-decom-G12-3} yields
\begin{equation*}
  \begin{split}
    \|\dot{\Lambda}_h^{-\lambda}(\mathcal{G}^{1}_{{Q}, 1},\mathcal{G}^{1}_{{Q}, 2})\|_{H^1}\lesssim E_{s_0}^{\frac{1}{2}}\|\dot{\Lambda}_h^{\sigma_0} \Lambda_h^{\frac{1}{2}} {Q}\|_{L^2(\Sigma_0)}.
\end{split}
\end{equation*}
Similarly, we prove
\begin{equation*}
  \begin{split}
    &\|\dot{\Lambda}_h^{-\lambda}(\mathcal{G}^{\beta}_{{Q}, 1},\mathcal{G}^{\beta}_{{Q}, 2})\|_{H^1}+\|\dot{\Lambda}_h^{\sigma_0}(\mathcal{G}^{j}_{{Q}, 1},\mathcal{G}^{j}_{{Q}, 2})\|_{H^1}\lesssim E_{s_0}^{\frac{1}{2}}\|\dot{\Lambda}_h^{\sigma_0} \Lambda_h^{\frac{1}{2}} {Q}\|_{L^2(\Sigma_0)},
\end{split}
\end{equation*}
and
\begin{equation*}
  \begin{split}
    &\|\dot{\Lambda}_h^{-\lambda}\mathcal{G}^{j}_{\partial_hv} \|_{L^2} \lesssim \|\dot{\Lambda}_h^{2-\sigma_0-\lambda} \mathcal{B}_{10+j, i}^{\alpha}\|_{H^1} \|\dot{\Lambda}_h^{\sigma_0-1} \partial_{\alpha}v^{i} \|_{L^2} \lesssim E_{s_0}^{\frac{1}{2}}\|\dot{\Lambda}_h^{\sigma_0}  v \|_{L^2},\\
   &\|\dot{\Lambda}_h^{\sigma_0}\mathcal{G}^{j}_{\partial_hv} \|_{L^2}\lesssim   \|\dot{\Lambda}_h^{\sigma_0} \Lambda_h^{s_0-1-\sigma_0} \mathcal{B}_{10+j, i}^{\alpha}\|_{H^1} \|\dot{\Lambda}_h^{\sigma_0} \partial_{\alpha}v^{i} \|_{L^2} \lesssim E_{s_0}^{\frac{1}{2}}\|\dot{\Lambda}_h^{\sigma_0+1}  v \|_{L^2},\\
  &\|\dot{\Lambda}_h^{-\lambda}\mathcal{G}^{j}_{\partial_hv} \|_{H^1} \lesssim \|\dot{\Lambda}_h^{2-\sigma_0-\lambda} \mathcal{B}_{10+j, i}^{\alpha}\|_{H^1} \|\dot{\Lambda}_h^{\sigma_0-1} \partial_{\alpha}v^{i} \|_{H^1} \lesssim E_{s_0}^{\frac{1}{2}}\|\dot{\Lambda}_h^{\sigma_0} \nabla\,v \|_{L^2},\\
   &\|\dot{\Lambda}_h^{\sigma_0}\mathcal{G}^{j}_{\partial_hv} \|_{H^1}\lesssim   \|\dot{\Lambda}_h^{\sigma_0} \Lambda_h^{s_0-1-\sigma_0} \mathcal{B}_{10+j, i}^{\alpha}\|_{H^1} \|\dot{\Lambda}_h^{\sigma_0} \partial_{\alpha}v^{i} \|_{H^1} \lesssim E_{s_0}^{\frac{1}{2}}\|\dot{\Lambda}_h^{\sigma_0+1} \nabla\,v \|_{L^2},
  \end{split}
\end{equation*}
which results in \eqref{est-GQ1-norm-1}.

Thanks to Lemmas \ref{lem-productlaw-1aa}, \ref{lem-est-harmonic-ext-1}, and \ref{lem-est-aij-1} again, we find for $\sigma \geq \sigma_0$
\begin{equation*}
  \begin{split}
   \|\dot{\Lambda}_h^{\sigma}\mathcal{G}^{1}_{{Q}, 1}\|_{L^2}&\lesssim \|\dot{\Lambda}_h^{\sigma}( a_{11}|\overrightarrow{\rm{a}_1}|^{-2}J -1) \|_{H^1}\|\dot{\Lambda}_h^{\sigma_0} \Lambda_h^{s_0-1-\sigma_0} \mathcal{H}({Q})\|_{L^2}\\
   &\qquad\qquad\qquad+\|\dot{\Lambda}_h^{\sigma_0} \Lambda_h^{s_0-1-\sigma_0}( a_{11}|\overrightarrow{\rm{a}_1}|^{-2}J -1)\|_{H^1}\|\dot{\Lambda}_h^{\sigma} \mathcal{H}({Q})\|_{L^2}\\
   &\lesssim E_{\sigma+1}^{\frac{1}{2}}\|\dot{\Lambda}_h^{\sigma_0} \Lambda_h^{s_0-\frac{1}{2}-\sigma_0} {Q}\|_{L^2(\Sigma_0)}+E_{s_0}^{\frac{1}{2}}\|\dot{\Lambda}_h^{\sigma-1} \Lambda_h^{\frac{1}{2}} {Q}\|_{L^2(\Sigma_0)},
   \end{split}
\end{equation*}
\begin{equation*}
  \begin{split}
&\|\dot{\Lambda}_h^{\sigma}\mathcal{G}^{1}_{{Q}, 1}\|_{H^1} \lesssim \|\dot{\Lambda}_h^{\sigma}( a_{11}|\overrightarrow{\rm{a}_1}|^{-2}J -1) \|_{H^1}\|\dot{\Lambda}_h^{\sigma_0} \Lambda_h^{s_0-1-\sigma_0} \mathcal{H}({Q})\|_{H^1}\\
&\qquad\qquad\qquad +\|\dot{\Lambda}_h^{\sigma_0} \Lambda_h^{s_0-1-\sigma_0}( a_{11}|\overrightarrow{\rm{a}_1}|^{-2}J -1)\|_{H^1}\|\dot{\Lambda}_h^{\sigma} \mathcal{H}({Q})\|_{H^1}\\
&\lesssim E_{\sigma+1}^{\frac{1}{2}}\|\dot{\Lambda}_h^{\sigma_0} \Lambda_h^{s_0-\frac{1}{2}-\sigma_0} {Q}\|_{L^2(\Sigma_0)}+E_{s_0}^{\frac{1}{2}}\|\dot{\Lambda}_h^{\sigma} \Lambda_h^{\frac{1}{2}} {Q}\|_{L^2(\Sigma_0)}\lesssim  E_{\sigma+1}^{\frac{1}{2}}\dot{\mathcal{D}}_{s_0}^{\frac{1}{2}}
  +E_{s_0}^{\frac{1}{2}}\dot{\mathcal{D}}_{\sigma+1}^{\frac{1}{2}},
   \end{split}
\end{equation*}
and
\begin{equation*}
  \begin{split}
   \|\dot{\Lambda}_h^{\sigma+\ell_0}\mathcal{G}^{1}_{{Q}, 1}\|_{L^2}&\lesssim \|\dot{\Lambda}_h^{\sigma+\ell_0}( a_{11}|\overrightarrow{\rm{a}_1}|^{-2}J -1) \|_{L^2_{x_1}L^{\frac{2}{\ell_0}}_{h}}\|\dot{\Lambda}_h^{\sigma_0} \Lambda_h^{s_0-1-\sigma_0} \mathcal{H}({Q})\|_{L^{\infty}_{x_1}L^{\frac{2}{1-\ell_0}}_{h}}\\
   &\qquad\qquad\qquad+\|\dot{\Lambda}_h^{\sigma_0} \Lambda_h^{s_0-1-\sigma_0}( a_{11}|\overrightarrow{\rm{a}_1}|^{-2}J -1)\|_{H^1}\|\dot{\Lambda}_h^{\sigma+\ell_0} \mathcal{H}({Q})\|_{L^2}\\
   &\lesssim E_{\sigma+1}^{\frac{1}{2}}\|\dot{\Lambda}_h^{\sigma_0} \Lambda_h^{s_0-\frac{1}{2}+\ell_0-\sigma_0} {Q}\|_{L^2(\Sigma_0)}+E_{s_0}^{\frac{1}{2}}\|\dot{\Lambda}_h^{\sigma+\ell_0-1} \Lambda_h^{\frac{1}{2}} {Q}\|_{L^2(\Sigma_0)},
   \end{split}
\end{equation*}
where we used the Sobolev embedding $\dot{H}^{1-\ell_0}(\mathbb{R}^2_h) \hookrightarrow L^{\frac{2}{\ell_0}}(\mathbb{R}^2_h)$ and $\dot{H}^{\ell_0}(\mathbb{R}^2_h) \hookrightarrow L^{\frac{2}{1-\ell_0}}(\mathbb{R}^2_h)$ in the last inequality.
In the same manner, one can see
\begin{equation*}
  \begin{split}
   &\|\dot{\Lambda}_h^{\sigma}\mathcal{G}^{\beta}_{{Q}, 1}\|_{L^2}\lesssim E_{\sigma+1}^{\frac{1}{2}}\|\dot{\Lambda}_h^{\sigma_0} \Lambda_h^{s_0-\frac{1}{2}-\sigma_0} {Q}\|_{L^2(\Sigma_0)}+E_{s_0}^{\frac{1}{2}}\|\dot{\Lambda}_h^{s-2} \Lambda_h^{\frac{1}{2}} {Q}\|_{L^2(\Sigma_0)},\\
     &\|\dot{\Lambda}_h^{\sigma+\ell_0}\mathcal{G}^{\beta}_{{Q}, 1}\|_{L^2}\lesssim E_{\sigma+1}^{\frac{1}{2}}\|\dot{\Lambda}_h^{\sigma_0} \Lambda_h^{s_0-\frac{1}{2}+\ell_0-\sigma_0} {Q}\|_{L^2(\Sigma_0)}+E_{s_0}^{\frac{1}{2}}\|\dot{\Lambda}_h^{\sigma+\ell_0-1} \Lambda_h^{\frac{1}{2}} {Q}\|_{L^2(\Sigma_0)},\\
   &\|\dot{\Lambda}_h^{\sigma}\mathcal{G}^{\beta}_{{Q}, 1}\|_{H^1}\lesssim E_{\sigma+1}^{\frac{1}{2}}\dot{\mathcal{D}}_{s_0}^{\frac{1}{2}}
  +E_{s_0}^{\frac{1}{2}}\dot{\mathcal{D}}_{\sigma+1}^{\frac{1}{2}},
\end{split}
\end{equation*}
\begin{equation*}
  \begin{split}
  &\|\dot{\Lambda}_h^{\sigma}\mathcal{G}^{j}_{\partial_hv} \|_{L^2}\lesssim  E_{\sigma+1}^{\frac{1}{2}}\|\dot{\Lambda}_h^{\sigma_0} \Lambda_h^{s_0-\sigma_0} v\|_{L^2}+E_{s_0}^{\frac{1}{2}}\|\dot{\Lambda}_h^{\sigma+1}v\|_{L^2},\\
    &\|\dot{\Lambda}_h^{\sigma+\ell_0}\mathcal{G}^{j}_{\partial_hv} \|_{L^2}\lesssim  E_{\sigma+1}^{\frac{1}{2}}\|\dot{\Lambda}_h^{\sigma_0} \Lambda_h^{s_0-1+\ell_0-\sigma_0} \nabla\,v\|_{L^2}+E_{s_0}^{\frac{1}{2}}\|\dot{\Lambda}_h^{\sigma+\ell_0}v\|_{L^2},\\
    &\|\dot{\Lambda}_h^{\sigma}\mathcal{G}^{j}_{\partial_hv} \|_{H^1}\lesssim   E_{\sigma+1}^{\frac{1}{2}}\dot{\mathcal{D}}_{s_0}^{\frac{1}{2}}
  +E_{s_0}^{\frac{1}{2}}\dot{\mathcal{D}}_{\sigma+1}^{\frac{1}{2}},
  \end{split}
\end{equation*}
\begin{equation*}
  \begin{split}
            & \|\dot{\Lambda}_h^{\sigma} \mathcal{G}^{j}_{{Q}, 2}\|_{L^2}\lesssim  E_{\sigma+1}^{\frac{1}{2}} \|\dot{\Lambda}_h^{\sigma_0} \Lambda_h^{s_0-\frac{1}{2}-\sigma_0} {Q}\|_{L^2(\Sigma_0)}+ E_{s_0}^{\frac{1}{2}}   \|\dot{\Lambda}_h^{\sigma_0} \Lambda_h^{\sigma-\frac{1}{2}-\sigma_0}  {Q}\|_{L^2(\Sigma_0)},\\
                        & \|\dot{\Lambda}_h^{\sigma+\ell_0} \mathcal{G}^{j}_{{Q}, 2}\|_{L^2}\lesssim  E_{\sigma+1}^{\frac{1}{2}}\|\dot{\Lambda}_h^{\sigma_0} \Lambda_h^{s_0-\frac{1}{2}+\ell_0-\sigma_0} {Q}\|_{L^2(\Sigma_0)}+E_{s_0}^{\frac{1}{2}}\|\dot{\Lambda}_h^{\sigma+\ell_0-1} \Lambda_h^{\frac{1}{2}} {Q}\|_{L^2(\Sigma_0)},\\
                        & \|\dot{\Lambda}_h^{\sigma} \mathcal{G}^{j}_{{Q}, 2}\|_{H^1} \lesssim E_{\sigma+1}^{\frac{1}{2}}\dot{\mathcal{D}}_{s_0}^{\frac{1}{2}}
  +E_{s_0}^{\frac{1}{2}}\dot{\mathcal{D}}_{\sigma+1}^{\frac{1}{2}}\quad (\forall\,\,j=1, 2, 3).
\end{split}
\end{equation*}
For $\partial_t\mathcal{G}^{\beta}_{{Q}, 1}$, we first compute
\begin{equation*}
  \begin{split}
& \partial_t\mathcal{G}^{\beta}_{{Q}, 1}=\frac{\bar{\rho}(0)}{(\delta+\frac{4}{3}\varepsilon)} \bigg(a_{\beta\,1}J \mathcal{H}( \partial_t{Q})+ \partial_t(a_{\beta\,1}J)\mathcal{H}( {Q})\bigg),
  \end{split}
\end{equation*}
then applying Lemma \ref{lem-productlaw-1aa}, we obtain that
\begin{equation*}
  \begin{split}
 & \|\dot{\Lambda}_h^{\sigma} \partial_t\mathcal{G}^{\beta}_{{Q}, 1}\|_{L^2}\lesssim  E_{\sigma+1}^{\frac{1}{2}}\dot{\mathcal{D}}_{s_0}^{\frac{1}{2}}
 +E_{s_0}^{\frac{1}{2}}\dot{\mathcal{D}}_{\sigma+1}^{\frac{1}{2}}.
  \end{split}
\end{equation*}
Repeating the above argument yields
\begin{equation*}
  \begin{split}
   &\|\dot{\Lambda}_h^{\sigma}\partial_t(\mathcal{G}^{j}_{{Q}, 1},\,\mathcal{G}^{j}_{{Q}, 2},\,\mathcal{G}^{j}_{\partial_hv})\|_{L^2}\lesssim E_{\sigma+1}^{\frac{1}{2}}\dot{\mathcal{D}}_{s_0}^{\frac{1}{2}}
  +E_{s_0}^{\frac{1}{2}}\dot{\mathcal{D}}_{\sigma+1}^{\frac{1}{2}} \quad(\forall\,\,\sigma\geq \sigma_0),\\
     &\|\dot{\Lambda}_h^{-\lambda}\partial_t(\mathcal{G}^{j}_{{Q}, 1},\,\mathcal{G}^{j}_{{Q}, 2},\,\mathcal{G}^{j}_{\partial_hv})\|_{L^2}+\|\dot{\Lambda}_h^{\sigma_0}\partial_t(\mathcal{G}^{j}_{{Q}, 1},\,\mathcal{G}^{j}_{{Q}, 2},\,\mathcal{G}^{j}_{\partial_hv})\|_{L^2}\lesssim E_{s_0}^{\frac{1}{2}}\dot{\mathcal{D}}_{\sigma_0+1}^{\frac{1}{2}}.
  \end{split}
\end{equation*}
This completes the proof of Lemma \ref{lem-est-G-norm-1}.
\end{proof}

\renewcommand{\theequation}{\thesection.\arabic{equation}}
\setcounter{equation}{0}
\section{Energy estimates of the tangential derivatives of the velocity}\label{sect-energy}

In this section, we will derive some global energy estimates. In view of the local well-posedness theorem (Theorem \ref{thm-local}), it suffices to get necessary {\it a priori} estimates. Here and in what follows, all the $C$-forms, such as $C_j$, $c_j$, $\mathfrak{C}_j$, $\widetilde{\mathfrak{C}}_j$, $\mathfrak{c}_j$, and $\widetilde{\mathfrak{c}}_j$, are generic positive constants, which may be different on different lines.

\subsection{Energy estimates of the horizontal derivatives of the velocity}\label{subsect-est-hori-1}

In general, since the equations of the derivatives of the solution $(v, \xi)$ to the system \eqref{eqns-pert-1} lack the nonlinear symmetric structure, we have no idea to get the high-order energy identity. For this reason, we turn to employ the linearized form \eqref{eqns-linear-1} of the system \eqref{eqns-pert-1} to study the high-order energy estimates.

We first apply the operator $\dot{\Lambda}_h^{\sigma}$ with $\sigma\in \mathbb{R}$ to the linearized form \eqref{eqns-linear-1} of the system \eqref{eqns-pert-1} to get
\begin{equation}\label{eqns-sigma-lin-1}
\begin{cases}
 & \bar{\rho}(x_1)\partial_t \dot{\Lambda}_h^{\sigma}v +\bar{\rho}(x_1)\nabla\,\dot{\Lambda}_h^{\sigma}q-  \nabla\cdot\mathbb{S}(\dot{\Lambda}_h^{\sigma}v)=\dot{\Lambda}_h^{\sigma}\mathfrak{g},\\
 &\nabla\,\cdot  (\bar{\rho}(x_1)\,\dot{\Lambda}_h^{\sigma}v)=g^{-1}\, \bar{\rho}'(x_1) \partial_t\dot{\Lambda}_h^{\sigma}q+\dot{\Lambda}_h^{\sigma}(B_{1, i}^{h, j}\partial_jv^i)\quad \text{in} \quad \Omega,\\
  &\bar{\rho}(0)\dot{\Lambda}_h^{\sigma}q\, e_1- \mathbb{S}(\dot{\Lambda}_h^{\sigma}v)\,e_1=
\left(
  \begin{array}{c}
   \bar{\rho}(0)\dot{\Lambda}_h^{\sigma}\xi^1+\dot{\Lambda}_h^{\sigma}\mathcal{B}(\partial_{\alpha}v, {Q}, \partial_t{Q})\\
    -\varepsilon\,\dot{\Lambda}_h^{\sigma}(\mathcal{B}_{7, i}^{\alpha}\partial_{\alpha}v^{i}+\widetilde{\mathcal{B}}_{2}\partial_t{Q})\\
    - \varepsilon\,\dot{\Lambda}_h^{\sigma}(\mathcal{B}_{8, i}^{\alpha}\partial_{\alpha}v^{i}+\widetilde{\mathcal{B}}_{3}\partial_t{Q})
  \end{array}
\right) \quad \text{on} \quad \Sigma_0,\\
&\dot{\Lambda}_h^{\sigma}v|_{\Sigma_b}=0.
    \end{cases}
\end{equation}

For the general horizontal derivatives of the velocity, we first derive
\begin{prop}\label{lem-tan-pseudo-energy-1}
Let $\sigma \in \mathbb{R}$, $(v, \xi)$ be smooth solution to the system \eqref{eqns-pert-1}, then there holds
\begin{equation}\label{tan-linear-pseudo-0}
\begin{split}
& \frac{1}{2}\frac{d}{dt}\bigg(\|\sqrt{\bar{\rho}(x_1)}\,\dot{\Lambda}_h^{\sigma}v\|_{L^2(\Omega)}^2+g^{-1} \|\sqrt{-\bar{\rho}'(x_1)}\, \dot{\Lambda}_h^{\sigma}q\|_{L^2(\Omega)}^2 + g\,\bar{\rho}(0)\|\dot{\Lambda}_h^{\sigma}\xi^1\|_{L^2(\Sigma_0)}^2\bigg)\\
 &\quad+\bigg(\frac{\varepsilon}{2}\|\mathbb{D}^0(\dot{\Lambda}_h^{\sigma}v)\|_{L^2(\Omega)}^2+\delta\, \|\grad \cdot \dot{\Lambda}_h^{\sigma}v\|_{L^2(\Omega)}^2\bigg)=\sum_{j=1}^3\mathfrak{K}_j
\end{split}
\end{equation}
with
\begin{equation}\label{tan-linear-pseudo-0a}
\begin{split}
 &\mathfrak{K}_1:=\int_{\Sigma_{0}} \bigg(\dot{\Lambda}_h^{\sigma}v^1\,\dot{\Lambda}_h^{\sigma}\mathcal{B}(\partial_{\alpha}v, {Q}, \partial_t{Q})-\varepsilon\,\dot{\Lambda}_h^{\sigma}v^2\,\dot{\Lambda}_h^{\sigma}(\mathcal{B}_{7, i}^{\alpha}\partial_{\alpha}v^{i}+\widetilde{\mathcal{B}}_{2}\partial_t{Q})\\
 &\qquad\qquad\qquad\qquad\qquad\qquad\qquad\qquad
    - \varepsilon\,\dot{\Lambda}_h^{\sigma}v^3\,\dot{\Lambda}_h^{\sigma}(\mathcal{B}_{8, i}^{\alpha}\partial_{\alpha}v^{i}+\widetilde{\mathcal{B}}_{3}\partial_t{Q})
\bigg) \,dS_0,\\
&\mathfrak{K}_2:=\int_{ \Omega} \dot{\Lambda}_h^{\sigma}q\,\dot{\Lambda}_h^{\sigma}(B_{1, i}^{h, j}\partial_jv^i)\, dx,\ \mathfrak{K}_3:=\int_{ \Omega} \dot{\Lambda}_h^{\sigma}\mathfrak{g}\cdot \dot{\Lambda}_h^{\sigma}v\, dx.
\end{split}
\end{equation}
\end{prop}

\begin{proof}
We multiply the $i$-th component of the momentum equations of \eqref{eqns-sigma-lin-1} by $\dot{\Lambda}_h^{\sigma}v^i$ , sum over $i$, and integrate over $\Omega$ to find
\begin{equation}\label{linear-v-L2-1-1}
\begin{split}
 & \frac{1}{2}\frac{d}{dt}\int_{ \Omega}\bar{\rho}(x_1)|\dot{\Lambda}_h^{\sigma}v|^2\, dx+I+II=\int_{ \Omega} \dot{\Lambda}_h^{\sigma}\mathfrak{g}\cdot \dot{\Lambda}_h^{\sigma}v\, dx
\end{split}
\end{equation}
with $I=\int_{ \Omega}  \bar{\rho}(x_1)\nabla\,\dot{\Lambda}_h^{\sigma}q\,\cdot \dot{\Lambda}_h^{\sigma}v\,dx$, $II=-\int_{ \Omega}(\grad \cdot \mathbb{S}(\dot{\Lambda}_h^{\sigma}v)) \cdot \dot{\Lambda}_h^{\sigma}v\,dx$.

Integrating by parts in $I$ and $II$ yields
\begin{equation*}
\begin{split}
 &I=\int_{\Sigma_{0}}\dot{\Lambda}_h^{\sigma}v\cdot (\bar{\rho}(0) \dot{\Lambda}_h^{\sigma}q\, e_1) \,dS_0-\int_{ \Omega}\dot{\Lambda}_h^{\sigma}q\,\grad \cdot (\bar{\rho}(x_1) \dot{\Lambda}_h^{\sigma}v) \, dx,
\end{split}
\end{equation*}
and
\begin{equation*}
\begin{split}
 II&=\int_{\Sigma_{0}} (- \mathbb{S}(\dot{\Lambda}_h^{\sigma}v)e_1 ) \cdot \dot{\Lambda}_h^{\sigma}v \, dS_0+\int_{\Omega} \mathbb{S}(\dot{\Lambda}_h^{\sigma}v): \grad\dot{\Lambda}_h^{\sigma} v  \, dx.
 \end{split}
\end{equation*}
Thanks to the second equation in \eqref{eqns-sigma-lin-1}, we have
\begin{equation*}
\begin{split}
 &\int_{ \Omega}\dot{\Lambda}_h^{\sigma}q\,\grad \cdot (\bar{\rho}(x_1)  \dot{\Lambda}_h^{\sigma}v) \, dx\\
 &=\frac{1}{2g}\frac{d}{dt}\int_{ \Omega} \bar{\rho}'(x_1) |\dot{\Lambda}_h^{\sigma}q|^2  \, dx+\int_{ \Omega} \dot{\Lambda}_h^{\sigma}q\,\dot{\Lambda}_h^{\sigma}(B_{1, i}^{h, j}\partial_jv^i)\, dx,
\end{split}
\end{equation*}
so
\begin{equation}\label{linear-v-L2-1-2}
\begin{split}
 I=\int_{\Sigma_{0}}\dot{\Lambda}_h^{\sigma}v\cdot (\bar{\rho}(0)\dot{\Lambda}_h^{\sigma}q\, e_1) \,dS_0&+\frac{1}{2g}\frac{d}{dt}\int_{ \Omega} (-\bar{\rho}'(x_1)) |\dot{\Lambda}_h^{\sigma}q|^2  \, dx\\
 &-\int_{ \Omega} \dot{\Lambda}_h^{\sigma}q\,\dot{\Lambda}_h^{\sigma}(B_{1, i}^{h, j}\partial_jv^i)\, dx.
\end{split}
\end{equation}
On the other hand, thanks to the identity $ \mathbb{S}(\dot{\Lambda}_h^{\sigma}v): \grad \dot{\Lambda}_h^{\sigma}v=\frac{\varepsilon}{2}|\mathbb{D}^0(\dot{\Lambda}_h^{\sigma}v)|^2+\delta\, |\grad \cdot \dot{\Lambda}_h^{\sigma}v|^2 $, one gets
\begin{equation}\label{linear-v-L2-1-3}
\begin{split}
 II=\int_{\Sigma_{0}} (- \mathbb{S}(\dot{\Lambda}_h^{\sigma}v)e_1 ) \cdot \dot{\Lambda}_h^{\sigma}v \, dS_0+\int_{\Omega}\bigg(\frac{\varepsilon}{2}|\mathbb{D}^0(\dot{\Lambda}_h^{\sigma}v)|^2+\delta\, |\grad \cdot \dot{\Lambda}_h^{\sigma}v|^2\bigg) \, dx.
 \end{split}
\end{equation}
Plugging \eqref{linear-v-L2-1-2} and \eqref{linear-v-L2-1-3} into \eqref{linear-v-L2-1-1} results in
\begin{equation}\label{linear-v-L2-1-6}
\begin{split}
 & \frac{1}{2}\frac{d}{dt}\int_{ \Omega}\bigg(\bar{\rho}(x_1)|\dot{\Lambda}_h^{\sigma}v|^2+\frac{-\bar{\rho}'(x_1)}{g} |\dot{\Lambda}_h^{\sigma}q|^2 \bigg) \, dx+\bigg(\frac{\varepsilon}{2}\|\mathbb{D}^0(\dot{\Lambda}_h^{\sigma}v)\|_{L^2(\Omega)}^2+\delta\, \|\grad \cdot \dot{\Lambda}_h^{\sigma}v\|_{L^2(\Omega)}^2\bigg)\\
 &\qquad+\int_{\Sigma_{0}}\dot{\Lambda}_h^{\sigma}v\cdot \bigg(\bar{\rho}(0)\dot{\Lambda}_h^{\sigma}q\, e_1- \mathbb{S}(\dot{\Lambda}_h^{\sigma}v)e_1\bigg)\,dS_0=\int_{ \Omega} \dot{\Lambda}_h^{\sigma}q\,\dot{\Lambda}_h^{\sigma}(B_{1, i}^{h, j}\partial_jv^i) \, dx.
\end{split}
\end{equation}

For the boundary integral in \eqref{linear-v-L2-1-6}, making use of the interface boundary condition in \eqref{eqns-sigma-lin-1} leads to
\begin{equation}\label{linear-v-L2-1-7}
\begin{split}
&\int_{\Sigma_{0}}\dot{\Lambda}_h^{\sigma}v\cdot \bigg(\bar{\rho}(0)\dot{\Lambda}_h^{\sigma}q\, e_1- \mathbb{S}(v)e_1\bigg)\,dS_0=\frac{1}{2}\frac{d}{dt}\int_{\Sigma_{0}} g\,\bar{\rho}(0)|\xi^1|^2\,dS_0\\
&\qquad\qquad+\int_{\Sigma_{0}} \bigg(\dot{\Lambda}_h^{\sigma}v^1\dot{\Lambda}_h^{\sigma}\mathcal{B}(\partial_{\alpha}v, {Q}, \partial_t{Q})-\varepsilon\,\dot{\Lambda}_h^{\sigma}v^2\dot{\Lambda}_h^{\sigma}(\mathcal{B}_{7, i}^{\alpha}\partial_{\alpha}v^{i}+\widetilde{\mathcal{B}}_{2}\partial_t{Q})\\
&\qquad\qquad\qquad\qquad\qquad\qquad\qquad\qquad\qquad\qquad
    - \varepsilon\,\dot{\Lambda}_h^{\sigma}v^3\dot{\Lambda}_h^{\sigma}(\mathcal{B}_{8, i}^{\alpha}\partial_{\alpha}v^{i}+\widetilde{\mathcal{B}}_{3}\partial_t{Q})
\bigg) \,dS_0.
\end{split}
\end{equation}
Inserting  \eqref{linear-v-L2-1-7} into \eqref{linear-v-L2-1-6} yields \eqref{tan-linear-pseudo-0}, which ends the proof of Proposition \ref{lem-tan-pseudo-energy-1}.
\end{proof}

With Proposition \ref{lem-tan-pseudo-energy-1}in hand, we will deal with the estimates $\|\dot{\Lambda}_h^{\sigma}(v, \,q)\|_{L^\infty_t(L^2(\Omega))}
+\|\dot{\Lambda}_h^{\sigma}\xi^1\|_{L^\infty_t(L^2(\Sigma_0))}$ with $\sigma=\sigma_0,\,s,\,-\lambda$.

\subsubsection{Estimate of the tangential derivatives $\|\dot{\Lambda}_h^{s}(v,\,q)\|_{L^\infty_t(L^2(\Omega))}+\|\dot{\Lambda}_h^{s}\xi^1\|_{L^\infty_t(L^2(\Sigma_0))}$}

\begin{lem}\label{lem-tan-decay-total-1}
Let $\sigma>2$, $s_0 >2$, under the assumption of Lemma \ref{lem-tan-pseudo-energy-1}, if $E_{s_0}(t) \leq 1$ for all $t\in [0, T]$, then there holds that $\forall\,t\in [0, T]$
\begin{equation}\label{tan-decay-total-1}
\begin{split}
&\frac{d}{dt}\bigg(\|\sqrt{\bar{\rho}(x_1)}\,\dot{\Lambda}_h^{\sigma_0}\Lambda_h^{\sigma-\sigma_0}v\|_{L^2(\Omega)}^2+g^{-1} \|\sqrt{-\bar{\rho}'(x_1)}\, \dot{\Lambda}_h^{\sigma_0}\Lambda_h^{\sigma-\sigma_0}q\|_{L^2(\Omega)}^2 \\
&\qquad \qquad \qquad \qquad \qquad + g\,\bar{\rho}(0)\|\dot{\Lambda}_h^{\sigma_0}\Lambda_h^{\sigma-\sigma_0}\xi^1\|_{L^2(\Sigma_0)}^2\bigg)+ c_1 \|\dot{\Lambda}_h^{\sigma_0}\Lambda_h^{\sigma-\sigma_0}\nabla\,v\|_{L^2(\Omega)}^2 \\
&\leq C_1 (E_{\sigma}^{\frac{1}{2}}\dot{\mathcal{D}}_{s_0}^{\frac{1}{2}}
+E_{s_0}^{\frac{1}{2}} \dot{\mathcal{D}}_{\sigma}^{\frac{1}{2}})\dot{\mathcal{D}}_{\sigma}^{\frac{1}{2}}.
\end{split}
\end{equation}
\end{lem}
\begin{proof}
We first take $\sigma=\sigma_0$ in \eqref{tan-linear-pseudo-0} to get
\begin{equation}\label{linear-v-sigma0-0}
\begin{split}
 & \frac{d}{dt}\bigg(\|\sqrt{\bar{\rho}(x_1)}\,\dot{\Lambda}_h^{\sigma_0}v\|_{L^2(\Omega)}^2+g^{-1} \|\sqrt{-\bar{\rho}'(x_1)}\, \dot{\Lambda}_h^{\sigma_0}q\|_{L^2(\Omega)}^2 + g\,\bar{\rho}(0)\|\dot{\Lambda}_h^{\sigma_0}\xi^1\|_{L^2(\Sigma_0)}^2\bigg)\\
 &\quad+\bigg(\varepsilon \|\mathbb{D}^0(\dot{\Lambda}_h^{\sigma_0}v)\|_{L^2(\Omega)}^2+2\delta\, \|\grad \cdot \dot{\Lambda}_h^{\sigma_0}v\|_{L^2(\Omega)}^2\bigg)=2\sum_{j=1}^3\mathfrak{K}_j,
\end{split}
\end{equation}
where the remainder terms $\mathfrak{K}_j$ with $j=1, 2, 3$ are defined in \eqref{tan-linear-pseudo-0a}.

For the first integral $\mathfrak{K}_1$, we first observe
\begin{equation*}
\begin{split}
&|\int_{\Sigma_{0}}\dot{\Lambda}_h^{\sigma_0} \mathcal{B}(\partial_{\alpha}v, {Q}, \partial_t{Q})\, \dot{\Lambda}_h^{\sigma_0}v^1\, dS_0|\\
&\lesssim (\|\dot{\Lambda}_h^{\sigma_0}(\widetilde{\mathcal{B}}_5\partial_t{Q} ,\,{Q}^2\mathcal{R}({Q}),\,\mathcal{B}_{6, i}^{\alpha}\partial_{\alpha}v^{i},\,\mathcal{B}_{9, i}^{\alpha}\partial_{\alpha}v^{i})\|_{L^2(\Sigma_0)}) \|\dot{\Lambda}_h^{\sigma_0}v^1\|_{L^2(\Sigma_0)}\\
  &\lesssim (\|\dot{\Lambda}_h^{\sigma_0}\Lambda_h^{s_0-1-\sigma_0}\widetilde{\mathcal{B}}_5 \|_{L^2(\Sigma_0)}\|\dot{\Lambda}_h^{\sigma_0}\partial_t{Q} \|_{L^2(\Sigma_0)}+\|\dot{\Lambda}_h^{\sigma_0}\Lambda_h^{s_0-1-\sigma_0}{Q} \|_{L^2(\Sigma_0)}^2 \\
  &\qquad\qquad\qquad\qquad\qquad\qquad\qquad+\|\dot{\Lambda}_h^{\sigma_0}(\mathcal{B}_{6, i}^{\alpha}\partial_{\alpha}v^{i},\,\mathcal{B}_{9, i}^{\alpha}\partial_{\alpha}v^{i})\|_{H^1(\Omega)})  \|\dot{\Lambda}_h^{\sigma_0}v^1\|_{H^1(\Omega)},
\end{split}
\end{equation*}
which, along with the fact $\|\dot{\Lambda}_h^{\sigma_0}\Lambda_h^{s_0-1-\sigma_0}\widetilde{\mathcal{B}}_5 \|_{L^2(\Sigma_0)}\lesssim \|\dot{\Lambda}_h^{\sigma_0}\Lambda_h^{s_0-1-\sigma_0}\widetilde{\mathcal{B}}_5 \|_{H^1(\Omega)}\lesssim E_{s_0}^{\frac{1}{2}}$ and Lemma \ref{lem-est-B-Bv-1}, follows
\begin{equation*}
|\int_{\Sigma_{0}}\dot{\Lambda}_h^{\sigma_0} \mathcal{B}(\partial_{\alpha}v, {Q}, \partial_t{Q})\, \dot{\Lambda}_h^{\sigma_0}v^1\, dS_0|\lesssim E_{s_0}^{\frac{1}{2}}  \dot{\mathcal{D}}_{s_0}.
\end{equation*}
Along the same line, it can be obtained that
\begin{equation*}
\begin{split}
 & |\int_{\Sigma_{0}}(-\varepsilon\,\dot{\Lambda}_h^{\sigma_0}v^2\,\dot{\Lambda}_h^{\sigma_0}(\mathcal{B}_{7, i}^{\alpha}\partial_{\alpha}v^{i}+\widetilde{\mathcal{B}}_{2}\partial_t{Q})- \varepsilon\,\dot{\Lambda}_h^{\sigma_0}v^3\,\dot{\Lambda}_h^{\sigma_0}(\mathcal{B}_{8, i}^{\alpha}\partial_{\alpha}v^{i}+\widetilde{\mathcal{B}}_{3}\partial_t{Q})\, dS_0|\lesssim E_{s_0}^{\frac{1}{2}}  \dot{\mathcal{D}}_{s_0}.
\end{split}
\end{equation*}
We thus get
\begin{equation}\label{nonlin-v-sigma0-3}
\begin{split}
 & |\mathfrak{K}_1|\lesssim E_{s_0}^{\frac{1}{2}}  \dot{\mathcal{D}}_{s_0}.
\end{split}
\end{equation}
On the other hand, due to the product law \eqref{product-law-1} and Lemma \ref{lem-est-B-Bv-1}, one can see that
\begin{equation}\label{nonlin-v-sigma0-4}
\begin{split}
&|\mathfrak{K}_2|\lesssim\|\dot{\Lambda}_h^{\sigma_0+1}q\|_{L^2}\,
(\|\dot{\Lambda}_h^{\sigma_0-1}(B_{1, i}^{h, j}\partial_jv^i )\|_{L^2}
+\|\dot{\Lambda}_h^{\sigma_0-1}(B_{2, i}^hv^i)\|_{L^2})\lesssim E_{s_0}^{\frac{1}{2}}  \dot{\mathcal{D}}_{s_0}.
\end{split}
\end{equation}
For the term $\mathfrak{K}_3=\int_{ \Omega}\dot{\Lambda}_h^{\sigma_0} \mathfrak{g}\cdot \dot{\Lambda}_h^{\sigma_0} v\,dx$, from \eqref{est-g-2}, it follows that
\begin{equation}\label{nonlin-v-sigma0-6}
\begin{split}
&|\mathfrak{K}_3|\lesssim \|\dot{\Lambda}_h^{\sigma_0} \mathfrak{g}\|_{L^2}\|\dot{\Lambda}_h^{\sigma_0} v\|_{L^2}\lesssim \|\dot{\Lambda}_h^{\sigma_0} \mathfrak{g}\|_{L^2}\|\dot{\Lambda}_h^{\sigma_0} v\|_{H^1}\lesssim E_{s_0}^{\frac{1}{2}}\dot{\mathcal{D}}_{s_0}.
\end{split}
\end{equation}
Substituting \eqref{nonlin-v-sigma0-3}-\eqref{nonlin-v-sigma0-6} into \eqref{linear-v-sigma0-0} leads to \begin{equation}\label{linear-decay-sigma0-1}
\begin{split}
& \frac{d}{dt}\bigg(\|\sqrt{\bar{\rho}(x_1)}\,\dot{\Lambda}_h^{\sigma_0}v\|_{L^2(\Omega)}^2+g^{-1} \|\sqrt{-\bar{\rho}'(x_1)}\, \dot{\Lambda}_h^{\sigma_0}q\|_{L^2(\Omega)}^2 + g\,\bar{\rho}(0)\|\dot{\Lambda}_h^{\sigma_0}\xi^1\|_{L^2(\Sigma_0)}^2\bigg)\\
 &\quad+\bigg(\varepsilon \|\mathbb{D}^0(\dot{\Lambda}_h^{\sigma_0}v)\|_{L^2(\Omega)}^2+2\delta\, \|\grad \cdot \dot{\Lambda}_h^{\sigma_0}v\|_{L^2(\Omega)}^2\bigg) \lesssim E_{s_0}^{\frac{1}{2}}\dot{\mathcal{D}}_{s_0}.
\end{split}
\end{equation}

Next, we estimate the remainder terms $\mathfrak{K}_j$ with $j=1, 2, 3$ in \eqref{tan-linear-pseudo-0} with $\sigma>2$.
\begin{equation}\label{linear-v-Ntan-0}
\begin{split}
 &\frac{d}{dt}\bigg(\|\sqrt{\bar{\rho}(x_1)}\,\dot{\Lambda}_h^{s}v\|_{L^2(\Omega)}^2+g^{-1} \|\sqrt{-\bar{\rho}'(x_1)}\, \dot{\Lambda}_h^{s}q\|_{L^2(\Omega)}^2 + g\,\bar{\rho}(0)\|\dot{\Lambda}_h^{s}\xi^1\|_{L^2(\Sigma_0)}^2\bigg)\\
 &\quad+2\bigg(\frac{\varepsilon}{2} \|\mathbb{D}^0(\dot{\Lambda}_h^{s}v)\|_{L^2(\Omega)}^2+\delta\, \|\grad \cdot \dot{\Lambda}_h^{s}v\|_{L^2(\Omega)}^2\bigg)=2\sum_{j=1}^3\mathfrak{K}_j,
\end{split}
\end{equation}
where $\mathfrak{K}_j$ with $j=1, 2, 3$ are defined in \eqref{tan-linear-pseudo-0a} in which we take $\sigma=s$.
For the boundary integral $\mathfrak{K}_1$, making use of the trace theorem ensures
\begin{equation*}
\begin{split}
|\int_{\Sigma_{0}}\dot{\Lambda}_h^{\sigma}(\mathcal{B}_{6, i}^{\alpha} \partial_{\alpha}v^{i})\, \dot{\Lambda}_h^{\sigma}v^1\, dS_0|&\lesssim \|\dot{\Lambda}_h^{\sigma-\frac{1}{2}}(\mathcal{B}_{6, i}^{\alpha} \partial_{\alpha}v^{i})\|_{L^2(\Sigma_0)}\, \|\dot{\Lambda}_h^{\sigma+\frac{1}{2}}v^1\|_{L^2(\Sigma_0)}\\
&\lesssim \|\dot{\Lambda}_h^{\sigma-1}(\mathcal{B}_{6, i}^{\alpha} \partial_{\alpha}v^{i})\|_{H^1(\Omega)}\, \|\dot{\Lambda}_h^{\sigma}v^1\|_{H^1(\Omega)}.
\end{split}
\end{equation*}
We thus obtain from Lemma \ref{lem-est-B-Bv-1} that
\begin{equation*}
\begin{split}
&|\int_{\Sigma_{0}}\dot{\Lambda}_h^{\sigma}(\mathcal{B}_{6, i}^{\alpha} \partial_{\alpha}v^{i})\,\dot{\Lambda}_h^{\sigma}v^1\, dS_0|\lesssim \bigg( E_{\sigma}^{\frac{1}{2}}\dot{\mathcal{D}}_{s_0}^{\frac{1}{2}}
+E_{s_0}^{\frac{1}{2}} \dot{\mathcal{D}}_{\sigma}^{\frac{1}{2}}\bigg)\,\dot{\mathcal{D}}_{\sigma}^{\frac{1}{2}}.
\end{split}
\end{equation*}
It can be similarly proved that
\begin{equation*}\label{linear-v-Ntan-5}
\begin{split}
 & |\int_{\Sigma_{0}}\dot{\Lambda}_h^{\sigma}(\mathcal{B}_{9, i}^{\alpha}\partial_{\alpha}v^{i})\dot{\Lambda}_h^{\sigma}v^1\, dS_0|+|\int_{\Sigma_{0}}\dot{\Lambda}_h^{\sigma}(\mathcal{B}_{7, i}^{\alpha}\partial_{\alpha}v^{i}) \dot{\Lambda}_h^{\sigma}v^2\, dS_0|\\
&\qquad+|\int_{\Sigma_{0}}\dot{\Lambda}_h^{\sigma}(\mathcal{B}_{8, i}^{\alpha}\partial_{\alpha}v^{i})\, \dot{\Lambda}_h^{\sigma}v^3\, dS_0|\lesssim \bigg( E_{\sigma}^{\frac{1}{2}}\dot{\mathcal{D}}_{s_0}^{\frac{1}{2}}
+E_{s_0}^{\frac{1}{2}} \dot{\mathcal{D}}_{\sigma}^{\frac{1}{2}}\bigg)\,\dot{\mathcal{D}}_{\sigma}^{\frac{1}{2}}.
\end{split}
\end{equation*}
While for the integrals $\int_{\Sigma_{0}}\dot{\Lambda}_h^{\sigma}(\widetilde{\mathcal{B}}_5\partial_t{Q})\, \dot{\Lambda}_h^{\sigma}v^1\, dS_0$ and $\int_{\Sigma_{0}}\dot{\Lambda}_h^{\sigma}({Q}^2\mathcal{R}({Q}))\, \dot{\Lambda}_h^{\sigma}v^1\, dS_0$, thanks to Lemmas \ref{lem-product-law-1} and \ref{lem-est-B-Bv-1}, we obtain
\begin{equation*}
\begin{split}
&|\int_{\Sigma_{0}}\dot{\Lambda}_h^{\sigma}(\widetilde{\mathcal{B}}_5\partial_t{Q})\, \dot{\Lambda}_h^{\sigma}v^1\, dS_0|\lesssim \|\dot{\Lambda}_h^{\sigma-\frac{1}{2}}(\widetilde{\mathcal{B}}_5\partial_t{Q})\|_{L^2(\Sigma_0)}\, \|\dot{\Lambda}_h^{\sigma+\frac{1}{2}}v^1\|_{L^2(\Sigma_0)}\\
&\lesssim (\|\dot{\Lambda}_h^{\sigma-1}\widetilde{\mathcal{B}}_5\|_{H^1(\Omega)}
\|\dot{\Lambda}_h^{\sigma_0}\Lambda_h^{s_0-\frac{1}{2}-\sigma_0}\partial_t{Q}\|_{L^2(\Sigma_0)}
+\|\dot{\Lambda}_h^{\sigma-\frac{1}{2}} \partial_t{Q} \|_{L^2(\Sigma_0)}
\|\dot{\Lambda}_h^{\sigma_0}\Lambda_h^{s_0-1-\sigma_0}\widetilde{\mathcal{B}}_5\|_{H^1(\Omega)})\, \|\dot{\Lambda}_h^{\sigma}v^1\|_{H^1(\Omega)}\\
&\lesssim \bigg( E_{\sigma}^{\frac{1}{2}}\dot{\mathcal{D}}_{s_0}^{\frac{1}{2}}
+E_{s_0}^{\frac{1}{2}} \dot{\mathcal{D}}_{\sigma}^{\frac{1}{2}}\bigg)\,\dot{\mathcal{D}}_{\sigma}^{\frac{1}{2}},
\end{split}
\end{equation*}
and
\begin{equation*}
\begin{split}
&|\int_{\Sigma_{0}}\dot{\Lambda}_h^{\sigma}({Q}^2\mathcal{R}({Q}))\, \dot{\Lambda}_h^{\sigma}v^1\, dS_0|\lesssim \|\dot{\Lambda}_h^{\sigma-\frac{1}{2}}({Q}^2\mathcal{R}({Q}))\|_{L^2(\Sigma_0)}\, \|\dot{\Lambda}_h^{\sigma+\frac{1}{2}}v^1\|_{L^2(\Sigma_0)}\\
&\lesssim (\|\dot{\Lambda}_h^{\sigma-\frac{1}{2}}{Q}^2\|_{L^2(\Sigma_0)}
+\|\dot{\Lambda}_h^{\sigma-\frac{1}{2}}({Q}^2)\|_{L^2(\Sigma_0)}
\|\mathcal{R}({Q})-\mathcal{R}(0)\|_{L^\infty(\Sigma_0)}\\
&\qquad\qquad
+\|\dot{\Lambda}_h^{\sigma-\frac{1}{2}}(\mathcal{R}({Q})
-\mathcal{R}(0))\|_{L^2(\Sigma_0)}\|{Q}^2\|_{L^\infty(\Sigma_0)})\, \|\dot{\Lambda}_h^{\sigma}v^1\|_{H^1(\Omega)}\lesssim
 {E}_{s_0}^{\frac{1}{2}} \dot{\mathcal{D}}_{\sigma}.
\end{split}
\end{equation*}
We thus immediately get
\begin{equation}\label{linear-v-Ntan-6}
\begin{split}
 & |\mathfrak{K}_1|\lesssim \bigg( E_{\sigma}^{\frac{1}{2}}\dot{\mathcal{D}}_{s_0}^{\frac{1}{2}}
+E_{s_0}^{\frac{1}{2}} \dot{\mathcal{D}}_{\sigma}^{\frac{1}{2}}\bigg)\,\dot{\mathcal{D}}_{\sigma}^{\frac{1}{2}}.\\
\end{split}
\end{equation}
On the other hand, for $\mathfrak{K}_2$, it is easy to produce
\begin{equation*}\label{linear-v-Ntan-7}
\begin{split}
&|\mathfrak{K}_2|\lesssim\|\dot{\Lambda}_h^{\sigma}q\|_{L^2}\,
 \|\dot{\Lambda}_h^{\sigma}B_{1, i}^{h, j}\partial_jv^i\|_{L^2},
\end{split}
\end{equation*}
then using Lemma \ref{lem-est-B-Bv-1} again, one has
\begin{equation*}
\begin{split}
&|\mathfrak{K}_2| \lesssim\|\dot{\Lambda}_h^{\sigma}q\|_{L^2}\,
\bigg( E_{\sigma}^{\frac{1}{2}}\dot{\mathcal{D}}_{s_0}^{\frac{1}{2}}
+E_{s_0}^{\frac{1}{2}} \dot{\mathcal{D}}_{\sigma}^{\frac{1}{2}}\bigg)\lesssim \bigg( E_{\sigma}^{\frac{1}{2}}\dot{\mathcal{D}}_{s_0}^{\frac{1}{2}}
+E_{s_0}^{\frac{1}{2}} \dot{\mathcal{D}}_{\sigma}^{\frac{1}{2}}\bigg)\,\dot{\mathcal{D}}_{\sigma}^{\frac{1}{2}}.
\end{split}
\end{equation*}
Finally, for $\mathfrak{K}_3=\int_{ \Omega}\dot{\Lambda}_h^{\sigma} \mathfrak{g}\cdot \dot{\Lambda}_h^{\sigma}v\,dx$, we use \eqref{est-g-2} to get
\begin{equation}\label{linear-v-Ntan-10}
\begin{split}
&|\mathfrak{K}_3|\lesssim \|\dot{\Lambda}_h^{\sigma-1} \mathfrak{g}\|_{L^2}\|\dot{\Lambda}_h^{\sigma+1}v\|_{L^2}\lesssim ( E_{\sigma}^{\frac{1}{2}}\dot{\mathcal{D}}_{s_0}^{\frac{1}{2}}
+E_{s_0}^{\frac{1}{2}} \dot{\mathcal{D}}_{\sigma}^{\frac{1}{2}})\dot{\mathcal{D}}_{\sigma}^{\frac{1}{2}}.
\end{split}
\end{equation}
Therefore, inserting \eqref{linear-v-Ntan-6}-\eqref{linear-v-Ntan-10} into \eqref{linear-v-Ntan-0} yields
\begin{equation}\label{linear-decay-s-1}
\begin{split}
&\frac{d}{dt}\bigg(\|\sqrt{\bar{\rho}(x_1)}\,\dot{\Lambda}_h^{\sigma}v\|_{L^2(\Omega)}^2+g^{-1} \|\sqrt{-\bar{\rho}'(x_1)}\, \dot{\Lambda}_h^{\sigma}q\|_{L^2(\Omega)}^2 + g\,\bar{\rho}(0)\|\dot{\Lambda}_h^{\sigma}\xi^1\|_{L^2(\Sigma_0)}^2\bigg)\\
 &\quad+2\bigg(\frac{\varepsilon}{2} \|\mathbb{D}^0(\dot{{\Lambda}}_h^{\sigma}v)\|_{L^2(\Omega)}^2+\delta\, \|\grad \cdot \dot{\Lambda}_h^{\sigma}v\|_{L^2(\Omega)}^2\bigg)\lesssim (E_{\sigma}^{\frac{1}{2}}\dot{\mathcal{D}}_{s_0}^{\frac{1}{2}}
+E_{s_0}^{\frac{1}{2}} \dot{\mathcal{D}}_{\sigma}^{\frac{1}{2}})\dot{\mathcal{D}}_{\sigma}^{\frac{1}{2}}.
\end{split}
\end{equation}
Thanks to \eqref{linear-decay-sigma0-1}, \eqref{linear-decay-s-1}, and Korn's inequality \eqref{korn-comp-2}, we obtain \eqref{tan-decay-total-1}, which gives the desired result.
\end{proof}

\subsubsection{Estimate of $\|\dot{\Lambda}_h^{-\lambda} v\|_{L^\infty_t(L^2)}$}

\begin{lem}\label{lem-tan-bdd-lambda-1}
Under the assumption of Lemma \ref{lem-tan-pseudo-energy-1}, if $(\lambda,\,\sigma_0) \in (0, 1)$ satisfies $1-\lambda< \sigma_0\leq 1-\frac{1}{2}\lambda$, and $E_{s_0}(t) \leq 1$ for all $t\in [0, T]$, then there holds that $\forall\,t\in [0, T]$
\begin{equation}\label{linear-bdd-lambda-1}
\begin{split}
&\frac{d}{dt}\bigg(\|\sqrt{\bar{\rho}(x_1)}\,\dot{\Lambda}_h^{-\lambda}v\|_{L^2(\Omega)}^2+g^{-1} \|\sqrt{-\bar{\rho}'(x_1)}\, \dot{\Lambda}_h^{-\lambda}q\|_{L^2(\Omega)}^2 + g\,\bar{\rho}(0)\|\dot{\Lambda}_h^{-\lambda}\xi^1\|_{L^2(\Sigma_0)}^2\bigg)\\
 &\quad+c_1\|\dot{\Lambda}_h^{-\lambda}\nabla\,v\|_{L^2(\Omega)}^2\leq C_1   E_{s_0} (\dot{\mathcal{D}}_{s_0}^{\frac{1}{2}}+   \dot{\mathcal{E}}_{s_0}+   \dot{\mathcal{D}}_{s_0}).
\end{split}
\end{equation}
\end{lem}
\begin{proof}
First, taking $\sigma=-\lambda$ in \eqref{tan-linear-pseudo-0} shows
\begin{equation}\label{linear-bdd-lambda-2}
\begin{split}
 &\frac{d}{dt}\bigg(\|\sqrt{\bar{\rho}(x_1)}\,\dot{\Lambda}_h^{-\lambda}v\|_{L^2(\Omega)}^2+g^{-1} \|\sqrt{-\bar{\rho}'(x_1)}\, \dot{\Lambda}_h^{-\lambda}q\|_{L^2(\Omega)}^2 + g\,\bar{\rho}(0)\|\dot{\Lambda}_h^{-\lambda}\xi^1\|_{L^2(\Sigma_0)}^2\bigg)\\
 &\quad+2\bigg(\frac{\varepsilon}{2} \|\mathbb{D}^0(\dot{\Lambda}_h^{-\lambda}v)\|_{L^2(\Omega)}^2+\delta\, \|\grad \cdot \dot{\Lambda}_h^{-\lambda}v\|_{L^2(\Omega)}^2\bigg)=\mathfrak{K}_1+\mathfrak{K}_2+\mathfrak{K}_3
\end{split}
\end{equation}
Let's estimate $\mathfrak{K}_i$ ($i=1, 2, 3$). For $\mathfrak{K}_1$, due to the product law in $\dot{H}^s$, we obtain
\begin{equation*}
\begin{split}
&|\int_{\Sigma_{0}} \dot{\Lambda}_h^{-\lambda}v^1\,\dot{\Lambda}_h^{-\lambda}(\widetilde{\mathcal{B}}_5\partial_t{Q})\,dS_0|
\lesssim\|\dot{\Lambda}_h^{-\lambda}v^1\|_{L^2(\Sigma_0)}
\|\dot{\Lambda}_h^{-\lambda}(\widetilde{\mathcal{B}}_5\partial_t{Q})\|_{L^2(\Sigma_0)}\\
&\lesssim\|\dot{\Lambda}_h^{-\lambda}v^1\|_{H^1(\Omega)}
\|\dot{\Lambda}_h^{\sigma_0-1}\widetilde{\mathcal{B}}_5\|_{L^2(\Sigma_0)}
\|\dot{\Lambda}_h^{2-\sigma_0-\lambda}\partial_t{Q}\|_{L^2(\Sigma_0)}\\
&\lesssim\|\dot{\Lambda}_h^{-\lambda}\nabla\,v\|_{L^2(\Omega)}
\|\dot{\Lambda}_h^{\sigma_0-1}\widetilde{\mathcal{B}}_5\|_{H^1(\Omega)}
\|\dot{\Lambda}_h^{2-\sigma_0-\lambda}\partial_t{Q}\|_{L^2(\Sigma_0)}\lesssim E_{s_0}^{\frac{1}{2}}\|\dot{\Lambda}_h^{-\lambda}\nabla\,v\|_{L^2(\Omega)}
\dot{\mathcal{D}}_{s_0}^{\frac{1}{2}},
\end{split}
\end{equation*}
where we used the fact $2-\sigma_0-\lambda\geq \sigma_0$ in the last inequality.

While
\begin{equation*}
\begin{split}
&|\int_{\Sigma_{0}} \dot{\Lambda}_h^{-\lambda}v^1\,\dot{\Lambda}_h^{-\lambda}(\mathcal{B}_{6, i}^{\alpha}\partial_{\alpha}v^{i} ) \,dS_0|\lesssim \|\dot{\Lambda}_h^{-\lambda}v^1\|_{L^2(\Sigma_0)}\|\dot{\Lambda}_h^{-\lambda}(\mathcal{B}_{6, i}^{\alpha}\partial_{\alpha}v^{i} )\|_{L^2(\Sigma_0)}\\
&\lesssim \|\dot{\Lambda}_h^{-\lambda}\nabla\,v\|_{L^2(\Omega)}\|\dot{\Lambda}_h^{2-\sigma_0-\lambda}(\mathcal{B}_{6, i}^{\alpha})\|_{H^1(\Omega)}\|\dot{\Lambda}_h^{\sigma_0}\nabla\,v\|_{L^2(\Omega)}\\
&\lesssim \|\dot{\Lambda}_h^{-\lambda}\nabla\,v\|_{L^2(\Omega)}
E_{s_0}^{\frac{1}{2}}\|\dot{\Lambda}_h^{\sigma_0}\nabla\,v\|_{L^2(\Omega)},
 \end{split}
\end{equation*}
and
\begin{equation*}
\begin{split}
&|\int_{\Sigma_{0}} \dot{\Lambda}_h^{-\lambda}v^1\,\dot{\Lambda}_h^{-\lambda}({Q}^2\mathcal{R}({Q}))\,dS_0|
\lesssim\|\dot{\Lambda}_h^{-\lambda}v^1\|_{L^2(\Sigma_0)}
\|\dot{\Lambda}_h^{-\lambda}({Q}^2\mathcal{R}({Q}))\|_{L^2(\Sigma_0)}\\
&\lesssim\|\dot{\Lambda}_h^{-\lambda}v^1\|_{H^1(\Omega)}
(\|\dot{\Lambda}_h^{-\lambda}({Q}^2)\|_{L^2(\Sigma_0)}
+\|\dot{\Lambda}_h^{-\lambda}({Q}^2(\mathcal{R}({Q})-\mathcal{R}(0)))\|_{L^2(\Sigma_0)})\\
&\lesssim\|\dot{\Lambda}_h^{-\lambda}\nabla\,v\|_{L^2(\Omega)}
\|\dot{\Lambda}_h^{-\lambda}{Q}\|_{L^2(\Sigma_0)}
 \|\dot{\Lambda}_h^{\sigma_0} {\Lambda}_h^{s_0-1-\sigma_0}{Q}\|_{L^2(\Sigma_0)}\lesssim \|\dot{\Lambda}_h^{-\lambda}\nabla\,v\|_{L^2(\Omega)} E_{s_0}^{\frac{1}{2}}
\dot{\mathcal{D}}_{s_0}^{\frac{1}{2}}.
\end{split}
\end{equation*}
In the same line, we have
\begin{equation*}
\begin{split}
&|\int_{\Sigma_{0}} \dot{\Lambda}_h^{-\lambda}v^1\,\dot{\Lambda}_h^{-\lambda}((\mathcal{B}_{7, i}^{\alpha},\,\mathcal{B}_{8, i}^{\alpha},\,\mathcal{B}_{9, i}^{\alpha})\partial_{\alpha}v^{i} ) \,dS_0|\lesssim \|\dot{\Lambda}_h^{-\lambda}\nabla\,v\|_{L^2(\Omega)}E_3^{\frac{1}{2}}
\|\dot{\Lambda}_h^{\sigma_0}\nabla\,v\|_{L^2(\Omega)},\\
&|\varepsilon\,\int_{\Sigma_{0}} \bigg(\dot{\Lambda}_h^{-\lambda}v^2\,\dot{\Lambda}_h^{-\lambda}(\mathcal{B}_{7, i}^{\alpha}\partial_{\alpha}v^{i}+\widetilde{\mathcal{B}}_{2}\partial_t{Q})+\dot{\Lambda}_h^{-\lambda}v^3\,\dot{\Lambda}_h^{-\lambda}(\mathcal{B}_{8, i}^{\alpha}\partial_{\alpha}v^{i}+\widetilde{\mathcal{B}}_{3}\partial_t{Q})
\bigg) \,dS_0\\
&\lesssim \|\dot{\Lambda}_h^{-\lambda}\nabla\,v\|_{L^2(\Omega)} E_{s_0}^{\frac{1}{2}}
\dot{\mathcal{D}}_{s_0}^{\frac{1}{2}}.
\end{split}
\end{equation*}
Hence, we infer
\begin{equation*}
\begin{split}
&|\mathfrak{K}_1|\lesssim \|\dot{\Lambda}_h^{-\lambda}\nabla\,v\|_{L^2(\Omega)} E_{s_0}^{\frac{1}{2}}
\dot{\mathcal{D}}_{s_0}^{\frac{1}{2}}.
\end{split}
\end{equation*}
For $\mathfrak{K}_3=\int_{ \Omega}\dot{\Lambda}_h^{-\lambda} \mathfrak{g}\cdot \dot{\Lambda}_h^{-\lambda}v\,dx$, we use \eqref{est-g-2} to get
\begin{equation*}
\begin{split}
&|\mathfrak{K}_3|\lesssim \|\dot{\Lambda}_h^{-\lambda} \mathfrak{g}\|_{L^2}\|\dot{\Lambda}_h^{-\lambda}v\|_{L^2}\lesssim  \|\dot{\Lambda}_h^{-\lambda}\nabla\,v\|_{L^2(\Omega)}E_{s_0}^{\frac{1}{2}}(\dot{\mathcal{E}}_{\sigma_0+\frac{1}{2}}^{\frac{1}{2}}  + \dot{\mathcal{D}}_{\sigma_0+1}^{\frac{1}{2}}).
\end{split}
\end{equation*}
Therefore, we get
\begin{equation}\label{linear-bdd-lambda-4ff}
\begin{split}
&|\mathfrak{K}_1|+|\mathfrak{K}_3| \lesssim  \|\dot{\Lambda}_h^{-\lambda}\nabla\,v\|_{L^2(\Omega)}E_{s_0}^{\frac{1}{2}}(\dot{\mathcal{E}}_{s_0}^{\frac{1}{2}}  + \dot{\mathcal{D}}_{s_0}^{\frac{1}{2}}).
\end{split}
\end{equation}
On the other hand, thanks to  Lemma \ref{lem-est-B-Bv-1}, we get
\begin{equation}\label{linear-bdd-lambda-4ee}
\begin{split}
 |\mathfrak{K}_2|&\lesssim \|\dot{\Lambda}_h^{2(1-\sigma_0-\lambda)}q\|_{L^2(\Omega)}\|\dot{\Lambda}_h^{2(\sigma_0-1)} B_{1, i}^{h, j}\partial_jv^i\|_{L^2(\Omega)}\\
  &\lesssim  \|\dot{\Lambda}_h^{2(1-\sigma_0-\lambda)}q\|_{L^2(\Omega)}
E_{s_0}^{\frac{1}{2}}
 \|\dot{\Lambda}_h^{\sigma_0}\nabla\,v\|_{L^2(\Omega)}\lesssim  E_{s_0}
\dot{\mathcal{D}}_{s_0}^{\frac{1}{2}},
 \end{split}
\end{equation}
where in the last inequality, we have used the fact $1>2(1-\sigma_0-\lambda)>-\lambda$ (from the condition $1-\lambda<\sigma_0\leq 1-\frac{\lambda}{2}$).

Therefore, substituting \eqref{linear-bdd-lambda-4ff} and \eqref{linear-bdd-lambda-4ee} into \eqref{linear-bdd-lambda-2} yields
\begin{equation*}\label{linear-bdd-lambda-14}
\begin{split}
 &\frac{d}{dt}\bigg(\|\sqrt{\bar{\rho}(x_1)}\,\dot{\Lambda}_h^{-\lambda}v\|_{L^2(\Omega)}^2+g^{-1} \|\sqrt{-\bar{\rho}'(x_1)}\, \dot{\Lambda}_h^{-\lambda}q\|_{L^2(\Omega)}^2 + g\,\bar{\rho}(0)\|\dot{\Lambda}_h^{-\lambda}\xi^1\|_{L^2(\Sigma_0)}^2\bigg)\\
 &\quad+2\bigg(\frac{\varepsilon}{2} \|\mathbb{D}^0(\dot{\Lambda}_h^{-\lambda}v)\|_{L^2(\Omega)}^2+\delta\, \|\grad \cdot \dot{\Lambda}_h^{-\lambda}v\|_{L^2(\Omega)}^2\bigg)\\
  &\lesssim \|\dot{\Lambda}_h^{-\lambda}\nabla\,v\|_{L^2(\Omega)} E_{s_0}^{\frac{1}{2}}(\dot{\mathcal{E}}_{s_0}^{\frac{1}{2}}  + \dot{\mathcal{D}}_{s_0}^{\frac{1}{2}})+   E_{s_0}
\dot{\mathcal{D}}_{s_0}^{\frac{1}{2}}.
\end{split}
\end{equation*}
Applying Young's inequality and Korn's inequality \eqref{korn-comp-2} implies \eqref{linear-bdd-lambda-1}, which is what we wanted to prove.
\end{proof}

Thanks to Lemmas \ref{lem-tan-decay-total-1} and \ref{lem-tan-bdd-lambda-1}, we have
\begin{lem}\label{lem-tan-bdd-N+1-1}
Let $\sigma>2$, $s_0>2$, under the assumption of Lemma \ref{lem-tan-pseudo-energy-1}, if $(\lambda,\,\sigma_0) \in (0, 1)$ satisfies $1-\lambda< \sigma_0\leq 1-\frac{1}{2}\lambda$, and $E_{s_0}(t) \leq 1$ for all $t\in [0, T]$, then there holds that $\forall\,t\in [0, T]$
\begin{equation}\label{tan-bdd-s+1-1}
\begin{split}
&\frac{d}{dt}\bigg(\|\sqrt{\bar{\rho}(x_1)}\,(\dot{\Lambda}_h^{-\lambda}v, \dot{\Lambda}_h^{\sigma_0}\Lambda_h^{\sigma-\sigma_0}v)\|_{L^2(\Omega)}^2+g^{-1} \|\sqrt{-\bar{\rho}'(x_1)}\, (\dot{\Lambda}_h^{-\lambda}q, \dot{\Lambda}_h^{\sigma_0}\Lambda_h^{\sigma-\sigma_0}q)\|_{L^2(\Omega)}^2 \\
&\qquad\qquad+ g\,\bar{\rho}(0)\|(\dot{\Lambda}_h^{-\lambda}\xi^1, \dot{\Lambda}_h^{\sigma_0}\Lambda_h^{\sigma-\sigma_0}\xi^1)\|_{L^2(\Sigma_0)}^2\bigg) \\
&\qquad+ c_1\|\nabla(\dot{\Lambda}_h^{-\lambda}v, \dot{\Lambda}_h^{\sigma_0}\Lambda_h^{\sigma-\sigma_0}v)\|_{L^2(\Omega)}^2\leq C_1 \bigg(E_{\sigma}^{\frac{1}{2}}\dot{\mathcal{D}}_{s_0}^{\frac{1}{2}}\dot{\mathcal{D}}_{\sigma}^{\frac{1}{2}}
+E_{s_0}^{\frac{1}{2}} \dot{\mathcal{D}}_{\sigma}+ E_{s_0} (\dot{\mathcal{D}}_{s_0}^{\frac{1}{2}}+\dot{\mathcal{E}}_{s_0})\bigg).
\end{split}
\end{equation}
\end{lem}

\subsection{Estimate of $Q$ on the interface}
In this subsection, we derive estimate in terms of $Q$ on the interface.

\begin{lem}\label{lem-bdry-energy-1}
Let $s_0>2$, under the assumption of Lemma \ref{lem-tan-pseudo-energy-1}, if $(\lambda,\,\sigma_0) \in (0, 1)$ satisfies $1-\lambda< \sigma_0\leq 1-\frac{1}{2}\lambda$, and $E_{s_0}(t) \leq 1$ for all $t\in [0, T]$, then there hold that $\forall\,t\in [0, T]$
\begin{equation}\label{est-Q-bdry-low-1}
  \begin{split}
 \frac{d}{dt}\|\dot{\Lambda}_h^{-\lambda}{Q}\|_{L^2(\Sigma_0)}^2
&+2c_2\|(\dot{\Lambda}_h^{-\lambda} {Q},\,\dot{\Lambda}_h^{-\lambda}\partial_t{Q})\|_{L^2(\Sigma_0)}^2 \\
&\leq C_2   \|\dot{\Lambda}_h^{\frac{1}{2}-\lambda} \nabla\,v \|_{L^2(\Omega)}^2+C_2 E_{s_0}\|(\dot{\Lambda}_h^{-\lambda} {Q},\,\dot{\Lambda}_h^{-\lambda}\partial_t{Q})\|_{L^2(\Sigma_0)}^2,
\end{split}
\end{equation}
\begin{equation}\label{est-Q-bdry-decay-1}
  \begin{split}
&\frac{d}{dt}\|\dot{\Lambda}_h^{\sigma_0-1} {\Lambda}_h^{s_0+\frac{1}{2}-\sigma_0}{Q}\|_{L^2(\Sigma_0)}^2
+ 2c_2\|\dot{\Lambda}_h^{\sigma_0-1} {\Lambda}_h^{s_0+\frac{1}{2}-\sigma_0}({Q},\,\partial_t{Q})\|_{L^2(\Sigma_0)}^2 \\
&\qquad\leq  C_2(\|\dot{\Lambda}_h^{\sigma_0} {\Lambda}_h^{s_0 -\sigma_0}\nabla\,v \|_{L^2(\Omega)}^2 +E_{s_0}\|\dot{\Lambda}_h^{\sigma_0-1} {\Lambda}_h^{s_0+\frac{1}{2}-\sigma_0}({Q},\,\partial_t{Q})\|_{L^2(\Sigma_0)}^2),
\end{split}
\end{equation}
and
\begin{equation}\label{est-Q-bdry-decay-2}
  \begin{split}
&\frac{d}{dt}\|\dot{\Lambda}_h^{\sigma}{Q}\|_{L^2(\Sigma_0)}^2
+c_2\|(\dot{\Lambda}_h^{\sigma} {Q},\,\dot{\Lambda}_h^{\sigma}\partial_t{Q})\|_{L^2(\Sigma_0)}^2 \\
&\qquad\leq C_2( \|\dot{\Lambda}_h^{\sigma+\frac{1}{2}} \nabla\,v \|_{L^2(\Omega)}^2+ E_{s_0}\dot{\mathcal{D}}_{\sigma+\frac{1}{2}}+E_{\sigma+\frac{1}{2}}\dot{\mathcal{D}}_{s_0})\quad(\forall\,\,\sigma \geq s_0-\frac{1}{2}),
\end{split}
\end{equation}
\end{lem}
\begin{proof}
From the boundary equation \eqref{eqns-Q-bdry-1}, we find for any $\sigma \in \mathbb{R}$
\begin{equation*}
  \begin{split}
\frac{d}{dt}\|\dot{\Lambda}_h^{\sigma}{Q}\|_{L^2(\Sigma_0)}^2
&+2c_2(\|\dot{\Lambda}_h^{\sigma}{Q}\|_{L^2(\Sigma_0)}^2+\|\dot{\Lambda}_h^{\sigma}\partial_t{Q}\|_{L^2(\Sigma_0)}^2)\\
&\lesssim \|\dot{\Lambda}_h^{\sigma}\partial_h v\|_{L^2(\Sigma_0)}^2+\|\dot{\Lambda}_h^{\sigma}(\mathcal{B}_{9, i}^{\alpha}\partial_{\alpha}v^{i}, {Q}^2\mathcal{R}({Q}), \widetilde{\mathcal{B}}_{4}\partial_t{Q})\|_{L^2(\Sigma_0)}^2.
\end{split}
\end{equation*}
Thanks to Lemma \ref{lem-product-law-1}, we find $\forall \,\, \sigma \in [-\lambda, 1)$
\begin{equation*}
  \begin{split}
\|\dot{\Lambda}_h^{\sigma}(\mathcal{B}_{9, i}^{\alpha}\partial_{\alpha}v^{i})\|_{L^2(\Sigma_0)} &\lesssim \|\dot{\Lambda}_h^{\sigma} \partial_{\alpha}v^{i} \|_{L^2(\Sigma_0)}\| \dot{\Lambda}_h^{\sigma_0}  \Lambda_h^{s_0-1-\sigma_0}\mathcal{B}_{9, i}^{\alpha} \|_{L^2(\Sigma_0)},\\
 \|\dot{\Lambda}_h^{\sigma}({Q}^2\mathcal{R}({Q}))\|_{L^2(\Sigma_0)}&\lesssim \|\dot{\Lambda}_h^{\sigma} {Q} \|_{L^2(\Sigma_0)} \|\dot{\Lambda}_h^{\sigma_0}  \Lambda_h^{s_0-1-\sigma_0}(Q\mathcal{R}({Q})) \|_{L^2(\Sigma_0)},\\
 \|\dot{\Lambda}_h^{\sigma}(\widetilde{\mathcal{B}}_{4}\partial_t{Q})\|_{L^2(\Sigma_0)} &\lesssim \|\dot{\Lambda}_h^{\sigma} \partial_t{Q} \|_{L^2(\Sigma_0)}
\|\dot{\Lambda}_h^{\sigma_0}  \Lambda_h^{s_0-1-\sigma_0}\widetilde{\mathcal{B}}_{4} \|_{L^2(\Sigma_0)}.
\end{split}
\end{equation*}
Notice that
\begin{equation}\label{est-infty-B-interface-1}
  \begin{split}
&\| \dot{\Lambda}_h^{\sigma_0}  \Lambda_h^{s_0-1-\sigma_0}(\mathcal{B}_{9, i}^{\alpha}, \widetilde{\mathcal{B}}_{4}) \|_{L^2(\Sigma_0)}\lesssim \| \dot{\Lambda}_h^{\sigma_0}  \Lambda_h^{s_0-1-\sigma_0}(\mathcal{B}_{9, i}^{\alpha}, \widetilde{\mathcal{B}}_{4}) \|_{H^1(\Omega)}\lesssim E_{s_0}^{\frac{1}{2}},\\
 &\|\dot{\Lambda}_h^{\sigma_0}  \Lambda_h^{s_0-1-\sigma_0}(Q\mathcal{R}({Q})) \|_{L^2(\Sigma_0)}\lesssim \|\dot{\Lambda}_h^{\sigma_0}  \Lambda_h^{s_0-1-\sigma_0}{Q} \|_{L^2(\Sigma_0)} \lesssim E_{s_0}^{\frac{1}{2}},
\end{split}
\end{equation}
we infer
\begin{equation*}
  \begin{split}
  \|\dot{\Lambda}_h^{-\lambda}(\mathcal{B}_{9, i}^{\alpha}\partial_{\alpha}v^{i},\,&{Q}^2\mathcal{R}({Q}),
  \,\widetilde{\mathcal{B}}_{4}\partial_t{Q})\|_{L^2(\Sigma_0)}^2 \\
  &\lesssim E_{s_0}(\| \dot{\Lambda}_h^{1-\lambda}v \|_{L^2(\Sigma_0)}^2+\|\dot{\Lambda}_h^{-\lambda} {Q} \|_{L^2(\Sigma_0)}^2+\|\dot{\Lambda}_h^{-\lambda} \partial_t{Q} \|_{L^2(\Sigma_0)}^2),
\end{split}
\end{equation*}
and  $\forall\,\,\sigma \in [\sigma_0-1, 1)$
\begin{equation*}
  \begin{split}
  \|\dot{\Lambda}_h^{\sigma}(\mathcal{B}_{9, i}^{\alpha}\partial_{\alpha}v^{i},\,&{Q}^2\mathcal{R}({Q}),
  \,\widetilde{\mathcal{B}}_{4}\partial_t{Q})\|_{L^2(\Sigma_0)}^2 \\
  &\lesssim E_{s_0}(\| \dot{\Lambda}_h^{1+\sigma}v \|_{L^2(\Sigma_0)}^2+\|\dot{\Lambda}_h^{\sigma} {Q} \|_{L^2(\Sigma_0)}^2+\|\dot{\Lambda}_h^{\sigma} \partial_t{Q} \|_{L^2(\Sigma_0)}^2).
\end{split}
\end{equation*}
Similarly, for any $\sigma \geq 1$, we have
\begin{equation*}
  \begin{split}
 \|\dot{\Lambda}_h^{\sigma}(\mathcal{B}_{9, i}^{\alpha}\partial_{\alpha}v^{i})\|_{L^2(\Sigma_0)}
 \lesssim &\|\dot{\Lambda}_h^{\sigma}\partial_{h}v\|_{L^2(\Sigma_0)}\| \dot{\Lambda}_h^{\sigma_0}\Lambda_h^{s_0-1-\sigma_0}\mathcal{B}_{9, i}^{\alpha} \|_{L^2(\Sigma_0)}\\
 &+\|\dot{\Lambda}_h^{\sigma}\mathcal{B}_{9, i}^{\alpha}\|_{L^2(\Sigma_0)}\| \dot{\Lambda}_h^{\sigma_0}\Lambda_h^{s_0-1-\sigma_0}\partial_{\alpha}v\|_{L^2(\Sigma_0)},
\end{split}
\end{equation*}
\begin{equation*}
  \begin{split}
\|\dot{\Lambda}_h^{\sigma}({Q}^2\mathcal{R}({Q}))\|_{L^2(\Sigma_0)}
\lesssim &\|\dot{\Lambda}_h^{\sigma}Q\|_{L^2(\Sigma_0)}
\|\dot{\Lambda}_h^{\sigma_0}\Lambda_h^{s_0-1-\sigma_0}({Q}\mathcal{R}({Q}))\|_{L^2(\Sigma_0)}\\
&
+\|\dot{\Lambda}_h^{\sigma}({Q}\mathcal{R}({Q})) \|_{L^2(\Sigma_0)}
\|\dot{\Lambda}_h^{\sigma_0}\Lambda_h^{s_0-1-\sigma_0}Q\|_{L^2(\Sigma_0)}\\
\lesssim &\|\dot{\Lambda}_h^{\sigma}Q\|_{L^2(\Sigma_0)}
\|\dot{\Lambda}_h^{\sigma_0}\Lambda_h^{s_0-1-\sigma_0}Q\|_{L^2(\Sigma_0)},
\end{split}
\end{equation*}
\begin{equation*}
  \begin{split}
 \|\dot{\Lambda}_h^{\sigma}(\widetilde{\mathcal{B}}_{4}\partial_t{Q})\|_{L^2(\Sigma_0)} \lesssim &\|\dot{\Lambda}_h^{\sigma}\widetilde{\mathcal{B}}_{4} \|_{L^2(\Sigma_0)}
\|\dot{\Lambda}_h^{\sigma_0}\Lambda_h^{s_0-1-\sigma_0} \partial_t{Q}\|_{L^2(\Sigma_0)}\\
&
+\|\dot{\Lambda}_h^{\sigma} \partial_t{Q} \|_{L^2(\Sigma_0)}
\|\dot{\Lambda}_h^{\sigma_0}\Lambda_h^{s_0-1-\sigma_0}\widetilde{\mathcal{B}}_{4} \|_{L^2(\Sigma_0)}.
\end{split}
\end{equation*}
Hence, thanks to \eqref{est-infty-B-interface-1} and
\begin{equation*}
  \begin{split}
   &\|\dot{\Lambda}_h^{\sigma}(\mathcal{B}_{9, i}^{\alpha}, \widetilde{\mathcal{B}}_{4})\|_{L^2(\Sigma_0)} \lesssim \|\dot{\Lambda}_h^{\sigma_0}{\Lambda}_h^{\sigma-\sigma_0}(\mathcal{B}_{9, i}^{\alpha}, \widetilde{\mathcal{B}}_{4})\|_{H^1(\Omega)} \lesssim E_{s_0}^{\frac{1}{2}}\quad (\forall\,\,\sigma \in [1, s_0-1]),\\
      &\|\dot{\Lambda}_h^{\sigma}(\mathcal{B}_{9, i}^{\alpha}, \widetilde{\mathcal{B}}_{4})\|_{L^2(\Sigma_0)} \lesssim \|\dot{\Lambda}_h^{\sigma_0}{\Lambda}_h^{\sigma-\sigma_0-\frac{1}{2}}(\mathcal{B}_{9, i}^{\alpha}, \widetilde{\mathcal{B}}_{4})\|_{H^1(\Omega)} \lesssim E_{s_0}^{\frac{1}{2}}\quad (\forall\,\,\sigma \in [s_0-1, s_0-\frac{1}{2}]),\\
 &\|\dot{\Lambda}_h^{\sigma}(\mathcal{B}_{9, i}^{\alpha}, \widetilde{\mathcal{B}}_{4})\|_{L^2(\Sigma_0)} \lesssim \|\dot{\Lambda}_h^{\sigma_0}{\Lambda}_h^{\sigma-\sigma_0-\frac{1}{2}}(\mathcal{B}_{9, i}^{\alpha}, \widetilde{\mathcal{B}}_{4})\|_{H^1(\Omega)} \lesssim E_{\sigma+\frac{1}{2}}^{\frac{1}{2}}\quad (\forall\,\,\sigma \geq s_0-\frac{1}{2}),
\end{split}
\end{equation*}
we find that, for $\sigma \in [1, s_0-\frac{1}{2}]$,
\begin{equation*}
  \begin{split}
 &\|\dot{\Lambda}_h^{\sigma}(\mathcal{B}_{9, i}^{\alpha}\partial_{\alpha}v^{i},\,{Q}^2\mathcal{R}({Q}),
  \,\widetilde{\mathcal{B}}_{4}\partial_t{Q})\|_{L^2(\Sigma_0)}^2 \\
 & \lesssim E_{s_0}^{\frac{1}{2}}(\| \dot{\Lambda}_h^{\sigma_0}(\Lambda_h^{\sigma+1-\sigma_0} v,\,\Lambda_h^{s_0-\sigma_0} v)\|_{L^2(\Sigma_0)}
  +\|\dot{\Lambda}_h^{\sigma}(Q, \partial_t{Q})\|_{L^2(\Sigma_0)} ),
\end{split}
\end{equation*}
and for $\sigma \geq s_0-\frac{1}{2}$,
\begin{equation*}
  \begin{split}
 &\|\dot{\Lambda}_h^{\sigma}(\mathcal{B}_{9, i}^{\alpha}\partial_{\alpha}v^{i},\,{Q}^2\mathcal{R}({Q}),
  \,\widetilde{\mathcal{B}}_{4}\partial_t{Q})\|_{L^2(\Sigma_0)}^2 \\
 &\lesssim E_{s_0}^{\frac{1}{2}}\| \dot{\Lambda}_h^{\sigma_0}\Lambda_h^{\sigma+1-\sigma_0} v\|_{L^2(\Sigma_0)} + E_{s_0}^{\frac{1}{2}}\dot{\mathcal{D}}_{\sigma+\frac{1}{2}}^{\frac{1}{2}}
 +E_{\sigma+\frac{1}{2}}^{\frac{1}{2}}\dot{\mathcal{D}}_{s_0}^{\frac{1}{2}}.
\end{split}
\end{equation*}
Therefore, we immediately get \eqref{est-Q-bdry-low-1}, \eqref{est-Q-bdry-decay-1}, and \eqref{est-Q-bdry-decay-2}, and then finish the proof of Lemma \ref{lem-bdry-energy-1}.
\end{proof}

\begin{col}\label{cor-bdry-energy-1}
Let $s_0>2$, under the assumption of Lemma \ref{lem-bdry-energy-1}, if, in addition, $E_{s_0}(t) \leq \frac{c_2}{C_2}$ for all $t\in [0, T]$, then there hold that $\forall\,t\in [0, T]$
\begin{equation}\label{est-Q-bdry-low-1aa1}
  \begin{split}
 \frac{d}{dt}\|\dot{\Lambda}_h^{-\lambda}{Q}\|_{L^2(\Sigma_0)}^2
&+c_2\|(\dot{\Lambda}_h^{-\lambda} {Q},\,\dot{\Lambda}_h^{-\lambda}\partial_t{Q})\|_{L^2(\Sigma_0)}^2\leq C_2   \|\dot{\Lambda}_h^{\frac{1}{2}-\lambda}\nabla\, v \|_{L^2(\Omega)}^2,
\end{split}
\end{equation}
\begin{equation}\label{est-Q-bdry-decay-1bb1}
  \begin{split}
&\frac{d}{dt}\|\dot{\Lambda}_h^{\sigma_0-1} {\Lambda}_h^{s_0+\frac{1}{2}-\sigma_0}{Q}\|_{L^2(\Sigma_0)}^2
+ c_2\|\dot{\Lambda}_h^{\sigma_0-1} {\Lambda}_h^{s_0+\frac{1}{2}-\sigma_0}({Q},\,\partial_t{Q})\|_{L^2(\Sigma_0)}^2 \\
&\qquad\leq  C_2\|\dot{\Lambda}_h^{\sigma_0} {\Lambda}_h^{s_0 -\sigma_0}\nabla\, v \|_{L^2(\Omega)}^2,
\end{split}
\end{equation}
\begin{equation}\label{est-Q-bdry-decay-2cc1}
  \begin{split}
&\frac{d}{dt}\|\dot{\Lambda}_h^{\sigma_0-1} {\Lambda}_h^{\sigma+1-\sigma_0}{Q}\|_{L^2(\Sigma_0)}^2
+c_2\| \dot{\Lambda}_h^{\sigma_0-1} {\Lambda}_h^{\sigma+1-\sigma_0}( {Q},\,\partial_t{Q})\|_{L^2(\Sigma_0)}^2 \\
&\qquad\leq C_2( \|\dot{\Lambda}_h^{\sigma_0} {\Lambda}_h^{\sigma+\frac{1}{2}-\sigma_0} \nabla\, v \|_{L^2(\Omega)}^2+ E_{s_0}\dot{\mathcal{D}}_{\sigma+\frac{1}{2}}+E_{\sigma+\frac{1}{2}}\dot{\mathcal{D}}_{s_0})\quad(\forall\,\,\sigma \geq s_0-\frac{1}{2}),
\end{split}
\end{equation}
and
\begin{equation}\label{est-Q-bdry-decay-2dd1}
  \begin{split}
&\frac{d}{dt}\|(\dot{\Lambda}_h^{-\lambda}{Q},\dot{\Lambda}_h^{\sigma_0-1} {\Lambda}_h^{\sigma+1-\sigma_0}{Q})\|_{L^2(\Sigma_0)}^2\\
&
+c_2(\| \dot{\Lambda}_h^{-\lambda}( {Q},\,\partial_t{Q})\|_{L^2(\Sigma_0)}^2+\| \dot{\Lambda}_h^{\sigma_0-1} {\Lambda}_h^{\sigma+1-\sigma_0}( {Q},\,\partial_t{Q})\|_{L^2(\Sigma_0)}^2) \\
&\,\leq C_2(  \|(\dot{\Lambda}_h^{\frac{1}{2}-\lambda}\nabla\, v , \dot{\Lambda}_h^{\sigma_0} {\Lambda}_h^{\sigma+\frac{1}{2}-\sigma_0} \nabla\, v) \|_{L^2(\Omega)}^2+ E_{s_0}\dot{\mathcal{D}}_{\sigma+\frac{1}{2}}+E_{\sigma+\frac{1}{2}}\dot{\mathcal{D}}_{s_0})\quad(\forall\,\,\sigma \geq s_0-\frac{1}{2}).
\end{split}
\end{equation}
\end{col}

\subsection{Energy estimates of the gradient of the velocity}

In this subsection, we derive energy estimates in terms of $\nabla\,v$ as well as its horizontal derivatives.
For this, we still use the linearized form \eqref{eqns-linear-1} of the system \eqref{eqns-pert-1} as in Section \ref{subsect-est-hori-1}. Set
\begin{equation*}
\begin{split}
 &   \mathring{\mathcal{E}}( \mathcal{P}(\partial_h)\nabla\, v)\eqdefa \frac{\varepsilon}{2}\sum_{\beta=2}^3(\|\mathcal{P}(\partial_h) \mathcal{G}^{\beta}\|_{L^2}^2-\|\mathcal{P}(\partial_h)(-\partial_{\beta}v^1
+\widetilde{\mathcal{G}}^{\beta})\|_{L^2}^2)\\
&\,+ \frac{(\delta+\frac{4}{3}\varepsilon)}{2}\bigg(\|\mathcal{P}(\partial_h)\mathcal{G}^1\|_{L^2}^2
-\|\mathcal{P}(\partial_h)
(-\frac{\delta-\frac{2}{3}\varepsilon}{\delta+\frac{4}{3}\varepsilon} \nabla_h\cdot v^h+\frac{\bar{\rho}(0)}{\delta+\frac{4}{3}\varepsilon}\, \mathcal{H}({Q})+\widetilde{\mathcal{G}}^1)\|_{L^2(\Omega)}^2\\
& \qquad\qquad\qquad\qquad\qquad-\frac{2\bar{\rho}(0)}{\delta+\frac{4}{3}\varepsilon}
\int_{\Omega}\mathcal{P}(\partial_h) v^1\mathcal{P}(\partial_h)\partial_1  \mathcal{H}({Q}) \,dx\bigg)\\
   &\qquad +\bar{\rho}(0) \int_{\Sigma_0}\mathcal{P}(\partial_h) v^1\, \mathcal{P}(\partial_h)q\,dS_0-\int_{\Omega} \mathcal{P}(\partial_h)\nabla\cdot(\bar{\rho}(x_1)v ) \mathcal{P}(\partial_h)q\,dx.
        \end{split}
\end{equation*}

The standard procedure for getting the energy estimates can be adopted to get
\begin{lem}\label{lem-pseudo-energy-tv-1}
Let $(v, \xi)$ be smooth solution to the system \eqref{eqns-pert-1}, there holds that
\begin{equation}\label{pseudo-energy-tv-1}
\begin{split}
 & \frac{d}{dt}\mathring{\mathcal{E}}(\dot{\Lambda}_h^{\sigma}\nabla\, v)+\| \sqrt{\bar{\rho}(x_1)}\dot{\Lambda}_h^{\sigma} \partial_t v\|_{L^2}^2=-\sum_{j=1}^{6}\mathfrak{J}_j,
     \end{split}
\end{equation}
where
\begin{equation*}
\begin{split}
&\mathfrak{J}_1:=-\sum_{i=1}^3\int_{\Omega}\dot{\Lambda}_h^{\sigma} \partial_t v^i  \dot{\Lambda}_h^{\sigma}\mathcal{L}_{i}(\partial_h\nabla\,v)\,dx\\
&\qquad\qquad+(\delta+\frac{1}{3}\varepsilon)\int_{\Omega}(\dot{\Lambda}_h^{\sigma}
 \partial_t v^h\cdot \dot{\Lambda}_h^{\sigma}\nabla_h \partial_1v^1 +\dot{\Lambda}_h^{\sigma}\partial_t v^1 \dot{\Lambda}_h^{\sigma}\partial_1 \nabla_h\cdot v^h) \,dx,\\
&\mathfrak{J}_2:=-\bar{\rho}(0) \int_{\Sigma_0} \dot{\Lambda}_h^{\sigma} v^1 \dot{\Lambda}_h^{\sigma} \partial_t {Q}\,dS_0-g\,\bar{\rho}(0)  \|\dot{\Lambda}_h^{\sigma} v^1 \|_{L^2(\Sigma_0)}^2,
    \end{split}
\end{equation*}
\begin{equation*}
\begin{split}
&\mathfrak{J}_3:=g\int_{\Omega} \dot{\Lambda}_h^{\sigma} \nabla\cdot(v\,\bar{\rho}(x_1))\, ( \bar{\rho}(x_1)(\bar{\rho}'(x_1))^{-1}\nabla \cdot \dot{\Lambda}_h^{\sigma} v+\dot{\Lambda}_h^{\sigma} v^1)\,dx\\
 &\qquad-g\int_{\Omega}(\bar{\rho}'(x_1))^{-1}\, \dot{\Lambda}_h^{\sigma} \nabla\cdot(v\,\bar{\rho}(x_1))\,  \dot{\Lambda}_h^{\sigma} (B_{1, i}^{h, j}\partial_jv^i)\,dx,\\
&\mathfrak{J}_4:=  \bar{\rho}(0) \int_{\Omega}\dot{\Lambda}_h^{\sigma} \, \mathcal{H}( \partial_t{Q})\dot{\Lambda}_h^{\sigma} \partial_1v^1  \,dx+ \bar{\rho}(0)\int_{\Omega}\dot{\Lambda}_h^{\sigma} v^1\dot{\Lambda}_h^{\sigma}\partial_1 \mathcal{H}(\partial_t {Q}) \,dx,
        \end{split}
\end{equation*}
\begin{equation*}
\begin{split}
 &\mathfrak{J}_5:=-(\delta+\frac{4}{3}\varepsilon)\int_{\Omega}\dot{\Lambda}_h^{\sigma}\partial_t v^1 \dot{\Lambda}_h^{\sigma}\partial_1 \widetilde{\mathcal{G}}^1  \,dx-\varepsilon\int_{\Omega}\dot{\Lambda}_h^{\sigma}\partial_t v^{\beta}\, \dot{\Lambda}_h^{\sigma}\partial_1\widetilde{\mathcal{G}}^{\beta} \,dx,\\
 &\mathfrak{J}_6:=(\delta+\frac{4}{3}\varepsilon)\int_{\Omega}\dot{\Lambda}_h^{\sigma}
 \partial_t \widetilde{\mathcal{G}}^1 \dot{\Lambda}_h^{\sigma} \partial_1v^1  \,dx+\varepsilon\int_{\Omega}\dot{\Lambda}_h^{\sigma}\partial_t  \widetilde{\mathcal{G}}^{\beta} \, \dot{\Lambda}_h^{\sigma} \partial_1v^{\beta} \,dx.
        \end{split}
\end{equation*}
\end{lem}
\begin{proof}
We multiply the $i$-th component of the momentum equations of \eqref{eqns-sigma-lin-1} by $\dot{\Lambda}_h^{\sigma}\partial_tv^i$, sum over $i$, and integrate
over $\Omega$ to find that
\begin{equation}\label{pseudo-energy-tv-2}
\begin{split}
 & \| \sqrt{\bar{\rho}(x_1)}\dot{\Lambda}_h^{\sigma} \partial_t v\|_{L^2}^2 +I+ II=\int_{\Omega} \dot{\Lambda}_h^{\sigma} \partial_t v\,\cdot\, \dot{\Lambda}_h^{\sigma}\mathfrak{g}\,dx,
     \end{split}
\end{equation}
where
$I:=\int_{\Omega} \dot{\Lambda}_h^{\sigma} \partial_t v\,\cdot\,\bar{\rho}(x_1)\nabla\,\dot{\Lambda}_h^{\sigma}q\,dx$, $II:=-\int_{\Omega} \dot{\Lambda}_h^{\sigma} \partial_t v\,\cdot\,[\nabla\cdot\mathbb{S}(\dot{\Lambda}_h^{\sigma}v)]\,dx$.

For $I_1$, we have
\begin{equation*}
\begin{split}
 &I=\frac{d}{dt}\int_{\Omega} \dot{\Lambda}_h^{\sigma} v\,\cdot\,\bar{\rho}(x_1)\nabla\,\dot{\Lambda}_h^{\sigma}q\,dx-\int_{\Omega} \dot{\Lambda}_h^{\sigma} v\,\cdot\,\bar{\rho}(x_1)\nabla\,\dot{\Lambda}_h^{\sigma} \partial_t q\,dx.
\end{split}
\end{equation*}
Integrating by parts gives rise to
\begin{equation*}
\begin{split}
 &\int_{\Omega} \dot{\Lambda}_h^{\sigma} v\,\cdot\,\bar{\rho}(x_1)\nabla\,\dot{\Lambda}_h^{\sigma}q\,dx=\int_{\Sigma_0} \dot{\Lambda}_h^{\sigma} v^1\, \bar{\rho}(x_1) \dot{\Lambda}_h^{\sigma}q\,dS_0-\int_{\Omega} \dot{\Lambda}_h^{\sigma}\nabla\cdot(\bar{\rho}(x_1)v ) \dot{\Lambda}_h^{\sigma}q\,dx,
\end{split}
\end{equation*}
and
\begin{equation*}
\begin{split}
 &-\int_{\Omega} \dot{\Lambda}_h^{\sigma} v\,\cdot\,\bar{\rho}(x_1)\nabla\,\dot{\Lambda}_h^{\sigma} \partial_t q\,dx\\
  &=-\int_{\Sigma_0} \dot{\Lambda}_h^{\sigma} v^1\bar{\rho}(0) \dot{\Lambda}_h^{\sigma} \partial_t q\,dS_0+ \int_{\Omega} \dot{\Lambda}_h^{\sigma} \nabla\cdot(v\,\bar{\rho}(x_1))\,\dot{\Lambda}_h^{\sigma} \partial_t q\,dx.
\end{split}
\end{equation*}
Due to the second equation in \eqref{eqns-sigma-lin-1}, we have
\begin{equation*}
\begin{split}
\int_{\Omega} \dot{\Lambda}_h^{\sigma} \nabla\cdot(v\,\bar{\rho}(x_1))\,\dot{\Lambda}_h^{\sigma} \partial_t q\,dx=&g\int_{\Omega} \dot{\Lambda}_h^{\sigma} \nabla\cdot(v\,\bar{\rho}(x_1))\, ( \bar{\rho}(x_1)(\bar{\rho}'(x_1))^{-1}\nabla \cdot \dot{\Lambda}_h^{\sigma} v+\dot{\Lambda}_h^{\sigma} v^1)\,dx\\
 &-g\int_{\Omega}(\bar{\rho}'(x_1))^{-1}\, \dot{\Lambda}_h^{\sigma} \nabla\cdot(v\,\bar{\rho}(x_1))\,  \dot{\Lambda}_h^{\sigma} (B_{1, i}^{h, j}\partial_jv^i)\,dx.
 \end{split}
\end{equation*}
While thanks to $q= {Q}+g\xi^1$ and $\partial_tq= \partial_t{Q}+g v^1$, one can see
\begin{equation*}
\begin{split}
 &-\int_{\Sigma_0} \dot{\Lambda}_h^{\sigma} v^1\bar{\rho}(0) \dot{\Lambda}_h^{\sigma} \partial_t q\,dS_0=-\int_{\Sigma_0} \dot{\Lambda}_h^{\sigma} v^1\bar{\rho}(0) \dot{\Lambda}_h^{\sigma} \partial_t {Q}\,dS_0-g\,\bar{\rho}(0)  \|\dot{\Lambda}_h^{\sigma} v^1 \|_{L^2(\Sigma_0)}^2,
\end{split}
\end{equation*}
which follows
\begin{equation*}
\begin{split}
 -\int_{\Omega} \dot{\Lambda}_h^{\sigma} v\,\cdot\,&\bar{\rho}(x_1)\nabla\,\dot{\Lambda}_h^{\sigma} \partial_t q\,dx
  =-\int_{\Sigma_0} \dot{\Lambda}_h^{\sigma} v^1\bar{\rho}(0) \dot{\Lambda}_h^{\sigma} \partial_t {Q}\,dS_0-g\,\bar{\rho}(0)  \|\dot{\Lambda}_h^{\sigma} v^1 \|_{L^2(\Sigma_0)}^2\\
  &\qquad+g\int_{\Omega} \dot{\Lambda}_h^{\sigma} \nabla\cdot(v\,\bar{\rho}(x_1))\, ( \bar{\rho}(x_1)(\bar{\rho}'(x_1))^{-1}\nabla \cdot \dot{\Lambda}_h^{\sigma} v+\dot{\Lambda}_h^{\sigma} v^1)\,dx\\
 &\qquad-g\int_{\Omega}(\bar{\rho}'(x_1))^{-1}\, \dot{\Lambda}_h^{\sigma} \nabla\cdot(v\,\bar{\rho}(x_1))\,  \dot{\Lambda}_h^{\sigma} (B_{1, i}^{h, j}\partial_jv^i)\,dx.
\end{split}
\end{equation*}
Therefore, we get
\begin{equation*}
\begin{split}
 I=\frac{d}{dt}\mathcal{E}_{I}(\dot{\Lambda}_h^{\sigma}\nabla\,v)+\mathcal{R}_{I}(\dot{\Lambda}_h^{\sigma}\nabla\,v)
\end{split}
\end{equation*}
with
\begin{equation*}
\begin{split}
 \mathcal{E}_{I}(\dot{\Lambda}_h^{\sigma}\nabla\,v):=&\bar{\rho}(0) \int_{\Sigma_0} \dot{\Lambda}_h^{\sigma} v^1\, \dot{\Lambda}_h^{\sigma}q\,dS_0-\int_{\Omega} \dot{\Lambda}_h^{\sigma}\nabla\cdot(\bar{\rho}(x_1)v ) \dot{\Lambda}_h^{\sigma}q\,dx,\\
\mathcal{R}_{I}(\dot{\Lambda}_h^{\sigma}\nabla\,v):=&-\bar{\rho}(0) \int_{\Sigma_0} \dot{\Lambda}_h^{\sigma} v^1 \dot{\Lambda}_h^{\sigma} \partial_t {Q}\,dS_0-g\,\bar{\rho}(0)  \|\dot{\Lambda}_h^{\sigma} v^1 \|_{L^2(\Sigma_0)}^2\\
  &+g\int_{\Omega} \dot{\Lambda}_h^{\sigma} \nabla\cdot(v\,\bar{\rho}(x_1))\, ( \bar{\rho}(x_1)(\bar{\rho}'(x_1))^{-1}\nabla \cdot \dot{\Lambda}_h^{\sigma} v+\dot{\Lambda}_h^{\sigma} v^1)\,dx\\
 &\qquad-g\int_{\Omega}(\bar{\rho}'(x_1))^{-1}\, \dot{\Lambda}_h^{\sigma} \nabla\cdot(v\,\bar{\rho}(x_1))\,  \dot{\Lambda}_h^{\sigma} (B_{1, i}^{h, j}\partial_jv^i)\,dx.
\end{split}
\end{equation*}
For $II$, we split the integral $II$ into three parts:
\begin{equation}\label{pseudo-energy-tv-10}
\begin{split}
   &II=II_1+II_2+II_3
    \end{split}
\end{equation}
with
\begin{equation*}
\begin{split}
   &II_i:=-\int_{\Omega} \dot{\Lambda}_h^{\sigma} \partial_t v^i  [\nabla\cdot\mathbb{S}(\dot{\Lambda}_h^{\sigma}v)]^i\,dx\quad\text{for} \quad i=1, 2, 3.
    \end{split}
\end{equation*}
For $II_1$, we have
\begin{equation*}
\begin{split}
   &II_1=-\int_{\Omega} \dot{\Lambda}_h^{\sigma} \partial_t v^1 [\nabla\cdot\mathbb{S}(\dot{\Lambda}_h^{\sigma}v)]^1\,dx=(\delta+\frac{4}{3}\varepsilon)II_{1, 1}+II_{1, 2},
    \end{split}
\end{equation*}
where
\begin{equation*}
\begin{split}
   &II_{1, 1}:=-\int_{\Omega} \dot{\Lambda}_h^{\sigma} \partial_t v^1 \partial_1^2 \dot{\Lambda}_h^{\sigma}v^1\,dx,\quad II_{1, 2}:= -\int_{\Omega}\dot{\Lambda}_h^{\sigma} \partial_t v^1  \dot{\Lambda}_h^{\sigma}\mathcal{L}_{1}(\partial_h\nabla\,v)\,dx.
    \end{split}
\end{equation*}
Thanks to the definition of $\mathcal{G}^1$ in \eqref{d-1-v-unknown-1}, one can see
\begin{equation}\label{pseudo-energy-tv-12}
\begin{split}
   II_{1, 1}&=-\int_{\Omega}\dot{\Lambda}_h^{\sigma}\partial_t v^1 \dot{\Lambda}_h^{\sigma}\partial_1(\mathcal{G}^1-\frac{\delta-\frac{2}{3}\varepsilon}{\delta+\frac{4}{3}\varepsilon} \nabla_h\cdot v^h+\frac{\bar{\rho}(0)}{\delta+\frac{4}{3}\varepsilon}\, \mathcal{H}({Q})+\widetilde{\mathcal{G}}^1) \,dx=:\sum_{j=1}^4II_{1, 1}^j
    \end{split}
\end{equation}
with
\begin{equation*}
\begin{split}
&II_{1, 1}^1:=-\int_{\Omega}\dot{\Lambda}_h^{\sigma}\partial_t v^1\dot{\Lambda}_h^{\sigma}\partial_1 \mathcal{G}^1 \,dx,\,II_{1, 1}^2:=-\frac{\bar{\rho}(0)}{\delta+\frac{4}{3}\varepsilon}\int_{\Omega}\dot{\Lambda}_h^{\sigma}\partial_t v^1\dot{\Lambda}_h^{\sigma}\partial_1  \mathcal{H}({Q}) \,dx,\\
   &II_{1, 1}^3:=\frac{\delta-\frac{2}{3}\varepsilon}{\delta+\frac{4}{3}\varepsilon}\int_{\Omega}\dot{\Lambda}_h^{\sigma}\partial_t v^1 \dot{\Lambda}_h^{\sigma}\partial_1 \nabla_h\cdot v^h  \,dx,\quad II_{1, 1}^4:=-\int_{\Omega}\dot{\Lambda}_h^{\sigma}\partial_t v^1 \dot{\Lambda}_h^{\sigma}\partial_1 \widetilde{\mathcal{G}}^1  \,dx.
    \end{split}
\end{equation*}
Thanks to the boundary conditions $v|_{\Sigma_b}=0$ and $\mathcal{G}^1|_{\Sigma_0}=0$, we use integration by parts to find
\begin{equation*}
\begin{split}
II_{1, 1}^1&=-\int_{\Omega}\dot{\Lambda}_h^{\sigma}\partial_t v^1\dot{\Lambda}_h^{\sigma}\partial_1 \mathcal{G}^1 \,dx=\int_{\Omega}\dot{\Lambda}_h^{\sigma}\partial_t\partial_1  v^1 \dot{\Lambda}_h^{\sigma} \mathcal{G}^1 \,dx\\
    &=\int_{\Omega}\dot{\Lambda}_h^{\sigma}\partial_t(\mathcal{G}^1-\frac{\delta-\frac{2}{3}\varepsilon}{\delta+\frac{4}{3}\varepsilon} \nabla_h\cdot v^h+\frac{\bar{\rho}(0)}{\delta+\frac{4}{3}\varepsilon}\, \mathcal{H}({Q})+\widetilde{\mathcal{G}}^1)
    \dot{\Lambda}_h^{\sigma} \mathcal{G}^1 \,dx\\
    &=\frac{1}{2}\frac{d}{dt}\|\dot{\Lambda}_h^{\sigma}\mathcal{G}^1\|_{L^2}^2
    +\int_{\Omega}\dot{\Lambda}_h^{\sigma}\partial_t(-\frac{\delta-\frac{2}{3}\varepsilon}{\delta+\frac{4}{3}\varepsilon} \nabla_h\cdot v^h+\frac{\bar{\rho}(0)}{\delta+\frac{4}{3}\varepsilon}\, \mathcal{H}({Q})+\widetilde{\mathcal{G}}^1)\dot{\Lambda}_h^{\sigma} \mathcal{G}^1 \,dx.
        \end{split}
\end{equation*}
Thanks to the definition of $\mathcal{G}^1$ in \eqref{d-1-v-unknown-1} again, we get
\begin{equation*}
\begin{split}
&\int_{\Omega}\dot{\Lambda}_h^{\sigma}\partial_t(-\frac{\delta-\frac{2}{3}\varepsilon}{\delta+\frac{4}{3}\varepsilon} \nabla_h\cdot v^h+\frac{\bar{\rho}(0)}{\delta+\frac{4}{3}\varepsilon}\, \mathcal{H}({Q})+\widetilde{\mathcal{G}}^1)\dot{\Lambda}_h^{\sigma} \mathcal{G}^1 \,dx\\
 &= -\frac{1}{2}\frac{d}{dt}\|\dot{\Lambda}_h^{\sigma}
(-\frac{\delta-\frac{2}{3}\varepsilon}{\delta+\frac{4}{3}\varepsilon} \nabla_h\cdot v^h+\frac{\bar{\rho}(0)}{\delta+\frac{4}{3}\varepsilon}\, \mathcal{H}({Q})+\widetilde{\mathcal{G}}^1)\|_{L^2(\Omega)}^2 \\
&\qquad+ \int_{\Omega}\dot{\Lambda}_h^{\sigma}
 \partial_t(-\frac{\delta-\frac{2}{3}\varepsilon}{\delta+\frac{4}{3}\varepsilon} \nabla_h\cdot v^h+\frac{\bar{\rho}(0)}{\delta+\frac{4}{3}\varepsilon}\, \mathcal{H}({Q})+\widetilde{\mathcal{G}}^1)\dot{\Lambda}_h^{\sigma} \partial_1v^1  \,dx,
        \end{split}
\end{equation*}
which follows immediately that
\begin{equation}\label{pseudo-energy-tv-13}
\begin{split}
II_{1, 1}^1&=\frac{1}{2}\frac{d}{dt}(\|\dot{\Lambda}_h^{\sigma}\mathcal{G}^1\|_{L^2}^2-\|\dot{\Lambda}_h^{\sigma}
(-\frac{\delta-\frac{2}{3}\varepsilon}{\delta+\frac{4}{3}\varepsilon} \nabla_h\cdot v^h+\frac{\bar{\rho}(0)}{\delta+\frac{4}{3}\varepsilon}\, \mathcal{H}({Q})+\widetilde{\mathcal{G}}^1)\|_{L^2(\Omega)}^2)\\
&\qquad+ \int_{\Omega}\dot{\Lambda}_h^{\sigma}
 \partial_t(-\frac{\delta-\frac{2}{3}\varepsilon}{\delta+\frac{4}{3}\varepsilon} \nabla_h\cdot v^h+\frac{\bar{\rho}(0)}{\delta+\frac{4}{3}\varepsilon}\, \mathcal{H}({Q})+\widetilde{\mathcal{G}}^1)\dot{\Lambda}_h^{\sigma} \partial_1v^1  \,dx.
        \end{split}
\end{equation}
On the other hand, for $II_{1, 1}^2$, we have
\begin{equation}\label{pseudo-energy-tv-14}
\begin{split}
&II_{1, 1}^2=-\frac{\bar{\rho}(0)}{\delta+\frac{4}{3}\varepsilon}\frac{d}{dt}\int_{\Omega}\dot{\Lambda}_h^{\sigma} v^1\dot{\Lambda}_h^{\sigma}\partial_1  \mathcal{H}({Q}) \,dx+\frac{\bar{\rho}(0)}{\delta+\frac{4}{3}\varepsilon}\int_{\Omega}\dot{\Lambda}_h^{\sigma} v^1\dot{\Lambda}_h^{\sigma}\partial_1 \partial_t \mathcal{H}({Q}) \,dx.
    \end{split}
\end{equation}
Plugging \eqref{pseudo-energy-tv-13} and \eqref{pseudo-energy-tv-14}  into \eqref{pseudo-energy-tv-12} yields
\begin{equation*}\label{pseudo-energy-tv-15aa}
\begin{split}
II_{1, 1}
=\frac{1}{2}\frac{d}{dt}\bigg(&\|\dot{\Lambda}_h^{\sigma}\mathcal{G}^1\|_{L^2}^2-\|\dot{\Lambda}_h^{\sigma}
(-\frac{\delta-\frac{2}{3}\varepsilon}{\delta+\frac{4}{3}\varepsilon} \nabla_h\cdot v^h+\frac{\bar{\rho}(0)}{\delta+\frac{4}{3}\varepsilon}\, \mathcal{H}({Q})+\widetilde{\mathcal{G}}^1)\|_{L^2(\Omega)}^2\\
& -\frac{2\bar{\rho}(0)}{\delta+\frac{4}{3}\varepsilon}\frac{d}{dt}\int_{\Omega}\dot{\Lambda}_h^{\sigma} v^1\dot{\Lambda}_h^{\sigma}\partial_1  \mathcal{H}({Q}) \,dx\bigg)+II_{1, 1}^R,
        \end{split}
\end{equation*}
where
\begin{equation*}
\begin{split}
II_{1, 1}^R:=&\int_{\Omega}\dot{\Lambda}_h^{\sigma}
 \partial_t(-\frac{\delta-\frac{2}{3}\varepsilon}{\delta+\frac{4}{3}\varepsilon} \nabla_h\cdot v^h+\frac{\bar{\rho}(0)}{\delta+\frac{4}{3}\varepsilon}\, \mathcal{H}({Q})+\widetilde{\mathcal{G}}^1)\dot{\Lambda}_h^{\sigma} \partial_1v^1  \,dx\\
 &+\frac{\bar{\rho}(0)}{\delta+\frac{4}{3}\varepsilon}\int_{\Omega}\dot{\Lambda}_h^{\sigma} v^1\dot{\Lambda}_h^{\sigma}\partial_1 \partial_t \mathcal{H}({Q}) \,dx\\
 &+\frac{\delta-\frac{2}{3}\varepsilon}{\delta+\frac{4}{3}\varepsilon}\int_{\Omega}\dot{\Lambda}_h^{\sigma}\partial_t v^1 \dot{\Lambda}_h^{\sigma}\partial_1 \nabla_h\cdot v^h  \,dx-\int_{\Omega}\dot{\Lambda}_h^{\sigma}\partial_t v^1 \dot{\Lambda}_h^{\sigma}\partial_1 \widetilde{\mathcal{G}}^1  \,dx.
        \end{split}
\end{equation*}
Therefore, we obtain
\begin{equation}\label{pseudo-energy-tv-15}
\begin{split}
   &II_1=  \frac{\delta+\frac{4}{3}\varepsilon}{2}\frac{d}{dt}\bigg(\|\dot{\Lambda}_h^{\sigma}\mathcal{G}^1\|_{L^2}^2-\|\dot{\Lambda}_h^{\sigma}
(-\frac{\delta-\frac{2}{3}\varepsilon}{\delta+\frac{4}{3}\varepsilon} \nabla_h\cdot v^h+\frac{\bar{\rho}(0)}{\delta+\frac{4}{3}\varepsilon}\, \mathcal{H}({Q})+\widetilde{\mathcal{G}}^1)\|_{L^2(\Omega)}^2\\
& \,-\frac{2\bar{\rho}(0)}{\delta+\frac{4}{3}\varepsilon}\int_{\Omega}\dot{\Lambda}_h^{\sigma} v^1\dot{\Lambda}_h^{\sigma}\partial_1  \mathcal{H}({Q}) \,dx\bigg)+(\delta+\frac{4}{3}\varepsilon)II_{1, 1}^R-\int_{\Omega}\dot{\Lambda}_h^{\sigma} \partial_t v^1  \dot{\Lambda}_h^{\sigma}\mathcal{L}_{1}(\partial_h\nabla\,v)\,dx.
    \end{split}
\end{equation}
For $II_{\beta}$ with $\beta=2,\,3$, we have
\begin{equation}\label{pseudo-energy-tv-16}
\begin{split}
   &II_{\beta}=-\int_{\Omega} \dot{\Lambda}_h^{\sigma} \partial_t v^{\beta} [\nabla\cdot\mathbb{S}(\dot{\Lambda}_h^{\sigma}v)]^{\beta}\,dx=\varepsilon\,II_{\beta, 1}-\int_{\Omega} \dot{\Lambda}_h^{\sigma} \partial_t v^{\beta} \dot{\Lambda}_h^{\sigma}\mathcal{L}_{\beta}(\partial_h\nabla\,v)\,dx
    \end{split}
\end{equation}
with $II_{\beta, 1}:=-\int_{\Omega} \dot{\Lambda}_h^{\sigma} \partial_t v^{\beta} \dot{\Lambda}_h^{\sigma}\partial_1^2v^{\beta} \,dx$.

Following our analysis of the term $II_{1, 1}$, we use
$\partial_1v^{\beta}=\mathcal{G}^{\beta}-\partial_{\beta}v^1
+\widetilde{\mathcal{G}}^{\beta}$ with $\widetilde{\mathcal{G}}^{\beta}:=\mathcal{G}^{\beta}_{{Q}, 1}+\mathcal{G}^{\beta}_{{Q}, 2}
+\mathcal{G}^{\beta}_{\partial_hv}$ in \eqref{partialv-q-1} to the integral $II_{\beta, 1}$ to obtain
\begin{equation}\label{pseudo-energy-tv-17}
\begin{split}
   II_{\beta, 1}&=-\int_{\Omega}\dot{\Lambda}_h^{\sigma}\partial_t v^{\beta}\, \dot{\Lambda}_h^{\sigma}\partial_1(\mathcal{G}^{\beta}-\partial_{\beta}v^1
+\widetilde{\mathcal{G}}^{\beta}) \,dx\\
&=:II_{\beta, 1}^1+\int_{\Omega}\dot{\Lambda}_h^{\sigma}\partial_t v^{\beta}\, \dot{\Lambda}_h^{\sigma}\partial_{\beta}\partial_1v^1\,dx-\int_{\Omega}\dot{\Lambda}_h^{\sigma}\partial_t v^{\beta}\, \dot{\Lambda}_h^{\sigma}\partial_1\widetilde{\mathcal{G}}^{\beta} \,dx
    \end{split}
\end{equation}
with
\begin{equation*}
\begin{split}
   &II_{\beta, 1}^1:=-\int_{\Omega}\dot{\Lambda}_h^{\sigma}\partial_t v^{\beta}\,  \dot{\Lambda}_h^{\sigma}\partial_1 \mathcal{G}^{\beta} \,dx.
    \end{split}
\end{equation*}
Thanks to the definition of $\mathcal{G}^{\beta}$ in \eqref{d-1-v-unknown-1}, one can see
\begin{equation}\label{pseudo-energy-tv-18}
\begin{split}
   &II_{\beta, 1}^1=-\int_{\Omega}\dot{\Lambda}_h^{\sigma}\partial_t v^{\beta} \dot{\Lambda}_h^{\sigma}\partial_1 \mathcal{G}^{\beta} \,dx=\int_{\Omega}\dot{\Lambda}_h^{\sigma}\partial_t \partial_1v^{\beta} \dot{\Lambda}_h^{\sigma} \mathcal{G}^{\beta} \,dx\\
 &=\frac{1}{2}\frac{d}{dt}\|\dot{\Lambda}_h^{\sigma} \mathcal{G}^{\beta}\|_{L^2}^2+\int_{\Omega}\dot{\Lambda}_h^{\sigma}\partial_t (-\partial_{\beta}v^1
+\widetilde{\mathcal{G}}^{\beta})\, \dot{\Lambda}_h^{\sigma} \mathcal{G}^{\beta} \,dx\\
&=\frac{1}{2}\frac{d}{dt}(\|\dot{\Lambda}_h^{\sigma} \mathcal{G}^{\beta}\|_{L^2}^2-\|\dot{\Lambda}_h^{\sigma} (-\partial_{\beta}v^1
+\widetilde{\mathcal{G}}^{\beta})\|_{L^2}^2)+\int_{\Omega}\dot{\Lambda}_h^{\sigma}\partial_t (-\partial_{\beta}v^1
+\widetilde{\mathcal{G}}^{\beta})\, \dot{\Lambda}_h^{\sigma} \partial_1v^{\beta} \,dx.
    \end{split}
\end{equation}
Inserting \eqref{pseudo-energy-tv-18} into \eqref{pseudo-energy-tv-17} gives rise to
\begin{equation*}
\begin{split}
II_{\beta, 1}=\frac{1}{2}\frac{d}{dt}(\|\dot{\Lambda}_h^{\sigma} \mathcal{G}^{\beta}\|_{L^2}^2-\|\dot{\Lambda}_h^{\sigma} (-\partial_{\beta}v^1
+\widetilde{\mathcal{G}}^{\beta})\|_{L^2}^2)+II_{\beta, 1}^{R}
    \end{split}
\end{equation*}
with
\begin{equation*}
\begin{split}
II_{\beta, 1}^{R}:=&\int_{\Omega}\dot{\Lambda}_h^{\sigma}\partial_t (-\partial_{\beta}v^1
+\widetilde{\mathcal{G}}^{\beta})\, \dot{\Lambda}_h^{\sigma} \partial_1v^{\beta} \,dx\\
&+\int_{\Omega}\dot{\Lambda}_h^{\sigma}\partial_t v^{\beta}\, \dot{\Lambda}_h^{\sigma}\partial_{\beta}\partial_1v^1\,dx-\int_{\Omega}\dot{\Lambda}_h^{\sigma}\partial_t v^{\beta}\, \dot{\Lambda}_h^{\sigma}\partial_1\widetilde{\mathcal{G}}^{\beta} \,dx,
    \end{split}
\end{equation*}
which along with \eqref{pseudo-energy-tv-16} yields
\begin{equation}\label{pseudo-energy-tv-19}
\begin{split}
   II_{\beta}=&\frac{\varepsilon}{2}\frac{d}{dt}\bigg(\|\dot{\Lambda}_h^{\sigma} \mathcal{G}^{\beta}\|_{L^2}^2-\|\dot{\Lambda}_h^{\sigma} (-\partial_{\beta}v^1
+\widetilde{\mathcal{G}}^{\beta})\|_{L^2}^2\bigg)\\
&\qquad\qquad\qquad\qquad+\varepsilon\,II_{\beta, 1}^{R}-\int_{\Omega} \dot{\Lambda}_h^{\sigma} \partial_t v^{\beta} \dot{\Lambda}_h^{\sigma}\mathcal{L}_{\beta}(\partial_h\nabla\,v)\,dx \quad \text{with}\,\,\,\beta=2, 3.
    \end{split}
\end{equation}
Plugging \eqref{pseudo-energy-tv-15} and \eqref{pseudo-energy-tv-19} into \eqref{pseudo-energy-tv-10} gives rise to
\begin{equation}\label{pseudo-energy-tv-20}
\begin{split}
   &II=\frac{d}{dt}\mathcal{E}_{II}(\dot{\Lambda}_h^{\sigma}\nabla\,v)+\mathcal{R}_{II}(\dot{\Lambda}_h^{\sigma}\nabla\,v),
        \end{split}
\end{equation}
where
\begin{equation*}
\begin{split}
   &\mathcal{E}_{II}(\dot{\Lambda}_h^{\sigma}\nabla\,v)\eqdefa \frac{\varepsilon}{2}\sum_{\beta=2}^3(\|\dot{\Lambda}_h^{\sigma} \mathcal{G}^{\beta}\|_{L^2}^2-\|\dot{\Lambda}_h^{\sigma} (-\partial_{\beta}v^1
+\widetilde{\mathcal{G}}^{\beta})\|_{L^2}^2)\\
   &\qquad+\frac{(\delta+\frac{4}{3}\varepsilon)}{2}\bigg(\|\dot{\Lambda}_h^{\sigma}\mathcal{G}^1\|_{L^2}^2-\|\dot{\Lambda}_h^{\sigma}
(-\frac{\delta-\frac{2}{3}\varepsilon}{\delta+\frac{4}{3}\varepsilon} \nabla_h\cdot v^h+\frac{\bar{\rho}(0)}{\delta+\frac{4}{3}\varepsilon}\, \mathcal{H}({Q})+\widetilde{\mathcal{G}}^1)\|_{L^2(\Omega)}^2\\
& \qquad\qquad\qquad\qquad\qquad-\frac{2\bar{\rho}(0)}{\delta+\frac{4}{3}\varepsilon}\int_{\Omega}\dot{\Lambda}_h^{\sigma} v^1\dot{\Lambda}_h^{\sigma}\partial_1  \mathcal{H}({Q}) \,dx\bigg),
        \end{split}
\end{equation*}
and
\begin{equation*}
\begin{split}
     &  \mathcal{R}_{II}(\dot{\Lambda}_h^{\sigma}\nabla\,v)\eqdefa(\delta+\frac{4}{3}\varepsilon)II_{1, 1}^R +\varepsilon\sum_{\beta=2}^3II_{\beta, 1}^{R}-\sum_{i=1}^3\int_{\Omega}\dot{\Lambda}_h^{\sigma} \partial_t v^i  \dot{\Lambda}_h^{\sigma}\mathcal{L}_{i}(\partial_h\nabla\,v)\,dx.
        \end{split}
\end{equation*}
Inserting \eqref{pseudo-energy-tv-20} into \eqref{pseudo-energy-tv-2} yields \eqref{pseudo-energy-tv-1}, which is the desired result.
\end{proof}

\subsection{Some estimates}

\subsubsection{Estimate of $\|\dot{\Lambda}_h^{\sigma_0}\nabla\,v\|_{L^2}$}

\begin{lem}\label{lem-grad-sigma0-s-1}
Under the assumption of Lemma \ref{lem-pseudo-energy-tv-1}, if $\sigma \geq s_0-1$, $E_{s_0}(t) \leq 1$ for all $t\in [0, T]$, then there holds that $\forall\,t\in [0, T]$
\begin{equation}\label{grad-sigma0-s-tv-1}
\begin{split}
 &\frac{d}{dt}\mathring{\mathcal{E}}(\dot{\Lambda}_h^{\sigma_0}{\Lambda}_h^{\sigma-\sigma_0}\nabla\, v)+2c_3\|\dot{\Lambda}_h^{\sigma_0}{\Lambda}_h^{\sigma-\sigma_0}\partial_t v\|_{L^2}^2 \\
 &\leq C_3(\|\dot{\Lambda}_h^{\sigma_0}{\Lambda}_h^{\sigma+1-\sigma_0}\nabla\,v\|_{L^2}^2
 +\|\dot{\Lambda}_h^{\sigma_0}\Lambda_h^{\sigma-\frac{1}{2}-\sigma_0} \partial_t{Q} \|_{L^2(\Sigma_0)}^2)\\
 &\qquad\qquad\qquad\qquad\qquad\qquad+ C_3 \dot{\mathcal{D}}_{\sigma+1}^{\frac{1}{2}}(E_{s_0}^{\frac{1}{2}}\dot{\mathcal{D}}_{\sigma+1}^{\frac{1}{2}}
   +E_{\sigma+1}^{\frac{1}{2}}\,\dot{\mathcal{D}}_{s_0}^{\frac{1}{2}}).
        \end{split}
\end{equation}
\end{lem}
\begin{proof}
We first take $\sigma=\sigma_0$ in \eqref{pseudo-energy-tv-1}, and estimate all the integrals in the right hand side of \eqref{pseudo-energy-tv-1} one by one.

For $\mathfrak{J}_1$ and $\mathfrak{J}_2$, we directly bound them by
\begin{equation*}
\begin{split}
&|\mathfrak{J}_1|\lesssim \|\dot{\Lambda}_h^{\sigma_0} \partial_t v\|_{L^2} \|\dot{\Lambda}_h^{\sigma_0}\nabla\partial_{h}v\|_{L^2}
 \end{split}
\end{equation*}
and
\begin{equation*}
\begin{split}
|\mathfrak{J}_2|&\lesssim  \|\dot{\Lambda}_h^{\sigma_0} v^1\|_{L^2(\Sigma_0)}^2+\|\dot{\Lambda}_h^{\sigma_0} v^1\|_{L^2(\Sigma_0)}\| \dot{\Lambda}_h^{\sigma_0}\partial_t {Q}\|_{L^2(\Sigma_0)}\\
&\lesssim \|\dot{\Lambda}_h^{\sigma_0}\nabla v\|_{L^2(\Omega)}^2+\|\dot{\Lambda}_h^{\sigma_0}\nabla v\|_{L^2(\Omega)}\| \dot{\Lambda}_h^{\sigma_0}\partial_t {Q}\|_{L^2(\Sigma_0)}.
 \end{split}
\end{equation*}
While for $\mathfrak{J}_3$, due to Poincare's inequality, we get
\begin{equation*}
\begin{split}|\mathfrak{J}_3|\lesssim \|\dot{\Lambda}_h^{\sigma_0} \nabla v\|_{L^2}^2+\|\dot{\Lambda}_h^{\sigma_0}B_{1, i}^{h, j}\partial_jv^i\|_{L^2}^2,
     \end{split}
\end{equation*}
 then applying Lemmas \ref{lem-productlaw-1aa} and \ref{lem-est-aij-1} ensures that
\begin{equation*}
\begin{split}
|\mathfrak{J}_3|&\lesssim \|\dot{\Lambda}_h^{\sigma_0}\nabla\,v\|_{L^2}^2+ \|\dot{\Lambda}_h^{\sigma_0}(B_{1, i}^{h, j}) \|_{H^1}^2\|\dot{\Lambda}_h^{\sigma_0} \Lambda_h \nabla\,v\|_{L^2}^2\lesssim \|\dot{\Lambda}_h^{\sigma_0}\nabla\,v\|_{L^2}^2+ E_{s_0}\dot{\mathcal{D}}_{s_0}.
\end{split}
\end{equation*}
For $\mathfrak{J}_4$, one has
\begin{equation*}
\begin{split}
|\mathfrak{J}_4|&\lesssim \|\dot{\Lambda}_h^{\sigma_0}\mathcal{H}( \partial_t{Q})\|_{L^2}\|\dot{\Lambda}_h^{\sigma_0} \nabla\,v\|_{L^2}+ \|\dot{\Lambda}_h^{\sigma_0} v\|_{L^2}\|\dot{\Lambda}_h^{\sigma_0}\partial_1 \mathcal{H}(\partial_t {Q})\|_{L^2} \\
 &\lesssim \|\dot{\Lambda}_h^{\sigma_0} \nabla\,v\|_{L^2}\|\dot{\Lambda}_h^{\sigma_0}\Lambda_h^{\frac{1}{2}} \partial_t {Q} \|_{L^2(\Sigma_0)}.
        \end{split}
\end{equation*}
Thanks to Lemma \ref{lem-est-G-norm-1}, we deal with the integrals $\mathfrak{J}_5$ and $\mathfrak{J}_6$ to obtain
\begin{equation*}
  \begin{split}
|\mathfrak{J}_5|\lesssim \|\dot{\Lambda}_h^{\sigma_0}\partial_t v\|_{L^2}\|\dot{\Lambda}_h^{\sigma_0}\partial_1 \widetilde{\mathcal{G}}^j\|_{L^2}\lesssim E_{s_0}^{\frac{1}{2}}\dot{\mathcal{D}}_{s_0},\quad  |\mathfrak{J}_6|\lesssim \|\dot{\Lambda}_h^{\sigma_0}
 \partial_t \widetilde{\mathcal{G}}^j\|_{L^2}\|\dot{\Lambda}_h^{\sigma_0} \partial_1v\|_{L^2} \lesssim E_{s_0}^{\frac{1}{2}}\dot{\mathcal{D}}_{s_0}.
\end{split}
\end{equation*}
Therefore, we conclude that
\begin{equation*}
  \begin{split}
  \sum_{j=1}^6|\mathfrak{J}_j|\lesssim  &\|\dot{\Lambda}_h^{\sigma_0}\nabla v\|_{L^2(\Omega)}^2+\|\dot{\Lambda}_h^{\sigma_0} \partial_t v\|_{L^2} \|\dot{\Lambda}_h^{\sigma_0}\nabla\partial_{h}v\|_{L^2} + \| \dot{\Lambda}_h^{\sigma_0}\Lambda_h^{\frac{1}{2}}\partial_t {Q}\|_{L^2(\Sigma_0)}^2+E_{s_0}^{\frac{1}{2}}\dot{\mathcal{D}}_{s_0},
  \end{split}
\end{equation*}
which leads to
\begin{equation}\label{grad-sigma0-tv-1}
\begin{split}
 &\frac{d}{dt}\mathring{\mathcal{E}}(\dot{\Lambda}_h^{\sigma_0}\nabla\, v)+2c_3\|\dot{\Lambda}_h^{\sigma_0}\partial_t v\|_{L^2}^2\lesssim  \|\dot{\Lambda}_h^{\sigma_0}\Lambda_h\nabla v\|_{L^2(\Omega)}^2+\| \dot{\Lambda}_h^{\sigma_0}\Lambda_h^{\frac{1}{2}}\partial_t {Q}\|_{L^2(\Sigma_0)}^2+E_{s_0}^{\frac{1}{2}}\dot{\mathcal{D}}_{s_0}.
        \end{split}
\end{equation}
Let's now consider the case $\sigma \geq s_0-1$ in \eqref{pseudo-energy-tv-1}, and estimate all the integrals in the right hand side of \eqref{pseudo-energy-tv-1}.

For $\mathfrak{J}_1$ and $\mathfrak{J}_2$, we directly bound them by
\begin{equation*}
\begin{split}
&|\mathfrak{J}_1|\lesssim \|\dot{\Lambda}_h^{\sigma}\partial_t v\|_{L^2} \|\dot{\Lambda}_h^{\sigma+1}\nabla\,v\|_{L^2}
 \end{split}
\end{equation*}
and
\begin{equation*}
\begin{split}
|\mathfrak{J}_2|&\lesssim  \| \dot{\Lambda}_h^{\sigma}v^1\|_{L^2(\Sigma_0)}^2+\|\dot{\Lambda}_h^{\sigma+\frac{1}{2}} v^1\|_{L^2(\Sigma_0)}\| \dot{\Lambda}_h^{\sigma-\frac{1}{2}} \partial_t {Q}\|_{L^2(\Sigma_0)} \lesssim \|\dot{\Lambda}_h^{\sigma}\nabla v\|_{L^2(\Omega)}^2+ \| \dot{\Lambda}_h^{\sigma-\frac{1}{2}} \partial_t {Q}\|_{L^2(\Sigma_0)}^2.
 \end{split}
\end{equation*}
For $\mathfrak{J}_3$, we get
\begin{equation*}
\begin{split}|\mathfrak{J}_3|\lesssim \|\dot{\Lambda}_h^{\sigma}\nabla v\|_{L^2}^2+\|\dot{\Lambda}_h^{\sigma}\nabla v\|_{L^2}\|\dot{\Lambda}_h^{\sigma}(B_{1, i}^{h, j}\partial_jv^i)\|_{L^2}^2,
     \end{split}
\end{equation*}
 then applying Lemmas \ref{lem-productlaw-1aa} and \ref{lem-est-aij-1} ensures that
\begin{equation*}
\begin{split}
|\mathfrak{J}_3|&\lesssim \|\dot{\Lambda}_h^{\sigma}\nabla\,v\|_{L^2}^2+ \|\dot{\Lambda}_h^{\sigma}\nabla v\|_{L^2}(\|\dot{\Lambda}_h^{\sigma}(B_{1, i}^{h, j}) \|_{L^2}\| \dot{\Lambda}_h^{\sigma_0} \Lambda_h^{s_0-1-\sigma_0}(\nabla\,v,\,v)\|_{H^1}\\
&\qquad +\|\dot{\Lambda}_h^{\sigma_0} \Lambda_h^{s_0-1-\sigma_0}(B_{1, i}^{h, j}) \|_{H^1}\|\dot{\Lambda}_h^{\sigma} (v,\,\nabla\,v)\|_{L^2})\\
&\lesssim \|\dot{\Lambda}_h^{\sigma}\nabla\,v\|_{L^2}^2+ (E_{\sigma}^{\frac{1}{2}}\dot{\mathcal{D}}_{s_0}^{\frac{1}{2}}
   +E_{s_0}^{\frac{1}{2}}\,\dot{\mathcal{D}}_{\sigma}^{\frac{1}{2}})\dot{\mathcal{D}}_{\sigma}^{\frac{1}{2}}.
\end{split}
\end{equation*}
For $\mathfrak{J}_4$, one has
\begin{equation*}
\begin{split}
|\mathfrak{J}_4|&\lesssim \|\dot{\Lambda}_h^{\sigma-1}\mathcal{H}( \partial_t{Q})\|_{L^2}\|\dot{\Lambda}_h^{\sigma+1} \nabla\,v\|_{L^2}+ \|\dot{\Lambda}_h^{\sigma+1} v\|_{L^2}\|\dot{\Lambda}_h^{\sigma-1}\partial_1 \mathcal{H}(\partial_t {Q})\|_{L^2} \\
&\lesssim  \|\dot{\Lambda}_h^{\sigma-1}\Lambda_h^{\frac{1}{2}} \partial_t{Q} \|_{L^2(\Sigma_0)}^2+ \|\dot{\Lambda}_h^{\sigma+1} \nabla\,v\|_{L^2}^2.
        \end{split}
\end{equation*}
Thanks to Lemma \ref{lem-est-G-norm-1}, we deal with the integrals $\mathfrak{J}_5$ and $\mathfrak{J}_6$ to obtain
\begin{equation*}
  \begin{split}
|\mathfrak{J}_5|\lesssim \|\dot{\Lambda}_h^{\sigma}\partial_t v\|_{L^2}\|\dot{\Lambda}_h^{\sigma}\partial_1 \widetilde{\mathcal{G}}^j\|_{L^2}\lesssim \dot{\mathcal{D}}_{\sigma+1}^{\frac{1}{2}}(E_{s_0}^{\frac{1}{2}}\dot{\mathcal{D}}_{\sigma+1}^{\frac{1}{2}}
   +E_{\sigma+1}^{\frac{1}{2}}\,\dot{\mathcal{D}}_{s_0}^{\frac{1}{2}}),
\end{split}
\end{equation*}
and
\begin{equation*}
  \begin{split}
|\mathfrak{J}_6|\lesssim \|\dot{\Lambda}_h^{\sigma}
 \partial_t \widetilde{\mathcal{G}}^j\|_{L^2}\|\dot{\Lambda}_h^{\sigma}\partial_1v\|_{L^2} \lesssim  \dot{\mathcal{D}}_{\sigma}^{\frac{1}{2}}(E_{s_0}^{\frac{1}{2}}\dot{\mathcal{D}}_{\sigma+1}^{\frac{1}{2}}
   +E_{\sigma+1}^{\frac{1}{2}}\,\dot{\mathcal{D}}_{s_0}^{\frac{1}{2}}).
\end{split}
\end{equation*}
Therefore, we conclude that
\begin{equation*}
  \begin{split}
  \sum_{j=1}^6|\mathfrak{J}_j|\lesssim  & \|\dot{\Lambda}_h^{\sigma}\partial_t v\|_{L^2} \|\dot{\Lambda}_h^{\sigma+1}\nabla\,v\|_{L^2}+\|\dot{\Lambda}_h^{\sigma}\Lambda_h\nabla\,v\|_{L^2}^2\\
  &+\|\dot{\Lambda}_h^{\sigma-1}\Lambda_h^{\frac{1}{2}} \partial_t{Q} \|_{L^2(\Sigma_0)}^2+ \dot{\mathcal{D}}_{\sigma+1}^{\frac{1}{2}}(E_{s_0}^{\frac{1}{2}}\dot{\mathcal{D}}_{\sigma+1}^{\frac{1}{2}}
   +E_{\sigma+1}^{\frac{1}{2}}\,\dot{\mathcal{D}}_{s_0}^{\frac{1}{2}}),
  \end{split}
\end{equation*}
which leads to
\begin{equation}\label{tan-grad-s-tv-1}
\begin{split}
 &\frac{d}{dt}\mathring{\mathcal{E}}(\dot{\Lambda}_h^{\sigma}\nabla\, v)+2c_3\|\dot{\Lambda}_h^{\sigma}\partial_t v\|_{L^2}^2 \\
 &\lesssim \|\dot{\Lambda}_h^{\sigma}\Lambda_h\nabla\,v\|_{L^2}^2+\|\dot{\Lambda}_h^{\sigma-1}\Lambda_h^{\frac{1}{2}} \partial_t{Q} \|_{L^2(\Sigma_0)}^2+ \dot{\mathcal{D}}_{\sigma+1}^{\frac{1}{2}}(E_{s_0}^{\frac{1}{2}}\dot{\mathcal{D}}_{\sigma+1}^{\frac{1}{2}}
   +E_{\sigma+1}^{\frac{1}{2}}\,\dot{\mathcal{D}}_{s_0}^{\frac{1}{2}}).
        \end{split}
\end{equation}
Combining \eqref{grad-sigma0-tv-1} with \eqref{tan-grad-s-tv-1}, we reach \eqref{grad-sigma0-s-tv-1}, which completes the proof of Lemma \ref{lem-grad-sigma0-s-1}.
\end{proof}

\subsubsection{Estimate of $\|\dot{\Lambda}_h^{-\lambda}\nabla\,v\|_{L^\infty_t(L^2)}$}

\begin{lem}\label{lem-grad-bdd-lambda-1}
Under the assumption of Lemma \ref{lem-tan-pseudo-energy-1}, if $(\lambda,\,\sigma_0) \in (0, 1)$ satisfies $1-\lambda< \sigma_0\leq 1-\frac{1}{2}\lambda$, and $E_{s_0}(t) \leq 1$ for all $t\in [0, T]$, then there holds that $\forall\,t\in [0, T]$
\begin{equation}\label{grad-lambda-tv-1}
\begin{split}
 & \frac{d}{dt}\mathring{\mathcal{E}}(\dot{\Lambda}_h^{-\lambda}\nabla\, v)+c_2\|\dot{\Lambda}_h^{-\lambda}\partial_t v\|_{L^2}^2  \\
 &\leq C_2\bigg(\|\dot{\Lambda}_h^{-\lambda}\Lambda_h\nabla v\|_{L^2(\Omega)}^2+ \| \dot{\Lambda}_h^{-\lambda}\Lambda_h^{\frac{1}{2}}\partial_t {Q}\|_{L^2(\Sigma_0)}^2+E_{s_0}^{\frac{1}{2}}\dot{\mathcal{D}}_{s_0}\bigg).
     \end{split}
\end{equation}
\end{lem}
\begin{proof}
Taking $ \sigma=-\lambda$ in  \eqref{pseudo-energy-tv-1}, we will estimate all the integrals in the right hand side of \eqref{pseudo-energy-tv-1}.

For $\mathfrak{J}_1$ and $\mathfrak{J}_2$, we directly bound them by
\begin{equation*}
\begin{split}
&|\mathfrak{J}_1|\lesssim \|\dot{\Lambda}_h^{-\lambda} \partial_t v\|_{L^2} \|\dot{\Lambda}_h^{-\lambda}\nabla\partial_{h}v\|_{L^2}
 \end{split}
\end{equation*}
and
\begin{equation*}
\begin{split}
|\mathfrak{J}_2|&\lesssim  \|\dot{\Lambda}_h^{-\lambda} v^1\|_{L^2(\Sigma_0)}^2+\|\dot{\Lambda}_h^{-\lambda} v^1\|_{L^2(\Sigma_0)}\| \dot{\Lambda}_h^{-\lambda}\partial_t {Q}\|_{L^2(\Sigma_0)} \lesssim \|\dot{\Lambda}_h^{-\lambda}\nabla v\|_{L^2(\Omega)}^2+ \| \dot{\Lambda}_h^{-\lambda}\partial_t {Q}\|_{L^2(\Sigma_0)}^2.
 \end{split}
\end{equation*}
While for $\mathfrak{J}_3$, we get
\begin{equation*}
\begin{split}|\mathfrak{J}_3|\lesssim \|\dot{\Lambda}_h^{-\lambda} v\|_{L^2}^2+\|\dot{\Lambda}_h^{-\lambda}(B_{1, i}^{h, j}\partial_jv^i)\|_{L^2}^2,
     \end{split}
\end{equation*}
 then applying Lemmas \ref{lem-productlaw-1aa} and \ref{lem-est-aij-1} ensures that
\begin{equation*}
\begin{split}
|\mathfrak{J}_3|&\lesssim \|\dot{\Lambda}_h^{-\lambda}\nabla\,v\|_{L^2}^2+ \|\dot{\Lambda}_h^{\sigma_0-1}(B_{1, i}^{h, j}) \|_{H^1}^2\|\dot{\Lambda}_h^{\sigma_0} \Lambda_h^{s_0-1-\sigma_0} \nabla\,v\|_{L^2}^2\lesssim \|\dot{\Lambda}_h^{-\lambda}\nabla\,v\|_{L^2}^2+ E_{s_0}\dot{\mathcal{D}}_{s_0}.
\end{split}
\end{equation*}
For $\mathfrak{J}_4$, one has
\begin{equation*}
\begin{split}
|\mathfrak{J}_4|&\lesssim \|\dot{\Lambda}_h^{-\lambda}\mathcal{H}( \partial_t{Q})\|_{L^2}\|\dot{\Lambda}_h^{-\lambda} \nabla\,v\|_{L^2}+ \|\dot{\Lambda}_h^{-\lambda} v\|_{L^2}\|\dot{\Lambda}_h^{-\lambda}\partial_1 \mathcal{H}(\partial_t {Q})\|_{L^2} \\
 &\lesssim \|\dot{\Lambda}_h^{-\lambda}\nabla\, v\|_{L^2}\|\dot{\Lambda}_h^{-\lambda}\Lambda_h^{\frac{1}{2}} \partial_t {Q} \|_{L^2(\Sigma_0)}.
        \end{split}
\end{equation*}
Thanks to Lemma \ref{lem-est-G-norm-1}, we deal with the integrals $\mathfrak{J}_5$ and $\mathfrak{J}_6$ to obtain
\begin{equation*}
  \begin{split}
|\mathfrak{J}_5|\lesssim \|\dot{\Lambda}_h^{-\lambda}\partial_t v\|_{L^2}\|\dot{\Lambda}_h^{-\lambda}\partial_1 \widetilde{\mathcal{G}}^j\|_{L^2}\lesssim \|\dot{\Lambda}_h^{-\lambda}\partial_t v\|_{L^2}E_{s_0}^{\frac{1}{2}}\dot{\mathcal{D}}_{s_0}^{\frac{1}{2}},
\end{split}
\end{equation*}
and
\begin{equation*}
  \begin{split}
|\mathfrak{J}_6|\lesssim \|\dot{\Lambda}_h^{-\lambda}
 \partial_t \widetilde{\mathcal{G}}^j\|_{L^2}\|\dot{\Lambda}_h^{-\lambda} \partial_1v\|_{L^2} \lesssim E_{s_0}^{\frac{1}{2}}\dot{\mathcal{D}}_{s_0}^{\frac{1}{2}}\|\dot{\Lambda}_h^{-\lambda} \partial_1v\|_{L^2}.
\end{split}
\end{equation*}
Therefore, we conclude that
\begin{equation*}
  \begin{split}
  \sum_{j=1}^6|\mathfrak{J}_j|\lesssim  &\|\dot{\Lambda}_h^{-\lambda}\nabla v\|_{L^2(\Omega)}^2+\|\dot{\Lambda}_h^{-\lambda}\partial_t v\|_{L^2} \|\dot{\Lambda}_h^{-\lambda}\nabla\partial_{h}v\|_{L^2} + \| \dot{\Lambda}_h^{-\lambda}\Lambda_h^{\frac{1}{2}}\partial_t {Q}\|_{L^2(\Sigma_0)}^2+E_{s_0}^{\frac{1}{2}}\dot{\mathcal{D}}_{s_0},
  \end{split}
\end{equation*}
which leads to \eqref{grad-lambda-tv-1}, and the desired result is proved.
\end{proof}

With Lemmas \ref{lem-grad-sigma0-s-1} in hand, we have
\begin{lem}\label{lem-tan-grad-total-1}
Let $\sigma>2$, $s_0>2$, under the assumption of Lemma \ref{lem-tan-pseudo-energy-1}, if $E_{s_0}(t) \leq 1$ for all $t\in [0, T]$, then there holds that $\forall\,t\in [0, T]$
\begin{equation}\label{tan-grad-decay-total-1}
\begin{split}
&\frac{d}{dt}\bigg(\mathring{\mathcal{E}}(\dot{\Lambda}_h^{\sigma_0}{\Lambda}_h^{\sigma-1-\sigma_0}\nabla\, v)\bigg)+ 2c_3\|\dot{\Lambda}_h^{\sigma_0}{\Lambda}_h^{\sigma-1-\sigma_0}\partial_t v\|_{L^2}^2 \\
&\leq C_3 \bigg(\|\dot{\Lambda}_h^{\sigma_0}\Lambda_h^{\sigma-\sigma_0}\nabla\,v\|_{L^2}^2 +\|\dot{\Lambda}_h^{\sigma_0}\Lambda_h^{\sigma-\frac{1}{2}-\sigma_0} \partial_t{Q} \|_{L^2(\Sigma_0)}^2 + \dot{\mathcal{D}}_{\sigma}^{\frac{1}{2}}(E_{s_0}^{\frac{1}{2}}\dot{\mathcal{D}}_{\sigma}^{\frac{1}{2}}
   +E_{\sigma}^{\frac{1}{2}}\,\dot{\mathcal{D}}_{s_0}^{\frac{1}{2}})\bigg),
\end{split}
\end{equation}
and
\begin{equation}\label{tan-grad-bdd-total-1}
\begin{split}
&\frac{d}{dt}\bigg(\mathring{\mathcal{E}}(\dot{\Lambda}_h^{\sigma_0}{\Lambda}_h^{\sigma-1-\sigma_0}\nabla\, v)+\mathring{\mathcal{E}}(\dot{\Lambda}_h^{-\lambda}\nabla\, v)\bigg)+ 2c_3(\|\dot{\Lambda}_h^{\sigma_0}{\Lambda}_h^{\sigma-1-\sigma_0}\partial_t v\|_{L^2}^2+\|\dot{\Lambda}_h^{-\lambda}\partial_t v\|_{L^2}^2) \\
&\leq C_3 \bigg(\|\dot{\Lambda}_h^{-\lambda} \nabla v\|_{L^2(\Omega)}^2+\|\dot{\Lambda}_h^{\sigma_0}\Lambda_h^{\sigma-\sigma_0}\nabla\,v\|_{L^2}^2 + \| \dot{\Lambda}_h^{-\lambda} \partial_t {Q}\|_{L^2(\Sigma_0)}^2\\
&\qquad\qquad \qquad\qquad+\|\dot{\Lambda}_h^{\sigma_0}\Lambda_h^{\sigma-\frac{1}{2}-\sigma_0} \partial_t{Q} \|_{L^2(\Sigma_0)}^2 + \dot{\mathcal{D}}_{\sigma}^{\frac{1}{2}}(E_{s_0}^{\frac{1}{2}}\dot{\mathcal{D}}_{\sigma}^{\frac{1}{2}}
   +E_{\sigma}^{\frac{1}{2}}\,\dot{\mathcal{D}}_{s_0}^{\frac{1}{2}})\bigg).
\end{split}
\end{equation}
\end{lem}

Define
\begin{equation*}
\begin{split}
& \widehat{\dot{\mathcal{E}}}_{\sigma,\text{tan},  \epsilon}(t)\eqdefa \|\sqrt{\bar{\rho}(x_1)}\,\dot{\Lambda}_h^{\sigma_0}\Lambda_h^{\sigma-\sigma_0}v\|_{L^2(\Omega)}^2+g^{-1} \|\sqrt{-\bar{\rho}'(x_1)}\, \dot{\Lambda}_h^{\sigma_0}\Lambda_h^{\sigma-\sigma_0}q\|_{L^2(\Omega)}^2 \\
&\,+ g\,\bar{\rho}(0)\|\dot{\Lambda}_h^{\sigma_0}\Lambda_h^{\sigma-\sigma_0}\xi^1\|_{L^2(\Sigma_0)}^2
+\epsilon\,\|\dot{\Lambda}_h^{\sigma_0-1} {\Lambda}_h^{\sigma+\frac{1}{2}-\sigma_0}{Q}\|_{L^2(\Sigma_0)}^2 +\epsilon^2 \, \mathring{\mathcal{E}}(\dot{\Lambda}_h^{\sigma_0}{\Lambda}_h^{\sigma-1-\sigma_0}\nabla\, v)\\
& \widehat{\dot{\mathcal{D}}}_{\sigma, \text{tan},\epsilon}(t)\eqdefa \frac{1}{2}c_1 \|\dot{\Lambda}_h^{\sigma_0}\Lambda_h^{\sigma-\sigma_0}\nabla\,v\|_{L^2(\Omega)}^2 +\frac{1}{2}c_2\,\epsilon\, \| \dot{\Lambda}_h^{\sigma_0-1} {\Lambda}_h^{\sigma+\frac{1}{2}-\sigma_0}( {Q},\,\partial_t{Q})\|_{L^2(\Sigma_0)}^2 \\
&\qquad\qquad\qquad\qquad+\frac{1}{2}c_3\,\epsilon^2\,\|\dot{\Lambda}_h^{\sigma_0}{\Lambda}_h^{\sigma-1-\sigma_0}\partial_t v\|_{L^2}^2.
\end{split}
\end{equation*}
Thanks to \eqref{tan-decay-total-1}, \eqref{est-Q-bdry-decay-2cc1}, and \eqref{tan-grad-decay-total-1}, we obtain
\begin{equation}\label{ineq-tan-decay-energy-1}
  \begin{split}
&\frac{d}{dt}\widehat{\dot{\mathcal{E}}}_{\sigma, \text{tan},\epsilon}(t)+2\,\widehat{\dot{\mathcal{D}}}_{\sigma, \text{tan},\epsilon}(t)\\
&\leq  (C_2\,\epsilon+C_3 \,\epsilon^2)\, \|\dot{\Lambda}_h^{\sigma_0} {\Lambda}_h^{\sigma-\sigma_0} \nabla\,v \|_{L^2(\Omega)}^2 +C_3 \,\epsilon^2\,\|\dot{\Lambda}_h^{\sigma_0}\Lambda_h^{\sigma-\frac{1}{2}-\sigma_0} \partial_t{Q} \|_{L^2(\Sigma_0)}^2\\
&\qquad+(C_1+C_3 \,\epsilon^2)\,(E_{\sigma}^{\frac{1}{2}}\dot{\mathcal{D}}_{s_0}^{\frac{1}{2}}
+E_{s_0}^{\frac{1}{2}} \dot{\mathcal{D}}_{\sigma}^{\frac{1}{2}})\dot{\mathcal{D}}_{\sigma}^{\frac{1}{2}}+ C_2\,\epsilon\,(E_{s_0}\dot{\mathcal{D}}_{s}+E_{\sigma}\dot{\mathcal{D}}_{s_0}).
\end{split}
\end{equation}
From this, for any small positive constant $\epsilon \leq \min\{\frac{c_1}{3C_2}, \frac{c_2}{2C_3}\}$ (which will be determined later), we have
\begin{lem}\label{lem-decay-total-1}
Under the assumption of Lemma \ref{lem-tan-pseudo-energy-1}, if $(\lambda,\,\sigma_0) \in (0, 1)$ satisfies $1-\lambda< \sigma_0\leq 1-\frac{1}{2}\lambda$, and $E_{s_0}(t) \leq 1$ for all $t\in [0, T]$, then there holds that $\forall\,t\in [0, T]$
\begin{equation*}
\begin{split}
\frac{d}{dt}\widehat{\dot{\mathcal{E}}}_{s, \text{tan},\epsilon}(t)&+ \widehat{\dot{\mathcal{D}}}_{s, \text{tan},\epsilon}(t)\\
&
\leq (C_1+C_3 \,\epsilon^2)(E_{s}^{\frac{1}{2}}\dot{\mathcal{D}}_{s_0}^{\frac{1}{2}}
+E_{s_0}^{\frac{1}{2}} \dot{\mathcal{D}}_s^{\frac{1}{2}})\dot{\mathcal{D}}_{s}^{\frac{1}{2}}+ C_2\,\epsilon\,(E_{s_0} \dot{\mathcal{D}}_{s}
+E_s \dot{\mathcal{D}}_{s_0}).
\end{split}
\end{equation*}
\end{lem}

With Lemmas \ref{lem-tan-grad-total-1} and \ref{lem-grad-bdd-lambda-1} in hand, we have
\begin{lem}\label{lem-tan-grad-N+1-1}
Let $N\geq 3$, under the assumption of Lemma \ref{lem-tan-pseudo-energy-1}, if $(\lambda,\,\sigma_0) \in (0, 1)$ satisfies $1-\lambda< \sigma_0\leq 1-\frac{1}{2}\lambda$, and $E_3(t) \leq 1$ for all $t\in [0, T]$, then there holds that $\forall\,t\in [0, T]$
\begin{equation*}
\begin{split}
&\frac{d}{dt}\bigg(\mathring{\mathcal{E}}(\dot{\Lambda}_h^{-\lambda}\nabla\, v) + \mathring{\mathcal{E}}(\dot{\Lambda}_h^{\sigma_0}\Lambda_h^{s-1-\sigma_0}\nabla\, v)\bigg) + c_2(\|\dot{\Lambda}_h^{-\lambda}\partial_t v\|_{L^2}^2  + \|\dot{\Lambda}_h^{\sigma_0}\Lambda_h^{s-1-\sigma_0}\partial_t v\|_{L^2}^2)\\
&\leq C_2 \bigg(\|\dot{\Lambda}_h^{-\lambda} \nabla\,v\|_{L^2}^2 +\|\dot{\Lambda}_h^{\sigma_0}\Lambda_h^{s-\sigma_0}\nabla\,v\|_{L^2}^2 +\| \dot{\Lambda}_h^{-\lambda} \partial_t {Q}\|_{L^2(\Sigma_0)}^2\\
&\qquad\qquad\qquad +\|\dot{\Lambda}_h^{\sigma_0}\Lambda_h^{s-\frac{1}{2}-\sigma_0} \partial_t{Q} \|_{L^2(\Sigma_0)}^2 + \dot{\mathcal{D}}_{s}^{\frac{1}{2}}(E_{s_0}^{\frac{1}{2}}\dot{\mathcal{D}}_{s}^{\frac{1}{2}}
   +E_{s}^{\frac{1}{2}}\,\dot{\mathcal{D}}_{s_0}^{\frac{1}{2}})\bigg).
\end{split}
\end{equation*}
\end{lem}

 Define
\begin{equation*}
\begin{split}
& \widehat{{\mathcal{E}}}_{\sigma,\text{tan},\epsilon}\eqdefa \|\sqrt{\bar{\rho}(x_1)}\,(\dot{\Lambda}_h^{-\lambda}v, \dot{\Lambda}_h^{\sigma_0}\Lambda_h^{\sigma-\sigma_0}v)\|_{L^2(\Omega)}^2+ g\,\bar{\rho}(0)\|(\dot{\Lambda}_h^{-\lambda}\xi^1, \dot{\Lambda}_h^{\sigma_0}\Lambda_h^{\sigma-\sigma_0}\xi^1)\|_{L^2(\Sigma_0)}^2\\
&\quad+g^{-1} \|\sqrt{-\bar{\rho}'(x_1)}\, (\dot{\Lambda}_h^{-\lambda}q, \dot{\Lambda}_h^{\sigma_0}\Lambda_h^{\sigma-\sigma_0}q)\|_{L^2(\Omega)}^2
+\epsilon\,\|(\dot{\Lambda}_h^{-\lambda}Q,\,\dot{\Lambda}_h^{\sigma_0-1} {\Lambda}_h^{\sigma+\frac{1}{2}-\sigma_0}{Q})\|_{L^2(\Sigma_0)}^2 \\
 &\qquad+\epsilon^2 \, (\mathring{\mathcal{E}}(\dot{\Lambda}_h^{-\lambda}\nabla\, v)+\mathring{\mathcal{E}}(\dot{\Lambda}_h^{\sigma_0}{\Lambda}_h^{\sigma-1-\sigma_0}\nabla\, v)),\\
& \widehat{{\mathcal{D}}}_{\sigma, \text{tan},\epsilon}\eqdefa \frac{c_1 }{2} \|(\dot{\Lambda}_h^{-\lambda}\nabla\,v,\,\dot{\Lambda}_h^{\sigma_0}\Lambda_h^{\sigma-\sigma_0}\nabla\,v)\|_{L^2(\Omega)}^2 +\frac{c_3\,\epsilon^2}{2}\,\|(\dot{\Lambda}_h^{-\lambda}\partial_tv,\,
\dot{\Lambda}_h^{\sigma_0}{\Lambda}_h^{\sigma-1-\sigma_0}\partial_t v)\|_{L^2}^2\\
&\qquad\qquad+\frac{c_2\,\epsilon}{2}\,( \|\dot{\Lambda}_h^{-\lambda}({Q},\,\partial_t{Q})\|_{L^2(\Sigma_0)}^2+\| \dot{\Lambda}_h^{\sigma_0-1} {\Lambda}_h^{\sigma+\frac{1}{2}-\sigma_0}( {Q},\,\partial_t{Q})\|_{L^2(\Sigma_0)}^2 ).
\end{split}
\end{equation*}
Thanks to \eqref{tan-bdd-s+1-1}, \eqref{est-Q-bdry-decay-2dd1}, and \eqref{tan-grad-bdd-total-1}, we obtain
\begin{equation*}
  \begin{split}
&\frac{d}{dt}\widehat{{\mathcal{E}}}_{s+\ell_0, \text{tan},\epsilon}(t)+2\,\widehat{{\mathcal{D}}}_{s+\ell_0, \text{tan},\epsilon}(t)
\leq C_0 \bigg(E_{s+\ell_0}^{\frac{1}{2}}\dot{\mathcal{D}}_{s_0}^{\frac{1}{2}}\dot{\mathcal{D}}_{s+\ell_0}^{\frac{1}{2}}
+E_{s_0}^{\frac{1}{2}} \dot{\mathcal{D}}_{s+\ell_0}+ E_{s_0} (\dot{\mathcal{D}}_{s_0}^{\frac{1}{2}}+ \dot{\mathcal{E}}_{s_0})\bigg)\\
&\qquad + C_2\,\epsilon\,(  \|(\dot{\Lambda}_h^{\frac{1}{2}-\lambda}\nabla\, v , \dot{\Lambda}_h^{\sigma_0} {\Lambda}_h^{s+\ell_0-\sigma_0} \nabla\, v) \|_{L^2(\Omega)}^2+ E_{s_0}\dot{\mathcal{D}}_{s+\ell_0}+E_{s+\ell_0}\dot{\mathcal{D}}_{s_0})\\
&\qquad +C_3\,\epsilon^2\,\bigg(\|\dot{\Lambda}_h^{-\lambda} \nabla v\|_{L^2(\Omega)}^2+\|\dot{\Lambda}_h^{\sigma_0}\Lambda_h^{s+\ell_0-\sigma_0}\nabla\,v\|_{L^2}^2 + \| \dot{\Lambda}_h^{-\lambda} \partial_t {Q}\|_{L^2(\Sigma_0)}^2\\
&\qquad\qquad \qquad\qquad+\|\dot{\Lambda}_h^{\sigma_0}\Lambda_h^{s+\ell_0-\frac{1}{2}-\sigma_0} \partial_t{Q} \|_{L^2(\Sigma_0)}^2 + \dot{\mathcal{D}}_{s+\ell_0}^{\frac{1}{2}}(E_{s_0}^{\frac{1}{2}}\dot{\mathcal{D}}_{s+\ell_0}^{\frac{1}{2}}
   +E_{s+\ell_0}^{\frac{1}{2}}\,\dot{\mathcal{D}}_{s_0}^{\frac{1}{2}})\bigg).
\end{split}
\end{equation*}
From this, for any small positive constant $\epsilon \leq \min\{\frac{c_1}{6C_2}, \frac{c_2}{4C_3}\}$ (which will be determined later), we have
\begin{lem}\label{lem-tan-energy-total-1}
Under the assumption of Lemma \ref{lem-tan-pseudo-energy-1}, if $(\lambda,\,\sigma_0) \in (0, 1)$ satisfies $1-\lambda< \sigma_0\leq 1-\frac{1}{2}\lambda$, and $E_{s_0}(t) \leq 1$ for all $t\in [0, T]$, then there holds that $\forall\,t\in [0, T]$
\begin{equation*}
\begin{split}
 \frac{d}{dt}\widehat{{\mathcal{E}}}_{s+\ell_0, \text{tan},\epsilon}(t)&+ \widehat{{\mathcal{D}}}_{s+\ell_0, \text{tan},\epsilon}(t)
\leq  C_2\,\epsilon\,(E_{s_0}  {\mathcal{D}}_{s+\ell_0}
+E_{s+\ell_0}  {\mathcal{D}}_{s_0})\\
& + (C_1+C_3 \,\epsilon^2) (E_{s+\ell_0}^{\frac{1}{2}}\dot{\mathcal{D}}_{s_0}^{\frac{1}{2}}\dot{\mathcal{D}}_{s+\ell_0}^{\frac{1}{2}}
+E_{s_0}^{\frac{1}{2}} \dot{\mathcal{D}}_{s+\ell_0}+ E_{s_0}  (\dot{\mathcal{D}}_{s_0}^{\frac{1}{2}}+ \dot{\mathcal{E}}_{s_0})).
\end{split}
\end{equation*}
\end{lem}

\renewcommand{\theequation}{\thesection.\arabic{equation}}
\setcounter{equation}{0}
\section{Stokes estimates}\label{sect-stokes}

\subsection{Estimate of the pressure term $\partial_1q$}

From the equation of $v^1$ in the momentum equations of the system \eqref{eqns-linear-1}, we deduce
\begin{equation}\label{comp-eqns-v1-1}
  \begin{split}
   &  (\delta+\frac{4}{3}\varepsilon)\partial_1^2v^1- \bar{\rho}(x_1)\partial_1\,q= \bar{\rho}(x_1)\partial_t v^1-\mathcal{L}_{1}(\partial_h\nabla\,v)-\mathfrak{g}_1.
    \end{split}
\end{equation}
While due to \eqref{eqns-q-form-3}, applying the operator $\partial_1$ to it ensures
\begin{equation}\label{comp-second-v1-1}
  \begin{split}
 \partial_1^2v^1 =&\frac{\bar{\rho}'(x_1) }{g\bar{\rho}(x_1)} \partial_t\partial_1q+\mathcal{K}_{1}+\widetilde{\mathcal{K}}_{1}.
 \end{split}
\end{equation}
with
\begin{equation*}
  \begin{split}
 &\mathcal{K}_{1}\eqdefa-\partial_1\nabla_h\cdot v^h+ \bigg(\frac{ \bar{\rho}''(x_1)}{ \bar{\rho}'(x_1)}-\frac{2\bar{\rho}'(x_1)}{\bar{\rho}(x_1)}\bigg) \partial_1v^1+ \bigg(\frac{ \bar{\rho}''(x_1)}{ \bar{\rho}'(x_1)}-\frac{\bar{\rho}'(x_1)}{\bar{\rho}(x_1)}\bigg) \nabla_h\cdot v^h,\\
 &\widetilde{\mathcal{K}}_{1}\eqdefa-\frac{ \bar{\rho}''(x_1)}{\bar{\rho}(x_1)\bar{\rho}'(x_1)}B_{1, i}^{h, j}\partial_jv^i +\frac{1}{\bar{\rho}(x_1)}\partial_1(B_{1, i}^{h, j}\partial_jv^i).
 \end{split}
\end{equation*}
Substituting \eqref{comp-second-v1-1} into \eqref{comp-eqns-v1-1} implies
\begin{equation}\label{comp-eqns-v1-2}
  \begin{split}
   & \partial_t\partial_1q+\frac{g\bar{\rho}(x_1)^2}{-(\delta+\frac{4}{3}\varepsilon)\bar{\rho}'(x_1) } \partial_1\,q= \frac{g\bar{\rho}(x_1)^2}{(\delta+\frac{4}{3}\varepsilon)\bar{\rho}'(x_1) } \partial_t v^1+\mathcal{K}_{2}+\widetilde{\mathcal{K}}_{2},\\
    \end{split}
\end{equation}
where
\begin{equation*}
  \begin{split}
   & \mathcal{K}_2\eqdefa \frac{g\bar{\rho}(x_1)}{-(\delta+\frac{4}{3}\varepsilon)\bar{\rho}'(x_1)}\mathcal{L}_{1}(\partial_h\nabla\,v)
   +\frac{g\bar{\rho}(x_1)}{-\bar{\rho}'(x_1) }\mathcal{K}_{1},\,\,\widetilde{\mathcal{K}}_{2}\eqdefa  \frac{g\bar{\rho}(x_1)}{-\bar{\rho}'(x_1) } \widetilde{\mathcal{K}}_{1}
   +\frac{g\bar{\rho}(x_1)}{-(\delta+\frac{4}{3}\varepsilon)\bar{\rho}'(x_1) }\mathfrak{g}_1.
    \end{split}
\end{equation*}
\begin{lem}\label{lem-est-pseudo-1q-energy-1}
  Let $\sigma \in \mathbb{R}$, there holds
  \begin{equation}\label{comp-eqns-v1-9}
  \begin{split}
   &\frac{d}{dt} \|\dot{\Lambda}_h^{\sigma}\partial_1q\|_{L^2}^2
   +c_4\|\dot{\Lambda}_h^{\sigma}(\partial_1\,q,\,\partial_t\partial_1\,q,\,\partial_1^2v^1)\|_{L^2}^2\\
   &\lesssim\|\dot{\Lambda}_h^{\sigma} (\nabla\,v,\,\partial_h\nabla\,v,\,\partial_tv)\|_{L^2}^2+\|\dot{\Lambda}_h^{\sigma} \mathfrak{g}_1\|_{L^2}^2+\|\dot{\Lambda}_h^{\sigma}(B_{1, i}^{h, j}\partial_jv^i)\|_{H^1}^2.
    \end{split}
\end{equation}
\end{lem}
\begin{proof}
From the equation \eqref{comp-eqns-v1-2}, the basic energy estimate gives rise to
\begin{equation*}
  \begin{split}
   &\frac{1}{2}\frac{d}{dt} \|\dot{\Lambda}_h^{\sigma}\partial_1\,q\|_{L^2}^2+\|\sqrt{\frac{g\bar{\rho}(x_1)^2}{-(\delta+\frac{4}{3}\varepsilon)\bar{\rho}'(x_1) }}\dot{\Lambda}_h^{\sigma}\partial_1\,q\|_{L^2}^2\\
   &=   \int_{\Omega}\frac{g\bar{\rho}(x_1)^2}{(\delta+\frac{4}{3}\varepsilon)\bar{\rho}'(x_1) }\partial_t\dot{\Lambda}_h^{\sigma} v^1\,\dot{\Lambda}_h^{\sigma}\partial_1q\,dx+  \int_{\Omega}\dot{\Lambda}_h^{\sigma}(\mathcal{K}_{2}+\widetilde{\mathcal{K}}_{2})
   \,\dot{\Lambda}_h^{\sigma}\partial_1q\,dx,
    \end{split}
\end{equation*}
which along with the fact $\mathfrak{e}_1:=\frac{1}{2}\inf_{x_1 \in [-\underline{b}, 0]}\frac{g\bar{\rho}(x_1)^2}{-(\delta+\frac{4}{3}\varepsilon)\bar{\rho}'(x_1) }>0$ implies
\begin{equation*}
  \begin{split}
   &\frac{1}{2}\frac{d}{dt} \|\dot{\Lambda}_h^{\sigma}\partial_1\,q\|_{L^2}^2+\mathfrak{e}_1\|\dot{\Lambda}_h^{\sigma}\partial_1\,q\|_{L^2}^2\lesssim (\|\dot{\Lambda}_h^{\sigma} \partial_tv^1\|_{L^2}+\|\dot{\Lambda}_h^{\sigma}(\mathcal{K}_{2}+\widetilde{\mathcal{K}}_{2})\|_{L^2})
   \|\dot{\Lambda}_h^{\sigma}\partial_1q\|_{L^2}.
    \end{split}
\end{equation*}
Hence, we have
\begin{equation}\label{comp-eqns-v1-5}
  \begin{split}
   &\frac{d}{dt} \|\dot{\Lambda}_h^{\sigma}\partial_1\,q\|_{L^2}^2+\mathfrak{e}_1\|\dot{\Lambda}_h^{\sigma}\partial_1\,q\|_{L^2}^2
   \lesssim\|\dot{\Lambda}_h^{\sigma}\partial_tv\|_{L^2}^2
   +\|\dot{\Lambda}_h^{\sigma}(\mathcal{K}_{2}+\widetilde{\mathcal{K}}_{2})\|_{L^2}^2.
    \end{split}
\end{equation}
Thanks to \eqref{comp-eqns-v1-1} and \eqref{comp-eqns-v1-2}, we obtain
\begin{equation}\label{comp-eqns-v1-6}
  \begin{split}
  \|\dot{\Lambda}_h^{\sigma}\partial_1^2v^1\|_{L^2} \lesssim&\|\dot{\Lambda}_h^{\sigma}\partial_1\,q\|_{L^2}+ \|\dot{\Lambda}_h^{\sigma}\partial_t v^1\|_{L^2}+\|\dot{\Lambda}_h^{\sigma}\mathcal{L}_{1}(\partial_h\nabla\,v)\|_{L^2}
  +\|\dot{\Lambda}_h^{\sigma}\mathfrak{g}_1\|_{L^2},\\
          \|\dot{\Lambda}_h^{\sigma}\partial_t\partial_1q\|_{L^2}\lesssim& \|\dot{\Lambda}_h^{\sigma}\partial_1\,q\|_{L^2}+\| \dot{\Lambda}_h^{\sigma}\partial_t v^1\|_{L^2}+\|\dot{\Lambda}_h^{\sigma}(\mathcal{K}_{2}+\widetilde{\mathcal{K}}_{2})\|_{L^2}.
    \end{split}
\end{equation}
Combining \eqref{comp-eqns-v1-5} with \eqref{comp-eqns-v1-6} ensures
\begin{equation*}
  \begin{split}
  &\frac{d}{dt}\|\dot{\Lambda}_h^{\sigma}\partial_1\,q\|_{L^2}^2
   +2c_4\|\dot{\Lambda}_h^{\sigma}(\partial_1\,q,\,\partial_t\partial_1\,q, \partial_1^2v^1)\|_{L^2}^2\\
   &\lesssim \|\dot{\Lambda}_h^{\sigma} (v, \, \nabla\,v,\,\partial_h\nabla\,v,\,\partial_tv)\|_{L^2}^2+\|\dot{\Lambda}_h^{\sigma} \widetilde{\mathcal{K}}_{2}\|_{L^2}^2+\|\dot{\Lambda}_h^{\sigma}\mathfrak{g}_1\|_{L^2}^2
    \end{split}
\end{equation*}
for some positive constant $c_4$, which follows \eqref{comp-eqns-v1-9}. This ends the proof of Lemma \ref{lem-est-pseudo-1q-energy-1}.
\end{proof}

\begin{lem}\label{lem-est-1q-1}
Under the assumption of Lemma \ref{lem-tan-pseudo-energy-1}, if $\sigma\geq \sigma_0$, $(\lambda,\,\sigma_0) \in (0, 1)$ satisfies $1-\lambda< \sigma_0\leq 1-\frac{1}{2}\lambda$, and $E_{s_0}(t) \leq 1$ for all $t\in [0, T]$, then there holds that $\forall\,t\in [0, T]$
\begin{equation}\label{comp-eqns-v11-s-1}
  \begin{split}
   &\frac{d}{dt} \|\dot{\Lambda}_h^{\sigma_0}{\Lambda}_h^{\sigma-\sigma_0}\partial_1q\|_{L^2}^2
   +c_4\|\dot{\Lambda}_h^{\sigma_0}{\Lambda}_h^{\sigma-\sigma_0}(\partial_1\,q,\,\partial_t\partial_1\,q,\,\partial_1^2v^1)\|_{L^2}^2\\
   &\leq C_4(\|\dot{\Lambda}_h^{\sigma_0}{\Lambda}_h^{\sigma-\sigma_0} (\nabla\,v,\,\partial_h\nabla\,v,\,\partial_tv)\|_{L^2}^2+E_{s_0}\,\dot{\mathcal{D}}_{\sigma+1}+E_{\sigma+1}\,
   \dot{\mathcal{D}}_{s_0}),
    \end{split}
    \end{equation}
    and
\begin{equation}\label{comp-eqns-v11-lambda-1}
  \begin{split}
   &\frac{d}{dt} \|\dot{\Lambda}_h^{-\lambda}\partial_1q\|_{L^2}^2
   +c_4\|\dot{\Lambda}_h^{-\lambda}(\partial_1\,q,\,\partial_t\partial_1\,q,\,\partial_1^2v^1)\|_{L^2}^2\\
     &\leq C_4(\|\dot{\Lambda}_h^{-\lambda} (\nabla\,v,\,\partial_h\nabla\,v,\,\partial_tv)\|_{L^2}^2+E_{s_0} \dot{\mathcal{E}}_{s_0}+E_{s_0}\,
   \dot{\mathcal{D}}_{s_0}).
    \end{split}
\end{equation}
\end{lem}
\begin{proof}
Thanks to Lemmas \ref{lem-est-B-Bv-1} and \ref{lem-est-g-1}, we obtain
\begin{equation}\label{comp-eqns-v11-s-1aa}
  \begin{split}
   & \|\dot{\Lambda}_h^{\sigma} \mathfrak{g}_1\|_{L^2}^2+\|\dot{\Lambda}_h^{\sigma}(B_{1, i}^{h, j}\partial_jv^i)\|_{H^1}^2 \lesssim E_{s_0}\,\dot{\mathcal{D}}_{\sigma+1}+E_{\sigma+1}\,
   \dot{\mathcal{D}}_{s_0} \quad (\forall \,\,\sigma \geq \sigma_0),
    \end{split}
\end{equation}
and
\begin{equation}\label{comp-eqns-v11-lambda-1aa}
  \begin{split}
   & \|\dot{\Lambda}_h^{-\lambda} \mathfrak{g}_1\|_{L^2}^2+\|\dot{\Lambda}_h^{-\lambda}(B_{1, i}^{h, j}\partial_jv^i)\|_{H^1}^2 \lesssim E_{s_0} \dot{\mathcal{E}}_{s_0}+E_{s_0}\,
   \dot{\mathcal{D}}_{s_0} .
    \end{split}
\end{equation}
Plugging \eqref{comp-eqns-v11-s-1aa} and \eqref{comp-eqns-v11-lambda-1aa} into \eqref{comp-eqns-v1-9} yields \eqref{comp-eqns-v11-s-1} and \eqref{comp-eqns-v11-lambda-1} respectively, which completes the proof of Lemma \ref{lem-est-1q-1}.
\end{proof}

\subsection{Stokes estimates}

In this subsection, we will investigate the dissipative estimates in terms of $\nabla^2v$ and $\nabla q$ as well as their horizontal derivatives in the $L^2$ framework, which is based on the Stokes estimates.

We first recall the classical regularity theory for the Stokes problem with mixed boundary conditions on the boundary stated in Theorem 10.5 on page 78 of \cite{Agmon-D-N-1964}.

\begin{lem}[\cite{Agmon-D-N-1964}]\label{lem-Stokes-1}
Suppose $v$, $q$ solve
\begin{equation*}\label{eqns-Stokes-1}
 \begin{cases}
  & -\nu\nabla\,\cdot \mathbb{D}(v )+\grad\, q=\phi \in H^{r-2}(\Omega),\\
   &\dive\,v=\psi \in H^{r-1}(\Omega),\\
   &(q\,\mathbb{I}-\nu\mathbb{D}(v))e_1|_{\Sigma_0}= \kappa \in H^{r-\frac{3}{2}}(\Sigma_0),\\
   &v|_{\Sigma_b}=0.
 \end{cases}
\end{equation*}
Then, for $r\geq 2$,
\begin{equation*}\label{est-Stokes-1}
 \begin{split}
  & \|v\|_{H^r(\Omega)}^2+ \|q\|_{H^{r-1}(\Omega)}^2\lesssim \|\phi\|_{H^{r-2}(\Omega)}^2+\|\psi\|_{H^{r-1}(\Omega)}^2
  +\|\kappa\|_{H^{r-\frac{3}{2}}(\Sigma_0)}^2.
 \end{split}
\end{equation*}
\end{lem}

Denote the cut-off function near the bottom by the smooth function
\begin{equation*}
\chi_{in}(x_1)=
\begin{cases}
&1, \quad x_1 \in [-\underline{b}, -\frac{2}{3}\underline{b});\\
& \mbox{smooth} \quad \in [0, 1], \quad x_1 \in (-\frac{2}{3}\underline{b}, -\frac{1}{3}\underline{b}];\\
&0, \quad x_1 \in (-\frac{1}{3}\underline{b}, 0],
\end{cases}
\end{equation*}
with $|\frac{d^n}{d x_1^n}\chi_{in}(x_1)| \lesssim \underline{b}^{-n}$ for any $n \in \mathbb{N}$, and the cut-off function near the free surface by the smooth function
\begin{equation*}
\chi_{f}(x_1)=
\begin{cases}
&1, \quad x_1 \in (-\frac{2}{3}\underline{b}, 0];\\
& \mbox{smooth} \quad \in [0, 1], \quad x_1 \in (-\frac{5}{6}\underline{b}, -\frac{2}{3}\underline{b}];\\
&0, \quad x_1 \in [-\underline{b}, -\frac{5}{6}\underline{b}),
\end{cases}
\end{equation*}
with $|\frac{d^n}{d x_1^n}\chi_{f}(x_1)| \lesssim \underline{b}^{-n}$ for any $n \in \mathbb{N}$.
Notice that $\chi_{f}(\text{Supp}{\chi'_{in}})\equiv 1$. We also set $\Omega_{in}:=\Omega\,\chi_{in}$ with the measure $\chi_{in}\,dx$ and $\Omega_{f}:=\Omega\,\chi_{f}$ with the measure $\chi_{f}\,dx$.

Since there are different boundary conditions on the top boundary $\Sigma_0$ and the bottom boundary $\Sigma_b$ in \eqref{eqns-pert-1}, we need to deal with these two different situations separately.

\subsubsection{Estimates near the bottom}

Introduce
\begin{equation}\label{def-hat-q-v-1}
\begin{split}
 & \widehat{q}:=q-  \frac{(\delta-\frac{2}{3}\varepsilon)}{\bar{\rho}(x_1)}\nabla\cdot v,\quad \widehat{v}:=\frac{v}{\bar{\rho}(x_1)}.
    \end{split}
\end{equation}
Using these new variable, we may rewrite the system \eqref{eqns-pert-1} as
\begin{equation}\label{Stokes-eqns-equiv-1}
\begin{cases}
&\nabla\,\widehat{q}-\varepsilon\nabla\cdot\mathbb{D}(\widehat{v})=F,\\
 &\nabla\,\cdot  \widehat{v}=G\quad \text{in} \quad \Omega,\\
   &\widehat{q}\, e_1-\varepsilon\mathbb{D}(\widehat{v})\,e_1=H\quad \text{on} \quad \Sigma_0,\\
&\widehat{v}|_{\Sigma_b}=0,
    \end{cases}
\end{equation}
where
\begin{equation*}
\begin{split}
  &F:=-\varepsilon\nabla\cdot\bigg(v\otimes\nabla(\frac{1}{\bar{\rho}(x_1)})+\nabla(\frac{1}{\bar{\rho}(x_1)})\otimes v\bigg)-\varepsilon \nabla(\frac{1}{\bar{\rho}(x_1)}) \cdot\mathbb{D}(v)\\
  &\qquad\qquad\qquad\qquad+ (\delta-\frac{2}{3}\varepsilon)\frac{\bar{\rho}'(x_1)}{\bar{\rho}(x_1)^{2}}\nabla\cdot v- \partial_t v+\frac{1}{\bar{\rho}(x_1)}\mathfrak{g},\\
  &G:=\bar{\rho}(x_1)^{-1}\partial_1v^1+ \bar{\rho}(x_1)^{-1}\nabla_h\,\cdot\,v^h-v^1\bar{\rho}(x_1)^{-2}\bar{\rho}'(x_1),\\
 &H:=-\varepsilon(v\otimes\nabla(\frac{1}{\bar{\rho}(x_1)})+\nabla(\frac{1}{\bar{\rho}(x_1)})\otimes v)\,e_1+\frac{1}{\bar{\rho}(x_1)}
\left(
  \begin{array}{c}
   \bar{\rho}(0)\xi^1+\mathcal{B}(\partial_{\alpha}v, {Q}, \partial_t{Q})\\
    -\varepsilon\,(\mathcal{B}_{7, i}^{\alpha}\partial_{\alpha}v^{i}+\widetilde{\mathcal{B}}_{2}\partial_t{Q})\\
    - \varepsilon\,(\mathcal{B}_{8, i}^{\alpha}\partial_{\alpha}v^{i}+\widetilde{\mathcal{B}}_{3}\partial_t{Q})
  \end{array}
\right) .
    \end{split}
\end{equation*}
Define $\widehat{v}_{in}\eqdefa \widehat{v}\,\chi_{in}$, $\widehat{q}_{in}\eqdefa \widehat{q}\,\chi_{in}$, then, according to \eqref{Stokes-eqns-equiv-1}, $(\widehat{v}_{in}, \widehat{q}_{in})$ solves
\begin{equation}\label{eqns-linear-stokes-bottom-1}
 \begin{cases}
 &  -\varepsilon\nabla\,\cdot \mathbb{D}(\widehat{v}_{in} ) +\nabla\,\widehat{q}_{in}=F_{in},\\
   &\dive\,\widehat{v}_{in}=G_{in},\\
    &(\widehat{q}_{in}\,\mathbb{I}-\varepsilon\mathbb{D}(\widehat{v}_{in}))e_1|_{\Sigma_0}=0,\\
   &\widehat{v}_{in}|_{\Sigma_b}=0,
 \end{cases}
\end{equation}
where
\begin{equation*}\label{stokes-bottom-2}
 \begin{split}
 &F_{in} \thicksim F\chi_{in}+\chi_{in}'\grad v+\chi_{in}''\,v+\chi_{in}'\,\widehat{q},\quad  G_{in} \thicksim G\chi_{in}+\chi_{in}' v^1.
 \end{split}
\end{equation*}

\begin{prop}\label{prop-linear-stokes-bottom-1}
Let $\sigma \in \mathbb{R}$, there holds
\begin{equation}\label{est-linear-stokes-bottom-1}
 \begin{split}
    &\|\dot{\Lambda}_h^{\sigma}v\|_{H^2(\Omega_{in})}^2+\|\dot{\Lambda}_h^{\sigma}q\|_{H^{1}(\Omega_{in})}^2\\
    &\lesssim   \|\dot{\Lambda}_h^{\sigma}(\partial_1^2v^1, \nabla\,v,\,\nabla\partial_hv)\|_{L^2}^2+\|\dot{\Lambda}_h^{\sigma}q\|_{L^2(\Omega_{f})}^2
  +\|\dot{\Lambda}_h^{\sigma} \mathfrak{g} \|_{L^2(\Omega)}^2.
 \end{split}
\end{equation}
\end{prop}
\begin{proof}
Applying Lemma \ref{lem-Stokes-1} to \eqref{eqns-linear-stokes-bottom-1} yields
\begin{equation*}
 \begin{split}
  & \|\dot{\Lambda}_h^{\sigma}\widehat{v}_{in}\|_{H^2(\Omega)}^2+ \|\dot{\Lambda}_h^{\sigma}\widehat{q}_{in}\|_{H^{1}(\Omega)}^2\lesssim   \|\dot{\Lambda}_h^{\sigma}F_{in}\|_{L^2(\Omega)}^2
  +\|\dot{\Lambda}_h^{\sigma}G_{in}\|_{H^{1}(\Omega)}^2,
 \end{split}
\end{equation*}
which implies
\begin{equation*}
 \begin{split}
  & \|\dot{\Lambda}_h^{\sigma}\widehat{v}\|_{H^2(\Omega_{in})}^2+ \|\dot{\Lambda}_h^{\sigma}\widehat{q}\|_{H^{1}(\Omega_{in})}^2\\
  &\lesssim  \|\dot{\Lambda}_h^{\sigma}\widehat{v} \|_{H^1(\Omega_{f})}^2+ \|\dot{\Lambda}_h^{\sigma}\widehat{q}\|_{L^2(\Omega_{f})}^2+ \|\dot{\Lambda}_h^{\sigma}F\|_{L^2(\Omega)}^2
  +\|\dot{\Lambda}_h^{\sigma}G\|_{H^{1}(\Omega)}^2,
 \end{split}
\end{equation*}
where we used the fact $\chi_{f}(\text{Supp}{\chi'_{in}})\equiv 1$.

On the other hand, due to \eqref{def-hat-q-v-1}, one has
\begin{equation*}
 \begin{split}
  & \|\dot{\Lambda}_h^{\sigma}\widehat{v}\|_{H^2(\Omega_{in})}^2+ \|\dot{\Lambda}_h^{\sigma}\widehat{q}\|_{H^{1}(\Omega_{in})}^2\\
  &
  =\|\frac{1}{\bar{\rho}(x_1)}\dot{\Lambda}_h^{\sigma}v\|_{H^2(\Omega_{in})}^2+ \|\dot{\Lambda}_h^{\sigma}(q-  \frac{\delta-\frac{2}{3}\varepsilon}{\bar{\rho}(x_1)}\nabla\cdot v)\|_{H^{1}(\Omega_{in})}^2\\
  &\geq c_0 \|\dot{\Lambda}_h^{\sigma}v\|_{H^2(\Omega_{in})}^2+\|\dot{\Lambda}_h^{\sigma}q\|_{H^{1}(\Omega_{in})}^2-(\|\dot{\Lambda}_h^{\sigma}\partial_1^2v^1\|_{L^2}^2
  +\|\dot{\Lambda}_h^{\sigma}(v,\,\partial_hv)\|_{H^1}^2).
 \end{split}
\end{equation*}
Hence, we get
\begin{equation*}
 \begin{split}
  \|\dot{\Lambda}_h^{\sigma}v\|_{H^2(\Omega_{in})}^2+\|\dot{\Lambda}_h^{\sigma}q\|_{H^{1}(\Omega_{in})}^2\lesssim &\|\dot{\Lambda}_h^{\sigma}\partial_1^2v^1\|_{L^2}^2
  +\|\dot{\Lambda}_h^{\sigma}(v,\,\partial_hv)\|_{H^1}^2+\|\dot{\Lambda}_h^{\sigma}\widehat{v} \|_{H^1(\Omega_{f})}^2\\
  &+ \|\dot{\Lambda}_h^{\sigma}\widehat{q}\|_{L^2(\Omega_{f})}^2+ \|\dot{\Lambda}_h^{\sigma}F\|_{L^2(\Omega)}^2
  +\|\dot{\Lambda}_h^{\sigma}G\|_{H^{1}(\Omega)}^2,
 \end{split}
\end{equation*}
which immediately follows \eqref{est-linear-stokes-bottom-1}. This completes the proof of Proposition \ref{prop-linear-stokes-bottom-1}.
\end{proof}

\subsubsection{Estimates near the free boundary}
Recall the equations in terms of $\partial_1^2\,v^{\beta}$  (with $\beta=2, 3$)
\begin{equation}\label{eqns-vh-sec-1}
\begin{split}
  \bar{\rho}(x_1)\partial_{\beta}\,q-  \varepsilon\partial_1^2v^{\beta}=&-\bar{\rho}(x_1)\partial_t v^{\beta} +c_{\beta, ij}\partial_{\beta} \partial_i\,v^j
  +\mathfrak{g}_{\beta}
     \end{split}
\end{equation}
with
$c_{\beta, ij}\partial_{\beta} \partial_i\,v^j\eqdefa(\delta-\frac{2}{3}\varepsilon)\partial_{\beta}\nabla\cdot v
  +\varepsilon \partial_{\beta}\partial_1v^1+2\varepsilon\partial_{\beta}^2v^{\beta}
  +\varepsilon\partial_{\hat{\beta}}(\partial_{\beta}v^{\hat{\beta}}+\partial_{\hat{\beta}}v^{\beta})$.
\begin{lem}\label{lem-pseudo-near-interface-1}
Let $(v, \xi)$ be smooth solution to the system \eqref{eqns-pert-1}, then there holds that
  \begin{equation}\label{pseudo-near-interface-2}
\begin{split}
&\varepsilon\|\dot{\Lambda}_h^{\sigma}\partial_1^2v^{\beta}\|_{L^2(\Omega_f)}^2
+\frac{g\bar{\rho}(0)}{2}\frac{d}{dt}\|\dot{\Lambda}_h^{\sigma}\partial_{\beta}\xi^1\|_{L^2(\Sigma_0)}^2
=\sum_{j=1}^7\mathfrak{I}_j,\\
\end{split}
\end{equation}
where
\begin{equation*}\label{def-mfrak-J-1}
\begin{split}
&\mathfrak{I}_1:=\int_{\Omega}\bar{\rho}(x_1)\dot{\Lambda}_h^{\sigma}\partial_1^2v^{\beta}\,
\dot{\Lambda}_h^{\sigma}\partial_tv^{\beta} \,\chi_f\,dx,\,\mathfrak{I}_2:=-c_{\beta, ij}\int_{\Omega}\dot{\Lambda}_h^{\sigma}\partial_1^2v^{\beta}\, \dot{\Lambda}_h^{\sigma}\partial_{\beta} \partial_i\,v^j \,\chi_f\,dx\\
 &\mathfrak{I}_3:=\bar{\rho}(0) \int_{\Sigma_0} \dot{\Lambda}_h^{\sigma}\partial_{\beta} v^1\, \dot{\Lambda}_h^{\sigma} \partial_{\beta}Q \,dS_0,\quad \mathfrak{I}_4:=-\int_{\Omega} \bar{\rho}(x_1)\dot{\Lambda}_h^{\sigma}\partial_1v^{\beta}\, \dot{\Lambda}_h^{\sigma}\partial_{\beta}\partial_1q\,\chi_f\,dx,
\end{split}
\end{equation*}
\begin{equation*}\label{def-mfrak-J-3}
\begin{split}
 &\mathfrak{I}_5:=-\int_{\Omega} \dot{\Lambda}_h^{\sigma}\partial_1v^{\beta}\, \dot{\Lambda}_h^{\sigma}\partial_{\beta}q\,\bigg(\bar{\rho}(x_1)\chi_f'+\bar{\rho}'(x_1)\chi_f\bigg)\,dx,
\end{split}
\end{equation*}
\begin{equation*}\label{def-mfrak-J-2}
\begin{split}
 &\mathfrak{I}_6:=\int_{\Sigma_0}  \bar{\rho}(0) \dot{\Lambda}_h^{\sigma}(\mathcal{B}_{\beta+5, i}^{\alpha}\partial_{\alpha}v^{i}+\widetilde{\mathcal{B}}_{\beta}\partial_t{Q})\, \dot{\Lambda}_h^{\sigma}\partial_{\beta}q\,dS_0,\,\mathfrak{I}_7:=-\int_{\Omega}\dot{\Lambda}_h^{\sigma}\partial_1^2v^{\beta}\,
\dot{\Lambda}_h^{\sigma}\mathfrak{g}_{\beta}\,\chi_f\,dx.
\end{split}
\end{equation*}
\end{lem}
\begin{proof}
Applying the operator $\dot{\Lambda}_h^{\sigma}$ to \eqref{eqns-vh-sec-1} yields
\begin{equation}\label{eqns-vh-sec-2}
\begin{split}
&\bar{\rho}(x_1)\dot{\Lambda}_h^{\sigma}\partial_{\beta}\,q-  \varepsilon\dot{\Lambda}_h^{\sigma}\partial_1^2v^{\beta}=-\bar{\rho}(x_1)\partial_t \dot{\Lambda}_h^{\sigma}v^{\beta} +c_{\beta, ij}\dot{\Lambda}_h^{\sigma}\partial_{\beta} \partial_i\,v^j
  +\dot{\Lambda}_h^{\sigma}\mathfrak{g}_{\beta}.
     \end{split}
\end{equation}
Multiplying \eqref{eqns-vh-sec-2} by $-\dot{\Lambda}_h^{\sigma}\partial_1^2\,v^{\beta}\chi_f$ and integrating it in $\Omega$, we get
\begin{equation}\label{eqns-vh-sec-3}
\begin{split}
&\varepsilon\int_{\Omega} |\dot{\Lambda}_h^{\sigma}\partial_1^2v^{\beta}|^2\chi_f\,dx+I=
\int_{\Omega}\bar{\rho}(x_1)\dot{\Lambda}_h^{\sigma}\partial_1^2v^{\beta}\,
\partial_t \dot{\Lambda}_h^{\sigma}v^{\beta} \,\chi_f\,dx\\
 &\qquad\qquad\qquad-c_{\beta, ij}\int_{\Omega}\dot{\Lambda}_h^{\sigma}\partial_1^2v^{\beta}\, \dot{\Lambda}_h^{\sigma}\partial_{\beta} \partial_i\,v^j \,\chi_f\,dx-\int_{\Omega}\dot{\Lambda}_h^{\sigma}\partial_1^2v^{\beta}\,
\dot{\Lambda}_h^{\sigma}\mathfrak{g}_{\beta}\,\chi_f\,dx
\end{split}
\end{equation}
with $I:=-\int_{\Omega} \bar{\rho}(x_1) \dot{\Lambda}_h^{\sigma}\partial_1^2v^{\beta}\,\dot{\Lambda}_h^{\sigma}\partial_{\beta}q\,\chi_f\,dx$.

Integrating by parts in $I$ gives rise to
\begin{equation}\label{eqns-vh-sec-3a}
\begin{split}
&I=-\int_{\Sigma_0} \bar{\rho}(0)\dot{\Lambda}_h^{\sigma}\partial_1v^{\beta}\, \dot{\Lambda}_h^{\sigma}\partial_{\beta}q\,dS_0+\int_{\Omega} \bar{\rho}(x_1)\dot{\Lambda}_h^{\sigma}\partial_1v^{\beta}\, \dot{\Lambda}_h^{\sigma}\partial_{\beta}\partial_1q\,\chi_f\,dx\\
&\qquad\qquad+\int_{\Omega} \dot{\Lambda}_h^{\sigma}\partial_1v^{\beta}\, \dot{\Lambda}_h^{\sigma}\partial_{\beta}q\,(\bar{\rho}(x_1)\chi_f'+\bar{\rho}'(x_1)\chi_f)\,dx.
\end{split}
\end{equation}
Thanks to the boundary conditions in the system \eqref{eqns-pert-1} on $\Sigma_0$, we have
\begin{equation*}
\begin{split}
&-\int_{\Sigma_0}\bar{\rho}(0) \dot{\Lambda}_h^{\sigma}\partial_1v^{\beta}\, \dot{\Lambda}_h^{\sigma}\partial_{\beta}q\,dS_0\\
&=\int_{\Sigma_0} \bar{\rho}(0) \dot{\Lambda}_h^{\sigma}\partial_{\beta} v^1\,\dot{\Lambda}_h^{\sigma} \partial_{\beta}q\,dS_0-\int_{\Sigma_0} \bar{\rho}(0) \dot{\Lambda}_h^{\sigma}(\mathcal{B}_{\beta+5, i}^{\alpha}\partial_{\alpha}v^{i}+\widetilde{\mathcal{B}}_{\beta}\partial_t{Q})\, \dot{\Lambda}_h^{\sigma}\partial_{\beta}q\,dS_0\\
&= \bar{\rho}(0) \int_{\Sigma_0} \dot{\Lambda}_h^{\sigma}\partial_{\beta} v^1\, \dot{\Lambda}_h^{\sigma}(g\partial_{\beta}\xi^1+\partial_{\beta}Q )\,dS_0-\int_{\Sigma_0}  \bar{\rho}(0) \dot{\Lambda}_h^{\sigma}(\mathcal{B}_{\beta+5, i}^{\alpha}\partial_{\alpha}v^{i}+\widetilde{\mathcal{B}}_{\beta}\partial_t{Q})\, \dot{\Lambda}_h^{\sigma}\partial_{\beta}q\,dS_0,
\end{split}
\end{equation*}
which follows
\begin{equation}\label{eqns-vh-sec-3d}
\begin{split}
&-\int_{\Sigma_0}\bar{\rho}(0) \dot{\Lambda}_h^{\sigma}\partial_1v^{\beta}\, \dot{\Lambda}_h^{\sigma}\partial_{\beta}q\,dS_0=\frac{g\,\bar{\rho}(0)}{2}\frac{d}{dt}
\int_{\Sigma_0}|\dot{\Lambda}_h^{\sigma}\partial_{\beta}\xi^1|^2\,dS_0\\
&\quad+\bar{\rho}(0) \int_{\Sigma_0} \dot{\Lambda}_h^{\sigma}\partial_{\beta} v^1\, \dot{\Lambda}_h^{\sigma} \partial_{\beta}Q \,dS_0-\int_{\Sigma_0}  \bar{\rho}(0) \dot{\Lambda}_h^{\sigma}(\mathcal{B}_{\beta+5, i}^{\alpha}\partial_{\alpha}v^{i}+\widetilde{\mathcal{B}}_{\beta}\partial_t{Q})\, \dot{\Lambda}_h^{\sigma}\partial_{\beta}q\,dS_0.
\end{split}
\end{equation}
Therefore, inserting \eqref{eqns-vh-sec-3d} into \eqref{eqns-vh-sec-3a} implies
\begin{equation}\label{eqns-vh-sec-7}
\begin{split}
&I=\frac{g\,\bar{\rho}(0)}{2}\frac{d}{dt}
\int_{\Sigma_0}|\dot{\Lambda}_h^{\sigma}\partial_{\beta}\xi^1|^2\,dS_0+\bar{\rho}(0) \int_{\Sigma_0} \dot{\Lambda}_h^{\sigma}\partial_{\beta} v^1\, \dot{\Lambda}_h^{\sigma} \partial_{\beta}Q \,dS_0\\
&+\int_{\Omega} \dot{\Lambda}_h^{\sigma}\partial_1v^{\beta}\, \dot{\Lambda}_h^{\sigma}\partial_{\beta}q\,(\bar{\rho}(x_1)\chi_f'+\bar{\rho}'(x_1)\chi_f)\,dx+\int_{\Omega} \bar{\rho}(x_1)\dot{\Lambda}_h^{\sigma}\partial_1v^{\beta}\, \dot{\Lambda}_h^{\sigma}\partial_{\beta}\partial_1q\,\chi_f\,dx\\
&\qquad-\int_{\Sigma_0}  \bar{\rho}(0) \dot{\Lambda}_h^{\sigma}(\mathcal{B}_{\beta+5, i}^{\alpha}\partial_{\alpha}v^{i}+\widetilde{\mathcal{B}}_{\beta}\partial_t{Q})\, \dot{\Lambda}_h^{\sigma}\partial_{\beta}q\,dS_0.
\end{split}
\end{equation}
Plugging \eqref{eqns-vh-sec-7} into \eqref{eqns-vh-sec-3}, we get \eqref{pseudo-near-interface-2}.
\end{proof}

\subsection{Estimate of $\|(\dot{\Lambda}_h^{\sigma_0}\nabla^2v, \,\dot{\Lambda}_h^{\sigma_0}\nabla\,q)\|_{L^2_t(L^2)}$}

\begin{lem}\label{lem-laplace-sigma-1}
Let $\sigma \geq \max\{\sigma_0+1,\,s_0-1\}$. Under the assumption of Lemma \ref{lem-pseudo-energy-tv-1}, if $E_{s_0}(t) \leq 1$ for all $t\in [0, T]$, then there holds that $\forall\,t\in [0, T]$
  \begin{equation}\label{laplace-sigma0-0}
\begin{split}
&\frac{d}{dt}\|\dot{\Lambda}_h^{\sigma_0}{\Lambda}_h^{\sigma-\sigma_0}\partial_{h}\xi^1\|_{L^2(\Sigma_0)}^2
+c_6\|\dot{\Lambda}_h^{\sigma_0}{\Lambda}_h^{\sigma-\sigma_0}(\nabla^2\,v,\,\nabla\,q)\|_{L^2}^2\\
&\leq C_6\bigg(\|\dot{\Lambda}_h^{\sigma_0}{\Lambda}_h^{\sigma+1-\sigma_0}\nabla\,v\|_{L^2}^2
+\|\dot{\Lambda}_h^{\sigma_0}{\Lambda}_h^{\sigma-\sigma_0} \partial_tv\|_{L^2}^2+\|\dot{\Lambda}_h^{\sigma_0}{\Lambda}_h^{\sigma-\sigma_0}\partial_1q\|_{L^2}^2\\
 &\qquad\quad+\|\dot{\Lambda}_h^{\sigma_0}{\Lambda}_h^{\sigma-\sigma_0}{Q} \|_{\dot{H}^{\frac{1}{2}}(\Sigma_0)}^2
  +(E_{s_0}^{\frac{1}{2}}\dot{\mathcal{D}}_{\sigma+1}^{\frac{1}{2}}
+E_{\sigma+1}^{\frac{1}{2}}\dot{\mathcal{D}}_{s_0}^{\frac{1}{2}})\dot{\mathcal{D}}_{\sigma+1}^{\frac{1}{2}}
+E_{\sigma+1}\dot{\mathcal{D}}_{s_0}\bigg).
\end{split}
\end{equation}
\end{lem}
\begin{proof}
First, from the equations \eqref{eqns-vh-sec-1}, $\forall \,\,\sigma_1 \geq \sigma_0$, we get
  \begin{equation}\label{laplace-sigma0-2}
\begin{split}
   &\|-\varepsilon\dot{\Lambda}_h^{\sigma_1 }\partial_1^2\,v^{h}
   +\bar{\rho}(x_1)\dot{\Lambda}_h^{\sigma_1 }\partial_hq\|_{L^2(\Omega)}
  \lesssim \|\dot{\Lambda}_h^{\sigma_1 }\partial_t v^{h}\|_{L^2}+\|\dot{\Lambda}_h^{\sigma_1 }\partial_h\,v \|_{H^1} +\sum_{\beta=2}^3\|\dot{\Lambda}_h^{\sigma_1 } \mathfrak{g}_{\beta}\|_{L^2},\\
  &\|-\varepsilon\dot{\Lambda}_h^{\sigma_1 }\partial_1^2\,v^{h}+\bar{\rho}(x_1)\dot{\Lambda}_h^{\sigma_1 }\partial_hq\|_{L^2(\Omega_f)}
  \lesssim \|\dot{\Lambda}_h^{\sigma_1 }\partial_t v^{h}\|_{L^2}+\|\dot{\Lambda}_h^{\sigma_1 }\partial_h\,v \|_{H^1} +\sum_{\beta=2}^3\|\dot{\Lambda}_h^{\sigma_1 } \mathfrak{g}_{\beta}\|_{L^2}.
     \end{split}
\end{equation}
Thanks to Lemma \ref{lem-est-g-1}, for $\sigma \geq \max\{\sigma_0+1,\,s_0-1\}$, we have
  \begin{equation}\label{est-g-sigma-s1-1}
\begin{split}
  &\sum_{\beta=2}^3\|\dot{\Lambda}_h^{\sigma_0} \mathfrak{g}\|_{L^2}\lesssim E_{s_0}^{\frac{1}{2}}\dot{\mathcal{D}}_{\sigma_0+1}^{\frac{1}{2}},\quad\sum_{\beta=2}^3\|\dot{\Lambda}_h^{\sigma} \mathfrak{g}\|_{L^2}\lesssim E_{s_0}^{\frac{1}{2}}\dot{\mathcal{D}}_{\sigma+1}^{\frac{1}{2}}
  +E_{\sigma+1}^{\frac{1}{2}}\dot{\mathcal{D}}_{s_0}^{\frac{1}{2}},
     \end{split}
\end{equation}
from this, we thus get
\begin{equation}\label{laplace-sigma0-15}
\begin{split}
  &\|\dot{\Lambda}_h^{\sigma_0}\partial_hq\|_{L^2(\Omega_f)}
  \lesssim\|\dot{\Lambda}_h^{\sigma_0}\partial_1^2\,v^{h}\|_{L^2(\Omega_f)}+ \|\dot{\Lambda}_h^{\sigma_0}\partial_t v \|_{L^2}+\|\dot{\Lambda}_h^{\sigma_0}\partial_h\,v \|_{H^1} +E_{s_0}^{\frac{1}{2}}\dot{\mathcal{D}}_{\sigma_0+1}^{\frac{1}{2}}
     \end{split}
\end{equation}
and
\begin{equation}\label{laplace-s1-15}
\begin{split}
  &\|\dot{\Lambda}_h^{\sigma} \partial_hq\|_{L^2(\Omega_f)}
  \lesssim\|\dot{\Lambda}_h^{\sigma} \partial_1^2\,v^{h}\|_{L^2(\Omega_f)}+ \|\dot{\Lambda}_h^{\sigma} (\partial_t v ,\nabla\partial_h\,v) \|_{L^2} +E_{s_0}^{\frac{1}{2}}\dot{\mathcal{D}}_{\sigma+1}^{\frac{1}{2}}
  +E_{\sigma+1}^{\frac{1}{2}}\dot{\mathcal{D}}_{s_0}^{\frac{1}{2}}.
     \end{split}
\end{equation}
While due to \eqref{est-linear-stokes-bottom-1}, one can see
\begin{equation}\label{laplace-sigma0-18aa}
 \begin{split}
    &\|\dot{\Lambda}_h^{\sigma_0}\partial_hv\|_{H^2(\Omega_{in})}
    +\|\dot{\Lambda}_h^{\sigma_0}\partial_hq\|_{H^{1}(\Omega_{in})}\\
  &\lesssim   \|\dot{\Lambda}_h^{\sigma_0}\partial_h(\partial_1^2v^1, \nabla\,v,\,\nabla\partial_hv)\|_{L^2}
  +\|\dot{\Lambda}_h^{\sigma_0}\partial_hq\|_{L^2(\Omega_{f})}
  +\|\dot{\Lambda}_h^{\sigma_0}\partial_h \mathfrak{g} \|_{L^2(\Omega)},
 \end{split}
\end{equation}
and
\begin{equation*}
 \begin{split}
  &\|\dot{\Lambda}_h^{\sigma} v\|_{H^2(\Omega_{in})}^2
  +\|\dot{\Lambda}_h^{\sigma} q\|_{H^{1}(\Omega_{in})}^2\\
  &\lesssim   \|\dot{\Lambda}_h^{\sigma} (\partial_1^2v^1, \nabla\,v,\,\nabla\partial_hv)\|_{L^2}^2
  +\|\dot{\Lambda}_h^{\sigma} q\|_{L^2(\Omega_{f})}^2+ \|\dot{\Lambda}_h^{\sigma} \mathfrak{g} \|_{L^2(\Omega)}^2
 \\
  &\lesssim   \|\dot{\Lambda}_h^{\sigma} (\partial_1^2v^1,\,\nabla\,v,\,\nabla\partial_hv)\|_{L^2}^2+\|\dot{\Lambda}_h^{\sigma_0+1} q\|_{L^2(\Omega_{f})}^2+\|\dot{\Lambda}_h^{\sigma}\partial_h q\|_{L^2(\Omega_{f})}^2 +\|\dot{\Lambda}_h^{\sigma} \mathfrak{g} \|_{L^2(\Omega)}^2.
 \end{split}
\end{equation*}
Hence, thanks to \eqref{est-g-sigma-s1-1}, \eqref{laplace-sigma0-15}, and \eqref{laplace-s1-15}, we find
\begin{equation*}
 \begin{split}
   &\|\dot{\Lambda}_h^{\sigma_0}\partial_hv\|_{H^2(\Omega_{in})}^2
  +\|\dot{\Lambda}_h^{\sigma_0}\partial_hq\|_{H^{1}(\Omega_{in})}^2\\
  &\lesssim  \|\dot{\Lambda}_h^{\sigma_0}\partial_1^2\,v^{h}\|_{L^2(\Omega_f)}^2+ \|\dot{\Lambda}_h^{\sigma_0}\partial_h(\partial_1^2v^1,\, \nabla\,v,\,\nabla\partial_hv)\|_{L^2}^2+ \|\dot{\Lambda}_h^{\sigma_0}\partial_t v \|_{L^2}^2
  +E_{s_0} \dot{\mathcal{D}}_{s_0},
 \end{split}
\end{equation*}
and for $\sigma \geq \sigma_0$
\begin{equation*}
 \begin{split}
    \|\dot{\Lambda}_h^{\sigma}\partial_hv\|_{H^2(\Omega_{in})}^2
    &+\|\dot{\Lambda}_h^{\sigma}\partial_hq\|_{H^{1}(\Omega_{in})}^2\lesssim \|\dot{\Lambda}_h^{\sigma_0+1} q\|_{L^2(\Omega_{f})}^2+\|\dot{\Lambda}_h^{\sigma} \partial_1^2\,v^{h}\|_{L^2(\Omega_f)}^2\\
    &+ \|\dot{\Lambda}_h^{\sigma} (\partial_t v ,\nabla\,v,\,\nabla\partial_hv, \partial_1^2v^1) \|_{L^2}^2
  +E_{s_0} \dot{\mathcal{D}}_{\sigma+1} +E_{\sigma+1} \dot{\mathcal{D}}_{s_0},
 \end{split}
\end{equation*}
which follows
\begin{equation*}
\begin{split}
  &\|\dot{\Lambda}_h^{\sigma_0}\partial_hq\|_{L^2(\Omega)}\lesssim \|\dot{\Lambda}_h^{\sigma_0}\partial_hq\|_{L^2(\Omega_f)}+\|\dot{\Lambda}_h^{\sigma_0}\partial_hq\|_{L^2(\Omega_{in})}\\
  &
  \lesssim\|\dot{\Lambda}_h^{\sigma_0}\partial_1^2\,v^{h}\|_{L^2(\Omega_f)}+ \|\dot{\Lambda}_h^{\sigma_0}\partial_h(\partial_1^2v^1,\,\nabla\,v,\,\nabla\partial_hv)\|_{L^2}+ \|\dot{\Lambda}_h^{\sigma_0}\partial_t v \|_{L^2}
  +E_{s_0}^{\frac{1}{2}} \dot{\mathcal{D}}_{s_0}^{\frac{1}{2}}
     \end{split}
\end{equation*}
and
\begin{equation*}
\begin{split}
&\|\dot{\Lambda}_h^{\sigma}\partial_hq\|_{L^2(\Omega)}\lesssim \|\dot{\Lambda}_h^{\sigma}\partial_hq\|_{L^2(\Omega_f)}+\|\dot{\Lambda}_h^{\sigma}\partial_hq\|_{L^2(\Omega_{in})}\\
  &
  \lesssim\|\dot{\Lambda}_h^{\sigma} \partial_1^2\,v^{h}\|_{L^2(\Omega_f)}
  +\|\dot{\Lambda}_h^{\sigma_0}\partial_1^2\,v^{h}\|_{L^2(\Omega_f)}+ \|\dot{\Lambda}_h^{\sigma_0}\partial_h(\partial_1^2v^1,\,\nabla\,v,\,\nabla\partial_hv)\|_{L^2}+ \|\dot{\Lambda}_h^{\sigma_0}\partial_t v \|_{L^2}\\
  &\qquad\qquad+ \|\dot{\Lambda}_h^{\sigma} (\partial_t v ,\nabla\,v,\,\nabla\partial_hv, \partial_1^2v^1) \|_{L^2}
+E_{s_0}^{\frac{1}{2}} \dot{\mathcal{D}}_{s_0}^{\frac{1}{2}}+E_{s_0}^{\frac{1}{2}}\dot{\mathcal{D}}_{\sigma+1}^{\frac{1}{2}}+E_{\sigma+1}^{\frac{1}{2}}\dot{\mathcal{D}}_{s_0}^{\frac{1}{2}}
     \end{split}
\end{equation*}
with $\sigma \geq \sigma_0$.
Therefore, thanks to the second inequality in \eqref{laplace-sigma0-2}, we have
  \begin{equation}\label{laplace-sigma0-26}
\begin{split}
   &\|\dot{\Lambda}_h^{\sigma_0}(\partial_1^2\,v^{h},\,\partial_hq)\|_{L^2(\Omega)}\lesssim\|\dot{\Lambda}_h^{\sigma_0}\partial_hq\|_{L^2}
  +\|\dot{\Lambda}_h^{\sigma_0}(\partial_t v^{h},\,\partial_h\nabla\,v )\|_{L^2} +\sum_{\beta=2}^3\|\dot{\Lambda}_h^{\sigma_0} \mathfrak{g}_{\beta}\|_{L^2}\\
  & \lesssim\|\dot{\Lambda}_h^{\sigma_0}\partial_1^2\,v^{h}\|_{L^2(\Omega_f)}+ \|\dot{\Lambda}_h^{\sigma_0}\partial_h(\partial_1^2v^1,\, \nabla\,v,\,\nabla\partial_hv)\|_{L^2}+ \|\dot{\Lambda}_h^{\sigma_0}\partial_t v \|_{L^2}
  +E_{s_0}^{\frac{1}{2}} \dot{\mathcal{D}}_{s_0}^{\frac{1}{2}},
     \end{split}
\end{equation}
and
  \begin{equation}\label{laplace-s1-27}
\begin{split}
   &\|\dot{\Lambda}_h^{\sigma}(\partial_1^2\,v^{h},\,\partial_hq)\|_{L^2}
   \lesssim\|\dot{\Lambda}_h^{\sigma}\partial_hq\|_{L^2}
  +\|\dot{\Lambda}_h^{\sigma}\partial_t v^{h}\|_{L^2}+\|\dot{\Lambda}_h^{\sigma}\partial_h\,v \|_{H^1} +\sum_{\beta=2}^3\|\dot{\Lambda}_h^{\sigma} \mathfrak{g}_{\beta}\|_{L^2}\\
  & \lesssim \|\dot{\Lambda}_h^{\sigma} \partial_1^2\,v^{h}\|_{L^2(\Omega_f)} +\|\dot{\Lambda}_h^{\sigma_0}\partial_1^2\,v^{h}\|_{L^2(\Omega_f)}
  + \|\dot{\Lambda}_h^{\sigma} (\partial_t v ,\nabla\,v,\,\nabla\partial_hv,   \partial_1^2v^1) \|_{L^2}\\
  &+ \|\dot{\Lambda}_h^{\sigma_0}\partial_h(\partial_1^2v^1,\,\nabla\,v,\,\nabla\partial_hv)\|_{L^2}+ \|\dot{\Lambda}_h^{\sigma_0}\partial_t v \|_{L^2}
  +E_{s_0}^{\frac{1}{2}} \dot{\mathcal{D}}_{s_0}^{\frac{1}{2}}
+E_{s_0}^{\frac{1}{2}}\dot{\mathcal{D}}_{\sigma+1}^{\frac{1}{2}}+E_{\sigma+1}^{\frac{1}{2}}\dot{\mathcal{D}}_{s_0}^{\frac{1}{2}}.
     \end{split}
\end{equation}
Combining \eqref{laplace-sigma0-26} with \eqref{laplace-s1-27} yields that, for $\sigma \geq \max\{\sigma_0+1,\,s_0-1\}$,
  \begin{equation}\label{laplace-s1-27aa}
\begin{split}
   &\|\dot{\Lambda}_h^{\sigma_0}{\Lambda}_h^{\sigma-\sigma_0}(\partial_1^2\,v^{h},\,\partial_hq)\|_{L^2}\lesssim \|\dot{\Lambda}_h^{\sigma_0}{\Lambda}_h^{\sigma-\sigma_0} \partial_1^2\,v^{h}\|_{L^2(\Omega_f)}\\
  & \qquad\qquad
  + \|\dot{\Lambda}_h^{\sigma_0}{\Lambda}_h^{\sigma-\sigma_0}(\partial_t v ,\nabla\,v,\,\nabla\partial_hv,   \partial_1^2v^1) \|_{L^2}
+E_{s_0}^{\frac{1}{2}}\dot{\mathcal{D}}_{\sigma+1}^{\frac{1}{2}}
+E_{\sigma+1}^{\frac{1}{2}}\dot{\mathcal{D}}_{s_0}^{\frac{1}{2}}.
     \end{split}
\end{equation}
On the other hand, taking $\sigma=\sigma_0$ and $\sigma \geq \max\{1+\sigma_0,\,s_0-1\}$ in \eqref{pseudo-near-interface-2} respectively gives
 \begin{equation*}
\begin{split}
&\frac{d}{dt}(\|\dot{\Lambda}_h^{\sigma_0}\partial_{h}\xi^1\|_{L^2(\Sigma_0)}^2
+\|\dot{\Lambda}_h^{\sigma}\partial_{h}\xi^1\|_{L^2(\Sigma_0)}^2)
+2c_5\|\dot{\Lambda}_h^{\sigma_0}{\Lambda}_h^{\sigma-\sigma_0}\partial_1^2v^{h}\|_{L^2(\Omega_f)}^2
\lesssim\sum_{j=1}^9|\mathfrak{I}_j|,
\end{split}
\end{equation*}
where we replaced $\dot{\Lambda}_h^{\sigma}$ by $\dot{\Lambda}_h^{\sigma_0}{\Lambda}_h^{\sigma-\sigma_0}$ in $\mathfrak{I}_j$ ($j=1,...,7$) in Lemma \ref{lem-pseudo-near-interface-1}.

For the linear remainder terms $\mathfrak{I}_j$ with $j=1, ...,6$, we find
\begin{equation*}
\begin{split}
 |\mathfrak{I}_1|&\lesssim \|\dot{\Lambda}_h^{\sigma_0}{\Lambda}_h^{\sigma-\sigma_0}\partial_1^2v^{h}\|_{L^2(\Omega_f)}
\|\dot{\Lambda}_h^{\sigma_0}{\Lambda}_h^{\sigma-\sigma_0}\partial_tv \|_{L^2},\\
 |\mathfrak{I}_2| &\lesssim \|\dot{\Lambda}_h^{\sigma_0}{\Lambda}_h^{\sigma-\sigma_0}\partial_1^2v^{h}\|_{L^2(\Omega_f)}\, \|\dot{\Lambda}_h^{\sigma_0}{\Lambda}_h^{\sigma-\sigma_0}\partial_{h} \nabla\,v\|_{L^2},\\
 |\mathfrak{I}_3| &\lesssim \|\dot{\Lambda}_h^{\sigma_0}{\Lambda}_h^{\sigma-\sigma_0}\partial_{\beta} v^1\|_{\dot{H}^{\frac{1}{2}}(\Sigma_0)}\, \|\dot{\Lambda}_h^{\sigma_0}{\Lambda}_h^{\sigma-\sigma_0}Q\|_{\dot{H}^{\frac{1}{2}}(\Sigma_0)}\\
 &\lesssim \|\dot{\Lambda}_h^{\sigma_0}{\Lambda}_h^{\sigma-\sigma_0}\partial_{h} v\|_{H^1(\Omega)}^2+ \|\dot{\Lambda}_h^{\sigma_0}{\Lambda}_h^{\sigma-\sigma_0}Q\|_{\dot{H}^{\frac{1}{2}}(\Sigma_0)}^2,\\
\end{split}
\end{equation*}
\begin{equation*}
\begin{split}
  &|\mathfrak{I}_4|+|\mathfrak{I}_5|\lesssim \|\dot{\Lambda}_h^{\sigma_0}{\Lambda}_h^{\sigma-\sigma_0}\partial_h\partial_1v\|_{L^2}\, \|\dot{\Lambda}_h^{\sigma_0}{\Lambda}_h^{\sigma-\sigma_0}\partial_1q\|_{L^2}+\|\dot{\Lambda}_h^{\sigma_0}{\Lambda}_h^{\sigma-\sigma_0}\partial_1v \|_{L^2}\, \|\dot{\Lambda}_h^{\sigma_0}{\Lambda}_h^{\sigma-\sigma_0}\partial_{h}q\|_{L^2}\\
  &\qquad\qquad\lesssim (\|\dot{\Lambda}_h^{\sigma_0}{\Lambda}_h^{\sigma-\sigma_0}\partial_h\partial_1v\|_{L^2}\, +\|\dot{\Lambda}_h^{\sigma_0}{\Lambda}_h^{\sigma-\sigma_0}\partial_1v \|_{L^2})\, \|\dot{\Lambda}_h^{\sigma_0}{\Lambda}_h^{\sigma-\sigma_0}\nabla\,q\|_{L^2},\\
\end{split}
\end{equation*}
which along with \eqref{laplace-s1-27aa} ensures
\begin{equation}\label{def-mfrak-J-1to6}
\begin{split}
&\sum_{j=1}^5|\mathfrak{I}_j|\lesssim \|\dot{\Lambda}_h^{\sigma_0}{\Lambda}_h^{\sigma-\sigma_0}\partial_1^2v^{h}\|_{L^2(\Omega_f)}
\|\dot{\Lambda}_h^{\sigma_0}{\Lambda}_h^{\sigma-\sigma_0}(\partial_tv,\partial_{h} \nabla\,v)\|_{L^2}+ \|\dot{\Lambda}_h^{\sigma_0}{\Lambda}_h^{\sigma-\sigma_0}\partial_{h}\nabla\, v\|_{L^2}^2\\
 & \qquad
+\|\dot{\Lambda}_h^{\sigma_0}{\Lambda}_h^{\sigma-\sigma_0}\partial_t{Q}\|_{\dot{H}^{\frac{1}{2}}(\Sigma_0)}^2+\|\dot{\Lambda}_h^{\sigma_0}{\Lambda}_h^{\sigma-\sigma_0}(\partial_h\partial_1v,\partial_1v ) \|_{L^2}\, \|\dot{\Lambda}_h^{\sigma_0}{\Lambda}_h^{\sigma-\sigma_0}(\partial_1q,\,\partial_hq)\|_{L^2}\\
 &\lesssim \|\dot{\Lambda}_h^{\sigma_0}{\Lambda}_h^{\sigma-\sigma_0}\partial_1^2v^{h}\|_{L^2(\Omega_f)}
\|\dot{\Lambda}_h^{\sigma_0}{\Lambda}_h^{\sigma-\sigma_0}(\partial_tv, \,\partial_{h} \nabla\,v,\,\nabla\,v)\|_{L^2} +\|\dot{\Lambda}_h^{\sigma_0}{\Lambda}_h^{\sigma-\sigma_0}{Q} \|_{\dot{H}^{\frac{1}{2}}(\Sigma_0)}^2\\
 &\,+\|\dot{\Lambda}_h^{\sigma_0}{\Lambda}_h^{\sigma-\sigma_0}(\partial_h\nabla\,v,\,\nabla\,v,\,\partial_tv, \partial_1q)\|_{L^2}^2
  +(E_{s_0}^{\frac{1}{2}}\dot{\mathcal{D}}_{\sigma+1}^{\frac{1}{2}}
+E_{\sigma+1}^{\frac{1}{2}}\dot{\mathcal{D}}_{s_0}^{\frac{1}{2}})\dot{\mathcal{D}}_{\sigma+1}^{\frac{1}{2}}.
\end{split}
\end{equation}
For the nonlinear remainder terms $\mathfrak{I}_j$ with $j=6, 7$, we have
\begin{equation*}
\begin{split}
  |\mathfrak{I}_6|&\lesssim ( \|\dot{\Lambda}_h^{\sigma_0}{\Lambda}_h^{\sigma-\sigma_0}(\mathcal{B}_{\beta+5, i}^{\alpha}\partial_{\alpha}v^{i})\|_{L^2(\Sigma_0)}+\|\dot{\Lambda}_h^{\sigma_0}{\Lambda}_h^{\sigma-\sigma_0}(\widetilde{\mathcal{B}}_{\beta}\partial_t{Q})\|_{L^2(\Sigma_0)}) \|\dot{\Lambda}_h^{\sigma_0}{\Lambda}_h^{\sigma-\sigma_0}\partial_{\beta}q\|_{L^2(\Sigma_0)}\\
  &\lesssim ( \|\dot{\Lambda}_h^{\sigma_0}{\Lambda}_h^{\sigma-\sigma_0}(\mathcal{B}_{\beta+5, i}^{\alpha}\partial_{\alpha}v^{i})\|_{H^1(\Omega)}+\|\dot{\Lambda}_h^{\sigma_0}{\Lambda}_h^{\sigma-\sigma_0}(\widetilde{\mathcal{B}}_{\beta}\partial_t{Q})\|_{L^2(\Sigma_0)}) \|\dot{\Lambda}_h^{\sigma_0}{\Lambda}_h^{\sigma-\sigma_0}\partial_{h}q\|_{H^1(\Omega)}\\
  &\lesssim (E_{s_0}^{\frac{1}{2}}\dot{\mathcal{D}}_{\sigma+1}^{\frac{1}{2}}
+E_{\sigma+1}^{\frac{1}{2}}\dot{\mathcal{D}}_{s_0}^{\frac{1}{2}})\dot{\mathcal{D}}_{\sigma+1}^{\frac{1}{2}},\\
   |\mathfrak{I}_7|&\lesssim \|\dot{\Lambda}_h^{\sigma_0}{\Lambda}_h^{\sigma-\sigma_0}\partial_1^2v^{\beta}\|_{L^2(\Omega_f)}
\|\dot{\Lambda}_h^{\sigma_0}{\Lambda}_h^{\sigma-\sigma_0}\mathfrak{g}_{\beta}\|_{L^2}\lesssim (E_{s_0}^{\frac{1}{2}}\dot{\mathcal{D}}_{\sigma+1}^{\frac{1}{2}}
+E_{\sigma+1}^{\frac{1}{2}}\dot{\mathcal{D}}_{s_0}^{\frac{1}{2}})\dot{\mathcal{D}}_{\sigma+1}^{\frac{1}{2}},
\end{split}
\end{equation*}
which implies
\begin{equation}\label{def-mfrak-J-7to9}
\begin{split}
&|\mathfrak{I}_6|+|\mathfrak{I}_7|\lesssim (E_{s_0}^{\frac{1}{2}}\dot{\mathcal{D}}_{\sigma+1}^{\frac{1}{2}}
+E_{\sigma+1}^{\frac{1}{2}}\dot{\mathcal{D}}_{s_0}^{\frac{1}{2}})\dot{\mathcal{D}}_{\sigma+1}^{\frac{1}{2}}.
\end{split}
\end{equation}
Hence, we get from \eqref{def-mfrak-J-1to6} and \eqref{def-mfrak-J-7to9} that
 \begin{equation*}
\begin{split}
&\frac{d}{dt}\|\dot{\Lambda}_h^{\sigma_0}{\Lambda}_h^{\sigma-\sigma_0}\partial_{h}\xi^1\|_{L^2(\Sigma_0)}^2
+2c_5\|\dot{\Lambda}_h^{\sigma_0}{\Lambda}_h^{\sigma-\sigma_0}\partial_1^2v^{h}\|_{L^2(\Omega_f)}^2\\
&
\lesssim \|\dot{\Lambda}_h^{\sigma_0}{\Lambda}_h^{\sigma-\sigma_0}\partial_1^2v^{h}\|_{L^2(\Omega_f)}
\|\dot{\Lambda}_h^{\sigma_0}{\Lambda}_h^{\sigma-\sigma_0}(\partial_tv, \,\partial_{h} \nabla\,v,\,\nabla\,v)\|_{L^2}  +\|\dot{\Lambda}_h^{\sigma_0}{\Lambda}_h^{\sigma-\sigma_0}{Q} \|_{\dot{H}^{\frac{1}{2}}(\Sigma_0)}^2\\
 &\quad+\|\dot{\Lambda}_h^{\sigma_0}{\Lambda}_h^{\sigma-\sigma_0}(\partial_h\nabla\,v,\,\nabla\,v,\,\partial_tv, \partial_1q)\|_{L^2}^2
  +(E_{s_0}^{\frac{1}{2}}\dot{\mathcal{D}}_{\sigma+1}^{\frac{1}{2}}
+E_{\sigma+1}^{\frac{1}{2}}\dot{\mathcal{D}}_{s_0}^{\frac{1}{2}})\dot{\mathcal{D}}_{\sigma+1}^{\frac{1}{2}},
\end{split}
\end{equation*}
which follows
 \begin{equation}\label{laplace-sigma0-14}
\begin{split}
&\frac{d}{dt}\|\dot{\Lambda}_h^{\sigma_0}{\Lambda}_h^{\sigma-\sigma_0}\partial_{h}\xi^1\|_{L^2(\Sigma_0)}^2
+c_5\|\dot{\Lambda}_h^{\sigma_0}{\Lambda}_h^{\sigma-\sigma_0}\partial_1^2v^{h}\|_{L^2(\Omega_f)}^2\\
&\leq C_5\bigg(\|\dot{\Lambda}_h^{\sigma_0}{\Lambda}_h^{\sigma-\sigma_0}{Q} \|_{\dot{H}^{\frac{1}{2}}(\Sigma_0)}^2\\
 &\quad+\|\dot{\Lambda}_h^{\sigma_0}{\Lambda}_h^{\sigma-\sigma_0}(\partial_h\nabla\,v,\,\nabla\,v,\,\partial_tv, \partial_1q)\|_{L^2}^2
  +(E_{s_0}^{\frac{1}{2}}\dot{\mathcal{D}}_{\sigma+1}^{\frac{1}{2}}
+E_{\sigma+1}^{\frac{1}{2}}\dot{\mathcal{D}}_{s_0}^{\frac{1}{2}})\dot{\mathcal{D}}_{\sigma+1}^{\frac{1}{2}}\bigg).
\end{split}
\end{equation}
Combining \eqref{laplace-s1-27aa} with \eqref{laplace-sigma0-14}, we obtain
 \begin{equation*}
\begin{split}
&\frac{d}{dt}\|\dot{\Lambda}_h^{\sigma_0}{\Lambda}_h^{\sigma-\sigma_0}\partial_{h}\xi^1\|_{L^2(\Sigma_0)}^2
+2c_6\|\dot{\Lambda}_h^{\sigma_0}{\Lambda}_h^{\sigma-\sigma_0}(\partial_1^2\,v,\,\nabla\,q)\|_{L^2}^2\\
&\leq C_5\bigg(\|\dot{\Lambda}_h^{\sigma_0}{\Lambda}_h^{\sigma+1-\sigma_0}\nabla\,v\|_{L^2}^2
+\|\dot{\Lambda}_h^{\sigma_0}{\Lambda}_h^{\sigma-\sigma_0} \partial_tv\|_{L^2}^2+\|\dot{\Lambda}_h^{\sigma_0}{\Lambda}_h^{\sigma-\sigma_0}\partial_1q\|_{L^2}^2\\
 &\qquad\quad+\|\dot{\Lambda}_h^{\sigma_0}{\Lambda}_h^{\sigma-\sigma_0}{Q} \|_{\dot{H}^{\frac{1}{2}}(\Sigma_0)}^2
  +(E_{s_0}^{\frac{1}{2}}\dot{\mathcal{D}}_{\sigma+1}^{\frac{1}{2}}
+E_{\sigma+1}^{\frac{1}{2}}\dot{\mathcal{D}}_{s_0}^{\frac{1}{2}})\dot{\mathcal{D}}_{\sigma+1}^{\frac{1}{2}}
+E_{\sigma+1}\dot{\mathcal{D}}_{s_0}\bigg),
     \end{split}
\end{equation*}
which along with the second inequality in \eqref{laplace-sigma0-2} follows \eqref{laplace-sigma0-0}.
The proof of the lemma is accomplished.
\end{proof}

\subsection{Estimates of $\|\dot{\Lambda}_h^{\sigma_0+1}\xi^1\|_{H^\frac{1}{2}(\Sigma_0)}$ and $\|\dot{\Lambda}_h^{s-1}\xi^1\|_{H^\frac{1}{2}(\Sigma_0)}$}

With Lemmas \ref{lem-laplace-sigma-1} in hand, we may prove
\begin{lem}\label{lem-laplace-xi-decay-1}
Under the assumption of Lemma \ref{lem-pseudo-energy-tv-1}, if $E_{s_0}(t) \leq 1$ for all $t\in [0, T]$, then there holds that $\forall\,t\in [0, T]$
\begin{equation}\label{laplace-xi-decay-0}
\begin{split}
&\frac{d}{dt}\|\dot{\Lambda}_h^{\sigma_0}{\Lambda}_h^{\sigma-\sigma_0}\partial_h\xi^1\|_{L^2(\Sigma_0)}^2
+c_7(\|\dot{\Lambda}_h^{\sigma_0}{\Lambda}_h^{\sigma-\sigma_0}(\nabla^2\,v,\,\nabla\,q)\|_{L^2}^2
+\|\dot{\Lambda}_h^{\sigma_0+1}\Lambda_h^{\sigma-1-\sigma_0}\xi^1\|_{H^\frac{1}{2}(\Sigma_0)}^2 )\\
&\leq C_7\bigg(\|\dot{\Lambda}_h^{\sigma_0}{\Lambda}_h^{\sigma+1-\sigma_0}\nabla\,v\|_{L^2}^2
+\|\dot{\Lambda}_h^{\sigma_0}{\Lambda}_h^{\sigma-\sigma_0} \partial_tv\|_{L^2}^2+\|\dot{\Lambda}_h^{\sigma_0}{\Lambda}_h^{\sigma-\sigma_0}\partial_1q\|_{L^2}^2\\
 &\qquad\quad+\|\dot{\Lambda}_h^{\sigma_0}{\Lambda}_h^{\sigma-\sigma_0}{Q} \|_{{H}^{\frac{1}{2}}(\Sigma_0)}^2
  +(E_{s_0}^{\frac{1}{2}}\dot{\mathcal{D}}_{\sigma+1}^{\frac{1}{2}}
+E_{\sigma+1}^{\frac{1}{2}}\dot{\mathcal{D}}_{s_0}^{\frac{1}{2}})\dot{\mathcal{D}}_{\sigma+1}^{\frac{1}{2}}
+E_{\sigma+1}\dot{\mathcal{D}}_{s_0}\bigg).
     \end{split}
\end{equation}

\end{lem}
\begin{proof}
Due to $Q=q-g\,\xi^1$, we know that, for $\sigma \in \mathbb{R}$,
\begin{equation*}
  \begin{split}
\|\dot{\Lambda}_h^{\sigma}\xi^1\|_{H^\frac{1}{2}(\Sigma_0)} \lesssim & \|\dot{\Lambda}_h^{\sigma}q\|_{H^\frac{1}{2}(\Sigma_0)}+\|\dot{\Lambda}_h^{\sigma}{Q}\|_{H^\frac{1}{2}(\Sigma_0)},
\end{split}
\end{equation*}
which follows
\begin{equation}\label{est-interface-xi-3}
  \begin{split}
\|\dot{\Lambda}_h^{\sigma_0+1}\Lambda_h^{\sigma-1-\sigma_0}\xi^1\|_{H^\frac{1}{2}(\Sigma_0)} \lesssim & \|\dot{\Lambda}_h^{\sigma_0+1}\Lambda_h^{\sigma-1-\sigma_0}q\|_{H^1(\Omega)}
+\|\dot{\Lambda}_h^{\sigma_0+1}\Lambda_h^{\sigma-1-\sigma_0}{Q} \|_{H^\frac{1}{2}(\Sigma_0)}.
\end{split}
\end{equation}
Combining \eqref{laplace-sigma0-0} with \eqref{est-interface-xi-3} yields \eqref{laplace-xi-decay-0}, which completes the proof of Lemma \ref{lem-laplace-xi-decay-1}.
\end{proof}

\renewcommand{\theequation}{\thesection.\arabic{equation}}
\setcounter{equation}{0}
\section{Total energy estimates}\label{sect-total}
Define
\begin{equation*}
\begin{split}
\widehat{\dot{\mathcal{E}}}_{\sigma, \epsilon}\eqdefa\, &\widehat{\dot{\mathcal{E}}}_{\sigma, \text{tan},\epsilon}+\epsilon^3 \|\dot{\Lambda}_h^{\sigma_0}{\Lambda}_h^{\sigma-1-\sigma_0}\partial_1q\|_{L^2}^2
+\epsilon^4\|\dot{\Lambda}_h^{\sigma_0}{\Lambda}_h^{\sigma-\sigma_0-1}\partial_h\xi^1\|_{L^2(\Sigma_0)}^2,\\
\widehat{\dot{\mathcal{D}}}_{\sigma, \epsilon}\eqdefa \,& \widehat{\dot{\mathcal{D}}}_{\sigma, \text{tan},\epsilon}+\frac{c_4}{2}\epsilon^3\|\dot{\Lambda}_h^{\sigma_0}{\Lambda}_h^{\sigma-1-\sigma_0}
(\partial_1\,q,\,\partial_t\partial_1\,q,\,\partial_1^2v^1)\|_{L^2}^2 \\
& +\frac{c_7}{2}\epsilon^4(\|\dot{\Lambda}_h^{\sigma_0}{\Lambda}_h^{\sigma-\sigma_0-1}(\nabla^2\,v,\,\nabla\,q)\|_{L^2}^2
+\|\dot{\Lambda}_h^{\sigma_0+1}\Lambda_h^{\sigma-2-\sigma_0}\xi^1\|_{H^\frac{1}{2}(\Sigma_0)}^2 ),
\end{split}
\end{equation*}
and
\begin{equation*}
\begin{split}
&\widehat{\mathcal{E}}_{\sigma, \epsilon}\eqdefa\,\widehat{\mathcal{E}}_{\sigma, \tan, \epsilon}+\widehat{\dot{\mathcal{E}}}_{\sigma, \epsilon},\quad  \widehat{\mathcal{D}}_{\sigma, \epsilon}\eqdefa\,\widehat{\mathcal{D}}_{\sigma, \tan, \epsilon}+\widehat{\dot{\mathcal{D}}}_{\sigma, \epsilon}.
\end{split}
\end{equation*}
Thanks to  \eqref{ineq-tan-decay-energy-1}, \eqref{comp-eqns-v11-s-1}, and \eqref{laplace-xi-decay-0}, we get
\begin{equation*}
  \begin{split}
&\frac{d}{dt}\widehat{\dot{\mathcal{E}}}_{\sigma,\epsilon}(t)+2\,\widehat{\dot{\mathcal{D}}}_{\sigma, \epsilon}(t)\leq  (C_2\,\epsilon+C_3 \,\epsilon^2+C_4 \,\epsilon^3)\,\|\dot{\Lambda}_h^{\sigma_0} {\Lambda}_h^{\sigma-\sigma_0} \nabla\,v \|_{L^2(\Omega)}^2 \\
&\qquad+(C_3 \,\epsilon^2+C_7\,\epsilon^4)\,\|\dot{\Lambda}_h^{\sigma_0}\Lambda_h^{\sigma-\frac{1}{2}-\sigma_0} {Q} \|_{L^2(\Sigma_0)}^2+(C_4\,\epsilon^3+C_7\,\epsilon^4)\,\|\dot{\Lambda}_h^{\sigma_0}{\Lambda}_h^{\sigma-1-\sigma_0} \partial_tv\|_{L^2}^2\\
   &\qquad
   +C_7\,\epsilon^4\,\|\dot{\Lambda}_h^{\sigma_0}{\Lambda}_h^{\sigma-1-\sigma_0}\partial_1q\|_{L^2}^2+(C_1+C_3 \,\epsilon^2+C_7\,\epsilon^4)\,(E_{\sigma}^{\frac{1}{2}}\dot{\mathcal{D}}_{s_0}^{\frac{1}{2}}
+E_{s_0}^{\frac{1}{2}} \dot{\mathcal{D}}_{\sigma}^{\frac{1}{2}})\dot{\mathcal{D}}_{\sigma}^{\frac{1}{2}}\\
&\qquad+ (C_2\,\epsilon+C_4 \,\epsilon^3+C_7\,\epsilon^4)\,(E_{s_0}\dot{\mathcal{D}}_{\sigma}+E_{s}\dot{\mathcal{D}}_{s_0}).
 \end{split}
\end{equation*}
Combining Lemma \ref{lem-laplace-xi-decay-1} with Lemma \ref{lem-decay-total-1}, and taking the positive $\epsilon$ so small that
\begin{equation*}
\epsilon \leq \min\{1, \frac{\min\{c_1, c_2, c_3, c_4\}}{4(C_2+C_3+C_4+C_7)}\},
 \end{equation*}
 we immediately get the following energy estimate.

\begin{lem}\label{lem-tangrad-decay-total-1}
Let $\sigma>2$, under the assumption of Lemma \ref{lem-tan-pseudo-energy-1}, if $E_{s_0}(t) \leq 1$ for all $t\in [0, T]$, then there holds that $\forall\,t\in [0, T]$
\begin{equation*}
\begin{split}
&\frac{d}{dt}\widehat{\dot{\mathcal{E}}}_{\sigma, \epsilon}+\widehat{\dot{\mathcal{D}}}_{\sigma, \epsilon}\leq C_8\,((E_{\sigma}^{\frac{1}{2}}\dot{\mathcal{D}}_{s_0}^{\frac{1}{2}}
+E_{s_0}^{\frac{1}{2}} \dot{\mathcal{D}}_{\sigma}^{\frac{1}{2}})\dot{\mathcal{D}}_{\sigma}^{\frac{1}{2}}+E_{s}\dot{\mathcal{D}}_{s_0}).
\end{split}
\end{equation*}
\end{lem}

The following comparison lemma discovers the equivalence of two types of the energy.

\begin{lem}\label{lem-comparision-s1-1}
Let $N\geq 3$, under the assumption of Lemma \ref{lem-tan-pseudo-energy-1}, if $(\lambda,\,\sigma_0) \in (0, 1)$ satisfies $1-\lambda< \sigma_0\leq 1-\frac{1}{2}\lambda$, and $E_{s}(t) \leq 1$ for all $t\in [0, T]$, then there exists a genius positive constant $C_0$  such that, for $\sigma=s$,
\begin{equation}\label{def-s1-sigma-energy-tv-0}
\begin{split}
 \mathring{\mathcal{E}}(\dot{\Lambda}_h^{\sigma_0} {\Lambda}_h^{\sigma-1-\sigma_0}\nabla\, v)\leq  \frac{3}{2}( \frac{\varepsilon}{2} \|\dot{\Lambda}_h^{\sigma_0} {\Lambda}_h^{\sigma-1-\sigma_0} \nabla\,v^{h}\|_{L^2}^2&+\frac{(\delta+\frac{4}{3}\varepsilon)}{2} \|\dot{\Lambda}_h^{\sigma_0} {\Lambda}_h^{\sigma-1-\sigma_0}\nabla\,v^1\|_{L^2}^2)\\
  &+ C_0\,\mathring{E}_1(\sigma, v, Q, q, \xi^1),\\
\mathring{\mathcal{E}}(\dot{\Lambda}_h^{\sigma_0} {\Lambda}_h^{\sigma-1-\sigma_0}\nabla\, v)\geq  \frac{3}{2}( \frac{\varepsilon}{2} \|\dot{\Lambda}_h^{\sigma_0} {\Lambda}_h^{\sigma-1-\sigma_0} \nabla\,v^{h}\|_{L^2}^2&+\frac{(\delta+\frac{4}{3}\varepsilon)}{2} \|\dot{\Lambda}_h^{\sigma_0} {\Lambda}_h^{\sigma-1-\sigma_0}\nabla\,v^1\|_{L^2}^2)\\
 & - C_0\,\mathring{E}_1(\sigma, v, Q, q, \xi^1),
         \end{split}
\end{equation}
\begin{equation}\label{def-lambda-energy-tv-0}
\begin{split}
   &   \mathring{\mathcal{E}}(\dot{\Lambda}_h^{-\lambda}\nabla\, v)\leq  \frac{3}{2} (\frac{\varepsilon}{2} \|\dot{\Lambda}_h^{-\lambda} \nabla\,v^{h}\|_{L^2}^2+\frac{(\delta+\frac{4}{3}\varepsilon)}{2} \|\dot{\Lambda}_h^{-\lambda}\nabla\,v^1\|_{L^2}^2)+ C_0\mathring{E}_2(-\lambda, v, Q, q, \xi^1),\\
   &   \mathring{\mathcal{E}}(\dot{\Lambda}_h^{-\lambda}\nabla\, v)\leq  \frac{2}{3} (\frac{\varepsilon}{2} \|\dot{\Lambda}_h^{-\lambda} \nabla\,v^{h}\|_{L^2}^2+\frac{(\delta+\frac{4}{3}\varepsilon)}{2} \|\dot{\Lambda}_h^{-\lambda}\nabla\,v^1\|_{L^2}^2)- C_0\mathring{E}_2(-\lambda, v, Q, q, \xi^1),
        \end{split}
\end{equation}
where
\begin{equation*}
\begin{split}
 &\mathring{E}_1(\sigma, v, Q, q, \xi^1)\eqdefa\,\|\dot{\Lambda}_h^{\sigma_0} \Lambda_h^{\sigma-\sigma_0} v\|_{L^2}^2+\|\dot{\Lambda}_h^{\sigma_0} {\Lambda}_h^{\sigma-\frac{3}{2}-\sigma_0}{Q}\|_{L^2(\Sigma_0)}^2+\|\dot{\Lambda}_h^{\sigma_0} {\Lambda}_h^{s_0-\frac{1}{2}-\sigma_0}{Q}\|_{L^2(\Sigma_0)}^2\\
 &\qquad\qquad\qquad\qquad
+ \|\dot{\Lambda}_h^{\sigma_0} {\Lambda}_h^{\sigma-\frac{3}{2}-\sigma_0}\xi^1\|_{L^2(\Sigma_0)}^2+ \|\dot{\Lambda}_h^{\sigma_0} {\Lambda}_h^{\sigma-1-\sigma_0}q\|_{L^2}^2,\\
 &\mathring{E}_2(-\lambda, v, Q, q, \xi^1)\eqdefa\,\|(\dot{\Lambda}_h^{1-\lambda} v, \,\dot{\Lambda}_h^{\sigma_0} \Lambda_h v,\,\dot{\Lambda}_h^{-\lambda}q)\|_{L^2}^2
+ \|(\dot{\Lambda}_h^{-\lambda} \Lambda_h^{\frac{1}{2}+\lambda+\sigma_0} {Q},\,\dot{\Lambda}_h^{-\lambda}\xi^1)\|_{L^2(\Sigma_0)}^2.
        \end{split}
\end{equation*}
If, in addition, $E_s \leq \frac{\epsilon}{18C_0}$, then the inequality \eqref{def-s1-sigma-energy-tv-0} with $\sigma=s+\ell_0$ holds true.
\end{lem}
\begin{proof}
In order to prove \eqref{equiv-decay-total-N+1-1}, we need only to estimate the nonlinear energy $\mathring{\mathcal{E}}(\dot{\Lambda}_h^{\sigma}\nabla\, v)$ from upper and lower bounds.
In fact, consider the definition of $\mathring{\mathcal{E}}(\dot{\Lambda}_h^{\sigma_0} {\Lambda}_h^{s-1-\sigma_0}\nabla\, v)$ in \eqref{def-pseudo-gene-energy-tv-0}, in which
\begin{equation*}\label{def-pseudo-gene-energy-tv-0}
\begin{split}
\|\mathcal{P}(\partial_h) \mathcal{G}^{\beta}\|_{L^2}^2&-\|\mathcal{P}(\partial_h)(-\partial_{\beta}v^1
+\widetilde{\mathcal{G}}^{\beta})\|_{L^2}^2\\
&\leq \frac{4}{3}\|\mathcal{P}(\partial_h) \partial_1v^{\beta}\|_{L^2}^2+C(\| \mathcal{P}(\partial_h)\partial_{\beta}v^1\|_{L^2}^2+\|\mathcal{P}(\partial_h)\widetilde{\mathcal{G}}^{\beta}\|_{L^2}^2),
        \end{split}
\end{equation*}
\begin{equation*}
\begin{split}
 \|\mathcal{P}(\partial_h)\mathcal{G}^{\beta}\|_{L^2}^2&-\|\mathcal{P}(\partial_h)(-\partial_{\beta}v^1
+\widetilde{\mathcal{G}}^{\beta})\|_{L^2}^2\\
 &\geq \frac{2}{3}\|\mathcal{P}(\partial_h)\partial_1v^{\beta}\|_{L^2}^2-C(\|\mathcal{P}(\partial_h)\partial_{\beta}v^1\|_{L^2}^2+\|\mathcal{P}(\partial_h)\widetilde{\mathcal{G}}^{\beta}\|_{L^2}^2),
        \end{split}
\end{equation*}
\begin{equation*}
\begin{split}
&\|\mathcal{P}(\partial_h)\mathcal{G}^1\|_{L^2}^2
-\|\mathcal{P}(\partial_h)(-\frac{\delta-\frac{2}{3}\varepsilon}{\delta+\frac{4}{3}\varepsilon} \nabla_h\cdot v^h+\frac{\bar{\rho}(0)}{\delta+\frac{4}{3}\varepsilon}\, \mathcal{H}({Q})+\widetilde{\mathcal{G}}^1)\|_{L^2}^2\\
&\qquad\leq \frac{4}{3}\|\mathcal{P}(\partial_h)\partial_1v^1\|_{L^2}^2
+C(\|\mathcal{P}(\partial_h)\partial_hv^h\|_{L^2}^2+
\|\mathcal{P}(\partial_h)\mathcal{H}({Q})\|_{L^2}^2+\|\mathcal{P}(\partial_h)\widetilde{\mathcal{G}}^1\|_{L^2}^2),
        \end{split}
\end{equation*}

\begin{equation*}
\begin{split}
&\|\mathcal{P}(\partial_h)\mathcal{G}^1\|_{L^2}^2
-\|\mathcal{P}(\partial_h)
(-\frac{\delta-\frac{2}{3}\varepsilon}{\delta+\frac{4}{3}\varepsilon} \nabla_h\cdot v^h+\frac{\bar{\rho}(0)}{\delta+\frac{4}{3}\varepsilon}\, \mathcal{H}({Q})+\widetilde{\mathcal{G}}^1)\|_{L^2}^2\\
&\qquad\geq \frac{2}{3}\|\mathcal{P}(\partial_h)\partial_1v^1\|_{L^2}^2
-C(\|\mathcal{P}(\partial_h)\partial_hv^h\|_{L^2}^2+
\|\mathcal{P}(\partial_h)\mathcal{H}({Q})\|_{L^2}^2+\|\mathcal{P}(\partial_h)\widetilde{\mathcal{G}}^1\|_{L^2}^2),
        \end{split}
\end{equation*}
where we have taken $\mathcal{P}(\partial_h)=\dot{\Lambda}_h^{\sigma_0} {\Lambda}_h^{s-1-\sigma_0}$, and
\begin{equation*}
\begin{split}
|\int_{\Omega}\dot{\Lambda}_h^{\sigma_0} {\Lambda}_h^{s-\sigma_0}v^1\dot{\Lambda}_h^{\sigma_0} {\Lambda}_h^{s-2-\sigma_0}\partial_1  \mathcal{H}({Q}) \,dx|&\lesssim \|\dot{\Lambda}_h^{\sigma_0} {\Lambda}_h^{s-\sigma_0}v^1\|_{L^2}\|\dot{\Lambda}_h^{\sigma_0} {\Lambda}_h^{s-2-\sigma_0}\partial_1  \mathcal{H}({Q})\|_{L^2(\Omega)} \\
 &\lesssim \|\dot{\Lambda}_h^{\sigma_0} {\Lambda}_h^{s-\sigma_0}v^1\|_{L^2}^2+\|\dot{\Lambda}_h^{\sigma_0} {\Lambda}_h^{s-\frac{3}{2}-\sigma_0} {Q}\|_{L^2(\Sigma_0)}^2,
         \end{split}
\end{equation*}
\begin{equation*}
\begin{split}
   &|\int_{\Sigma_0}\dot{\Lambda}_h^{\sigma_0} {\Lambda}_h^{s-\frac{1}{2}-\sigma_0} v^1\, \dot{\Lambda}_h^{\sigma_0} {\Lambda}_h^{s-\frac{3}{2}-\sigma_0}q\,dS_0|\\
   &\leq \|\dot{\Lambda}_h^{\sigma_0} {\Lambda}_h^{s-\frac{1}{2}-\sigma_0} v^1\|_{L^2(\Sigma_0)}\, (\|\dot{\Lambda}_h^{\sigma_0} {\Lambda}_h^{s-\frac{3}{2}-\sigma_0}Q\|_{L^2(\Sigma_0)}+g\|\dot{\Lambda}_h^{\sigma_0} {\Lambda}_h^{s-\frac{3}{2}-\sigma_0}\xi^1\|_{L^2(\Sigma_0)})\\
   &\lesssim \|\dot{\Lambda}_h^{\sigma_0} {\Lambda}_h^{s-1-\sigma_0}\nabla v^1\|_{L^2(\Omega)}\, (\|\dot{\Lambda}_h^{\sigma_0} {\Lambda}_h^{s-\frac{3}{2}-\sigma_0}Q\|_{L^2(\Sigma_0)}+\|\dot{\Lambda}_h^{\sigma_0} {\Lambda}_h^{s-\frac{3}{2}-\sigma_0}\xi^1\|_{L^2(\Sigma_0)}),
           \end{split}
\end{equation*}
\begin{equation*}
\begin{split}
   &|\int_{\Omega} \dot{\Lambda}_h^{\sigma_0} {\Lambda}_h^{s-2-\sigma_0}\nabla\cdot(\bar{\rho}(x_1)v ) \dot{\Lambda}_h^{\sigma_0} {\Lambda}_h^{s-\sigma_0}q\,dx|\lesssim \|\dot{\Lambda}_h^{\sigma_0} {\Lambda}_h^{s-1-\sigma_0}\nabla v\|_{L^2} \|\dot{\Lambda}_h^{\sigma_0} {\Lambda}_h^{s-1-\sigma_0}q\|_{L^2}.
              \end{split}
\end{equation*}
Thanks to \eqref{est-GQ2-norm-1}, we have
\begin{equation*}
\begin{split}
   \|\dot{\Lambda}_h^{\sigma_0} {\Lambda}_h^{s-1-\sigma_0}(\widetilde{\mathcal{G}}^{1}, \widetilde{\mathcal{G}}^{2}, \widetilde{\mathcal{G}}^{3})\|_{L^2}^2\lesssim &E_{s}  (\|\dot{\Lambda}_h^{\sigma_0} \Lambda_h^{s_0-\sigma_0} v\|_{L^2}^2+\|\dot{\Lambda}_h^{\sigma_0} \Lambda_h^{s_0-\frac{1}{2}-\sigma_0} {Q}\|_{L^2(\Sigma_0)}^2)\\
   &+ E_{s_0}(\|\dot{\Lambda}_h^{\sigma_0} \Lambda_h^{s-\frac{3}{2}-\sigma_0}  {Q}\|_{L^2(\Sigma_0)}^2+\|\dot{\Lambda}_h^{\sigma_0} \Lambda_h^{s-\sigma_0} v\|_{L^2}^2).
              \end{split}
\end{equation*}
\begin{equation*}
\begin{split}
      &\|\dot{\Lambda}_h^{\sigma_0} {\Lambda}_h^{s+\ell_0-1-\sigma_0}(\widetilde{\mathcal{G}}^{1}, \widetilde{\mathcal{G}}^{2}, \widetilde{\mathcal{G}}^{3})\|_{L^2}^2\\
      &\lesssim E_{s}\,(\|\dot{\Lambda}_h^{\sigma_0} \Lambda_h^{s_0-1+\ell_0-\sigma_0} \nabla\,v\|_{L^2}^2+\|\dot{\Lambda}_h^{\sigma_0} \Lambda_h^{s_0-\frac{1}{2}+\ell_0-\sigma_0} {Q}\|_{L^2(\Sigma_0)}^2)\\
      &\qquad+E_{s_0}(\|\dot{\Lambda}_h^{\sigma_0} \Lambda_h^{s+\ell_0-\frac{3}{2}-\sigma_0}  {Q}\|_{L^2(\Sigma_0)}^2+\|\dot{\Lambda}_h^{\sigma_0} {\Lambda}_h^{s-1-\sigma_0+\ell_0}v\|_{L^2}^2).
        \end{split}
\end{equation*}
Hence, owing to $E_s \leq 1$, we obtain \eqref{def-s1-sigma-energy-tv-0}.

Similarly, we may readily obtain \eqref{def-lambda-energy-tv-0}.

For $\mathring{\mathcal{E}}(\dot{\Lambda}_h^{\sigma_0} {\Lambda}_h^{s+\ell_0-1-\sigma_0}\nabla\, v)$, in the same manner, thanks to \eqref{est-GQ2-s-norm-1}, we infer
\begin{equation*}
\begin{split}
 &\mathring{\mathcal{E}}(\dot{\Lambda}_h^{\sigma_0} {\Lambda}_h^{s+\ell_0-1-\sigma_0}\nabla\, v)\leq  \frac{3}{2}( \frac{\varepsilon}{2} \|\dot{\Lambda}_h^{\sigma_0} {\Lambda}_h^{s+\ell_0-1-\sigma_0} \nabla\,v^{h}\|_{L^2}^2+\frac{(\delta+\frac{4}{3}\varepsilon)}{2} \|\dot{\Lambda}_h^{\sigma_0} {\Lambda}_h^{s+\ell_0-1-\sigma_0}\nabla\,v^1\|_{L^2}^2)\\
 & + C_0\,\mathring{E}_1(s+\ell_0, v, Q, q, \xi^1)+C_0\bigg(E_{s}\,(\|\dot{\Lambda}_h^{\sigma_0} \Lambda_h^{s+\ell_0-1-\sigma_0} \nabla\,v\|_{L^2}^2+\|\dot{\Lambda}_h^{\sigma_0} \Lambda_h^{s_0-\frac{1}{2}+\ell_0-\sigma_0} {Q}\|_{L^2(\Sigma_0)}^2)\\
      &\qquad\qquad\qquad\qquad\qquad\qquad+E_{s_0}(\|\dot{\Lambda}_h^{\sigma_0} \Lambda_h^{s+\ell_0-\frac{3}{2}-\sigma_0}  {Q}\|_{L^2(\Sigma_0)}^2+\|\dot{\Lambda}_h^{\sigma_0} {\Lambda}_h^{s+\ell_0-1-\sigma_0}v\|_{L^2}^2)\bigg),\\
         \end{split}
\end{equation*}

\begin{equation*}
\begin{split}
 &\mathring{\mathcal{E}}(\dot{\Lambda}_h^{\sigma_0} {\Lambda}_h^{s+\ell_0-1-\sigma_0}\nabla\, v)\geq  \frac{2}{3}( \frac{\varepsilon}{2} \|\dot{\Lambda}_h^{\sigma_0} {\Lambda}_h^{s+\ell_0-1-\sigma_0} \nabla\,v^{h}\|_{L^2}^2+\frac{(\delta+\frac{4}{3}\varepsilon)}{2} \|\dot{\Lambda}_h^{\sigma_0} {\Lambda}_h^{s+\ell_0-1-\sigma_0}\nabla\,v^1\|_{L^2}^2)\\
 & - C_0\,\mathring{E}_1(s+\ell_0, v, Q, q, \xi^1)-C_0\bigg(E_{s}\,(\|\dot{\Lambda}_h^{\sigma_0} \Lambda_h^{s+\ell_0-1-\sigma_0} \nabla\,v\|_{L^2}^2+\|\dot{\Lambda}_h^{\sigma_0} \Lambda_h^{s_0-\frac{1}{2}+\ell_0-\sigma_0} {Q}\|_{L^2(\Sigma_0)}^2)\\
      &\qquad\qquad\qquad\qquad\qquad+E_{s_0}(\|\dot{\Lambda}_h^{\sigma_0} \Lambda_h^{s+\ell_0-\frac{3}{2}-\sigma_0}  {Q}\|_{L^2(\Sigma_0)}^2+\|\dot{\Lambda}_h^{\sigma_0} {\Lambda}_h^{s+\ell_0-1-\sigma_0}v\|_{L^2}^2)\bigg).
         \end{split}
\end{equation*}
From this, and $E_s \leq \frac{\epsilon}{18C_0}$, we obtain that the inequality \eqref{def-s1-sigma-energy-tv-0} with $\sigma=s+\ell_0$ holds true. This ends the proof of Lemma \ref{lem-comparision-s1-1}.
\end{proof}

\begin{lem}[Comparison lemma]\label{lem-equiv-decay-total-N+1-1}
Let $N\geq 3$, under the assumption of Lemma \ref{lem-tan-pseudo-energy-1}, if $(\lambda,\,\sigma_0) \in (0, 1)$ satisfies $1-\lambda< \sigma_0\leq 1-\frac{1}{2}\lambda$, and $E_{s}(t) \leq \frac{\varepsilon}{18C_0}$ for all $t\in [0, T]$, then there exists a genius positive constant $\epsilon_0$  such that
\begin{equation}\label{equiv-decay-total-N+1-1}
\begin{split}
& \dot{\mathcal{E}}_{s}\thicksim \widehat{\dot{\mathcal{E}}}_{s, \epsilon_0},\quad   \dot{\mathcal{D}}_{s} \thicksim \widehat{\dot{\mathcal{D}}}_{s, \epsilon_0},\quad \mathcal{E}_{s+\ell_0}\thicksim \widehat{{\mathcal{E}}}_{s+\ell_0, \epsilon_0},\quad  \mathcal{D}_{s+\ell_0}\thicksim \widehat{{\mathcal{D}}}_{s+\ell_0, \epsilon_0}.
\end{split}
\end{equation}
\end{lem}
\begin{proof}
Since $E_{s}(t) \leq \frac{\varepsilon}{18C_0}$ for all $t\in [0, T]$, by using Lemma \ref{lem-comparision-s1-1}, we obtain, from the definitions of $\dot{\mathcal{E}}_{s}$, $\widehat{\dot{\mathcal{E}}}_{s, \epsilon_0}$, $\dot{\mathcal{D}}_{s}$, $\widehat{\dot{\mathcal{D}}}_{s, \epsilon_0}$, $\mathcal{E}_{s+\ell_0}\thicksim \widehat{{\mathcal{E}}}_{s+\ell_0, \epsilon_0}$, $\mathcal{D}_{s+\ell_0}$, and $\widehat{{\mathcal{D}}}_{s+\ell_0, \epsilon_0}$, that \eqref{equiv-decay-total-N+1-1} holds if we take the positive constant $\epsilon_0$ small enough, which completes the proof of Lemma \ref{lem-equiv-decay-total-N+1-1}.
\end{proof}

With Lemma \ref{lem-equiv-decay-total-N+1-1} in hand, for simplicity, we denote $\widehat{\dot{\mathcal{E}}}_{s, \epsilon_0}$, $\widehat{\dot{\mathcal{D}}}_{s, \epsilon_0}$,
$\widehat{\mathcal{E}}_{s+\ell_0, \epsilon_0}$, $\widehat{\mathcal{D}}_{s+\ell_0, \epsilon_0}$ by $\widehat{\dot{\mathcal{E}}}_{s}$, $\widehat{\dot{\mathcal{D}}}_{s}$,
$\widehat{\mathcal{E}}_{s+\ell_0}$, $\widehat{\mathcal{D}}_{s+\ell_0}$ respectively. And we also denote the positive constant $\mathfrak{C}_0\geq 1$ such that
\begin{equation}\label{equiv-total-cont-1}
\begin{split}
&\mathfrak{C}_0^{-1}\widehat{\dot{\mathcal{E}}}_{s}\leq  \dot{\mathcal{E}}_s\leq \mathfrak{C}_0 \widehat{\dot{\mathcal{E}}}_{s},\quad   \mathfrak{C}_0^{-1} \widehat{\dot{\mathcal{D}}}_{s}\leq\dot{\mathcal{D}}_{s} \leq \mathfrak{C}_0 \widehat{\dot{\mathcal{D}}}_{s},\\
&\mathfrak{C}_0^{-1}\widehat{\mathcal{E}}_{s+\ell_0}\leq \mathcal{E}_{s+\ell_0}\leq \mathfrak{C}_0\widehat{\mathcal{E}}_{s+\ell_0},\quad  \mathfrak{C}_0^{-1} \widehat{\mathcal{D}}_{s+\ell_0}\leq \mathcal{D}_{s+\ell_0}\leq \mathfrak{C}_0 \widehat{\mathcal{D}}_{s+\ell_0}.
\end{split}
\end{equation}

Therefore, we restate Lemmas \ref{lem-tangrad-decay-total-1} and \ref{lem-tan-energy-total-1} to get the total energy estimates.
\begin{lem}\label{lem-restate-decay-total-N+1-1}
Under the assumption of Lemma \ref{lem-equiv-decay-total-N+1-1}, there hold that $\forall\,t\in [0, T]$
\begin{equation}\label{restate-decay-total-N+1-1}
\begin{split}
&\frac{d}{dt}\widehat{\dot{\mathcal{E}}}_{s}+2\mathfrak{c}_1\dot{\mathcal{D}}_{s}\leq \mathfrak{C}_1 E_s^{\frac{1}{2}} \,\dot{\mathcal{D}}_s,\\
&\frac{d}{dt}\widehat{\dot{\mathcal{E}}}_{s+\ell_0}+2\mathfrak{c}_1\dot{\mathcal{D}}_{s+\ell_0}\leq \mathfrak{C}_1 (E_{s_0}^{\frac{1}{2}} \,\dot{\mathcal{D}}_{s+\ell_0}+E_{s+\ell_0}^{\frac{1}{2}} \,\dot{\mathcal{D}}_{s_0}^{\frac{1}{2}}\dot{\mathcal{D}}_{s+\ell_0}^{\frac{1}{2}}),
\end{split}
\end{equation}
and
\begin{equation}\label{restate-bdddecay-total-N+1-1}
\begin{split}
&\frac{d}{dt}\widehat{\mathcal{E}}_{s+\ell_0} +2\mathfrak{c}_1\mathcal{D}_{s+\ell_0}\leq \mathfrak{C}_1 \bigg( E_{s_0}^{\frac{1}{2}} \dot{\mathcal{D}}_{s+\ell_0}+E_{s+\ell_0}\,\dot{\mathcal{D}}_{s_0}+ E_{s_0} (\dot{\mathcal{D}}_{s_0}^{\frac{1}{2}}+\dot{\mathcal{E}}_{s_0})\bigg).
\end{split}
\end{equation}
\end{lem}

\renewcommand{\theequation}{\thesection.\arabic{equation}}
\setcounter{equation}{0}

\section{Global well-posedness: Proof of Theorem \ref{thm-main}}\label{sect-proof-mainthm}

With the two {\it a priori} energy estimates \eqref{restate-decay-total-N+1-1} and \eqref{restate-bdddecay-total-N+1-1} in hand, by using the similar argument in \cite{Gui-2020}, we will show that the low order energy $E_{s}$ is uniformly bounded, while the high order energy $E_{s+\ell_0}$ grows algebraicly as the time $t$ goes to infinity by an inductive argument. Hence, we may find the decay estimate of the low order energy $\dot{\mathcal{E}}_s$.

For this, we first recall the following decay estimate.

\begin{lem}[Decay estimate, \cite{Guo-Tice-2, Gui-2020}]\label{lem-diff-ineq-1}
If the non-negative function $f$ satisfies the differential inequality
\begin{equation}\label{diff-ineq-1}
\begin{cases}
&\frac{d}{dt}f+c_1\,f^{1+\alpha}
 \leq 0, \quad \forall\,\, t>0, \\
 &f|_{t=0}=f_0
\end{cases}
\end{equation}
for two positive constants $\alpha>0$ and $c_1>0$, then there holds for any $t>0$
\begin{equation}\label{diff-ineq-2}
\begin{split}
&f(t)\leq (c_1\alpha)^{-\frac{1}{\alpha}}((c_1\alpha)^{-1}f^{-\alpha}_0+t)^{-\frac{1}{\alpha}}.
\end{split}
\end{equation}
\end{lem}
We are now in a position to complete the proof of Theorem \ref{thm-main}.

\begin{proof}[Proof of Theorem \ref{thm-main}]
Thanks to the local well-posedness theorem (Theorem \ref{thm-local}), there
exists a positive time $T$ such that the system \eqref{eqns-pert-1} with initial data $(\xi_0,\,v_0)$  has a
unique solution  $(\xi, \, v) \in \mathcal{C}([0, T); \mathfrak{F}_{s+\ell_0})$.

We define $T^{\ast}$ the maximal time of the existence to this solution, so
what we need to do in what follows is to show that $T^{\ast}=+\infty$.

Suppose, by way of contradiction, that  $T^{\ast}<+\infty$. We will show that solutions can actually be
extended past $T^{\ast}$ and satisfy $E_{s+\ell_0}(T^{\ast\ast})<+\infty$ for some $T^{\ast\ast}>T^{\ast}$, which contradicts the fact that $T^{\ast}$ is the maximal time of the existence to this solution, and then obtain that $T^{\ast}=+\infty$. For this, it suffices to prove that there is a positive constant $\kappa_0$ such that
\begin{equation}\label{bdd-soln-extend-1}
\begin{split}
&\lim_{t\nearrow T^{\ast}}E_{s}(t) \leq \kappa_0\,E_{s}(0),\quad \lim_{t\nearrow T^{\ast}}((1+t)^{-1}E_{s+\ell_0}(t)) \leq \kappa_0\,E_{s+\ell_0}(0).
 \end{split}
\end{equation}
Indeed, if \eqref{bdd-soln-extend-1} holds, then thanks to Lemma \ref{lem-est-aij-1}, $\|J-1\|_{L^\infty} \leq C_1 E_{s_0}(t)\leq \kappa_0C_1\,E_{s_0}(0)\leq  \kappa_0C_1\, \nu_0$. Take $\nu_0>0$ so small that
\begin{equation*}
\nu_0 \leq \min\{\frac{1}{2\kappa_0},\,\frac{1}{2C_1\kappa_0}\}
\end{equation*}
 we have $E_{s_0}(t) \leq \frac{1}{2}$ and $\|J-1\|_{L^\infty} \leq \frac{1}{2}$, which ensures that $J(t) \geq \frac{1}{2}$ for any $t\in [0, T^{\ast})$ and $\|J^{-1}(t)\|_{L^\infty}$ doesn't blow up. On the other hand, the second inequality in \eqref{bdd-soln-extend-1} implies $\lim_{t\nearrow T^{\ast}}E_{s+\ell_0}(t)\leq \kappa_0(1+T^{\ast})E_{s+\ell_0}(0)$. Hence, the equality \eqref{loc-blow-crit} in Theorem \ref{thm-local}  doesn't hold true, which contradicts the assumption $T^{\ast}<+\infty$, and then we can
extend the solution past $T^{\ast}$ to some $T^{\ast\ast}>T^{\ast}$.

For $\kappa>0$, set
\begin{equation}\label{equal-time-1}
\begin{split}
T_1&=T_1(\kappa)\eqdefa\sup\bigg\{T\in [0, T^{\ast}]|(\xi,\,v) \in \mathcal{C}([0, T); \mathfrak{F}_{s+\ell_0}), \quad\,\mathcal{E}_{s+\ell_0}(t) \leq 2\mathfrak{C}_0^2\mathcal{E}_{s+\ell_0}(0),\\
&E_{s}(t) \leq \kappa\,(E_{s}(0)+\mathcal{E}_{s+\ell_0}(0)), \,E_{s+\ell_0}(t) \leq \kappa\,(1+t)E_{s+\ell_0}(0)\quad \forall \,t \in [0, T)\bigg\}.
 \end{split}
\end{equation}
In order to prove \eqref{bdd-soln-extend-1}, we need only to show that, there exists a universal positive constant $\kappa_0$ such that
\begin{equation}\label{equal-time-T1-a}
\begin{split}
T^{\ast}=T_1(\kappa_0).
 \end{split}
\end{equation}

Since the inequality $T_1(\kappa)\leq T^{\ast}$ always holds true for any $\kappa>0$ by the definition of $T_1(\kappa)$ in \eqref{equal-time-1}, we will employ a bootstrap argument to prove \eqref{equal-time-T1-a}: find a constant $\kappa_0 >0$ such that, under the conditions in \eqref{equal-time-1} (with $\kappa=\kappa_0$), the following inequalities hold for any $t\in [0, T_1(\kappa_0))$
 \begin{equation}\label{equal-time-boot-1}
\begin{split}
&\mathcal{E}_{s+\ell_0}(t) \leq \frac{3}{2}\mathfrak{C}_0^2\mathcal{E}_{s+\ell_0}(0),\quad E_{s}(t) \leq \frac{1}{2}\kappa_0\,(E_{s}(0)+\mathcal{E}_{s+\ell_0}(0)),\\
&E_{s+\ell_0}(t) \leq \frac{1}{2}\kappa_0\,(1+t)E_{s+\ell_0}(0).
 \end{split}
\end{equation}
This immediately follows from  the definition of $T_1(\kappa_0)$ that the equality \eqref{equal-time-T1-a} holds.

For this, we fix $\kappa_0>0$, which will be determined later, and assume that $t \in [0, T_1(\kappa_0))$ in what follows.

To prove \eqref{equal-time-boot-1}, we first employ the differential inequality \eqref{restate-decay-total-N+1-1} in Lemma \ref{lem-restate-decay-total-N+1-1} to investigate the decay estimate of the lower order energy $\dot{\mathcal{E}}_{s}$. For this, we use the definition of $T_1(\kappa_0)$ in \eqref{equal-time-1} to get
\begin{equation}\label{c-0-1a}
\mathfrak{C}_1 E_s^{\frac{1}{2}}(t) \leq  \mathfrak{C}_1\kappa_0^{\frac{1}{2}}(E_{s}(0)+\mathcal{E}_{s+\ell_0}(0))^{\frac{1}{2}} \leq  \mathfrak{C}_1\kappa_0^{\frac{1}{2}}E_{s+\ell_0}(0)^{\frac{1}{2}}\leq  \mathfrak{C}_1\kappa_0^{\frac{1}{2}}\nu_0^{\frac{1}{2}}.
 \end{equation}
Take $\nu_0>0 $ so small that
\begin{equation}\label{require-esp-2}
\nu_0 \leq \min\{\frac{1}{2C_1\kappa_0},\, \frac{1}{16\kappa_0},\,\frac{\mathfrak{c}_1^2}{4\mathfrak{C}_1^2\kappa_0}\},
\end{equation}
one can obtain from \eqref{c-0-1a} that $\mathfrak{C}_1 E_s^{\frac{1}{2}}(t) \leq \frac{1}{2}\mathfrak{c}_1$ and $E_s(t)\leq \frac{1}{2}$, which along with \eqref{restate-decay-total-N+1-1} implies
\begin{equation}\label{decay-est-N-1}
\begin{split}
&\frac{d}{dt}\widehat{\dot{\mathcal{E}}}_{s}+\mathfrak{c}_1\dot{\mathcal{D}}_{s} \leq 0.\\
\end{split}
\end{equation}

Since the dissipation $\dot{\mathcal{D}}_{s} $ can't control linearly the energy $\widehat{\dot{\mathcal{E}}}_{s}$ in \eqref{decay-est-N-1}, we turn to apply the interpolation between the dissipation $\dot{\mathcal{D}}_{s} $ and the high-order energy $\mathcal{E}_{s+\ell_0}$ to bound the energy $\widehat{\dot{\mathcal{E}}}_{s}$, which can be used to get the decay estimate of $\widehat{\dot{\mathcal{E}}}_{s}$ according to \eqref{decay-est-N-1}.

Indeed, thanks to the definitions of $\dot{\mathcal{E}}_{s}$, ${\mathcal{E}}_{s}$, and $\dot{\mathcal{D}}_{s}$, and the interpolation inequality
\begin{equation*}
\begin{split}
& \|\dot{\Lambda}_h^{\sigma_0}\xi^1\|_{L^2(\Sigma_0)}
\lesssim\|\dot{\Lambda}_h^{1+\sigma_0}\xi^1\|_{L^2(\Sigma_0)}^{\frac{2\ell_0}{1+2\ell_0}}
\|\dot{\Lambda}_h^{-\lambda}\xi^1\|_{L^2(\Sigma_0)}^{\frac{1}{1+2\ell_0}},\\
& \|\dot{\Lambda}_h^{s}\xi^1\|_{L^2(\Sigma_0)}\lesssim\|\dot{\Lambda}_h^{s-\frac{1}{2}}\xi^1
\|_{L^2(\Sigma_0)}^{\frac{2\ell_0}{1+2\ell_0}}
\|\dot{\Lambda}_h^{s+\ell_0}\xi^1\|_{L^2(\Sigma_0)}^{\frac{1}{1+2\ell_0}},
\end{split}
\end{equation*}
we have $\dot{\mathcal{E}}_{s} \leq \mathfrak{C}_2  \dot{\mathcal{D}}_{s}^{\frac{2\ell_0}{1+2\ell_0}} {\mathcal{E}}_{s+ \ell_0}^{\frac{1}{1+2\ell_0}}$, where $\mathfrak{C}_2$ is the interpolation constant independent of $\kappa_0$.

Hence, from \eqref{equiv-total-cont-1}, one obtains
\begin{equation*}
\begin{split}
\widehat{\dot{\mathcal{E}}}_{s} &\leq \mathfrak{C}_0 \dot{\mathcal{E}}_{s} \leq  \mathfrak{C}_0 \mathfrak{C}_2 \dot{\mathcal{D}}_{s}^{\frac{2\ell_0}{1+2\ell_0}} \mathcal{E}_{s+\ell_0}^{\frac{1}{1+2\ell_0}}\leq \widetilde{\mathfrak{C}}_2\, \dot{\mathcal{D}}_{s}^{\frac{2\ell_0}{1+2\ell_0}} \mathcal{E}_{s+\ell_0}(0) ^{\frac{1}{1+2\ell_0}},
\end{split}
\end{equation*}
where $\widetilde{\mathfrak{C}}_2:=\mathfrak{C}_0 \mathfrak{C}_2  (2\mathfrak{C}_0^2)^{\frac{1}{1+2\ell_0}}$,
which implies
\begin{equation}\label{decay-est-N-1a}
\begin{split}
&[\widetilde{\mathfrak{C}}_2^{-1}\widehat{\dot{\mathcal{E}}}_{s}]^{\frac{1+2\ell_0}{2\ell_0}}
\mathcal{E}_{s+\ell_0}(0)^{-\frac{1}{2\ell_0}} \leq  \dot{\mathcal{D}}_{s}.
\end{split}
\end{equation}
Combining \eqref{decay-est-N-1} with \eqref{decay-est-N-1a} yields
\begin{equation}\label{decay-est-N-1b}
\begin{split}
&\frac{d}{dt}\widehat{\dot{\mathcal{E}}}_{s}+\mathfrak{c}_2\,\mathcal{E}_{s+\ell_0}(0)^{-\frac{1}{2\ell_0}}\,
\widehat{\dot{\mathcal{E}}}_{s}^{\frac{1+2\ell_0}{2\ell_0}}
 \leq 0
\end{split}
\end{equation}
with $\mathfrak{c}_2:=\mathfrak{c}_1[\widetilde{\mathfrak{C}}_2^{-1}
]^{\frac{1+2\ell_0}{2\ell_0}}=\mathfrak{c}_1[\mathfrak{C}_0\mathfrak{C}_2
]^{-\frac{1+2\ell_0}{2\ell_0}}(2\mathfrak{C}_0^2)^{-\frac{1}{2\ell_0}}$.

Applying Lemma \ref{lem-diff-ineq-1} to \eqref{decay-est-N-1b}, where we take $c_1:=\mathfrak{c}_2\,\mathcal{E}_{s+\ell_0}(0)^{-\frac{1}{2\ell_0}}$, $\alpha:=\frac{1}{2\ell_0}$,
$f(t):=\widehat{\dot{\mathcal{E}}}_{s}(t)$, we infer that
\begin{equation*}
\begin{split}
&\widehat{\dot{\mathcal{E}}}_{s}(t)\leq (\frac{\mathfrak{c}_2}{2\ell_0})^{-2\ell_0}\mathcal{E}_{s+\ell_0}(0)\bigg(2\ell_0\mathfrak{c}_2^{-1}
(\mathcal{E}_{s+\ell_0}(0)[\widehat{\dot{\mathcal{E}}}_{s}(0)]^{-1})^{\frac{1}{2\ell_0}}
+t\bigg)^{-2\ell_0}.
\end{split}
\end{equation*}
According to the definitions of $\widehat{\dot{\mathcal{E}}}_{s}(0)$ and $\mathcal{E}_{s+\ell_0}(0)$  and \eqref{def-bdd-energy-dissi-v-1},  we know that $\widehat{\dot{\mathcal{E}}}_{s}(0) \leq C_1(s) \mathcal{E}_{s+\ell_0}(0)$ with the positive constant $C_1(s)$ depending only on $s$, so one can see that
$ 2\ell_0\mathfrak{c}_2^{-1}
(\mathcal{E}_{s+\ell_0}(0)[\widehat{\dot{\mathcal{E}}}_{s}(0)]^{-1})^{\frac{1}{2\ell_0}} \geq 2\ell_0\mathfrak{c}_2^{-1}
C_1(s)^{-\frac{1}{2\ell_0}}$,
which leads to
\begin{equation*}
\begin{split}
&\widehat{\dot{\mathcal{E}}}_{s}(t)\leq (\frac{\mathfrak{c}_2}{2\ell_0})^{-2\ell_0}\mathcal{E}_{s+\ell_0}(0)\bigg(2\ell_0\mathfrak{c}_2^{-1}
C_1(s)^{-\frac{1}{2\ell_0}}+t\bigg)^{-2\ell_0}.
\end{split}
\end{equation*}
Therefore, we obtain $\widehat{\dot{\mathcal{E}}}_{s}(t)  \leq \mathfrak{C}_3 \mathcal{E}_{s+\ell_0}(0)(1+t)^{-2\ell_0}$
with $\mathfrak{C}_3: =\mathfrak{C}_3(\ell_0, C_1(s), \mathfrak{c}_2)$,
from this, it follows
\begin{equation}\label{decay-EN-1a}
\begin{split}
&\dot{\mathcal{E}}_{s}(t) \leq \mathfrak{C}_0\mathfrak{C}_3 \mathcal{E}_{s+\ell_0}(0)(1+t)^{-2\ell_0},\quad \int_0^t\dot{\mathcal{E}}_{s}\,d\tau \leq \widetilde{\mathfrak{C}}_3\mathcal{E}_{s+\ell_0}(0)
\end{split}
\end{equation}
with $\widetilde{\mathfrak{C}}_3:=\frac{\mathfrak{C}_0\mathfrak{C}_3 }{2\ell_0-1}$.

Set $\langle\,t\,\rangle:=(1+t)$. Since
\begin{equation*}
\begin{split}
&\frac{d}{dt}(\langle\,t\,\rangle^{\ell_0+\frac{1}{2}}\widehat{\dot{\mathcal{E}}}_{s})
+\mathfrak{c}_1\langle\,t\,\rangle^{\ell_0+\frac{1}{2}}\dot{\mathcal{D}}_{s}  =\langle\,t\,\rangle^{\ell_0+\frac{1}{2}}\frac{d}{dt} \widehat{\dot{\mathcal{E}}}_{s}
+\mathfrak{c}_1\langle\,t\,\rangle^{\ell_0+\frac{1}{2}}\dot{\mathcal{D}}_{s}+ \widehat{\dot{\mathcal{E}}}_{s}\langle\,t\,\rangle^{\ell_0-\frac{1}{2}}( \ell_0+\frac{1}{2}),\\
\end{split}
\end{equation*}
from \eqref{decay-est-N-1}, we get $\frac{d}{dt}(\langle\,t\,\rangle^{\ell_0+\frac{1}{2}}\widehat{\dot{\mathcal{E}}}_{s})
+\mathfrak{c}_1\langle\,t\,\rangle^{\ell_0+\frac{1}{2}}\dot{\mathcal{D}}_{s}  \leq (\ell_0+\frac{1}{2}) \widehat{\dot{\mathcal{E}}}_{s}\langle\,t\,\rangle^{\ell_0-\frac{1}{2}}$, so one can see
\begin{equation*}
\begin{split}
&\frac{d}{dt}(\langle\,t\,\rangle^{\ell_0+\frac{1}{2}}\widehat{\dot{\mathcal{E}}}_{s})
+\mathfrak{c}_1\langle\,t\,\rangle^{\ell_0+\frac{1}{2}}\dot{\mathcal{D}}_{s}  \leq (\ell_0+\frac{1}{2}) \mathfrak{C}_3\mathcal{E}_{s+\ell_0}(0)
\langle\,t\,\rangle^{-(\ell_0+\frac{1}{2})} .
\end{split}
\end{equation*}
Hence, we arrive at
\begin{equation}\label{decay-DN-1}
\begin{split}
&\int_0^t\langle\,\tau\,\rangle^{\ell_0+\frac{1}{2}}\dot{\mathcal{D}}_{s}(\tau)\,d\tau \leq \mathfrak{C}_4 \mathcal{E}_{s+\ell_0}(0),\quad \int_0^t\dot{\mathcal{D}}_{s}(\tau)^{\frac{1}{2}}\,d\tau \leq \widetilde{\mathfrak{C}}_4 \mathcal{E}_{s+\ell_0}(0)^{\frac{1}{2}}.
\end{split}
\end{equation}

With this in hand, we are ready to estimate $E_s$.

Thanks to the equations $\partial_t\xi=v$, we find
\begin{equation*}
\begin{split}
&\|\dot{\Lambda}_h^{\sigma_0}{\Lambda}_h^{s-1-\sigma_0}\nabla\xi(t)\|_{H^1}\leq \|\dot{\Lambda}_h^{\sigma_0}{\Lambda}_h^{s-1-\sigma_0}\nabla\xi(0)\|_{H^1}+\int_0^t \|\dot{\Lambda}_h^{\sigma_0}{\Lambda}_h^{s-1-\sigma_0}\nabla\,v(\tau)\|_{H^1}\,d\tau,
\end{split}
\end{equation*}
which follows
\begin{equation}\label{def-energy-hn-2}
\begin{split}
&\|\dot{\Lambda}_h^{\sigma_0}{\Lambda}_h^{s-1-\sigma_0}\nabla\xi(t)\|_{H^1}^2\leq 2\|\dot{\Lambda}_h^{\sigma_0}{\Lambda}_h^{s-1-\sigma_0}\nabla\xi(0)\|_{H^1}^2
+2\bigg(\int_0^t\dot{\mathcal{D}}_{s}^{\frac{1}{2}}\,d\tau\bigg)^2.
\end{split}
\end{equation}

Thanks to \eqref{def-energy-hn-2} and \eqref{decay-DN-1}, we find
\begin{equation*}
\begin{split}
E_s(t)&=\|\dot{\Lambda}_h^{\sigma_0}{\Lambda}_h^{s-1-\sigma_0}\nabla\xi(t)\|_{H^1}^2+\mathcal{E}_{s}(t)\\
&\leq  2\|\dot{\Lambda}_h^{\sigma_0}{\Lambda}_h^{s-1-\sigma_0}\nabla\xi(0)\|_{H^1}^2
+2\bigg(\int_0^t\dot{\mathcal{D}}_{s}^{\frac{1}{2}}\,d\tau\bigg)^2+\mathcal{E}_{s}(t)\\
 &\leq  2\|\dot{\Lambda}_h^{\sigma_0}{\Lambda}_h^{s-1-\sigma_0}\nabla\xi(0)\|_{H^1}^2+2( \widetilde{\mathfrak{C}}_4^2+\mathfrak{C}_0^2) \mathcal{E}_{s+\ell_0}(0).
\end{split}
\end{equation*}
Hence, we obtain
\begin{equation}\label{bdd-EN-decay-1}
\begin{split}
&E_s(t)\leq  \mathfrak{C}_5(E_s(0)+\mathcal{E}_{s+\ell_0}(0)) \quad \text{with}\quad \mathfrak{C}_5:=2 (1+\widetilde{\mathfrak{C}}_4^2+\mathfrak{C}_0^2).
\end{split}
\end{equation}

Let's now turn to estimate ${E}_{s+\ell_0}$. Notice that
\begin{equation*}
\begin{split}
&\|\dot{\Lambda}_h^{s+\ell_0-1}\nabla\xi(t)\|_{H^1}\leq \|\dot{\Lambda}_h^{s+\ell_0-1}\nabla\xi(0)\|_{H^1}
+\int_0^t\dot{\mathcal{D}}_{s+\ell_0}^{\frac{1}{2}}\,d\tau,
\end{split}
\end{equation*}
which along with \eqref{bdd-EN-decay-1} and the definition of $T_1(\kappa_0)$ leads to
\begin{equation}\label{est-E-N+1-bdd-1}
\begin{split}
&{E}_{s+\ell_0}(t) \leq\|\dot{\Lambda}_h^{s+\ell_0-1}\nabla^2\xi\|_{L^2}^2+2(\mathcal{E}_{s+\ell_0}+E_{s})\\
&\leq 2\|\dot{\Lambda}_h^{s+\ell_0-1}\nabla^2\xi(0)\|_{L^2}^2+2\,t\,\int_0^t \mathcal{D}_{s+\ell_0}\,d\tau+4\mathfrak{C}_0^2 \mathcal{E}_{s+\ell_0}(0)+2\mathfrak{C}_5(E_s(0)+\mathcal{E}_{s+\ell_0}(0))\\
&\leq 2\,t\,\int_0^t \mathcal{D}_{s+\ell_0}\,d\tau+\widetilde{\mathfrak{C}}_5 E_{s+\ell_0}(0)
\end{split}
\end{equation}
with $\widetilde{\mathfrak{C}}_5:=2+4\mathfrak{C}_0^2 +2\mathfrak{C}_5$.

To deal with $\int_0^t \mathcal{D}_{s+\ell_0}\,d\tau$ in \eqref{est-E-N+1-bdd-1}, we first apply the energy inequality \eqref{restate-bdddecay-total-N+1-1} to find
\begin{equation*}
\begin{split}
\frac{d}{dt}{\widehat{\mathcal{E}}}_{s+\ell_0}+\mathfrak{c}_1{\mathcal{D}}_{s+\ell_0} \leq \mathfrak{C}_1 \bigg(E_{s+\ell_0}\,\dot{\mathcal{D}}_{s_0}+ E_{s_0} (\dot{\mathcal{D}}_{s_0}^{\frac{1}{2}} +\dot{\mathcal{E}}_{s_0})\bigg),
\end{split}
\end{equation*}
which along with \eqref{est-E-N+1-bdd-1} gives rise to
\begin{equation}\label{diff-ineq-gron-EN+1-1}
\begin{split}
\frac{d}{dt}({\widehat{\mathcal{E}}}_{s+\ell_0}+\mathfrak{c}_1\int_0^t \mathcal{D}_{s+\ell_0}\,d\tau) \leq &  2\mathfrak{C}_1\,t\,\dot{\mathcal{D}}_{s_0}\,\int_0^t \mathcal{D}_{s+\ell_0}\,d\tau\\
&+ 2\mathfrak{C}_1 \mathfrak{C}_0^2\mathcal{E}_{s+\ell_0}(0)(\dot{\mathcal{D}}_{s_0}^{\frac{1}{2}} +\dot{\mathcal{E}}_{s_0})+ \mathfrak{C}_1  \widetilde{\mathfrak{C}}_5 E_{s+\ell_0}(0)\,\dot{\mathcal{D}}_{s_0}.
\end{split}
\end{equation}
Applying the Gronwall inequality to \eqref{diff-ineq-gron-EN+1-1} yields
\begin{equation*}
\begin{split}
& {\widehat{\mathcal{E}}}_{s+\ell_0}(t)+\mathfrak{c}_1\int_0^t \mathcal{D}_{s+\ell_0}\,d\tau  \leq \exp\{2\mathfrak{C}_1\mathfrak{c}_1^{-1}\,\int_0^t\tau\,\dot{\mathcal{D}}_{s_0}\,d\tau\} \\
&\times \bigg[{\widehat{\mathcal{E}}}_{s+\ell_0}(0)+  2\mathfrak{C}_1 \mathfrak{C}_0^2\mathcal{E}_{s+\ell_0}(0)\int_0^t (\dot{\mathcal{D}}_{s_0}^{\frac{1}{2}} +\dot{\mathcal{E}}_{s_0})\,d\tau+ \mathfrak{C}_1\widetilde{\mathfrak{C}}_5 E_{s+\ell_0}(0)\,\int_0^t \dot{\mathcal{D}}_{s_0}\,d\tau\bigg],
\end{split}
\end{equation*}
which along with \eqref{decay-DN-1} implies
\begin{equation*}
\begin{split}
& {\widehat{\mathcal{E}}}_{s+\ell_0}(t)+\mathfrak{c}_1\int_0^t \mathcal{D}_{s+\ell_0}\,d\tau  \leq \exp\{2\mathfrak{C}_1\mathfrak{c}_1^{-1}\,\mathfrak{C}_4 \mathcal{E}_{s+\ell_0}(0)\} \\
&\times \bigg[{\widehat{\mathcal{E}}}_{s+\ell_0}(0)+  2\mathfrak{C}_1 \mathfrak{C}_0^2\mathcal{E}_{s+\ell_0}(0)\, (\widetilde{\mathfrak{C}}_4 \mathcal{E}_{s+\ell_0}(0)^{\frac{1}{2}}+\widetilde{\mathfrak{C}}_3 \mathcal{E}_{s+\ell_0}(0))+ \mathfrak{C}_1  \widetilde{\mathfrak{C}}_5 E_{s+\ell_0}(0)\,\mathfrak{C}_4\mathcal{E}_{s+\ell_0}(0)\bigg].
\end{split}
\end{equation*}
Therefore, we arrive at
\begin{equation}\label{E-N+1-est-bdd-1a}
\begin{split}
& {\widehat{\mathcal{E}}}_{s+\ell_0}(t)+\mathfrak{c}_1\int_0^t \mathcal{D}_{s+\ell_0}\,d\tau\leq \exp\{2\mathfrak{C}_1\mathfrak{c}_1^{-1}\,\mathfrak{C}_4 \nu_0\}  \widehat{\mathcal{E}}_{s+\ell_0}(0)(1+ \mathfrak{C}_6\nu_0^{\frac{1}{2}})
\end{split}
\end{equation}
with $\mathfrak{C}_6:=2\mathfrak{C}_1 \mathfrak{C}_0^3 (\widetilde{\mathfrak{C}}_4 +\widetilde{\mathfrak{C}}_3)+ \mathfrak{C}_1  \widetilde{\mathfrak{C}}_5 \mathfrak{C}_4\mathfrak{C}_0 $.

Take $\nu_0>0 $ so small that
\begin{equation*}
\nu_0 \leq \min\{\frac{1}{2C_1\kappa_0},\, \frac{1}{16\kappa_0},\,\frac{\mathfrak{c}_1^2}{4\mathfrak{C}_1^2\kappa_0},\,\frac{\mathfrak{c}_1}{2\mathfrak{C}_1\,\mathfrak{C}_4 }\log(\frac{5}{4}),\,\frac{1}{25\mathfrak{C}_6^2}\},
\end{equation*}
from \eqref{E-N+1-est-bdd-1a}, we obtain $\sup_{\tau\in [0, t]}\widehat{\mathcal{E}}_{s+\ell_0}(\tau) +\mathfrak{c}_1\int_0^t\mathcal{D}_{s+\ell_0}(\tau)\,d\tau \leq \frac{3}{2}\widehat{\mathcal{E}}_{s+\ell_0}(0)$, which follows
\begin{equation}\label{E-N+1-est-bdd-3}
\begin{split}
&\mathcal{E}_{s+\ell_0}(t) \leq \frac{3}{2}\mathfrak{C}_0^2 \mathcal{E}_{s+\ell_0}(0) \quad\text{and}\quad \int_0^t\mathcal{D}_{s+\ell_0}(\tau)\,d\tau \leq \frac{3}{2\mathfrak{c}_1}\mathfrak{C}_0 \mathcal{E}_{s+\ell_0}(0).
\end{split}
\end{equation}
Consequently, owing to \eqref{est-E-N+1-bdd-1} and \eqref{E-N+1-est-bdd-3}, one can see
\begin{equation*}\label{E-N+1-est-bdd-4a}
\begin{split}
{E}_{s+\ell_0}(t) &\leq 2\,t\,\int_0^t \mathcal{D}_{s+\ell_0}\,d\tau+\mathfrak{C}_5 E_{s+\ell_0}(0)\leq  t\,\frac{3}{\mathfrak{c}_1}\mathfrak{C}_0^2 \mathcal{E}_{s+\ell_0}(0)+\mathfrak{C}_5 E_{s+\ell_0}(0),
\end{split}
\end{equation*}
so we arrive at
\begin{equation}\label{E-N+1-est-bdd-4}
\begin{split}
{E}_{s+\ell_0}(t) \leq  \widetilde{\mathfrak{C}}_6\,(t+1)\,E_{s+\ell_0}(0)
\end{split}
\end{equation}
with $\widetilde{\mathfrak{C}}_6:=\frac{3}{\mathfrak{c}_1}\mathfrak{C}_0^2 +\mathfrak{C}_5$.

Observe that all the constants $\mathfrak{C}_i,\, \widetilde{\mathfrak{C}}_j$ (with $i=0, 1, ..., 6$, $j=1,...,6$) above are independent of $\kappa_0$. We take the constant $\kappa_0:=2(\mathfrak{C}_5+\widetilde{\mathfrak{C}}_6)$ in \eqref{equal-time-1} and
\begin{equation*}
\nu_0:=\frac{1}{2}\min\{\frac{1}{2C_1\kappa_0},\,\frac{1}{16\kappa_0},\,
\frac{\mathfrak{c}_1^2}{4\mathfrak{C}_1^2\kappa_0},\,\frac{\mathfrak{c}_1}{2\mathfrak{C}_1\,\mathfrak{C}_4 }\log(\frac{5}{4}),\,\frac{1}{25\mathfrak{C}_6^2}\},
 \end{equation*}
 then for any $t \in [0, T_1(\kappa_0))$, combining \eqref{bdd-EN-decay-1}, \eqref{E-N+1-est-bdd-3}, and \eqref{E-N+1-est-bdd-4} ensure that \eqref{equal-time-boot-1} holds true.

Thanks to \eqref{decay-EN-1a} and \eqref{equal-time-boot-1}, we get the inequality \eqref{totalenerg-high-thm-1}. This ends the proof of Theorem \ref{thm-main}.
\end{proof}

\bigbreak \noindent {\bf Acknowledgments.} G. Gui's research is supported by NSFC under Grant 11571279. Z. Zhang is partially supported by NSFC under Grant 12171010.


\begin{thebibliography}{50}


     \bibitem{Abels-2005} H. Abels, The initial-value problem for the Navier-Stokes equations with a free surface in $L^q$-Sobolev spaces. {\it Adv. Differential Equations}, {\bf 10} (2005), 45--64.

     \bibitem{Agmon-D-N-1964}S. Agmon, A. Douglis, and L. Nirenberg, Estimates near the boundary for solutions of ellipic partial differential equations satisfying general boundary conditions II, {\it Comm. Pure Appl. Math.}, {\bf XVII} (1964), 35--92.



   \bibitem{Alinhac-1986} S. Alinhac, Paracomposition et op\'{e}rateurs paradiff\'{e}rentiels, {\it Communications in Partial Differential Equations}, {\bf 11} (1986), 87--121.

\bibitem{BCD}H. Bahouri, J.~Y. Chemin, and R. Danchin, {\it Fourier analysis and nonlinear partial differential equations}, Grundlehren der mathematischen Wissenschaften 343, Springer-Verlag Berlin Heidelberg, 2011.

    \bibitem{Beale-1981} J. T. Beale, The initial value problem for the Navier-Stokes equations with a free surface, {\it Comm. Pure Appl. Math.}, {\bf 34} (3) (1981), 359--392.



 \bibitem{Beale-1984} J. T. Beale, Large-time regularity of viscous surface waves, {\it Arch Rational. Mech. Anal.},{\bf 84} (1984), 307--352.

   \bibitem{Beale-Nishida-1985} J. T. Beale and T. Nishida, Large-time behavior of viscous surface waves, {\it Recent topics in nonlinear PDE}, II (Sendai, 1984), 1--14, North-Holland Math. Stud., 128, North-Holland, Amsterdam, 1985.

       \bibitem{DM-1990} B. Dacorogna and J. Moser, On a partial differential equation involving the Jacobian determinant, Ann. Inst. Henri Poincar\'{e}, {\bf 7} (1) (1990), 1--26.

    \bibitem{Gilbarg-Trudinger-1983} D. Gilbarg and N. S. Trudinger, Elliptic partial differential equations of second order, {\it Grundlehren}, Vol. 224, Springer-Verlag, Berlin, 1983.

 \bibitem{Gui-WW-2019} G. Gui, C. Wang, and Y. Wang, Local well-posedness of the vacuum free boundary of 3-D compressible Navier-Stokes equations, {\it Calculus of Variations and PDE}, {\bf 58} (2019):166.

   \bibitem{Gui-2020} G. Gui, Lagrangian approach to global well-posedness of the viscous surface wave equations without surface tension, {\it Peking Mathematical Journal}, {\bf 4} (2021), 1--82.

    \bibitem{GuiL-2022} G. Gui and Y. Li, Well-posedness of the viscous surface wave equations in anisotropic Sobolev spaces, 2022, preprint.

\bibitem{Guo-Tice-2010} Y. Guo and I. Tice, Linear Rayleigh-Taylor instability for viscous, compressible fluids, {\it SIAM J. Math. Anal.}, {\bf 42}(4) (2010), 1688-1720.

 \bibitem{Guo-Tice-1} Y. Guo and I. Tice, Local well-posedness of the viscous surface wave problem without surface tension, {\it Analysis and PDE}, {\bf 6} (2013), 287--369.

 \bibitem{Guo-Tice-2} Y. Guo and I. Tice, Decay of viscous surface waves without surface tension in horizontally infinite domains, {\it Analysis and PDE}, {\bf 6} (2013), 1429--1533.

       \bibitem{Hataya-2011} Y. Hataya, A remark on Beale-Nishida¡¯s paper, {\it Bulletin of the Institute of Mathematics, Academia Sinica (New Series)}, {\bf 6} (2011), 293--303.

 \bibitem{Heywood-1980} J.G. Heywood, The Navier-Stokes Equations: On the existence, regularity and decay of Solutions, {\it Indiana Univ. Math. J.}, {\bf 29} (1980), 639--681.

  \bibitem{HLuo-2021} Y. Huang and T. Luo, Compressible viscous heat-conducting surface wave without surface tension, {\it J. Math. Phys.}, {\bf 62} (2021), 061501.

    \bibitem{WTK-2016} J. Jang, I. Tice, and Y. Wang, ¡°The compressible viscous surface-internal wave problem: Local well-posedness, {\it SIAM J. Math. Anal.}, {\bf 48}(4) (2016), 2602--2673.

  \bibitem{WTK-2014} J. Jang, I. Tice, and Y. Wang, The compressible viscous surface-internal wave problem: stability and vanishing surface tension limit, {\it Commun. Math. Phys.}, {\bf 343} (2016), 1039--1113.




 \bibitem{Jin-2003}  B.J. Jin, Existence of viscous compressible barotropic flow in a moving domain with free upper surface via Galerkin method, {\it Ann. Univ. Ferrara Sez.} VII (N.S.), {\bf 49} (2003), 43-71.

 \bibitem{Jin-Padula-2002} B.J. Jin and M. Padula, On the existence of compressible viscous flows in a horizontal layer with free upper surface, {\it Commun. Pure Appl. Anal.}, {\bf 1}(3) (2002), 379-415.

\bibitem{Masm-Rou-2017} N. Masmoudi and F. Rousset, Uniform regularity and vanishing viscosity limit for the free surface Navier-Stokes equations, {\it Arch Rational. Mech. Anal.}, {\bf 223} (2017), 301-417.

 \bibitem{Rayleigh-1883} L. Rayleigh, Analytic solutions of the Rayleigh equation for linear density profiles, {\it Proc. Lond. Math. Soc.}, {\bf 14} (1883), 170-177.



  \bibitem{Ren-X-Zhang-2019} X. Ren, Z. Xiang, and Z. Zhang, Low regularity well-posedness for the viscous surface wave equation, {\it Science China Mathematics}, {\bf 62} (2019), 1887--1924.


 \bibitem{Solonnikov-1977} V. A. Solonnikov, Solvability of the problem of the motion of a viscous incompressible
fluid that is bounded by a free surface, {\it Izv Akad Nauk SSSR Ser. Mat.}, {\bf 41} (1977), 1388--1424.

 \bibitem{Strain-Guo-2006} R. M. Strain and Y. Guo, Almost exponential decay near Maxwellian, {\it Commun. Partial Differ. Equ.}, {\bf 31} (2006), 417-429.

 \bibitem{Sylvester-1990} D. L. G. Sylvester, Large time existence of small viscous surface waves without surface tension, {\it Comm. Partial Differential Equations}, {\bf 5} (1990), 823--903.



 \bibitem{Tani-Tanaka-1995} A. Tani and N. Tanaka, Large-time existence of surface waves in incompressible viscous fuids with or without surface tension, {\it Arch. Rational Mech. Anal.}, {\bf 130} (1995), 303--314.

 \bibitem{TaniT-2003} A. Tani and N. Tanaka, Surface waves for a compressible viscous fluid, {\it J. Math. Fluid Mech.}, {\bf 5} (4) (2003), 303--363.


\bibitem{Taylor-1950} G. I. Taylor, The instability of liquid surfaces when accelerated in a direction perpendicular to their planes, {\it Proc. R. Soc. Lond.} Ser. A., {\bf 201} (1950), 192-196.

     \bibitem{Taylor-2000} M. Taylor, {\it Tools for PDE}, Math. Surveys Monogr. 81, AMS, 2000.



  \bibitem{Wang-2020} Y. Wang,  Anisotropic decay and global well-posedness of viscous surface waves without surface tension, {\it Advances in Mathematics} {\bf 374} (18) (2020), 107330.

       \bibitem{Wu-2014} L. Wu, Well-posedness and decay of the viscous surface wave, {\it SIAM J. Math. Anal.}, {\bf 46} (2014), 2084-2135.


     \bibitem{Zadrzynska2001}E. Zadrzy\'nska, {\it Evolution free boundary problem for equations of viscous compressible heat-conducting capillary fluids}, Math. Methods Appl.  Sci. {\bf 24} (2001), 713-743.


\bibitem{Zajaczkowski1999}E.Zadrzy\'nska and W.M. Zaj\c{a}czkowski, {\it On nonstationary motion of a fixed mass of a viscous compressible barotropic fluid bounded by a free boundary}, Colloq. Math. {\bf 79} (1999), 283-310.







\end{thebibliography}
\end{document}